\documentclass[12pt]{amsart}%

\usepackage{amsmath,amssymb,enumitem,verbatim,stmaryrd,xcolor,microtype,graphicx,aliascnt,fancyvrb}

\usepackage[T1]{fontenc}
\usepackage[utf8]{inputenc}
\usepackage[english]{babel} 
\usepackage[top=3.2cm,bottom=3.2cm,left=3.2cm,right=3.2cm]{geometry}
\usepackage[bookmarksopen=true,bookmarksopenlevel=2,bookmarksdepth=2,linktoc=all,colorlinks,linkcolor={red!80!black},citecolor={red!80!black},urlcolor={blue!80!black}]{hyperref}
\usepackage{tikz}\usetikzlibrary{matrix,arrows,decorations.markings,shapes,calc,intersections,through}
\usepackage{anyfontsize}

\usepackage{newtxtext,newtxmath}                              
\usepackage[mathcal]{euscript}                               
\DeclareSymbolFont{sfoperators}{OT1}{bch}{m}{n} \DeclareSymbolFontAlphabet{\mathsf}{sfoperators} \makeatletter\def\operator@font{\mathgroup\symsfoperators}\makeatother 
\usepackage[colorinlistoftodos,bordercolor=orange,backgroundcolor=orange!20,linecolor=orange,textsize=footnotesize]{todonotes} \makeatletter \providecommand \@dotsep{5} \def\listtodoname{List of Todos} \def\listoftodos{\@starttoc{tdo}\listtodoname} \makeatother 

\usepackage{etoolbox}
\makeatletter
\patchcmd{\@startsection}{\@afterindenttrue}{\@afterindentfalse}{}{}             
\patchcmd{\part}{\bfseries}{\bfseries\LARGE}{}{}
\patchcmd{\section}{\scshape}{\bfseries}{}{}\renewcommand{\@secnumfont}{\bfseries} 
\patchcmd{\@settitle}{\uppercasenonmath\@title}{\large}{}{}
\patchcmd{\@setauthors}{\MakeUppercase}{}{}{}
\addto{\captionsenglish}{} 
\addto{\captionsenglish}{} 
\addto{\captionsenglish}{} 
\makeatother
\allowdisplaybreaks[1]                               

\usepackage{fancyhdr}

\addto\extrasenglish{} 

\addto\extrasenglish{}  
\newtheorem{thm}{Theorem}[section] 
\newaliascnt{lemma}{thm}\newtheorem{lemma}[lemma]{Lemma}\aliascntresetthe{lemma}
\newaliascnt{cor}{thm}\newtheorem{cor}[cor]{Corollary}\aliascntresetthe{cor}
\newaliascnt{prop}{thm}\newtheorem{prop}[prop]{Proposition}\aliascntresetthe{prop}
\newaliascnt{conj}{thm}\aliascntresetthe{conj}
\newaliascnt{problem}{thm}\newtheorem{problem}[problem]{Problem}\aliascntresetthe{problem}
\newaliascnt{question}{thm}\aliascntresetthe{question}
\newaliascnt{propA}{thmA}\aliascntresetthe{propA}

\newtheorem*{thm*}{Theorem}
\newtheorem*{lem*}{Lemma}
\newtheorem*{cor*}{Corollary}
\newtheorem*{problem*}{Problem}

\theoremstyle{definition}
\newaliascnt{df}{thm}\newtheorem{df}[df]{Definition}\aliascntresetthe{df}
\newaliascnt{rem}{thm}\newtheorem{rem}[rem]{Remark}\aliascntresetthe{rem}
\newaliascnt{ex}{thm}\newtheorem{ex}[ex]{Example}\aliascntresetthe{ex}

\newtheorem*{df*}{Definition}
\newtheorem*{ex*}{Example}
\newtheorem*{rem*}{Remark}

\numberwithin{equation}{section}
\newcommand{\nocontentsline}[3]{}
\newcommand{\tocless}[2]{\bgroup\let\addcontentsline=\nocontentsline#1{#2}\egroup}

\DeclareMathOperator{\id}{id}

\DeclareMathOperator{\Hom}{Hom}
\DeclareMathOperator{\Maps}{Maps}
\DeclareMathOperator{\Aut}{Aut}

\DeclareMathOperator{\rk}{rk}
\DeclareMathOperator{\cork}{cork}
\DeclareMathOperator{\colim}{colim}
\DeclareMathOperator{\upR}{R}

\DeclareMathOperator{\Dr}{Dr}

\DeclareMathOperator{\Emb}{{Emb}}
\DeclareMathOperator{\USL}{{USL}}
\DeclareMathOperator{\Sym}{{Sym}}
\DeclareMathOperator{\sign}{{sign}}
\DeclareMathOperator{\ord}{{ord}}
\DeclareMathOperator{\Pastures}{{Pastures}}
\DeclareMathOperator{\Fields}{{Fields}}
\DeclareMathOperator{\Sets}{{Sets}}

\newcommand\D{{\mathbb D}}

\newcommand\F{{\mathbb F}}

\renewcommand\H{{\mathbb H}}

\newcommand\K{{\mathbb K}}

\newcommand\Q{{\mathbb Q}}
\newcommand\R{{\mathbb R}}
\renewcommand\S{{\mathbb S}}
\newcommand\T{{\mathbb T}}
\newcommand\U{{\mathbb U}}
\newcommand\V{{\mathbb V}}

\newcommand\Z{{\mathbb Z}}

\newcommand\cC{{\mathcal C}}

\newcommand\cF{{\mathcal F}}

\newcommand\cH{{\mathcal H}}

\newcommand\cL{{\mathcal L}}

\newcommand\cS{{\mathcal S}}

\newcommand\cX{{\mathcal X}}

\renewcommand\geq{\geqslant}
\renewcommand\leq{\leqslant}
\newcommand{\univ}{\textup{univ}}
\newcommand{\Funpm}{{\F_1^\pm}}
\newcommand{\gen}[1]{\langle #1 \rangle}

\newcommand{\pastgen}[2]{#1 \!\sslash\!\gen{#2}}
\newcommand{\cross}[5]{\mathchoice{\scalebox{1.3}{$\big[$}\,\raisebox{1pt}{$\begin{matrix}{\scalebox{0.9}{$#1$}}\hspace{-5pt}&{\scalebox{0.9}{$#2$}}\\[-2pt]{\scalebox{0.9}{$#3$}}\hspace{-5pt}&{{\scalebox{0.9}{$#4$}}}\end{matrix}$}\,\scalebox{1.3}{$\big]_{#5}$}}{\big[\begin{smallmatrix}{#1}&{#2}\\{#3}&{#4}\end{smallmatrix}\big]_{#5}}{}{}}   
\newcommand{\crossinv}[5]{\mathchoice{\scalebox{1.3}{$\big[$}\,\raisebox{1pt}{$\begin{matrix}{\scalebox{0.9}{$#1$}}\hspace{-5pt}&{\scalebox{0.9}{$#2$}}\\[-2pt]{\scalebox{0.9}{$#3$}}\hspace{-5pt}&{{\scalebox{0.9}{$#4$}}}\end{matrix}$}\,\scalebox{1.3}{$\big]_{#5}^{-1}$}}{\big[\begin{smallmatrix}{#1}&{#2}\\{#3}&{#4}\end{smallmatrix}\big]_{#5}^{-1}}{}{}}   
\newcommand{\tripleratio}[7]{\mathchoice{\scalebox{1.3}{$\big[$}\,\raisebox{1pt}{$\begin{matrix}{\scalebox{0.9}{$#1$}}\hspace{-5pt}&{\scalebox{0.9}{$#2$}}\hspace{-5pt}&{\scalebox{0.9}{$#3$}}\\[-2pt]{\scalebox{0.9}{$#4$}}\hspace{-5pt}&{\scalebox{0.9}{$#5$}}\hspace{-5pt}&{{\scalebox{0.9}{$#6$}}}\end{matrix}$}\,\scalebox{1.3}{$\big]_{#7}$}}{\big[\begin{smallmatrix}{#1}&{#2}&{#3}\\{#4}&{#5}&{#6}\end{smallmatrix}\big]_{#7}}{}{}}   
\newcommand{\minus}{\backslash}
\newcommand{\minor}[2]{\backslash #1 / #2}
\renewcommand\emptyset\varnothing

\title{A modern perspective on Tutte's homotopy theorem}

\author{Matthew Baker}
\address{\rm Matthew Baker, School of Mathematics, Georgia Institute of Technology, Atlanta, USA}
\email{mbaker@math.gatech.edu}

\author{Tong Jin}
\address{\rm Tong Jin, Department of Mathematics, Vanderbilt University, Nashville, USA}
\email{tong.jin@vanderbilt.edu}

\address{\rm Ju\v s Kocutar, TU Berlin, Germany}
\email{j.kocutar@campus.tu-berlin.de}

\author{Oliver Lorscheid}
\address{\rm Oliver Lorscheid, University of Groningen, the Netherlands}
\email{o.lorscheid@rug.nl}

\hypersetup{pdftitle={A modern perspective on Tutte's homotopy theorem},pdfauthor={Matthew Baker, Tong Jin, and Oliver Lorscheid}}

\begin{document}

\begin{abstract}
We begin with a review of Tutte's homotopy theory, which concerns the structure of a certain graph associated to a matroid (together with some extra data). Concretely, Tutte's path theorem asserts that this graph is connected, and his homotopy theorem asserts that every cycle in the graph is a composition of ``elementary cycles'', which come in four different flavors. We present an extended version of the homotopy theorem, in which we give a more refined classification of the different types of elementary cycles. We explain in detail how the path theorem allows one to prove that the foundation of a matroid (in the sense of Baker--Lorscheid) is generated by universal cross-ratios, and how the extended homotopy theorem allows one to classify all algebraic relations between universal cross-ratios. The resulting ``fundamental presentation'' of the foundation was previously established in \cite{Baker-Lorscheid20}, but the argument here is more self-contained. We then recall a few applications of the fundamental presentation to the representation theory of matroids. Finally, in the most novel but also the most speculative part of the paper, we discuss what a ``higher Tutte homotopy theorem'' might look like, and we present some preliminary computations along these lines.

The appendices include modern language translations of Tutte's proofs of the path theorem and the homotopy theorem (the latter is a contribution by Ju\v s Kocutar).
\end{abstract}

\maketitle

\begin{center}
 \textit{Dedicated to Andreas Dress.}
\end{center}

\begin{small}
 \tableofcontents
\end{small}


\section*{Introduction}
\label{section: introduction}

Of William Tutte's many seminal contributions to matroid theory, his ``homotopy theory'' is undoubtedly one of the deepest, and yet it remains relatively poorly understood. One reason is that Tutte's papers are written using rather archaic terminology; another is that the most important modern books and surveys do not discuss this aspect of Tutte's work. The omission of Tutte's homotopy theory from \cite{Oxley11}, for example, is presumably due to the fact that it is rather complicated to state and prove, and simpler proofs were subsequently found for the main applications that Tutte originally had in mind (e.g., his excluded minor characterization of regular matroids). Nevertheless, as we shall argue in this paper, Tutte's homotopy theory remains an important result; furthermore, we believe that it may be the tip of an iceberg, in the sense that there may very well be ``higher homotopy theorems'' still awaiting discovery. 

\medskip

The ``modern'' perspective on Tutte's homotopy theory mentioned in the title of our paper is, in a nutshell, that it allows us to write down generators and relations for the {\em foundation} of a matroid (with the generators in question being {\em canonical}). The foundation $F_M$ of a matroid $M$ is a {\em pasture} (an algebraic structure generalizing fields) possessing the universal property that rescaling classes of representations of $M$ over a field $F$ (or, more generally, over any pasture $P$) are naturally in one-to-one correspondence with homomorphisms from $F_M$ to $F$ (resp. $P$). If $F_M$ is given to us in terms of generators and relations, then computing the set of homomorphisms from $F_M$ to some pasture $P$ becomes a manageable task, at least when $P$ has a simple structure.\footnote{In general --- for example, if $P=\Q$ is the field of rational numbers --- computing $\Hom(F_M,\Q)$ might be very difficult, perhaps even algorithmically unsolvable, even if we know explicit generators and relations for $F_M$, thanks to Mn\"ev's universality theorem and the conjectured negative answer to Hilbert's 10th problem over $\Q$; see \cite{Sturmfels87,Mnev88}.} This perspective on Tutte's homotopy theory has its origins in the pioneering work of Andreas Dress and his student Walter Wenzel \cite{Dress-Wenzel89,Dress-Wenzel90,Wenzel89,Wenzel91}, in which they demonstrated for the first time that Tutte's homotopy theory could be ``encoded'', loosely speaking, in an algebraic way. More precisely, Dress and Wenzel introduced finitely generated abelian groups associated to $M$, which they called the {\em Tutte group} and {\em inner Tutte group} of $M$, respectively, and they showed that Tutte's homotopy theory could be used to better understand these groups. Their work was further extended in the paper of Gelfand, Rybnikov, and Stone \cite{Gelfand-Rybnikov-Stone95}, which was the direct inspiration for our own approach to Tutte's homotopy theory and its applications.\footnote{We owe a debt of gratitude here to Rudi Pendavingh, who first suggested to us that the results of \cite{Gelfand-Rybnikov-Stone95} might be clarified and enhanced by reformulating their multiplicative relations between cross-ratios in terms of the foundation of certain special embedded minors. This suggestion ended up being very fruitful, as the present paper hopefully makes clear.}

\medskip

More precisely, Tutte's path theorem asserts that a certain graph $G = G_{M,\Gamma}$, associated to a pair consisting of a matroid $M$ on $E$ and a modular cut\footnote{For the purposes of this introduction, it suffices to note that modular cuts are in one-to-one correspondence with the single-element extensions of $M$, i.e., matroids $\hat{M}$ together with a non-coloop $a$ such that $\hat{M} \backslash a = M$.} $\Gamma$ in the lattice of flats of $M$, is connected if the matroid $M$ is connected. In the simplest case where $\Gamma = \{ E \}$ is the trivial modular cut, the graph in question has vertices corresponding to hyperplanes of $M$ and edges corresponding to modular pairs of hyperplanes whose union is not all of $E$. The connectivity of $G$, which is equivalent to the statement that $H_0(G,\Z)$ is one-dimensional, is the most non-trivial ingredient needed to show that the foundation of $M$ is generated by universal cross-ratios. Similarly, Tutte's homotopy theorem asserts that every cycle in $G$ can be decomposed into ``elementary cycles'' belonging to one of four different kinds. As already noted in \cite{Bjorner95}, this result can be reformulated as the vanishing of $H_1(\Sigma,\Z)$ for a certain two-dimensional cell complex $\Sigma$ whose 1-skeleton is $G$. This result, in turn, is the most non-trivial ingredient needed to classify all relations between the universal cross-ratios which hold in the foundation of $M$.

\subsection*{Structure of the paper}

The present paper is partly expository, but also contains some previously unpublished material. Our first goal is to provide a modern formulation of Tutte's two main theorems from \cite{Tutte58a}: the {\em path theorem} (\autoref{thm: Tutte's path theorem}) and the {\em homotopy theorem} (\autoref{thm: Tutte's homotopy theorem}). By ``modern'', we simply mean in this context that we formulate these results in terms of the {\em lattice of flats} of a matroid $M$, rather than as Tutte did, in terms of the lattice of unions of circuits of $M$ (a less familiar object to modern readers).\footnote{Tutte's lattice coincides with the opposite lattice of the dual matroid $M^*$, so in principle there is no difficulty translating between the two perspectives, but in practice it can be a challenge.} The translation between the two points of view has already been discussed in \cite{Kung86}, but we recall it here for the reader's convenience. 

\medskip

We provide a complete, self-contained proof in this paper of Tutte's path theorem, relegating certain technical lemmas about indecomposable flats to \autoref{appendix:flats}. And in \autoref{appendix: proof of the homotopy theorem} we give a complete, self-contained proof of Tutte's homotopy theorem using modern terminology.  Since the proof of the homotopy theorem is rather elaborate, we provide an outline of the structure of the proof in \autoref{sec:tutte-homotopy-outline}. 

\medskip

We also provide a generalization of Tutte's homotopy theorem which we call the {\em extended homotopy theorem} (\autoref{thm:extended-homotopy-theorem}). While this result has not previously appeared in the literature, it is in some sense implicit in the earlier work of \cite{Gelfand-Rybnikov-Stone95} and \cite{Baker-Lorscheid20}. The extended homotopy theorem provides a more precise classification of the different cases which arise in Tutte's original homotopy theorem, and it is this extension which is directly relevant for the applications we have in mind.

\medskip

We then turn to a discussion of different presentations for the foundation of a matroid. Since our paper is designed to be accessible to readers with only a rudimentary knowledge of matroid theory, we begin by recalling the definition of a pasture and some of the basic constructions one can make with them (e.g., tensor products and free algebras). We provide a number of illustrative examples. We then define the foundation of a matroid in terms of a suitable universal property, as described above. In particular, this requires defining matroid representations over a pasture $P$ and rescaling classes thereof; this material is adapted largely from \cite{Baker-Lorscheid20}. We provide a new proof, different from the one in \cite{Baker-Lorscheid20}, that the foundation $F_M$ of a matroid $M$ exists, by giving an explicit algorithm to write down a presentation of $F_M$ in terms of certain generators and relations encoded in the ``hyperplane incidence matrix'' of $M$. This description of $F_M$ is implicit in \cite{Chen-Zhang}, but we provide many more details. It is important to note that this algorithm does {\em not} require Tutte's homotopy theory. We give several examples illustrating how to implement the algorithm, thereby explicitly calculating the foundation of important matroids such as $U_{2,4}, U_{2,5}$, and the Fano matroid $F_7$, and we observe in the process that in all of these cases the foundation is generated by ``universal cross-ratios'' (which we rigorously define). 

\medskip

The proof that universal cross-ratios generate the foundation for an arbitrary matroid $M$ (\autoref{thm: the foundation is generated by cross ratios}) requires Tutte's path theorem. The proof proceeds by induction on the size of $M$; replacing $M$ by its dual if necessary, one chooses an element $a \in M$ such that $M' := M \backslash a$ is connected, assumes that cross-ratios determine the foundation of $M'$, and uses this to ``bootstrap'' from $M'$ to $M$. The bootstrapping step involves repeatedly choosing paths between pairs of vertices in the Tutte graph $G_{M',a}$ associated to $M'$ and the modular cut corresponding to the single-element extension $M$ of $M'$; the path theorem is used to show that such paths exist.

\medskip

Previously, in \cite[Theorem 4.5]{Baker-Lorscheid21b}, we referred to the work of Wenzel \cite[Proposition 6.4]{Wenzel91} in the course of proving that universal cross-ratios generate the foundation, noting that the multiplicative group of the foundation coincides with the inner Tutte group. However, unpacking this argument requires understanding Wenzel's notation and terminology, which is different from ours. In any case, we believe that the argument given in the present paper is simpler and more intuitive than the one given in the union of \cite{Baker-Lorscheid21b} and \cite{Wenzel91}.

\medskip

We then show that Tutte's homotopy theorem, or more precisely the extended homotopy theorem, can be used to describe all relations in $F_M$ between universal cross-ratios, thereby yielding a presentation of $F_M$ which is different from, and more natural than, the presentation stemming from the hyperplane incidence matrix; cf.~\autoref{thm: fundamental presentation of the foundation}. We call this the ``fundamental presentation'' of $F_M$.\footnote{Technically speaking, we use the term {\em fundamental presentation} to refer to the statement that the foundation of $M$ is the colimit of the foundations of all special embedded minors of $M$, where ``special'' means isomorphic to $U_{2,4}, U_{2,5}, U_{3,5}, C_5, F_7$, or $F_7^*$. This implies, in a rather straightforward way, the GRS-style presentation in terms of explicit relations between cross-ratios, as formalized in \autoref{thm:relations}.} The proof of this result in \cite[Theorem 4.18]{Baker-Lorscheid20} is based on results by Dress and Wenzel \cite{Dress-Wenzel89,Dress-Wenzel90} and by Gelfand, Rybnikov, and Stone \cite[Theorem 4]{Gelfand-Rybnikov-Stone95}, which are themselves intricately rooted in Tutte's path theorem and homotopy theorem \cite{Tutte58a}, respectively. Each of those papers uses different notation and terminology, making it extremely difficult to understand the full proof without a concerted effort to master all the prior literature on the subject. One of the principal contributions of the present paper, we hope, is to make the proof of the fundamental presentation more accessible, understandable, and self-contained. In the process, we highlight the role of the extended homotopy theorem, which was merely implicit in the prior literature.

\medskip

Like the proof of \autoref{thm: the foundation is generated by cross ratios}, the proof of \autoref{thm: fundamental presentation of the foundation} proceeds by induction on the size of $M$; again, replacing $M$ by its dual if necessary, one chooses an element $a \in M$ such that $M' := M \backslash a$ is connected. We then assume that we are given a $P$-representation $\varphi'$ of $M'$ (for some pasture $P$) and a collection of cross-ratios satisfying certain necessary relations, and our task is to prove that there exists a representation $\varphi$ of $M$ extending $\varphi'$ and having the desired cross-ratios. Again, the bootstrapping step involves choosing paths between pairs of vertices in the Tutte graph $G_{M',a}$ in order to define the extension $\varphi$ in terms of the given cross-ratios, but this time instead of just knowing that a path exists, it is necessary to show that choosing different paths leads to the same result. Our extended version of Tutte's homotopy theorem reduces this verification to the case where the difference between the two paths is an elementary cycle whose corresponding configuration of hyperplanes ``comes from'' a special embedded minor $N$ of $M$, and this can be checked inside the foundation of $N$, where the relations between cross-ratios are exactly the relations we are given.

\medskip

Next, we give several applications of the fundamental presentation of $F_M$ to the representation theory of matroids. Most of these are taken from \cite{Baker-Lorscheid20} and \cite{Baker-Lorscheid21}. For example, we give a self-contained proof, using the fundamental presentation, of Tutte's theorem (\autoref{thm: excluded minors for regular matroids}) that a matroid is regular if and only if it has no minors of type $U_{2,4}$, $F_7$, or $F_7^\ast$. As mentioned above, this was one of Tutte's original motivations for developing his homotopy theory (though it was later reproved by Gerards \cite{Gerards89} without Tutte's homotopy theory). Tutte's original deduction of the excluded minor theorem for regular matroids from the homotopy theorem, as presented in \cite{Tutte58b}, involved elaborate casework and detailed analysis; we believe the proof presented below to be clearer and more conceptually satisfying. We also recall the statement and proof of the ``structure theorem for foundations of matroids without large uniform minors'' (\autoref{thm: structure theorem for foundations of wlum matroids}) from \cite{Baker-Lorscheid20}. This yields, as a particular consequence, the excluded minor characterization of ternary matroids, which was originally proved by Bixby and Reid \cite{Bixby79} using Tutte's homotopy theory (and later reproved by Seymour without the use of Tutte's theory \cite{Seymour79}).

\medskip

In addition to these ``classical'' applications, we also present the short, conceptual proof from \cite{Baker-Lorscheid20} of (a generalization of) the Lee--Scobee theorem~\cite{Lee-Scobee99}, which asserts that a matroid is both ternary and orientable if and only if it is dyadic. And we provide a new result, \autoref{thm: Dressians for wlum matroids}, which may be of independent interest: it classifies all possibilities, up to homeomorphism, for the {\em Dressian}\footnote{The {\em Dressian} of a matroid $M$ is a topological space whose underlying set consists of all valuated matroids with underlying matroid $M$.} of a matroid $M$ without $U_{2,5}$ or $U_{3,5}$ minors. 

\medskip

Last, but not least, in the final --- and most original --- section of the present paper, we discuss what a ``higher homotopy theorem'' might look like. As alluded to above, Tutte's path theorem asserts that a certain graph $G_{M,\Gamma}$ 
(which, for the purposes of generalization, we repackage as a 1-dimensional simplicial complex $\Sigma^1_{M,\Gamma}$) associated to a pair consisting of a connected matroid $M$ and a modular cut $\Gamma$ in $M$ is connected. We give a careful definition of a two-dimensional simplicial complex $\Sigma^2_{M,\Gamma}$ with the property that Tutte's homotopy theorem 
implies the assertion that $H_1(\Sigma^2_{M,\Gamma},\Z)=0$. We then speculate about how one might define a corresponding three-dimensional complex $\Sigma^3_{M,\Gamma}$ with the property that $H_2(\Sigma^3_{M,\Gamma},\Z)=0$, etc. The key point here is that the definition of $\Sigma^3_{M,\Gamma}$ should involve only a finite number of ``types'' of $3$-simplices, independent of the matroid $M$ (and modular cut) that we start with. We provide some preliminary computations illustrating certain types of $3$-simplices which would need to be included in the definition of $\Sigma^3_{M,\Gamma}$ in order for such a ``higher homotopy theorem'' to be true. We leave the conjectural determination of a complete list of such types of $3$-simplices, along with a proof of the resulting conjecture, to future work.

\medskip

Although we do not explore concrete potential applications of a higher version of Tutte's homotopy theory in this paper, we expect that such a theory could be quite useful. The fact that $\Sigma^1_{M,\Gamma}$ is connected implies that the foundation of $M$ is generated by universal cross-ratios, and the fact that $\Sigma^2_{M,\Gamma}$ has trivial first homology 
allows us to classify the relations between these generators. It is natural to expect that $H_2$ of the (yet to be precisely defined) complex $\Sigma^3_{M,\Gamma}$ is related to ``syzygies'' --- i.e., relations between relations --- among the universal cross-ratios. We leave it to future work to make this more precise, to rigorously establish such connections, and to deduce concrete results about matroid representations from such considerations. For now, we content ourselves with the following observation: in trying to apply our fundamental presentation for $F_M$ to concrete open problems about matroid representations (for example, the classification of quaternary orientable matroids \cite{Robbins-Slilaty23}), we have repeatedly run into difficulties related to the way in which the special embedded minors of a matroid $M$ ``interact'' with one another.\footnote{More precisely, we require additional information about the fundamental diagram over which the colimit is taken in the fundamental presentation of $F_M$.} Conjecturally, we believe that such information should be encoded in a suitable higher version of Tutte's homotopy theorem.

\medskip

\subsection*{Acknowledgements:} The authors thank Ju\v s Kocutar for several corrections and useful observations on early drafts of the paper. In particular, his comments led to several improvements in \autoref{section: towards higher homotopy theorems}. We also thank the anonymous referees for their detailed and helpful reports. 
The first author was partially supported by NSF grant DMS-2154224 and a Simons Fellowship in Mathematics (1037306, Baker).
The second author was also partially supported by NSF grant DMS-2154224.

\subsection*{A.I. usage statement:} 
After receiving feedback on our original submission (which involved no A.I. usage) from two anonymous human referees and implementing their suggested corrections, we ran the revised manuscript through ChatGPT 6.0 Astra for additional correctness audits. GPT-6.0 identified several remaining typos and errors, which we subsequently fixed. 

The proofs of Tutte's path theorem and homotopy theorem presented in this paper have been formalized in Lean~4. The formalization is available
at \url{https://github.com/mbaker386/tutte-homotopy-lean}. This formalization process, conducted using GPT-6.0 in conjunction with Codex and Lean~4,
led to some minor changes to our original exposition. See~\autoref{footnote:LeanPath} and~\autoref{remark: tutte version of fourth path} for details. 

The authors wrote and take full responsibility for all content in this paper.


\section{Tutte's homotopy theory}
\label{section: Tutte's homotopy theorem}

In this section, we describe Tutte's homotopy theory in detail. We begin with some background information on the lattice of flats of a matroid, and then state and prove Tutte's path theorem. We then formulate Tutte's homotopy theorem and provide a brief outline of the proof. Finally, we state and prove an extended version of the homotopy theorem which lends itself more readily to applications than Tutte's original formulation.

\subsection{Embedded minors and upper sublattices}
\label{subsection: embedded minors and upper sublattices}

Let $M$ be a matroid on the finite set $E$. 
We denote by $\gen{S}_{M}$ the closure of a subset $S \subseteq E$ in $M$.

An \emph{embedded minor of $M$} is a minor $M\minor JI$ of $M$ together with a fixed choice of an independent subset $I$ and a coindependent subset $J$ of $M$. Note that by the ``scum theorem'' (\cite[Lemma 3.3.2]{Oxley11}), every minor of $M$ is isomorphic to an embedded minor of $M$. 

Let $\Lambda = \Lambda_M$ be the lattice of flats of $M$, which is a geometric lattice, i.e., $\Lambda$ is finite, atomistic, and semimodular. 
We write $F_1 \lor F_2$ for the join of two flats $F_1$ and $F_2$, which is the smallest flat containing their union, i.e., $F_1 \lor F_2 = \langle F_1 \cup F_2 \rangle$.
The meet $F_1 \wedge F_2$ is given simply by the intersection $F_1 \cap F_2$, and we use these two notations interchangeably.

We denote by $\Lambda^{(d)}$ the set of all corank $d$ flats of $M$. In particular, $\cH = \Lambda^{(1)}$ is the set of all hyperplanes. An \emph{upper sublattice of $\Lambda$} is a subset $\Lambda'$ closed under binary joins in $\Lambda$ which is a geometric lattice in the induced order, and whose rank is equal to the corank of its bottom element in $\Lambda$. Its meets are taken within $\Lambda'$ and need not agree with ambient intersections.

The lattice of flats $\Lambda_N$ of an embedded minor $N=M\minor JI$ is canonically isomorphic, via ambient
closure, to the join-closed subset
\[
   \bigl\{\gen{S}_M: I\subseteq S\subseteq E-J\bigr\}
   \ \subseteq\ \Lambda.
\]

Since $J$ is coindependent, $E-J$ is spanning and thus $\Lambda_N$ is an upper sublattice. This defines a map $\Psi\colon\Emb_M\to\USL_\Lambda$ from the set $\Emb_M$ of embedded minors of $M$ to the set $\USL_\Lambda$ of upper sublattices of $\Lambda$. By \cite[Proposition 6.7]{Baker-Lorscheid-Zhang24}, the map $\Psi$ is surjective.

If $I$ and $J$ are disjoint subsets of the ground set $E$ of $M$, with $I$ independent and $J$ coindependent (or, equivalently, with $E-J$ spanning), 
we denote by $\Lambda\minor JI$ the upper sublattice of $\Lambda$ that consists of all flats of the form $\gen{S}$ of $\Lambda$ for which $I \subseteq S\subseteq E-J$.

By a theorem of Birkhoff \cite{Birkhoff67}, two matroids $M_1$ and $M_2$ have isomorphic lattices of flats if and only if $M_1$ and $M_2$ have the same simplification. In particular, the lattice of flats of a simple matroid $M$ determines $M$.


\subsection{Linear subclasses and modular cuts}
\label{subsection: linear subclasses and modular cuts}

Let $M$ be a matroid of rank $r$ on $E$ and $\Lambda$ its lattice of flats. 
Two flats $F_1, F_2 \in \Lambda$ form a {\em modular pair} in $M$ (or in $\Lambda$) if $\rk(F_1) + \rk(F_2) = \rk(F_1 \wedge F_2) + \rk(F_1 \lor F_2)$.

\begin{df}
A \emph{modular cut in $\Lambda$} is a subset $\Gamma$ of $\Lambda$ such that
\begin{enumerate}
 \item for all $F\in\Gamma$ and $F'\in\Lambda$ with $F \subseteq F'$, we have $F'\in\Gamma$; and
 \item for all modular pairs $(F_1, F_2)$ in $\Lambda$ with $F_1, F_2 \in \Gamma$, we have $F_1 \cap F_2 \in \Gamma$.
\end{enumerate}
\end{df}

An example of a modular cut is the collection $\Gamma=\{F\in\Lambda\mid F_0\subseteq F\}$ of all flats that contain a fixed flat $F_0$. The \emph{empty modular cut} is $\Gamma=\emptyset$ and the \emph{trivial modular cut} is $\Gamma=\{E\}$. If $\Lambda'$ is an upper sublattice of $\Lambda$, then $\Gamma'=\Gamma\cap\Lambda'$ is a modular cut of $\Lambda'$. 

A \emph{single-element extension} of a matroid $M$ is a matroid $\hat{M}$ of the same rank as $M$ such that $M = \hat{M} \minus a$ for some $a \in \hat{M}$.\footnote{In the literature, the extension $\hat M$ of $M$ by a coloop, which has rank $\rk\, M+1$, is often considered as a single-element extension of $M$ whose corresponding modular cut is $\Gamma=\emptyset$. This extension does not play a role in this text, and we will ignore it without further comment.} The following result shows that modular cuts characterize the single-element extensions of a matroid; see~\cite[Theorem 7.2.3]{Oxley11}. 

\begin{prop}\label{prop: modular cuts and one element extensions}
 The association
 \[
  \big\{\text{\emph{single-element extensions of }} M \big\} \ \longrightarrow \ \big\{\text{\emph{non-empty modular cuts in }} \Lambda \big\}
 \]
 that sends a single-element extension $\widehat M$ of $M$ with $M=\widehat M\minus a$ to the associated modular cut $\Gamma=\{F\in\Lambda\mid a\in\gen{F}_{\widehat M}\}$ is a bijection.
\end{prop}

A \emph{linear subclass of $\Lambda$} (or of $M$) is a collection $\cL\subset\cH$ of hyperplanes of $M$ such that whenever $\cL$ contains two distinct hyperplanes $H$ and $H'$ that intersect in a corank $2$ flat $L=H\cap H'$, $\cL$ contains every hyperplane $H''$ with $L\subset H''$.

We recall from~\cite[Theorem 10.5]{Crapo-Rota} (see also \cite[Exercise 7.2.6]{Oxley11}) the following intimate relation between modular cuts and linear subclasses:

\begin{prop}\label{prop: linear subclasses correspond to modular cuts}
 Let $\Lambda$ be a geometric lattice with $\cH=\Lambda^{(1)}$. The association $\Gamma\mapsto\Gamma\cap\cH$ defines a bijection
 \[
  \big\{\text{\emph{non-empty modular cuts in $\Lambda$}}\big\} \quad \longrightarrow \quad \big\{\text{\emph{linear subclasses of $\Lambda$}}\big\}.
 \]
\end{prop}

One ingredient in the proof of \autoref{prop: linear subclasses correspond to modular cuts} is the following result (cf.~\cite[Corollary 7.2.5]{Oxley11}), which also plays a crucial role in the proofs of \autoref{thm: the foundation is generated by cross ratios} and \autoref{thm: fundamental presentation of the foundation}:

\begin{lemma} \label{lemma: hyperplanes of a deletion}
Let $M$ be a matroid on $E$, let $a \in E$ be an element that is not a coloop, and let $M' = M \backslash a$. Let $F$ be a flat of rank $k$ of $M$. Then exactly one of the following holds:
\begin{enumerate}
    \item $a \not\in F$. \emph{(}In this case, $F$ is also a rank $k$ flat of $M'$.\emph{)}
    \item $a \in F$ and $F-a$ is a rank $k$ flat of $M'$.
    \item $a \in F$ and $F-a$ is a rank $k-1$ flat of $M'$.
\end{enumerate}
\end{lemma}

For the rest of the paper, we always assume that the modular cut $\Gamma$ is non-empty. 

\subsection{The path theorem}
\label{subsection: path theorem}

Let $M$ be a connected matroid of rank $r$ on the ground set $E$. In this section, we recall the statement and proof of \emph{Tutte's path theorem} (called the \emph{Fundamental Theorem of Linear Subclasses} in~\cite{Tutte65}). The proof presented here is more or less the same as Tutte's original proof in~\cite[Theorem 5.1]{Tutte58a}, but is written in the more modern language of lattices of flats, in contrast to Tutte's lattices of unions of circuits (of the dual matroid $M^\ast$). The proofs of all relevant propositions can be found in~\autoref{appendix:flats}. 

\begin{df}
A flat $F$ is {\em indecomposable} if the contraction $M / F$ is connected. Otherwise, $F$ is {\em decomposable}. 
\end{df}

\begin{ex}\label{ex:conn-line}
 Since we assume that $M$ is connected, the smallest flat $\gen{\emptyset}$ is indecomposable. The ground set $E$, as well as all hyperplanes $H \in \cH$, are indecomposable. As a consequence of~\autoref{lemma:connected-contraction}, a corank $2$ flat $L \in \Lambda^{(2)}$ is indecomposable if and only if it is contained in at least $3$ hyperplanes.
\end{ex}

\begin{df}
 A {\em Tutte path in $M$} is a finite sequence $\gamma = (H_1, \cdots, H_k)$ of one or more hyperplanes of $M$, not necessarily all distinct, such that any two consecutive terms are distinct hyperplanes of $M$ intersecting in an indecomposable corank $2$ flat. If all terms of a Tutte path $\gamma$ contain a flat $F$, then we say $\gamma$ is {\em on} the flat $F$. 
\end{df}

\begin{thm}[Tutte's path theorem~\cite{Tutte58a}]\label{thm: Tutte's path theorem}
 Let $\Gamma$ be a modular cut of a connected matroid $M$, and let $F \ne E$ be an indecomposable flat. Suppose $X$ and $Y$ are two hyperplanes containing $F$, and suppose $X, Y \notin \Gamma$. Then, there exists a Tutte path $\gamma$ on $F$ from $X$ to $Y$ such that no term of $\gamma$ belongs to $\Gamma$. 
\end{thm}

\begin{proof}
 We use induction on the corank of $F$. The cases where $F$ has corank $1$ or $2$ are trivial. Suppose $F$ has corank $c \geqslant 3$. Since $M$ is connected, $\gen{\emptyset}$ is an indecomposable flat. By \autoref{prop:conn-path} and \autoref{prop:conn-diamond}, there exist indecomposable flats $U, V$, and $W$ with $X \supseteq U \supseteq V \ne W \supseteq F$ and $\rk(U) = \rk(V) + 1 = \rk(W) + 1 = \rk(F) + 2$; see~\autoref{fig: subposet structure in the proof of thm: Tutte's path theorem} for an illustration. We assume without loss of generality that $Y$ contains neither $V$ nor $W$; otherwise, we replace $F$ with $V$ or $W$, both of corank $c-1$. Applying~\autoref{prop:conn-complement} to $U \cap Y \supseteq F$ and $U \lor Y = E$, there exists an indecomposable corank $2$ flat $L$ with $F \subseteq L \subseteq Y$ and $U \lor L = E$. By the submodularity of the rank function, $H_1 = L \lor V$ and $H_2 = L \lor W$ are distinct hyperplanes. Since $L \subseteq Y$ has corank 2 in $\Lambda$, $\Gamma$ is a modular cut, and $Y \not\in \Gamma$, at least one of $H_1$ and $H_2$ does not belong to $\Gamma$. Say $H_1 \notin \Gamma$. Since $X$ and $H_1$ contain $V$, which is an indecomposable flat of corank $c - 1$, there exists a Tutte path $\gamma_1$ on $V$ (and hence on $F$) from $X$ to $H_1$ such that no term of $\gamma_1$ belongs to $\Gamma$. Adjoining $Y$ to $\gamma_1$ gives us the desired Tutte path from $X$ to $Y$. 
\end{proof}

\begin{figure}[th]
\centering
\begin{tikzpicture}[x=1cm,y=0.8cm, vertices/.style={draw, fill=black, circle, inner sep=0pt},font=\footnotesize]
	\node (0) at (0, 0){$F$};
	\node (1) at (-1,1){$V$};
	\node (2) at (1, 1){$W$};
	\node (3) at (0, 2){$U$};
	\node (4) at (-1.5,3.5){$X$};
	\node (5) at (0, 3.5){$H_1$};
	\node (6) at (1, 3.5){$Y$};
	\node (7) at (2, 3.5){$H_2$};
	\node (8) at (1, 2.5){$L$};
\foreach \to/\from in {0/1, 0/2, 1/3, 2/3, 5/8, 6/8, 7/8}
\draw [-] (\to)--(\from);
\draw[dotted] (3) -- (4);
\draw[dotted] (1) -- (5);
\draw[dotted] (2) -- (7);
\draw[dotted] (1) -- (4);
\draw[dotted] (0) -- (8);
\end{tikzpicture}
\caption{The subposet structure in \autoref{thm: Tutte's path theorem}}
\label{fig: subposet structure in the proof of thm: Tutte's path theorem}
\end{figure}

\begin{rem}
Tutte's original formulation of~\autoref{thm: Tutte's path theorem} allows $X$ to be an arbitrary hyperplane, not necessarily not in $\Gamma$, and asserts the existence of a Tutte path $\gamma$ from $X$ to $Y$ such that no term of $\gamma$ other than the first belongs to $\Gamma$. We will not require this more general statement.
\end{rem}

Let $\Gamma$ be a modular cut of a matroid $M$, and let $G_{M, \Gamma}$ be the graph whose vertex set is the set of hyperplanes of $M$ not in $\Gamma$, such that two vertices $H_1, H_2$ are adjacent if and only if $H_1 \cap H_2$ is an indecomposable corank $2$ flat. When we apply \autoref{thm: Tutte's path theorem} to proving that universal cross-ratios generate the foundation of a matroid, we shall need only the case where $F = \gen{\emptyset}$ is the minimal flat of $M$. In this case, Tutte's path theorem is equivalent to the following graph-theoretic reformulation:

\begin{cor}\label{cor: graph version of the path theorem}
    If $M$ is connected, then the graph $G_{M, \Gamma}$ as defined above is connected. 
\end{cor}

\begin{proof}
    This is a consequence of~\autoref{thm: Tutte's path theorem}, without requiring all terms of the Tutte path $\gamma$ to contain the flat $F = \emptyset$. 
\end{proof}

\subsection{The homotopy theorem}
\label{subsection: homotopy theorem}

A Tutte path $\gamma = (H_1, \cdots, H_k)$ in a connected matroid $M$ is {\em closed} (or is a \emph{re-entrant path}, in Tutte's original terminology) if $H_1 = H_k$. Let $\Gamma \subseteq \Lambda$ be a modular cut of $M$. If no term of the Tutte path $\gamma$ is in $\Gamma$, then $\gamma$ is {\em off} the modular cut $\Gamma$. 

Loosely speaking, Tutte's homotopy theorem asserts that every closed Tutte path off a fixed modular cut $\Gamma$ can be decomposed into ``short'' closed Tutte paths of a small number of types. We shall see in~\autoref{subsection:extended-homotopy-theorem} that these short closed Tutte paths necessarily occur in ``small'' minors of $M$ belonging to a fixed finite set of isomorphism types (independent of the matroid $M$ itself). 

We make this more precise in the following.

\subsubsection{Tutte constellations}
\label{subsubsection: constellations}

\begin{df}
 A \emph{Tutte constellation} is a pair $\tau=(\Lambda,\Gamma)$ consisting of a geometric lattice $\Lambda$ and a modular cut $\Gamma$ in $\Lambda$. The \emph{type of $\tau$} is the isomorphism type of the simple matroid $M$ with lattice of flats $\Lambda$. Given a Tutte constellation $\tau$, we denote its geometric lattice by $\Lambda_\tau$ and its modular cut by $\Gamma_\tau$. A \emph{subconstellation $\sigma$ of $\tau$} is a pair consisting of an upper sublattice $\Lambda_\sigma$ of $\Lambda$ and the modular cut $\Gamma_\sigma=\Gamma\cap\Lambda_\sigma$. A \emph{closed Tutte path in $\tau$} is a closed Tutte path in $\Lambda$ off $\Gamma$.
\end{df}

We fix a Tutte constellation $\tau=(\Lambda,\Gamma)$ for the rest of this section. In \cite{Tutte58a}, Tutte defines four kinds of so-called \emph{elementary re-entrant paths in $\tau$}, which appear in certain subconstellations.\footnote{Note, however, that elementary paths of the second kind, in Tutte's sense, can appear in two different subconstellations, as explained below.} 

Let $K_{2,3}$ be the complete bipartite graph on $2+3$ vertices; see\ \autoref{fig: graphs}. Let $M(K_{2,3})$ be the associated graphic matroid, which is the triple serial extension of $U_{1,3}$ with serial pairs $(1,4)$, $(2,5)$ and $(3,6)$. 

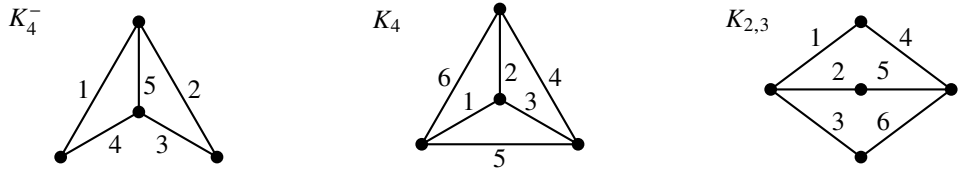
\begin{figure}[t]
\centering
 \begin{tikzpicture}[x=0.6cm,y=0.6cm, font=\footnotesize,decoration={markings,mark=at position 0.6 with {\arrow{latex}}}]
  \node at (-2.5,1.5) {$K_{4}^-$};
  \node at (0,0) {
  \begin{tikzpicture}[thick]
   \draw (0:0) -- node[right=-2pt,pos=0.3] {$5$} (90:2) -- node[left] {$1$} (210:2) -- node[below,pos=0.7] {$4$} (0:0) -- node[below,pos=0.3] {$3$} (330:2) -- node[right] {$2$} (90:2);
   \filldraw (0:0) circle (2pt);  
   \filldraw (90:2) circle (2pt);  
   \filldraw (210:2) circle (2pt);  
   \filldraw (330:2) circle (2pt);  
  \end{tikzpicture}
  };
  \node at (5.5,1.5) {$K_{4}$};
  \node at (8,0) {
  \begin{tikzpicture}[thick]
   \draw (0:0) -- node[right=-2pt,pos=0.3] {$2$} (90:2) -- node[left] {$6$} (210:2) -- node[above=-1pt,pos=0.6] {$1$} (0:0) -- node[above=-1pt,pos=0.4] {$3$} (330:2) -- node[right] {$4$} (90:2);
   \draw (210:2) -- node[below=-2pt] {$5$} (330:2);
   \filldraw (0:0) circle (2pt);  
   \filldraw (90:2) circle (2pt);  
   \filldraw (210:2) circle (2pt);  
   \filldraw (330:2) circle (2pt);  
  \end{tikzpicture}
  };
  \node at (13.5,1.5) {$K_{2,3}$};
  \node at (16,0) {
  \begin{tikzpicture}[thick]
   \draw (-2,0) -- node[above,near end] {$2$} (0,0) -- node[above, near start] {$5$} (2,0) -- node[above] {$4$} (0,1.5) -- node[above] {$1$} (-2,0) -- node[above,near end] {$3$} (0,-1.5) -- node[above,near start] {$6$} (2,0);
   \filldraw (0,0) circle (2pt);  
   \filldraw (-2,0) circle (2pt);  
   \filldraw (2,0) circle (2pt);  
   \filldraw (0,1.5) circle (2pt);  
   \filldraw (0,-1.5) circle (2pt);  
  \end{tikzpicture}
  };
\end{tikzpicture}
 \caption{Some important graphs for the current section and~\autoref{subsection:extended-homotopy-theorem}}
 \label{fig: graphs}
\end{figure}

\medskip\noindent\textbf{First kind.}
Let $\sigma$ be a subconstellation of type $U_{2,2}$ with $\Gamma_\sigma=\{E\}$,
and assume that the bottom element of $\Lambda_{U_{2,2}}$ is indecomposable in $\Lambda$.\footnote{Here, and below, when we consider a subconstellation of $\tau=(\Lambda,\Gamma)$ of type $N$, we identify $\Lambda' = \Lambda_N$ with an upper sublattice of $\Lambda$, and in particular it makes sense to ask whether an element of $\Lambda'$ is indecomposable as an element of $\Lambda$.} 
Let $H_1$ and $H_2$ be the two hyperplanes of $U_{2,2}$. Then the closed Tutte path $(H_1,H_2,H_1)$ is called \emph{elementary of the first kind}. This subconstellation and the elementary Tutte path are illustrated in \autoref{fig: kind1}.
\begin{figure}[t]
 \centering
 \begin{tikzpicture}[x=1.0cm,y=0.8cm, font=\footnotesize,decoration={markings,mark=at position 0.6 with {\arrow{latex}}}]
   \node at (-1,2) {$\Lambda_{U_{2,2}}$};  
   \filldraw[fill=green!30!white,draw=green!80!black,rounded corners=2pt] (0.7,2.3) -- (0.7,1.7) -- (1.3,1.7) -- (1.3,2.3) -- cycle;
   \node (0) at (1,0) {$\emptyset$};  
   \node (1) at (0.5,1) {$1$};  
   \node (2) at (1.5,1) {$2$};  
   \node (12) at (1,2) {$12$};  
   \draw (0) to (1);
   \draw (0) to (2);
   \draw (1) to (12);
   \draw (2) to (12);
  \node at (5.5,2) {\scriptsize elementary Tutte path of the first kind};  
  \draw (5,1.08) edge [-,red,thick,postaction={decorate}] (6,1.08);
  \draw (6,0.92) edge [-,red,thick,postaction={decorate}] (5,0.92);
  \draw (4.5,1) -- (6.5,1);
  \filldraw (5,1) circle (2pt);  
  \node at (5,0.5) {$1$};  
  \filldraw (6,1) circle (2pt);  
  \node at (6,0.5) {$2$};  
\end{tikzpicture}
 \caption{The Tutte constellations and elementary Tutte paths of the first kind}
 \label{fig: kind1}
\end{figure}
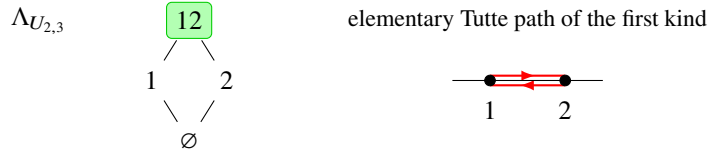

\medskip\noindent\textbf{Second kind.}
There are two different subconstellations which give rise to closed Tutte paths which Tutte refers to as ``elementary of the second kind'':
\begin{enumerate}
    \item[(a)] For every subconstellation $\sigma$ of type $U_{2,3}$ with hyperplanes $H_{1}$, $H_{2}$, and $H_{3}$, all off $\Gamma_\sigma=\{E\}$, the closed Tutte path $(H_{1},H_{2},H_{3},H_{1})$ is called \emph{elementary of the second kind}.
    \item[(b)] For every subconstellation $\sigma$ of type $U_{3,3}$ with hyperplanes $H_{12}$, $H_{13}$, and $H_{23}$, all off $\Gamma_\sigma= \{E \}$, whose three corank $2$ flats are indecomposable in $\Lambda$, the closed Tutte path $(H_{12},H_{13},H_{23},H_{12})$ is also called \emph{elementary of the second kind}. 
\end{enumerate}
Both subconstellations and their corresponding elementary Tutte paths are illustrated in \autoref{fig: kind2}.

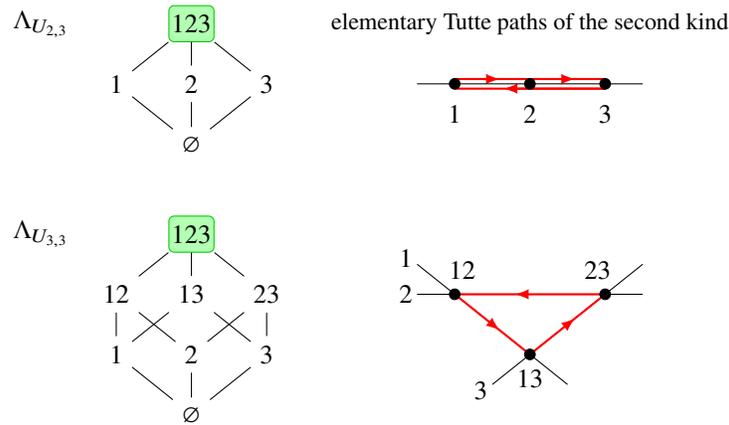
\begin{figure}[t]
 \centering
 \begin{tikzpicture}[x=1.0cm,y=0.8cm, font=\footnotesize,decoration={markings,mark=at position 0.6 with {\arrow{latex}}}]
  \node at (0,4)
  {
   \begin{tikzpicture}
   \clip(-1.5,-0.5) rectangle (2.5,2.5);
   \node at (-1,2) {$\Lambda_{U_{2,3}}$};  
   \filldraw[fill=green!30!white,draw=green!80!black,rounded corners=2pt] (0.7,2.3) -- (0.7,1.7) -- (1.3,1.7) -- (1.3,2.3) -- cycle;
   \node (0) at (1,0) {$\emptyset$};  
   \node (1) at (0,1) {$1$};  
   \node (2) at (1,1) {$2$};  
   \node (3) at (2,1) {$3$};  
   \node (123) at (1,2) {$123$};  
   \draw (0) to (1);
   \draw (0) to (2);
   \draw (0) to (3);
   \draw (1) to (123);
   \draw (2) to (123);
   \draw (3) to (123);
   \end{tikzpicture} 
  };
  \node at (5,5) {\scriptsize elementary Tutte paths of the second kind};  
  \node at (5,4)
  {
   \begin{tikzpicture}
   \clip(4,0) rectangle (8,2);
   \draw (5,1.08) edge [-,red,thick,postaction={decorate}] (6,1.08);
   \draw (6,1.08) edge [-,red,thick,postaction={decorate}] (7,1.08);
   \draw (7,0.92) edge [-latex,red,thick,] (5.65,0.92);
   \draw (7,0.92) edge [-,red,thick,] (5,0.92);
   \draw (4.5,1) -- (7.5,1);
   \filldraw (5,1) circle (2pt);  
   \node at (5,0.5) {$1$};  
   \filldraw (6,1) circle (2pt);  
   \node at (6,0.5) {$2$};  
   \filldraw (7,1) circle (2pt);  
   \node at (7,0.5) {$3$}; 
   \end{tikzpicture} 
  };
  \node at (0,0)
  {
   \begin{tikzpicture}
   \clip(-0.5,-0.5) rectangle (3.5,3.5);
   \node at (0,3) {$\Lambda_{U_{3,3}}$};  
   \filldraw[fill=green!30!white,draw=green!80!black,rounded corners=2pt] (1.7,3.3) -- (1.7,2.7) -- (2.3,2.7) -- (2.3,3.3) -- cycle;
   \node (0) at (2,0) {$\emptyset$};  
   \node (1) at (1,1) {$1$};  
   \node (2) at (2,1) {$2$};  
   \node (3) at (3,1) {$3$};  
   \node (12) at (1,2) {$12$};  
   \node (13) at (2,2) {$13$};  
   \node (23) at (3,2) {$23$};  
   \node (123) at (2,3) {$123$};  
   \draw (0) to (1);
   \draw (0) to (2);
   \draw (0) to (3);
   \draw (1) to (12);
   \draw (1) to (13);
   \draw (2) to (12);
   \draw (2) to (23);
   \draw (3) to (13);
   \draw (3) to (23);
   \draw (12) to (123);
   \draw (13) to (123);
   \draw (23) to (123);
   \end{tikzpicture} 
   };
  \node at (5,0.5)
  {
   \begin{tikzpicture}
   \clip(8,0) rectangle (12,4);
  \draw (8.5,2) -- (11.5,2);
  \draw (8.5,2.5) -- (10.5,0.5);
  \draw (11.5,2.5) -- (9.5,0.5);
  \draw (9,2) edge [-,red,thick,postaction={decorate}] (10,1);
  \draw (10,1) edge [-,red,thick,postaction={decorate}] (11,2);
  \draw (11,2) edge [-,red,thick,postaction={decorate}] (9,2);
  \filldraw (9,2) circle (2pt);  
  \filldraw (10,1) circle (2pt);  
  \filldraw (11,2) circle (2pt);  
  \node at (8.35,2.6) {$1$};  
  \node at (8.35,2) {$2$};  
  \node at (9.35,0.4) {$3$};  
  \node at (9.1,2.4) {$12$};  
  \node at (10,0.6) {$13$};  
  \node at (10.9,2.4) {$23$};  
   \end{tikzpicture} 
   };  
\end{tikzpicture}
 \caption{Tutte constellations and elementary Tutte paths of the second kind}
 \label{fig: kind2}
\end{figure}

\medskip\noindent\textbf{Third kind.}
For every subconstellation $\sigma$ of type $U_{3,4}$ with hyperplanes $H_{12},\dotsc,H_{34}$ and $\Gamma_\sigma=\{H_{23},\ H_{14},\ E\}$, the closed Tutte path $(H_{12},H_{13},H_{34},H_{24},H_{12})$ is called \emph{elementary of the third kind}. This subconstellation and the corresponding Tutte path are illustrated in \autoref{fig: kind3}.

\begin{figure}[t]
 \begin{tikzpicture}[x=1.0cm,y=0.8cm, font=\footnotesize,decoration={markings,mark=at position 0.6 with {\arrow{latex}}}]
  \node at (0,0)
  {
   \begin{tikzpicture}
   \node at (-1,3) {$\Lambda_{U_{3,4}}$};  
   \filldraw[fill=green!20!white,draw=green!80!black,rounded corners=2pt] (2.15,3.35) -- (1.75,2.3) -- (1.75,1.7) -- (3.25,1.7) -- (3.25,2.3) -- (2.85,3.35) -- cycle;
   \node (0) at (2.5,0) {$\emptyset$};  
   \node (1) at (1,1) {$1$};  
   \node (2) at (2,1) {$2$};  
   \node (3) at (3,1) {$3$};  
   \node (4) at (4,1) {$4$};  
   \node (12) at (0,2) {$12$};  
   \node (13) at (1,2) {$13$};  
   \node (23) at (2,2) {$23$};  
   \node (14) at (3,2) {$14$};  
   \node (24) at (4,2) {$24$};  
   \node (34) at (5,2) {$34$};  
   \node (1234) at (2.5,3) {$1234$};  
   \draw (0) to (1);
   \draw (0) to (2);
   \draw (0) to (3);
   \draw (0) to (4);
   \draw (1) to (12);
   \draw (1) to (13);
   \draw (1) to (14);
   \draw (2) to (12);
   \draw (2) to (23);
   \draw (2) to (24);
   \draw (3) to (13);
   \draw (3) to (23);
   \draw (3) to (34);
   \draw (4) to (14);
   \draw (4) to (24);
   \draw (4) to (34);
   \draw (12) to (1234);
   \draw (13) to (1234);
   \draw (14) to (1234);
   \draw (23) to (1234);
   \draw (24) to (1234);
   \draw (34) to (1234);
   \end{tikzpicture} 
  };
  \node at (7,2) {\scriptsize elementary Tutte path of the third kind};  
  \node at (7,0)
  {
   \begin{tikzpicture}[x=0.4cm,y=0.4cm]
    \filldraw[fill=green!20!white,draw=green!80!black,rounded corners=2pt] (0,0) circle (6pt);
    \filldraw[fill=green!20!white,draw=green!80!black,rounded corners=2pt] (10,0) circle (6pt);
    \node at (2.5,4.4) {$1$};  
    \node at (7.5,4.4) {$2$};  
    \node at (10.5,2.1) {$3$};  
    \node at (-0.5,2.1) {$4$};  
    \draw[name path=L1] (12,-1.2) -- (3,4.2);
    \draw[name path=L2] (-2,-1.2) -- (7,4.2);
    \draw[name path=L3] (-3,-0.6) -- (10,2);
    \draw[name path=L4] (13,-0.6) -- (0,2);
    \fill[name intersections={of=L1 and L2,by=12}] (intersection-1) circle (2pt) node[above=2pt] {$12$};
    \fill[name intersections={of=L1 and L3,by=13}] (intersection-1) circle (2pt) node[above=2pt] {$13$};
    \fill[name intersections={of=L1 and L4,by=14}] (intersection-1) circle (2pt) node[below=5pt] {$14$};
    \fill[name intersections={of=L2 and L3,by=23}] (intersection-1) circle (2pt) node[below=5pt] {$23$};
    \fill[name intersections={of=L2 and L4,by=24}] (intersection-1) circle (2pt) node[above=2pt] {$24$};
    \fill[name intersections={of=L3 and L4,by=34}] (intersection-1) circle (2pt) node[below=2pt] {$34$};
    \draw (12) edge [-,red,thick,postaction={decorate}] (13);
    \draw (13) edge [-,red,thick,postaction={decorate}] (34);
    \draw (34) edge [-,red,thick,postaction={decorate}] (24);
    \draw (24) edge [-,red,thick,postaction={decorate}] (12);
    \fill (12) circle (2pt);
    \fill (13) circle (2pt);
    \fill (34) circle (2pt);
    \fill (24) circle (2pt);
   \end{tikzpicture} 
  };
\end{tikzpicture}
 \caption{The Tutte constellation and elementary Tutte path of the third kind}
 \label{fig: kind3}
\end{figure}
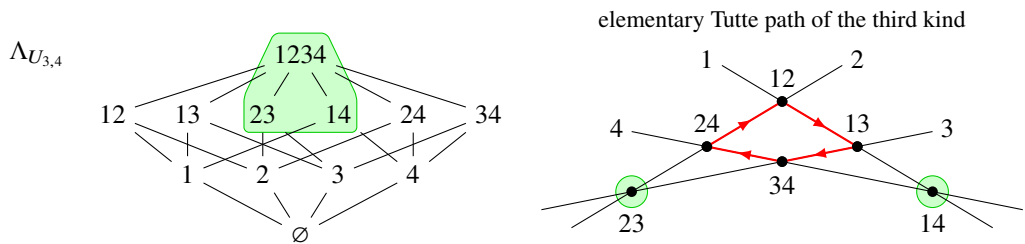

\medskip\noindent\textbf{Fourth kind.}
For every subconstellation $\sigma$ of type $M(K_{2,3})$ with hyperplanes as in \autoref{fig: Tutte constellation of kind4} and $\Gamma_\sigma=\{H_{123},\ H_{156},\ H_{246},\ H_{345}, E\}$ such that $14$, $25$, and $36$ are decomposable in $\Lambda$, the closed Tutte path $(H_{1245},H_{126},H_{1346},H_{456},H_{1245})$ is called \emph{elementary of the fourth kind}. This elementary Tutte path is illustrated in \autoref{fig: kind4}. 

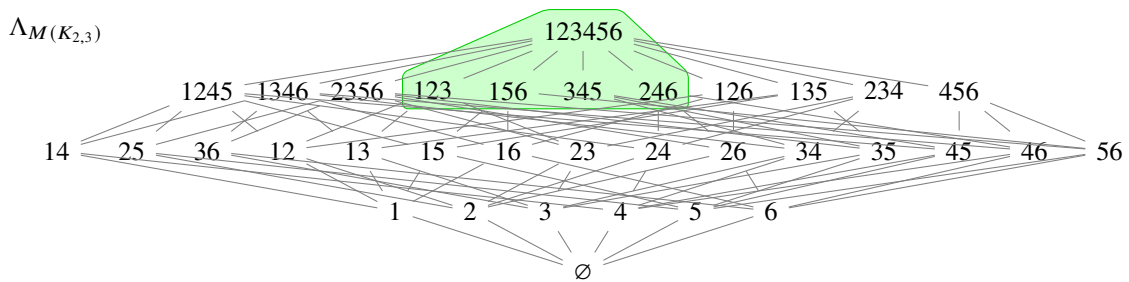
\begin{figure}[t]
 \begin{tikzpicture}[x=1.0cm,y=0.8cm, font=\footnotesize,decoration={markings,mark=at position 0.6 with {\arrow{latex}}}]
   \node at (0,4) {$\Lambda_{M(K_{2,3})}$};  
   \filldraw[fill=green!20!white,draw=green!80!black,rounded corners=2pt] (6.5,4.35) -- (4.6,3.3) -- (4.6,2.7) -- (8.4,2.7) -- (8.4,3.3) -- (7.5,4.35) -- cycle;
   \node (0) at (7,0) {$\emptyset$};  
   \node (1) at (4.5,1) {$1$};  
   \node (2) at (5.5,1) {$2$};  
   \node (3) at (6.5,1) {$3$};  
   \node (4) at (7.5,1) {$4$};  
   \node (5) at (8.5,1) {$5$};  
   \node (6) at (9.5,1) {$6$};  
   \node (14) at (0,2)  {$14$};  
   \node (25) at (1,2)  {$25$};  
   \node (36) at (2,2)  {$36$};  
   \node (12) at (3,2)  {$12$};  
   \node (13) at (4,2)  {$13$};  
   \node (15) at (5,2)  {$15$};  
   \node (16) at (6,2)  {$16$};  
   \node (23) at (7,2)  {$23$};  
   \node (24) at (8,2)  {$24$};  
   \node (26) at (9,2)  {$26$};  
   \node (34) at (10,2) {$34$};  
   \node (35) at (11,2) {$35$};  
   \node (45) at (12,2) {$45$};  
   \node (46) at (13,2) {$46$};  
   \node (56) at (14,2) {$56$};  
   \node (1245) at (2,3) {$1245$};  
   \node (1346) at (3,3) {$1346$};  
   \node (2356) at (4,3) {$2356$};  
   \node (123) at (5,3) {$123$};  
   \node (156) at (6,3) {$156$};  
   \node (345) at (7,3) {$345$};  
   \node (246) at (8,3) {$246$};  
   \node (126) at (9,3) {$126$};  
   \node (135) at (10,3) {$135$};  
   \node (234) at (11,3) {$234$};  
   \node (456) at (12,3) {$456$};  
   \node (123456) at (7,4) {$123456$};  
   \foreach \to/\from in {0/1,0/2,0/3,0/4,0/5,0/6,
1/12,1/13,1/14,1/15,1/16,
2/12,2/23,2/24,2/25,2/26,
3/13,3/23,3/34,3/35,3/36,
4/14,4/24,4/34,4/45,4/46,
5/15,5/25,5/35,5/45,5/56,
6/16,6/26,6/36,6/46,6/56,
14/1245,25/1245,12/1245,15/1245,24/1245,45/1245,
14/1346,36/1346,13/1346,16/1346,34/1346,46/1346,
25/2356,36/2356,23/2356,26/2356,35/2356,56/2356,
12/123,13/123,23/123,
15/156,16/156,56/156,
34/345,35/345,45/345,
24/246,26/246,46/246,
12/126,16/126,26/126,
13/135,15/135,35/135,
23/234,24/234,34/234,
45/456,46/456,56/456,
1245/123456,1346/123456,2356/123456,123/123456,126/123456,135/123456,156/123456,234/123456,345/123456,246/123456,456/123456}
      \draw [-,gray] (\to)--(\from);
\end{tikzpicture}
 \caption{The Tutte constellation of the fourth kind}
 \label{fig: Tutte constellation of kind4}
\end{figure}

\begin{figure}[t]
\begin{tikzpicture}[x=0.6cm,y=0.6cm,font=\scriptsize,decoration={markings,mark=at position 0.3 with {\arrow{latex}}}]
    \coordinate[label=left:1245] (1245) at (0,10);  
    \coordinate[label=right:1346] (1346) at (12,9);    
    \coordinate[label=above:2356] (2356) at (8,17);    
    \coordinate[label=left:34] (34) at (1,7.5);    
    \coordinate[label=below:35] (35) at (3.3,6);    
    \coordinate[label=below:23] (23) at (5.7,5);    
    \coordinate[label=below:46] (46) at (7.8,4);    
    \coordinate[label=below:56] (56) at (9.2,4);    
    \coordinate[label=below:26] (26) at (11.15,4);    
    \coordinate[label=right:45] (45) at (12,4);    
    \coordinate[label=right:24] (24) at (12,6.65);    
    \coordinate[label=right:12] (12) at (12,12.2);    
    \coordinate[label=right:15] (15) at (12,13.7);    
    \coordinate[label=left:16] (16) at (6,15);    
    \coordinate[label=left:13] (13) at (3,13);    
    \fill (1245) circle (2pt);
    \fill (1346) circle (2pt);
    \fill (2356) circle (2pt);
    \path[draw,dashed,bend left=0pt] (1245) to node[below] {$14$} (1346);
    \path[draw,dashed,bend left=30pt] (1245) to node[above] {$25$} (2356);
    \path[draw,dashed,bend left=30pt] (2356) to node[above=5pt] {$36$} (1346);
    \draw[name path=L45] (1245) -- (45);
    \draw[name path=L24] (1245) -- (24);
    \draw[name path=L12] (1245) -- (12);
    \draw[name path=L15] (1245) -- (15);
    \draw[name path=L16] (1346) -- (16);
    \draw[name path=L13] (1346) -- (13);
    \draw[name path=L34] (1346) -- (34);
    \draw[name path=L46] (1346) -- (46);
    \draw[name path=L35] (2356) -- (35);
    \draw[name path=L23] (2356) -- (23);
    \draw[name path=L56] (2356) -- (56);
    \draw[name path=L26] (2356) -- (26);
    \fill[name intersections={of=L15 and L35,by=135}] (intersection-1) circle (2pt) node[xshift=-4pt, yshift=8pt] {$135$};
    \fill[name intersections={of=L15 and L56,by=156}] (intersection-1) circle (2pt) node[xshift=-6pt, yshift=-8.5pt] {$156$};
    \fill[name intersections={of=L12 and L23,by=123}] (intersection-1) circle (2pt) node[xshift=-9pt, yshift=-7pt] {$123$};
    \fill[name intersections={of=L12 and L26,by=126}] (intersection-1) circle (2pt) node[xshift=6pt, yshift=6pt] {$126$};
    \fill[name intersections={of=L24 and L23,by=234}] (intersection-1) circle (2pt) node[xshift=6pt, yshift=-8pt] {$234$};
    \fill[name intersections={of=L24 and L26,by=246}] (intersection-1) circle (2pt) node[xshift=12pt, yshift=2pt] {$246$};
    \fill[name intersections={of=L45 and L35,by=345}] (intersection-1) circle (2pt) node[xshift=4pt, yshift=-10pt] {$345$};
    \fill[name intersections={of=L45 and L56,by=456}] (intersection-1) circle (2pt) node[xshift=-10pt, yshift=-2pt] {$456$};
    \fill[fill=green!20!white,draw=green!80!black,rounded corners=2pt] (156) circle (5pt);
    \fill[fill=green!20!white,draw=green!80!black,rounded corners=2pt] (123) circle (5pt);
    \fill[fill=green!20!white,draw=green!80!black,rounded corners=2pt] (345) circle (5pt);
    \fill[fill=green!20!white,draw=green!80!black,rounded corners=2pt] (246) circle (5pt);
    \draw (0,9.94) edge [-,red,thick,postaction={decorate}] (9.25,11.62);
    \draw (9.3,11.6) edge [-,red,thick,postaction={decorate}] (11.9,9);
    \draw (11.9,8.98) edge [-,red,thick,postaction={decorate}] (9.02,5.55);
    \draw (9,5.57) edge [-,red,thick,postaction={decorate}] (0,10.07);
    \draw[name path=L45] (1245) -- (45);
    \draw[name path=L24] (1245) -- (24);
    \draw[name path=L12] (1245) -- (12);
    \draw[name path=L15] (1245) -- (15);
    \draw[name path=L16] (1346) -- (16);
    \draw[name path=L13] (1346) -- (13);
    \draw[name path=L34] (1346) -- (34);
    \draw[name path=L46] (1346) -- (46);
    \draw[name path=L35] (2356) -- (35);
    \draw[name path=L23] (2356) -- (23);
    \draw[name path=L56] (2356) -- (56);
    \draw[name path=L26] (2356) -- (26);
    \fill (156) circle (2pt);
    \fill (123) circle (2pt);
    \fill (345) circle (2pt);
    \fill (246) circle (2pt);
    \fill (1245) circle (2pt);
    \fill (126) circle (2pt);
    \fill (1346) circle (2pt);
    \fill (456) circle (2pt);
\end{tikzpicture}
 \caption{The elementary Tutte path of the fourth kind}
 \label{fig: kind4}
\end{figure}
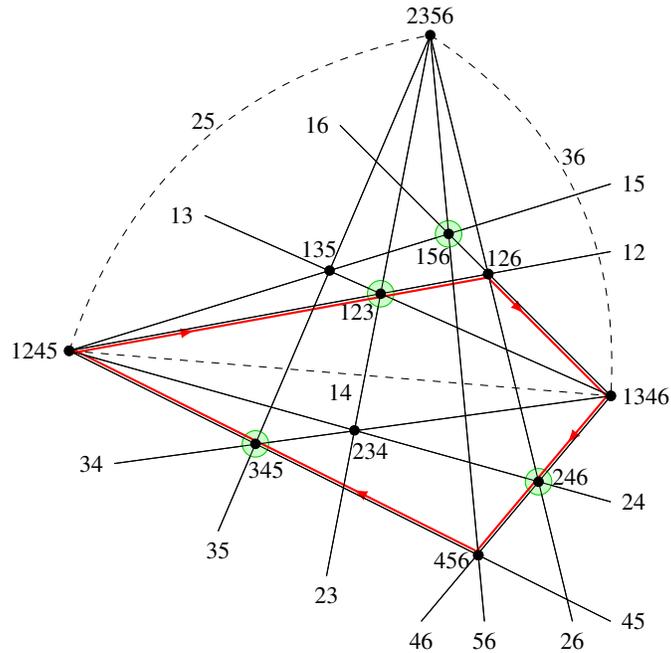

\subsubsection{Statement of the homotopy theorem}
We now introduce notations for path operations. If the last term of a Tutte path $\alpha$ is the first term
of a Tutte path $\beta$, their concatenation $\alpha\beta$ lists the common term only once. We write $\alpha^{-1}$ for the reversed path. Cyclic rotations and reversals of the elementary closed paths listed above are also regarded as elementary. 

Note that if
\[
\alpha = (H_1, H_2, \dots,H_i, \dots, H_k, H_1) \quad \textrm{and} \quad \beta = (H_i, G_1, G_2,\dots, G_\ell, H_i) 
\]
are both closed Tutte paths off $\Gamma$, then so is
\[
\gamma = (H_1, \dots, H_{i-1}, H_i, G_1, G_2,\dots, G_\ell, H_i, H_{i+1}, \dots, H_k, H_1).
\]
Conversely, if $\beta$ and $\gamma$ are both closed Tutte paths off $\Gamma$, then so is $\alpha$. Inserting or deleting an elementary closed path off $\Gamma$ at a hyperplane, as in the above construction, is called an \emph{elementary deformation}.

\begin{df}
A closed Tutte path $\gamma$ off $\Gamma$ is \emph{null-homotopic} if a finite sequence of elementary deformations transforms it into the one-term path at its origin. In this case, we say that $\gamma$ is \emph{decomposed into elementary Tutte paths}.
\end{df}

\begin{df}
Two Tutte paths $\gamma_1,\gamma_2$ off $\Gamma$ with the same origin and terminus are \emph{homotopic}, denoted $\gamma_1\sim\gamma_2$, if a finite sequence of elementary deformations off $\Gamma$ transforms one into the other. 
\end{df}

A null-homotopic closed path may therefore be inserted or
deleted inside any other path with matching endpoints. In particular, if $\alpha\beta_1\gamma$ is a Tutte path off $\Gamma$ and $\beta_1\beta_2^{-1}$ is a null-homotopic closed Tutte path off $\Gamma$, then we have $\alpha\beta_1\gamma \sim \alpha\beta_2\gamma$.

\begin{thm}[Tutte's homotopy theorem, original version~\cite{Tutte58a}]\label{thm: Tutte's homotopy theorem}
Let $M$ be a connected matroid and let $(\Lambda,\Gamma)$ be a Tutte constellation of type $M$. Then every closed Tutte path in $(\Lambda,\Gamma)$ can be decomposed into elementary Tutte paths in $(\Lambda,\Gamma)$. 
\end{thm}

\subsubsection{Outline of the proof of the homotopy theorem}\label{sec:tutte-homotopy-outline}

The technical heart of the proof of \autoref{thm: Tutte's homotopy theorem} is the \emph{Special Lemma} (\autoref{theorem: good paths}), which shows that a special class of closed Tutte paths $\delta = (H_1, H_2, H_3, H_4, H_1),$ for which $H_1\cap H_2\cap H_3$ and $H_1\cap H_3\cap H_4$ are indecomposable corank 3 flats and $H_1\cap H_3$ is a decomposable corank 2 flat, is null-homotopic. The rank of the subconstellation in which these special closed Tutte paths are embedded can be arbitrarily large.

The Special Lemma is proved by contradiction, assuming that $\delta$ is not null-homotopic and that $\delta$ has the least corank among all paths which are not null-homotopic. Extensive casework is used to determine the structure of the lattice below the path $\delta$, based on results from the proof of the path theorem (\autoref{prop:conn-diamond} and \autoref{prop:conn-complement}) as well as the uniqueness of a decomposable corank $2$ flat in an indecomposable rank $3$ matroid (\autoref{theorem: connected corank 3 flats}). In most cases, the path $\delta$ is deformed to a path $\delta' = \delta_1\cdots  \delta_s$ such that every path $\delta_i$ lies on a flat of smaller corank than $\cork(F)$, where $F$ is the intersection of all hyperplanes in the path $\delta$. The remaining case is deformed to an elementary path of the fourth kind. This leads to a contradiction, because all paths $\delta_i$ are thus null-homotopic.

 Based on the Special Lemma, the general proof proceeds as follows. Given a closed Tutte path $\gamma = (H_1, H_2, \dots , H_k, H_1)$ off $\Gamma$, we perform induction on the corank $c(\gamma)$ of $F(\gamma) := \bigcap_{i = 1}^k H_i$. The base case $c(\gamma) = 1$ is clear. We assume $c(\gamma)>1$ and that the statement is true for all closed Tutte paths off $\Gamma$ with corank smaller than $c(\gamma)$. By~\autoref{prop:connected-intersection}, $F(\gamma)$ is indecomposable.
Thus, there exists an indecomposable flat $F'$ of corank $c(\gamma) - 1$ lying between $F(\gamma)$ and $H_1$. Let $\gamma' = (H_1, G_2, G_3, \dots, G_\ell, H_1)$ be any closed Tutte path off $\Gamma$ such that all terms in $\gamma'$ contain $F(\gamma)$. Write $u(\gamma')$ for the number of indices of the hyperplanes in $\gamma'$ that do not contain $F'$. If $u(\gamma')>0$, let $v(\gamma')$ be the corank of $G_{i-1} \cap G_i \cap G_{i+1}$, where $G_i$ is the first term in $\gamma'$ that does not contain $F'$, else let $v(\gamma') = 0$. We choose $\gamma'$ so that: 
\begin{enumerate}
\item \label{homotopy-proof-1} $\gamma'$ can be derived from $\gamma$ by a finite sequence of deformations using elementary Tutte paths off $\Gamma$.
\item \label{homotopy-proof-2} $u(\gamma')$ attains the minimum among all closed Tutte paths satisfying~(\ref{homotopy-proof-1}).
\item \label{homotopy-proof-3} $v(\gamma')$ attains the minimum among all closed Tutte paths satisfying~(\ref{homotopy-proof-2}).
\end{enumerate}
There are four cases to consider: $v(\gamma') = 2$, $v(\gamma') = 3$, $v(\gamma') > 3$, and $v(\gamma') = 0$. The first case is straightforward, and the second and the third cases can be handled using the Special Lemma. Elementary Tutte paths of the third kind are used to handle the case $v(\gamma') = 3.$ The strategy is similar to the proof of the Special Lemma: one uses tools from \autoref{appendix:flats} to determine the structure of the lattice above $F(\gamma)$ and deforms $\gamma'$ to a path $\gamma''$ which contradicts the choice of $\gamma'$. 

 One concludes, in the end, that the first three cases are impossible, and we must have $v(\gamma') = 0$ and $u(\gamma') = 0$. The latter condition means that all terms in $\gamma'$ contain $F'$, which is an indecomposable flat of corank $c(\gamma) - 1$. We conclude by induction that $\gamma'$, and hence $\gamma$, is null-homotopic.

\subsection{The extended homotopy theorem}
\label{subsection:extended-homotopy-theorem}

Recall that the definition of a Tutte path $\gamma$ in a Tutte constellation $\tau$ requires that any two consecutive terms in $\gamma$ intersect in an indecomposable corank $2$ flat. While this property is satisfied with respect to the corresponding subconstellation $\sigma$ for Tutte paths of the third and fourth kind, the corank $2$ flats of subconstellations of type $U_{2,2}$ and $U_{3,3}$ (first and second kind, respectively) are decomposable. 

In this section, we determine certain minimal extensions of both subconstellations which remedy this deficiency, leading to an extended version \autoref{thm:extended-homotopy-theorem} of Tutte's homotopy theorem which is more useful for applications. The extended homotopy theorem will involve
a number of new ``types'' (as opposed to ``kinds'') of elementary Tutte paths.

Note that the modular cut $\Gamma'$ of a subconstellation $(\Lambda',\Gamma')$ of type $N$ defines a single-element extension $\widehat N$ of $N$, whose isomorphism type we call the {\em extended type} of $(\Lambda',\Gamma')$ . We include a description of the extended type $\widehat N$ in the following list of subconstellations.

\medskip\noindent\textbf{Type 1.}
Elementary Tutte paths of type 1 are of the form $(1,2,1)$ in the Tutte subconstellation $(\Lambda',\Gamma')$ of type $N = U_{2,3}$ with $\Gamma'=\{123\}$; see~\autoref{fig: type123}. In this case, $\widehat N = U_{2,4}$.

\medskip\noindent\textbf{Type 2.}
Elementary Tutte paths of type 2 are of the form $(1,2,1)$ in the Tutte subconstellation $(\Lambda',\Gamma')$ of type $N = U_{2,3}$ with $\Gamma'=\{3,\ 123\}$; see~\autoref{fig: type123}. In this case, $\widehat N$ is the parallel extension $\tilde{U}_{2,3}$ of $U_{2,3}$ with parallel elements $3$ and $4$.

\medskip\noindent\textbf{Type 3.}
Elementary Tutte paths of type 3 are of the form $(1,2,3,1)$ in the Tutte subconstellation $(\Lambda', \Gamma')$ of type $N = U_{2,3}$ with $\Gamma' = \{123\}$; see~\autoref{fig: type123}. In this case, $\widehat N = U_{2,4}$.

\begin{figure}[t]
 \begin{tikzpicture}[x=1.0cm,y=0.8cm, font=\footnotesize,decoration={markings,mark=at position 0.6 with {\arrow{latex}}}]
  \node at (0,0)
  {
   \begin{tikzpicture}
   \node at (-0.5,2) {$\Lambda_{U_{2,3}}$};  
   \filldraw[fill=green!10!white,draw=green!80!black,rounded corners=2pt] (0.65,2.35) -- (0.65,1.65) -- (1.25,1.65) -- (1.8,1.1) -- (1.8,0.65) -- (2.2,0.65) -- (2.2,1.25) -- (1.3,2.35) -- cycle;
   \filldraw[fill=green!30!white,draw=green!80!black,rounded corners=2pt] (0.7,2.3) -- (0.7,1.7) -- (1.3,1.7) -- (1.3,2.3) -- cycle;
   \node (0) at (1,0) {$\emptyset$};  
   \node (1) at (0,1) {$1$};  
   \node (2) at (1,1) {$2$};  
   \node (3) at (2,1) {$3$};  
   \node (123) at (1,2) {$123$};  
   \draw (0) to (1);
   \draw (0) to (2);
   \draw (0) to (3);
   \draw (1) to (123);
   \draw (2) to (123);
   \draw (3) to (123);
   \end{tikzpicture} 
  };
  \node at (4,0)
  {
   \begin{tikzpicture}
  \node at (5,2) {\scriptsize Type 1};  
  \draw (5,1.08) edge [-,red,thick,postaction={decorate}] (6,1.08);
  \draw (6,0.92) edge [-,red,thick,postaction={decorate}] (5,0.92);
  \draw (4.5,1) -- (7.5,1);
  \filldraw (5,1) circle (2pt);  
  \node at (5,0.5) {$1$};  
  \filldraw (6,1) circle (2pt);  
  \node at (6,0.5) {$2$};  
  \filldraw (7,1) circle (2pt);  
  \node at (7,0.5) {$3$};  
   \end{tikzpicture} 
  };
  \node at (7.5,0)
  {
   \begin{tikzpicture}
  \node at (5,2) {\scriptsize Type 2};  
  \filldraw[fill=green!20!white,draw=green!80!black,rounded corners=2pt] (7,1) circle (6pt);
  \draw (5,1.08) edge [-,red,thick,postaction={decorate}] (6,1.08);
  \draw (6,0.92) edge [-,red,thick,postaction={decorate}] (5,0.92);
  \draw (4.5,1) -- (7.5,1);
  \filldraw (5,1) circle (2pt);  
  \node at (5,0.5) {$1$};  
  \filldraw (6,1) circle (2pt);  
  \node at (6,0.5) {$2$};  
  \filldraw (7,1) circle (2pt);  
  \node at (7,0.5) {$3$};     
   \end{tikzpicture} 
  };
  \node at (11,0)
  {
   \begin{tikzpicture}
  \node at (10,2) {\scriptsize Type 3};  
  \draw (10,1.08) edge [-,red,thick,postaction={decorate}] (11,1.08);
  \draw (11,1.08) edge [-,red,thick,postaction={decorate}] (12,1.08);
  \draw (12,0.92) edge [-latex,red,thick,] (10.65,0.92);
  \draw (12,0.92) edge [-,red,thick,] (10,0.92);
  \draw (9.5,1) -- (12.5,1);
  \filldraw (10,1) circle (2pt);  
  \node at (10,0.5) {$1$};  
  \filldraw (11,1) circle (2pt);  
  \node at (11,0.5) {$2$};  
  \filldraw (12,1) circle (2pt);  
  \node at (12,0.5) {$3$};  
   \end{tikzpicture} 
  };
\end{tikzpicture}
 \caption{Subconstellations and elementary Tutte paths of types 1--3}
 \label{fig: type123}
\end{figure}
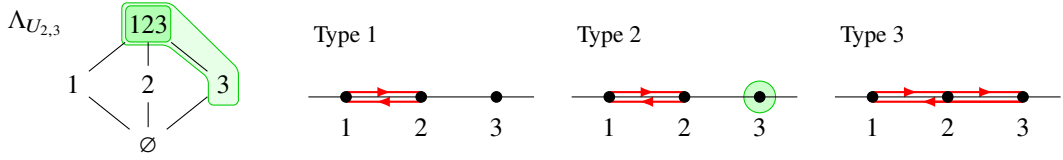

\medskip\noindent\textbf{Type 4.}
Elementary Tutte paths of type 4 are of the form $(12,13,23,12)$ in the Tutte subconstellation $(\Lambda',\Gamma')$ of type $N = U_{3,4}$ with $\Gamma' = \{1234\}$; see~\autoref{fig: type456}. In this case, $\widehat N=U_{3,5}$.
 
\medskip\noindent\textbf{Type 5.}
Elementary Tutte paths of type 5 are of the form $(12,13,23,12)$ in the Tutte subconstellation $(\Lambda',\Gamma')$ of type $N = U_{3,4}$ with $\Gamma' = \{14,1234\}$; see\ \autoref{fig: type456}. In this case, $\widehat N = C_5$, the matroid on $E = \{1,2,3,4,5\}$ whose set of bases is $\binom{E}{3} - \{123\}$.

\medskip\noindent\textbf{Type 6.}
Elementary Tutte paths of type 6 are of the form $(12,13,23,12)$ in the Tutte subconstellation $(\Lambda',\Gamma')$ of type $N=U_{3,4}$ with $\Gamma'=\{4,14,24,34,1234\}$; see\ \autoref{fig: type456}. In this case, $\widehat N$ is the parallel extension $\tilde{U}_{3,4}$ of $N$ with parallel elements $4$ and $5$.

\begin{figure}[t]
\begin{tikzpicture}[x=1.0cm,y=0.8cm, font=\footnotesize,decoration={markings,mark=at position 0.6 with {\arrow{latex}}}]
   \node at (0,0)
  {
   \begin{tikzpicture}
   \node at (-1,3) {$\Lambda_{U_{3,4}}$};  
   \filldraw[fill=green!10!white,draw=green!80!black,rounded corners=2pt] (2,3.4) -- (2,2.6) -- (2.9,1.65) -- (3.9,0.7) -- (4.1,0.7) -- (5.25,1.7) -- (5.25,2.3) -- (3,3.4) -- cycle;
   \filldraw[fill=green!20!white,draw=green!80!black,rounded corners=2pt] (2.05,3.35) -- (2.05,2.65) -- (2.95,1.7) -- (3.25,1.7) -- (3.25,2.3) -- (2.95,3.35) -- cycle;
   \filldraw[fill=green!30!white,draw=green!80!black,rounded corners=2pt] (2.1,3.3) -- (2.1,2.7) -- (2.9,2.7) -- (2.9,3.3) -- cycle;
   \node (0) at (2.5,0) {$\emptyset$};  
   \node (1) at (1,1) {$1$};  
   \node (2) at (2,1) {$2$};  
   \node (3) at (3,1) {$3$};  
   \node (4) at (4,1) {$4$};  
   \node (12) at (0,2) {$12$};  
   \node (13) at (1,2) {$13$};  
   \node (23) at (2,2) {$23$};  
   \node (14) at (3,2) {$14$};  
   \node (24) at (4,2) {$24$};  
   \node (34) at (5,2) {$34$};  
   \node (1234) at (2.5,3) {$1234$};  
   \draw (0) to (1);
   \draw (0) to (2);
   \draw (0) to (3);
   \draw (0) to (4);
   \draw (1) to (12);
   \draw (1) to (13);
   \draw (1) to (14);
   \draw (2) to (12);
   \draw (2) to (23);
   \draw (2) to (24);
   \draw (3) to (13);
   \draw (3) to (23);
   \draw (3) to (34);
   \draw (4) to (14);
   \draw (4) to (24);
   \draw (4) to (34);
   \draw (12) to (1234);
   \draw (13) to (1234);
   \draw (14) to (1234);
   \draw (23) to (1234);
   \draw (24) to (1234);
   \draw (34) to (1234);
   \end{tikzpicture} 
  };
  \node at (7,2) {Types 4--6};  
  \node at (9,0)
  {
   \begin{tikzpicture}[x=0.4cm,y=0.4cm]
    \filldraw[fill=green!10!white,draw=green!80!black,rounded corners=7pt] (1.7,4.3) -- (2.3,5.3) -- (12.8,-1) -- (12.2,-2) -- cycle;
    \filldraw[fill=green!20!white,draw=green!80!black,rounded corners=2pt] (10,0) circle (5.5pt);
    \node at (-0.5,2.1) {$1$};  
    \node at (7.5,4.4) {$2$};  
    \node at (10.5,2.1) {$3$};  
    \node at (2.5,4.5) {$4$};
    \draw[name path=L1] (13,-0.6) -- (0,2);
    \draw[name path=L2] (-2,-1.2) -- (7,4.2);
    \draw[name path=L3] (-3,-0.6) -- (10,2);
    \draw[name path=L4] (12,-1.2) -- (3,4.2);
    \fill[name intersections={of=L1 and L2,by=12}] (intersection-1) circle (2pt) node[above=2pt] {$12$};
    \fill[name intersections={of=L1 and L3,by=13}] (intersection-1) circle (2pt) node[below=2pt] {$13$};
    \fill[name intersections={of=L1 and L4,by=14}] (intersection-1) circle (2pt) node[below=5pt] {$14$};
    \fill[name intersections={of=L2 and L3,by=23}] (intersection-1) circle (2pt) node[below=5pt] {$23$};
    \fill[name intersections={of=L2 and L4,by=24}] (intersection-1) circle (2pt) node[above=4pt] {$24$};
    \fill[name intersections={of=L3 and L4,by=34}] (intersection-1) circle (2pt) node[above=5pt] {$34$};
    \draw (12) edge [-,red,thick,postaction={decorate}] (13);
    \draw (13) edge [-,red,thick,postaction={decorate}] (23);
    \draw (23) edge [-,red,thick,postaction={decorate}] (12);
    \fill (12) circle (2pt);
    \fill (13) circle (2pt);
    \fill (34) circle (2pt);
    \fill (24) circle (2pt);
    \fill (23) circle (2pt);
   \end{tikzpicture} 
  };
\end{tikzpicture}
 \caption{Tutte subconstellation and elementary Tutte path of types 4--6}
 \label{fig: type456}
\end{figure}
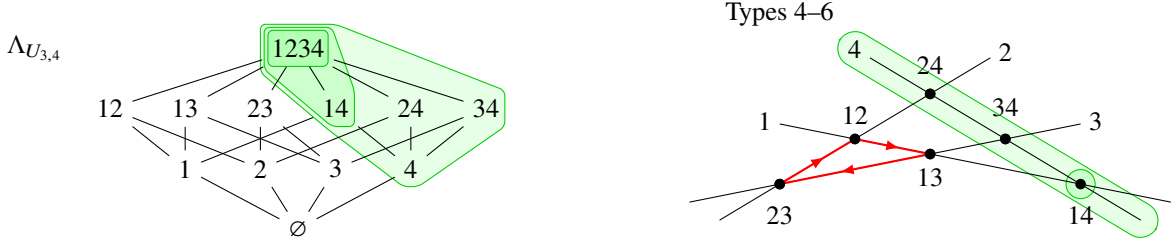

\medskip\noindent\textbf{Type 7.}
Elementary Tutte paths of type 7 are of the form $(12,24,34,13,12)$ in the Tutte subconstellation $(\Lambda',\Gamma')$ of type $N=U_{3,4}$ with $\Gamma'
 = \{23,14,1234\}$; see\ \autoref{fig: type7}. In this case, $\widehat N = M(K_4^-)$.

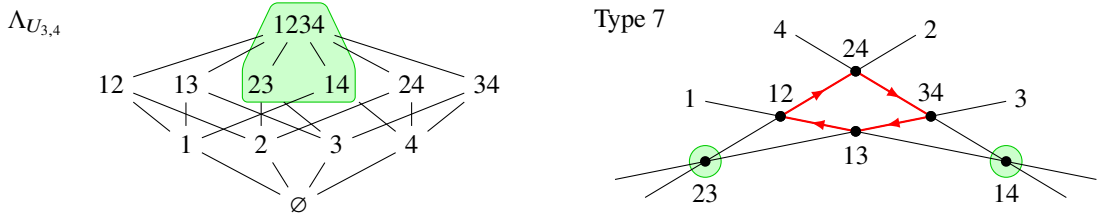
\begin{figure}[t]
\begin{tikzpicture}[x=1.0cm,y=0.8cm, font=\footnotesize,decoration={markings,mark=at position 0.6 with {\arrow{latex}}}]
  \node at (0,0)
  {
   \begin{tikzpicture}
  \node at (-1,3) {$\Lambda_{U_{3,4}}$};  
   \filldraw[fill=green!20!white,draw=green!80!black,rounded corners=2pt] (2.15,3.35) -- (1.75,2.3) -- (1.75,1.7) -- (3.25,1.7) -- (3.25,2.3) -- (2.85,3.35) -- cycle;
   \node (0) at (2.5,0) {$\emptyset$};  
   \node (1) at (1,1) {$1$};  
   \node (2) at (2,1) {$2$};  
   \node (3) at (3,1) {$3$};  
   \node (4) at (4,1) {$4$};  
   \node (12) at (0,2) {$12$};  
   \node (13) at (1,2) {$13$};  
   \node (23) at (2,2) {$23$};  
   \node (14) at (3,2) {$14$};  
   \node (24) at (4,2) {$24$};  
   \node (34) at (5,2) {$34$};  
   \node (1234) at (2.5,3) {$1234$};  
   \draw (0) to (1);
   \draw (0) to (2);
   \draw (0) to (3);
   \draw (0) to (4);
   \draw (1) to (12);
   \draw (1) to (13);
   \draw (1) to (14);
   \draw (2) to (12);
   \draw (2) to (23);
   \draw (2) to (24);
   \draw (3) to (13);
   \draw (3) to (23);
   \draw (3) to (34);
   \draw (4) to (14);
   \draw (4) to (24);
   \draw (4) to (34);
   \draw (12) to (1234);
   \draw (13) to (1234);
   \draw (14) to (1234);
   \draw (23) to (1234);
   \draw (24) to (1234);
   \draw (34) to (1234);
   \end{tikzpicture} 
  };
  \node at (5,1.5) {Type 7};
    \node at (8,0)
  {
   \begin{tikzpicture}[x=0.4cm,y=0.4cm]
    \filldraw[fill=green!20!white,draw=green!80!black,rounded corners=2pt] (0,0) circle (6pt);
    \filldraw[fill=green!20!white,draw=green!80!black,rounded corners=2pt] (10,0) circle (6pt);
    \node at (2.5,4.4) {$4$};  
    \node at (7.5,4.4) {$2$};  
    \node at (10.5,2.1) {$3$};  
    \node at (-0.5,2.1) {$1$};  
    \draw[name path=L4] (12,-1.2) -- (3,4.2);
    \draw[name path=L2] (-2,-1.2) -- (7,4.2);
    \draw[name path=L3] (-3,-0.6) -- (10,2);
    \draw[name path=L1] (13,-0.6) -- (0,2);
    \fill[name intersections={of=L1 and L2,by=12}] (intersection-1) circle (2pt) node[above=2pt] {$12$};
    \fill[name intersections={of=L1 and L3,by=13}] (intersection-1) circle (2pt) node[below=2pt] {$13$};
    \fill[name intersections={of=L1 and L4,by=14}] (intersection-1) circle (2pt) node[below=5pt] {$14$};
    \fill[name intersections={of=L2 and L3,by=23}] (intersection-1) circle (2pt) node[below=5pt] {$23$};
    \fill[name intersections={of=L2 and L4,by=24}] (intersection-1) circle (2pt) node[above=2pt] {$24$};
    \fill[name intersections={of=L3 and L4,by=34}] (intersection-1) circle (2pt) node[above=2pt] {$34$};
    \draw (12) edge [-,red,thick,postaction={decorate}] (24);
    \draw (24) edge [-,red,thick,postaction={decorate}] (34);
    \draw (34) edge [-,red,thick,postaction={decorate}] (13);
    \draw (13) edge [-,red,thick,postaction={decorate}] (12);
    \fill (12) circle (2pt);
    \fill (13) circle (2pt);
    \fill (34) circle (2pt);
    \fill (24) circle (2pt);
   \end{tikzpicture} 
  };
\end{tikzpicture}
 \caption{Tutte subconstellation and elementary Tutte path of type 7}
 \label{fig: type7}
\end{figure}

\medskip\noindent\textbf{Type 8.}
Elementary Tutte paths of type 8 are of the form $(126,135,234,126)$ in the Tutte subconstellation $(\Lambda',\Gamma')$ of type $N = M(K_4)$ with $\Gamma' = \{14,25,36,123456\}$; see\ \autoref{fig: type8}. In this case, $\widehat N=F_7$.

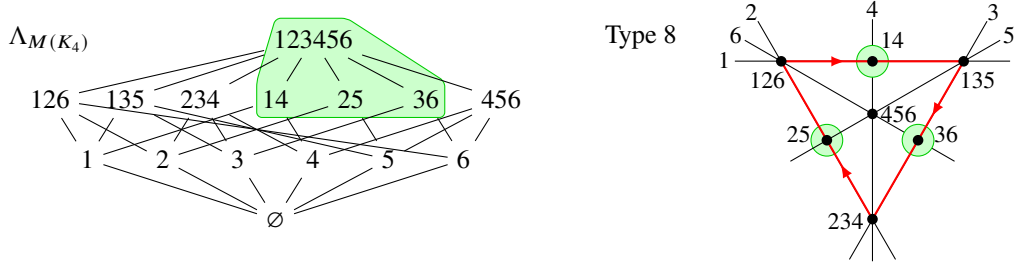
\begin{figure}[t]
\begin{tikzpicture}[x=1.0cm,y=0.8cm, font=\footnotesize,decoration={markings,mark=at position 0.35 with {\arrow{latex}}}]
    \node at (0,0)
  {
   \begin{tikzpicture}
   \node at (-1,3) {$\Lambda_{M(K_4)}$};  
   \filldraw[fill=green!20!white,draw=green!80!black,rounded corners=2pt] (2,3.35) -- (1.75,2.3) -- (1.75,1.7) -- (4.25,1.7) -- (4.25,2.3) -- (2.95,3.35) -- cycle;
   \node (0) at (2,0) {$\emptyset$};  
   \node (1) at (-0.5,1) {$1$};  
   \node (2) at (0.5,1) {$2$};  
   \node (3) at (1.5,1) {$3$};  
   \node (4) at (2.5,1) {$4$};  
   \node (5) at (3.5,1) {$5$};  
   \node (6) at (4.5,1) {$6$};  
   \node (126) at (-1,2) {$126$};  
   \node (135) at (0,2) {$135$};  
   \node (234) at (1,2) {$234$};  
   \node (14) at (2,2) {$14$};  
   \node (25) at (3,2) {$25$};  
   \node (36) at (4,2) {$36$};  
   \node (456) at (5,2) {$456$};  
   \node (123456) at (2.5,3) {$123456$};  
   \foreach \to/\from in {0/1,0/2,0/3,0/4,0/5,0/6,1/126,2/126,6/126,1/135,3/135,5/135,2/234,3/234,4/234,1/14,4/14,2/25,5/25,3/36,6/36,4/456,5/456,6/456,126/123456,135/123456,234/123456,14/123456,25/123456,36/123456,456/123456}
      \draw [-] (\to)--(\from);
   \end{tikzpicture} 
  };
  \node at (5,1.5) {Type 8};
  \node at (8,0)
  {
   \begin{tikzpicture}[x=0.7cm,y=0.7cm, font=\scriptsize]
    \fill (  0:0) circle (2pt);
    \foreach \a in {1,...,3}
    {
     \filldraw[fill=green!20!white,draw=green!80!black,rounded corners=2pt] ( 90+\a*120:1) circle (6pt);
     \draw (30+\a*120:2.8) -- (210+\a*120:1.8);
     \draw (21+\a*120:2.8) -- (159+\a*120:2.8);
     \draw (150+\a*120:2) edge [-,red,thick,postaction={decorate}] (30+\a*120:2);
    };  
    \foreach \a in {1,...,3}
    {
     \fill ( 30+\a*120:2) circle (2pt);
     \fill ( 90+\a*120:1) circle (2pt);
    };  
    \node at (160:3.0) {$1$};
    \node at (150:3.0) {$6$};
    \node at (140:3.0) {$2$};
    \node at ( 90:2.0) {$4$};
    \node at ( 40:3.0) {$3$};
    \node at ( 30:3.0) {$5$};
    \node at ( 18:2.1) {$135$};
    \node at (162:2.1) {$126$};
    \node at (256:2.1) {$234$};
    \node at ( 75:1.45) {$14$};
    \node at (195:1.45) {$25$};
    \node at (-15:1.45) {$36$};
    \node at ( 0:0.5) {$456$};
   \end{tikzpicture} 
  };
\end{tikzpicture}
 \caption{Tutte constellation and elementary path of type 8}
 \label{fig: type8}
\end{figure}
 
\medskip\noindent\textbf{Type 9.}
Elementary Tutte paths of type 9 are of the form $(1245,126,1346,456,1245)$ in the Tutte subconstellation $(\Lambda',\Gamma')$ of type $N=M(K_{2,3})$ with $\Gamma'=\{123,156,246,345,123456\}$ such that the corank $2$ flats $14$, $25$ and $36$ are decomposable in $\Lambda$; see\ \autoref{fig: Tutte constellation of kind4} and \autoref{fig: kind4}. In this case, $\widehat N=F_7^\ast$.

\begin{thm}[Extended homotopy theorem]\label{thm:extended-homotopy-theorem}
 Let $M$ be a connected matroid and let $(\Lambda,\Gamma)$ be a Tutte constellation of type $M$. Then every closed Tutte path in $(\Lambda,\Gamma)$ can be decomposed into elementary Tutte paths in $(\Lambda,\Gamma)$ of types 1--9.
\end{thm}

As an immediate consequence, we obtain the following easier-to-state result, which suffices for our application to the fundamental presentation of the foundation of a matroid.

\begin{cor}[Extended homotopy theorem, simplified version]\label{cor: extended minors of the homotopy theorem}
 Let $M$ be a connected matroid and let $(\Lambda,\Gamma)$ be a Tutte constellation of type $M$. Then every closed Tutte path in $(\Lambda,\Gamma)$ can be decomposed into elementary Tutte paths in $(\Lambda,\Gamma)$ whose extended type belongs to $\{ \tilde{U}_{2,3},U_{2,4},M(K_4^-),\tilde{U}_{3,4},C_5,U_{3,5},F_7,F_7^\ast  \}$.
\end{cor}

The proof of \autoref{thm:extended-homotopy-theorem} rests on the following result.

\begin{prop}\label{prop:extended-second-kind}
Let $M$ be a connected matroid and let $(\Lambda,\Gamma)$ be a Tutte constellation of type $M$. Let $\gamma = (H_{12}, H_{13}, H_{23}, H_{12})$ be an elementary Tutte path of the second kind with $F = H_{12} \cap H_{13} \cap H_{23}$ being a corank $3$ flat of $M$. Then $\gamma$ can be decomposed into elementary Tutte paths in $(\Lambda, \Gamma)$ of types 1--6 and 8. 
\end{prop}

\begin{proof}
Denote by $L_1$ the intersection of $H_{12}$ and $H_{13}$, $L_2$ the intersection of $H_{12}$ and $H_{23}$, and $L_3$ the intersection of $H_{13}$ and $H_{23}$. For each $i = 1, 2, 3$, since $L_i$ is an indecomposable corank $2$ flat of $M$ (by the definition of a Tutte path),
there exists a hyperplane $H_i' \in \cH$ not in $\gamma$ with $L_i \subset H_i'$. Since the join of any two of $L_1, L_2$, and $L_3$ is in $\gamma$, $H_1', H_2'$, and $H_3'$ are necessarily pairwise distinct. Note $H_1' \cap H_2' \cap H_3' \supseteq L_1 \cap L_2 \cap L_3 = F$. In the following, we consider two cases based on the corank of $H_1' \cap H_2' \cap H_3'$:

\medskip\noindent\textbf{Case 1.} Suppose $L = H_1' \cap H_2' \cap H_3'$ is of corank $2$. Then $\{L_1, L_2, L_3, L\}$ determines an upper sublattice of type $N = U_{3,4}$ in which every corank $2$ flat is contained in exactly $3$ hyperplanes. Since $H_{12}, H_{23}, H_{13} \notin \Gamma$, the only possibilities for $\Gamma' = \Gamma_N \cap \Lambda'$ are $\{E\}$, $\{E, H_i'\}$ for exactly one $i \in \{1, 2, 3\}$, and $\{E, H_1', H_2', H_3', L\}$, which correspond to types 4--6, respectively. (Using the notation from~\autoref{fig: type456}, we have $L = L_4$ and $H_i' = H_{i4}$.)

\medskip\noindent\textbf{Case 2.} If $H_1' \cap H_2' \cap H_3'$ is of corank $3$, then it must be equal to $F$. For each $i = 1, 2, 3$, pick $x_{i+3} \in H_i' - L_i$ and form $L_{i+3} = \gen{Fx_{i+3}}$. If at least one of $L_4, L_5$, and $L_6$, say $L = L_4$, is not contained in any hyperplane in $\gamma$, we can replace each $H_i'$ by $L_i \lor L$ and reduce the problem to the case where $L = H_1' \cap H_2' \cap H_3'$ is of corank $2$. Therefore, we can assume that each of $L_4, L_5$ and $L_6$ is contained in one of $H_{12}, H_{13}$, and $H_{23}$. Since $L_1$ is contained in $H_{12}, H_{13}$, and $H_1'$, the corank $2$ flat $L_4$ must be contained only in $H_{23}$; similarly, we have $L_5 \subset H_{13}$ and $L_6 \subset H_{12}$. We write $H_{126}$ for $H_{12}$, $H_{135}$ for $H_{13}$, and $H_{234}$ for $H_{23}$.

\medskip\noindent\textbf{Case 2a.} If there exists some $H_i'$, say $H_1'$, off $\Gamma$, then we can perform the following deformation of closed paths via elementary Tutte paths:
\[
\begin{split}
\gamma = (H_{126}, H_{135}, H_{234}, H_{126})
& \sim (H_{126}, H_1', H_{135}, H_{234}, H_{126}) \\
& \sim (H_{126}, H_{234}, H_1', H_{135}, H_{234}, H_{126}) \\
& \sim (H_{126}, H_{234}, H_{126}) \\
& \sim (H_{126}). 
\end{split}
\]
In this way, the closed path $\gamma$ is decomposed into elementary Tutte paths $(H_{126}, H_1', H_{135}, H_{126})$, $(H_{126}, H_{234}, H_1',H_{126})$, $(H_{234}, H_1', H_{135}, H_{234})$, and $(H_{126}, H_{234}, H_{126})$ in $(\Lambda, \Gamma)$, which are of types 1--6. 

\medskip\noindent\textbf{Case 2b.} Suppose $\Gamma$ contains $H_1', H_2'$, and $H_3'$. Denote $H_{56} = L_5 \lor L_6$, $H_{46} = L_4 \lor L_6$, and $H_{45} = L_4 \lor L_5$. Clearly, none of $H_{56}, H_{46}$, and $H_{45}$ is in $\Gamma$. 

Assume there exist two distinct members among $H_{56}, H_{46}$, and $H_{45}$; say $H_{56} \ne H_{46}$. Then all three hyperplanes must be pairwise distinct. In this case, the three closed Tutte paths $(H_{46}, H_{234}, H_{45}, H_{46})$, $(H_{45}, H_{135}, H_{56}, H_{45})$, and $(H_{56}, H_{126}, H_{46}, H_{56})$ are elementary of type 3, and the four closed Tutte paths $(H_{56}, H_{45}, H_{46}, H_{56})$, $(H_{126}, H_{56}, H_{135}, H_{126})$, $(H_{135}, H_{45}, H_{234}, H_{135})$, 
and $(H_{234}, H_{46}, H_{126}, H_{234})$ fall into Case 1, and hence are elementary in $(\Lambda, \Gamma)$ of types 4--6; for an illustration, see~\autoref{fig: ExtendedCase2b}. We perform the following deformation of closed paths via elementary Tutte paths:
\[
\begin{split}
\gamma = (H_{126}, H_{135}, H_{234}, H_{126}) 
& \sim (H_{126}, H_{56}, H_{135}, H_{45}, H_{234}, H_{46}, H_{126})\\
& \sim (H_{126}, H_{56}, H_{45}, H_{46}, H_{126}) \\
& \sim (H_{126}, H_{56}, H_{46}, H_{126}) \\
& \sim (H_{126}). 
\end{split}
\]
Thus, the closed path $\gamma$ is decomposed into elementary Tutte paths in $(\Lambda, \Gamma)$ of types 3--6. 

\begin{figure}[t]
 \begin{tikzpicture}[x=0.9cm,y=0.9cm, font=\footnotesize,decoration={markings,mark=at position 0.6 with {\arrow{latex}}}] 
  \node at (0,-0.28) {$H_{56}$};
  \node at (-1.45, 1.732) {$H_{126}$}; 
  \node at (1.47, 1.732) {$H_{135}$}; 
  \node at (-2.40, 3.474) {$H_{46}$}; 
  \node at (0, 3.754) {$H_{234}$}; 
  \node at (2.40, 3.474) {$H_{45}$};
  \draw (0,0) -- (-2,3.464);
  \draw (0,0) -- (2,3.464);
  \draw (-2,3.464) -- (2,3.464);
  \draw (-1,1.732) edge [-,red,thick,postaction={decorate}] (1, 1.732);
  \draw (1,1.732) edge [-,red,thick,postaction={decorate}] (0, 3.464);
  \draw (0,3.464) edge [-,red,thick,postaction={decorate}] (-1, 1.732);
  \filldraw (0, 0) circle (2pt);
  \filldraw (-1, 1.732) circle (2pt);
  \filldraw (1, 1.732) circle (2pt);
  \filldraw (-2, 3.464) circle (2pt);
  \filldraw (2, 3.464) circle (2pt);
  \filldraw (0, 3.464) circle (2pt);
\end{tikzpicture}
 \caption{Elementary Tutte paths in Case 2b in which $H_{56} \ne H_{46}$}
 \label{fig: ExtendedCase2b}
\end{figure}
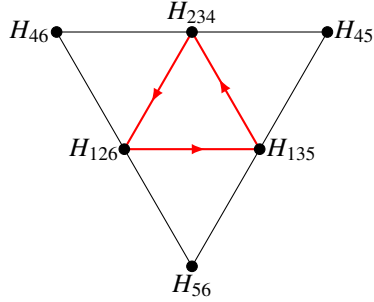

Finally, if $H_{56} = H_{46} = H_{45} =: H_{456}$, then the six corank $2$ flats $L_1, \dots, L_6$ together determine an upper sublattice of type $N = M(K_4)$ with $\Gamma' = \Lambda_N \cap \Gamma = \{E, H_1', H_2', H_3'\}$. Therefore, $\gamma = (H_{126}, H_{135}, H_{234}, H_{126})$ is of type $8$. 
(Adapting the notations from~\autoref{fig: type8}, for each $i = 1, 2, 3$, the hyperplane $H_i'$ is $H_{i, i+3}$.) 
\end{proof}

\begin{proof}[Proof of \autoref{thm:extended-homotopy-theorem}]
 The definitions of elementary Tutte paths in $M$ of the first, third, and fourth kinds force them to be elementary Tutte paths in $(\Lambda, \Gamma)$ of types 1--2, 7, and 9, respectively. Similarly, if $\gamma = (H_1, H_2, H_3, H_1)$ is an elementary Tutte path of the second kind with $H_1 \cap H_2 \cap H_3$ having corank $2$, then it is an elementary Tutte path in $(\Lambda, \Gamma)$ of type 3. ~\autoref{prop:extended-second-kind} verifies that every elementary Tutte path $\gamma = (H_1, H_2, H_3, H_1)$ of the second kind with $H_1 \cap H_2 \cap H_3$ having corank $3$ can be decomposed into elementary Tutte paths in $(\Lambda, \Gamma)$ of types 1--6 and 8. This shows that every elementary Tutte path in $M$ off $\Gamma$, and hence (by~\autoref{thm: Tutte's homotopy theorem}) every closed Tutte path in $(\Lambda,\Gamma)$, can be decomposed into elementary Tutte paths in $(\Lambda, \Gamma)$ of types 1--9. 
\end{proof}

\begin{rem}
 The proof of \autoref{prop:extended-second-kind} classifies, in particular, the minimal extensions of the Tutte constellation $(U_{3,3},\{E\})$ in which the elementary Tutte path of the second kind is indeed a Tutte path (i.e.,\ the corresponding corank $2$ flats are indecomposable) as
 \[
  \big(U_{3,4}, \{E\}\big), \quad \big(U_{3,4}, \{E,H\}\big), \quad \big(U_{3,4}, \{E,H,H',H'',L\}\big), \quad \big(M(K_4), \{E,H,H',H''\}\big),
 \]
 where $H\cap H'\cap H''=L$ in the third case and $H\cap H'\cap H''=\emptyset$ in the fourth case.
\end{rem}


\section{Presentations for the foundation}
\label{section: presentation of the foundation}

The foundation of a matroid belongs to a category of algebraic objects called \emph{pastures}.
In this section, we define pastures and describe their basic properties, we define the foundation of a matroid, and we discuss a couple of important presentations of the foundation by generators and relations.

\subsection{Pastures}
\label{subsection: pastures}

In \cite{Baker-Bowler19}, Baker and Bowler introduced an algebraic category called \emph{tracts}, which generalize fields, as a tool for unifying various definitions in matroid theory such as oriented matroids, valuated matroids, and matroids over partial fields. In particular, they defined two different kinds of matroids over a tract, called \emph{weak} and \emph{strong} matroids. In this paper, we are concerned exclusively with weak matroids, and for weak matroids the closely related category of \emph{pastures}, as defined in \cite{Baker-Lorscheid20}, is more convenient to use than tracts. Like tracts, pastures also generalize fields, and they form a category with various pleasant properties, which allows one to employ arguments from category theory in order to streamline and strengthen various techniques from matroid theory.

\subsubsection{Definitions}
A \emph{pointed group} is a multiplicatively written commutative monoid $P$ with a neutral element $1$ and an absorbing element $0 \neq 1$ such that every nonzero element $a\neq0$ of $P$ is invertible, i.e., $ab=1$ for some $b\in P$. We call $P^\times=P-\{0\}$ the \emph{unit group of $P$}. The symmetric group $S_3$ acts on $P^3$ by permutation of coordinates. We denote the equivalence classes of $\Sym^3(P):=P^3/S_3$ by $a+b+c:=[(a,b,c)]$. We often simply write $a+b$ for $a+b+0$ and $a$ for $a+0+0$. 

A \emph{pasture} is a pointed group $P$ together with a subset $N_P$ of $\Sym^3(P)$, called the \emph{null set of $P$}, which is invariant under multiplication by elements of $P$ and contains, for every $a\in P$, a unique $b\in P$, called the \emph{additive inverse of $a$}, such that $a+b\in N_P$. A \emph{pasture morphism $g\colon P\to Q$} is a multiplicative map between pastures $P$ and $Q$ that preserves $1$ and $0$ such that $g(a)+g(b)+g(c)\in N_Q$ for all $a+b+c\in N_P$. This defines the category $\Pastures$ of all pastures.

We write $-a=b$ for the additive inverse of $a$ and $a-b$ for $a+(-b)$. It follows from the axioms that $(-1)^2=1$, $-a=(-1)\cdot a$ for all $a \in P$, and $a\in N_P$ if and only if $a=0$. In addition, it follows that a pasture morphism $g\colon P\to Q$ must preserve additive inverses, i.e., $g(-a)=-g(a)$.

\subsubsection{First examples}
\label{subsubsection: first examples of pastures}

Every field $K$ is naturally a pasture with null set $N_K=\{a+b+c\mid a+b+c=0\textrm{ in }K\}$. In fact, this defines a fully faithful embedding $\Fields\to\Pastures$. For example, $\F_2=\{0,1\}$ has the null set $N_{\F_2}=\{0, \ 1+1\}$ and $\F_3=\{0,1,-1\}$ has the null set $\{0,\ 1-1,\ 1+1+1,\ -1-1-1\}$.

The \emph{regular partial field} is the pasture $\Funpm=\{0,1,-1\}$ with the obvious multiplication, together with the null set $N_\Funpm=\{0,\ 1-1\}$. The regular partial field is the initial object of $\Pastures$.

The \emph{Krasner hyperfield} is the pasture $\K=\{0,1\}$ with the obvious multiplication, together with the null set $N_\K=\{0,\ 1+1,\ 1+1+1\}$. Note that $-1=1$ in $\K$. The Krasner hyperfield is the terminal object of $\Pastures$.

The \emph{sign hyperfield} is the pasture $\S=\{0,1,-1\}$ with the obvious multiplication, together with the null set $N_\S=\{0,\ 1-1,\ 1+1-1,\ 1-1-1\}$. The sign map $\sign\colon \R\to\S$ is a morphism of pastures.

The \emph{tropical hyperfield} is the pasture $\T=\R_{\geq0}$ with the obvious multiplication, together with the null set $N_\T=\{a+b+b\mid 0\leq a\leq b\}$. The tautological inclusion $\K\to\T$ is a pasture morphism. For a field $K$, a map $v\colon K\to\T=\R_{\geq0}$ is a pasture morphism if and only if it is a non-Archimedean absolute value.

More generally, $\Pastures$ contains the categories of both partial fields and hyperfields naturally as full subcategories; for details, see~\cite[Section 2.1.5]{Baker-Lorscheid20}.

\subsubsection{Tensor products}
The category $\Pastures$ is complete and cocomplete. In particular, it has a coproduct $P_1\otimes P_2$, called the \emph{tensor product}, which is characterized by the universal property that
\[
 \Hom(P_1\otimes P_2,\ Q) \ = \ \Hom(P_1,\ Q) \ \times \ \Hom(P_2,\ Q)
\]
for every pasture $Q$, functorially in $Q$.
Details on the construction of the tensor product, and of other limits and colimits in the category of pastures, can be found in~\cite{Creech21}.

\subsubsection{Free algebras and quotients}
For $n\geq0$, the \emph{free algebra on $x_1,\dotsc,x_n$} is the pasture
\[
 \Funpm(x_1,\dotsc,x_n) \ := \ \{0\} \ \cup \{\pm x_1^{e_1}\dotsb x_n^{e_n}\mid e_1,\dots,e_n\in\Z\}
\]
with the obvious multiplication, together with the null set $\{a-a\mid a\in\Funpm(x_1,\dotsc,x_n)\}$. It satisfies the universal property
\[
\Hom\big(\Funpm(x_1,\dotsc,x_n),\ Q \big) \ = \ \Maps\big(\{x_1,\dotsc,x_n\},\ Q^\times\big)
\]
for every pasture $Q$, functorially in $Q$. 
If $g : \{x_1,\dotsc,x_n\}\to Q^\times$ is a set-theoretic map, we write $\Phi(g)$ for the corresponding morphism from $\Funpm(x_1,\dotsc,x_n)$ to $Q$.

Let $S\subset\Sym^3\big(\Funpm(x_1,\dotsc,x_n)\big)$, and assume that $S$ does not contain any term of the form $a+0+0$ with $a\neq0$. The \emph{quotient} $\pastgen{\Funpm(x_1,\dotsc,x_n)}S$ of $\Funpm(x_1,\dotsc,x_n)$ by $S$ can be characterized by the universal property 
\begin{multline*}
 \Hom\big(\pastgen{\Funpm(x_1,\dotsc,x_n)}S,\ Q\big) \\
 = \ \bigg\{ g\colon \{x_1,\dotsc,x_n\}\to Q^\times \ \bigg| \ \begin{array}{c} \Phi(g)(a)+\Phi(g)(b)+\Phi(g)(c)\in N_Q\\\textrm{for }a+b+c\in S\end{array}\bigg\}.
\end{multline*}
For details on the construction of $\pastgen{\Funpm(x_1,\dotsc,x_n)}S$, see~\cite[Section 2.1.1]{Baker-Lorscheid20}.

\subsubsection{Further examples}
\label{subsubsection: further examples of pastures}

The construction of quotients of free $\Funpm$-algebras allows us to present pastures in terms of generators and relations. Some examples of importance for this text are the following:
\begin{align*}
 \U   \ &= \ \pastgen{\Funpm(x,y)}{x+y-1}             && \textrm{(the \emph{near regular partial field})} \\
 \D   \ &= \ \pastgen{\Funpm(z)}{z-1-1}               && \textrm{(the \emph{dyadic partial field})} \\
 \H   \ &= \ \pastgen{\Funpm(z)}{z^3+1,\ \ z^2-z+1} && \textrm{(the \emph{hexagonal partial field})} \\
\V \ &= \ \pastgen{\Funpm(x_1,\dotsc,x_5)}{x_i+x_{i-1}x_{i+1}-1\mid i \in \Z/5\Z} && \textrm{(the \emph{$2$-regular partial field})} \\
\end{align*}


\subsection{Matroid representations}
\label{subsection: matroid representations}

Let $M$ be a matroid on $E$ and $\cH$ its collection of hyperplanes. A \emph{modular tuple of hyperplanes} is a tuple $(H_1,\dotsc,H_s)$ of hyperplanes that intersect in a corank $2$ flat $F=H_1\cap\dotsc\cap H_s$.  

Let $P$ be a pasture, let $P^E$ be the set of functions from $E$ to $P$, and let $H$ be a hyperplane of $M$. A \emph{$P$-hyperplane function for $H$} is a map $\varphi_H: E \to P$ such that $\varphi_H(e)=0$ if and only if $e \in H$. A \emph{family of $P$-hyperplane functions for $M$} is a map $\varphi : \cH \to P^E$ such that $\varphi_H := \varphi(H)$ is a $P$-hyperplane function for every $H\in\cH$. A triple $(\varphi_{H_1}, \varphi_{H_2},\varphi_{H_3})$ of $P$-hyperplane functions is \emph{linearly dependent} if there exist $a,b,c\in P$ with $(a,b,c)\neq(0,0,0)$ and $a\varphi_{H_1}(e) + b\varphi_{H_2}(e) + c\varphi_{H_3}(e) \in N_P$ for all $e\in E$.

\begin{df}
 A \emph{\emph{(}weak\emph{)} $P$-hyperplane representation} or, for short, a \emph{$P$-representation of $M$} is a family $\varphi: \cH \to P^E$ of $P$-hyperplane functions for $M$ such that for every modular triple $(H_1,H_2,H_3)$ of distinct hyperplanes, the triple $(\varphi_{H_1},\varphi_{H_2},\varphi_{H_3})$ of $P$-hyperplane functions is linearly dependent.
\end{df}

A pasture morphism $g\colon P\to Q$ defines a push-forward on hyperplane representations: given a $P$-representation $\{\varphi_H \colon E \to P\}_{H \in \cH}$ of $M$, the composition with $g$ defines a $Q$-representation $\{g \circ \varphi_H \colon E \to Q\}_{H \in \cH}$ of $M$.

\medskip

The following result (which is a combination of~\cite[Theorem 3.21]{Baker-Bowler19} and~\cite[Theorem 2.16]{Baker-Lorscheid20}) exhibits the relation between $P$-representations in the sense of this text and the notion of a weak $P$-circuit set.\footnote{We refer the reader to \cite[Definition 3.8]{Baker-Bowler19} for the definition of (weak) $P$-circuits. Note that the theory in \cite{Baker-Bowler19} is developed for \emph{tracts} (whose null sets can contain additive relations with more than three terms) rather than pastures, but the definition of a weak $P$-circuit set also makes sense for pastures, since it only refers to $3$-term relations in the null set. For a concise discussion of the relation between tracts and pastures, see~\cite[Section 2.1.5]{Baker-Lorscheid20}.}

\begin{lemma}\label{lemma:hyperplane-functions-and-circuit-set}
 Let $P$ be a pasture and $M$ a matroid. Suppose $\varphi : \cH \to P^E$ is a family of $P$-hyperplane functions for $M$. Then $\varphi$ is a $P$-representation of $M$ if and only if $\cC = \{a \cdot \varphi_H \mid a \in P^\times, H \in \cH\}$ is a weak $P$-circuit set with underlying matroid $M^\ast$.
\end{lemma}

Let $M$ be a matroid on $E$, let $\varphi: \cH \to P^E$ be a $P$-representation of $M$ and let $A \subseteq E$. Let $X = E - A$.
We write $\varphi_H|_X$ for the restriction of the function $\varphi_H$ to $X$.
Define $\varphi \backslash A = \{\varphi_H|_{X} \mid H - A \textrm{ is a hyperplane of } M \backslash A\}$.

\begin{prop} \cite[Theorem 3.29]{Baker-Bowler19} \label{prop: induced representation for embedded minor}
Up to multiplying functions by scalars, $\varphi \backslash A$ is a $P$-representation of $M \backslash A$.
\end{prop}

\subsection{Foundations}
\label{subsection: foundation}

Let $P$ be a pasture. Two families of $P$-hyperplane functions $\varphi = \{\varphi_H\}_{H \in \cH}$ and $\varphi' = \{\varphi'_H\}_{H \in \cH}$ for $M$ are said to be \emph{rescaling equivalent} if there are $(a_H)_{H\in\cH}\in (P^\times)^\cH$ and $(t_e)_{e\in E}\in (P^\times)^E$ such that $\varphi'_H(e) = a_H \cdot t_e \cdot \varphi_H(e)$ for all $H \in \cH$ and $e\in E$. It is easy to see that a family $\varphi$ of $P$-hyperplane functions for $M$ is a $P$-representation if and only if all families of $P$-hyperplane functions that are rescaling equivalent to $\varphi$ are $P$-representations. The \emph{realization space of $M$ over $P$} is the set $\upR_M(P)$ of all rescaling equivalence classes of $P$-representations of $M$. 

Note that the realization space $\upR_M(P)$ agrees, up to a canonical identification, with the set of rescaling classes of weak Grassmann-Pl\"ucker functions of $M$ in $P$, as verified in \cite[Remark 2.4]{Baker-Lorscheid-Walsh-Zhang25}. 

A pasture morphism $g: P \to Q$ induces a map $g_\ast : \upR_M(P) \to \upR_M(Q)$ that sends the rescaling class of a $P$-representation $\{\varphi_H\}_{H \in \cH}$ to the rescaling class of the $Q$-representation $\{g \circ \varphi_H\}_{H \in \cH}$. We can thus view $\upR_M$ as a functor $\Pastures$ to $\Sets$.

The following is \cite[Corollary 7.28]{Baker-Lorscheid21b} (cf.\ also \cite[Theorem 4.3]{Baker-Lorscheid20}):

\begin{thm}\label{thm: characterizing property of the foundation}
 The functor $\upR_M(-)$ is represented by a pasture $F_M$ \emph{(}called the \emph{foundation of $M$}\emph{)}, i.e., for every pasture $P$ we have $\upR_M(P)=\Hom(F_M,P)$ functorially in $P$.
\end{thm}

An explicit construction of the foundation in terms of Pl\"ucker coordinates can be found in \cite[Definition 4.2]{Baker-Lorscheid21b}. We shall explain another explicit construction of $F_M$ and a proof of~\autoref{thm: characterizing property of the foundation} in \autoref{subsection: hyperplane matrix representation} in terms of hyperplane functions. 

It turns out that the unit group $F_M^\times$ of the foundation is finitely generated, and is canonically isomorphic to the \emph{inner Tutte group} defined by Dress and Wenzel in~\cite{Dress-Wenzel89}. 

There is a canonical bijection between $\upR_{M^\ast}(P)$ and $\upR_{M}(P)$ for every pasture $P$, and thus (cf.~\cite[Theorem 4.7]{Baker-Lorscheid20}):

\medskip

\begin{prop}\label{prop:founadtion of dual}
     The foundation of $M^\ast$ is canonically isomorphic to the foundation of $M$. 
\end{prop}

We also have (cf.~\cite[Theorem 5.1]{Baker-Lorscheid-Zhang24}):

\begin{prop}\label{prop:founadtion of direct sum}
  Let $M_1$ and $M_2$ be matroids. Then $F_{M_1 \oplus M_2} \cong F_{M_1} \otimes F_{M_2}$.   
\end{prop}

\subsection{Cross-ratios}
\label{subsection: cross ratios}

Let $\Xi_M$ be the set of all tuples $(H_1, H_2, a, b)$ of hyperplanes $H_1, H_2 \in \cH_M$ and elements $a, b \in E - (H_1 \cup H_2)$ such that either $H_1 = H_2$ are the same or $L = H_1 \cap H_2$ is a corank $2$ flat of $M$. Let $\Xi^\diamondsuit_M$ be the subset of all tuples $(H_1,H_2,a,b) \in \Xi_M$ such that $L = H_1 \cap H_2$ is of corank $2$ and $\gen{La}, \gen{Lb}$ are distinct hyperplanes. 

\begin{df}
Let $P$ be a pasture and let $\varphi = \{\varphi_H : E \to P \mid H \in \cH_M\}$ be a family of $P$-hyperplane functions for $M$. The \emph{cross-ratio} $\cross{{H_1}}{{H_2}}{a}{b}{\varphi}$ is defined as
 \[
  \cross{{H_1}}{{H_2}}{a}{b}{\varphi} \ = \ \frac{\varphi_{H_1}(a) \cdot \varphi_{H_2}(b)}{\varphi_{H_1}(b) \cdot  \varphi_{H_2}(a)} \ \in \ P.
 \]
\end{df}

\begin{lemma}\label{lemma: first properties of cross ratios}
 Let $\varphi  =\{\varphi_H : E \to P \mid 
 H \in \cH_M\}$ be a $P$-representation of $M$. Then for all $(H_1, H_2, a, b), (H_1, H_2, a', b') \in \Xi_M$ we have: 
 \begin{enumerate}
  \item If $(H_1,H_2,a,b)\notin\Xi_M^\diamondsuit$, then $\cross{H_1}{H_2}ab\varphi=1$. 
  \item  If $(H_1,H_2,a,b)\in\Xi_M^\diamondsuit$, $\gen{La}=\gen{La'}$, and $\gen{Lb}=\gen{Lb'}$ for $L=H_1\cap H_2$, then
   \[
    \cross{{H_1}}{{H_2}}ab\varphi \ = \ \cross{{H_1}}{{H_2}}{a'}{b'}\varphi.
   \]
 \end{enumerate}
\end{lemma}

\begin{proof}
 The first assertion is clear when $H_1 = H_2$. Now assume $L = H_1 \cap H_2$ is of corank $2$. Write $H_0 = \gen{La} = \gen{Lb}$. Then $(H_0, H_1, H_2)$ forms a modular triple of distinct hyperplanes. Therefore there exist $x, y, z \in P^\times$ with $x\varphi_{H_0}(e) + y\varphi_{H_1}(e) + z\varphi_{H_2}(e) \in N_P$ for all $e\in E$. It follows that 
 \[
\cross{H_1}{H_2}ab\varphi = \frac{\varphi_{H_1}(a) \cdot \varphi_{H_2}(b)}{\varphi_{H_1}(b) \cdot  \varphi_{H_2}(a)} = \frac{\varphi_{H_1}(a)}{\varphi_{H_2}(a)} \cdot \frac{\varphi_{H_2}(b)}{\varphi_{H_1}(b)} = \left( - \frac{z}{y} \right) \cdot \left(- \frac{y}{z} \right) = 1. 
 \]
 
 For the second assertion, since $\gen{La}=\gen{La'}$ and $\gen{Lb}=\gen{Lb'}$, we know from the first part of the lemma that 
 \[
\cross{H_1}{H_2}{a}{a'}{\varphi} = 1 = \cross{H_1}{H_2}{b}{b'}{\varphi} . 
 \] 
 This implies 
 \[
\frac{\varphi_{H_1}(a) \varphi_{H_2}(a')}{\varphi_{H_1}(a') \varphi_{H_2}(a)} = \frac{\varphi_{H_1}(b) \varphi_{H_2}(b')}{\varphi_{H_1}(b') \varphi_{H_2}(b)}. 
 \]
 Rearranging both sides, we obtain 
 \[
\cross{{H_1}}{{H_2}}ab\varphi = \frac{\varphi_{H_1}(a) \varphi_{H_2}(b)}{\varphi_{H_1}(b) \varphi_{H_2}(a)} = \frac{\varphi_{H_1}(a') \varphi_{H_2}(b')}{\varphi_{H_1}(b') \varphi_{H_2}(a')} = \cross{{H_1}}{{H_2}}{a'}{b'}{\varphi}, 
 \]
 which completes the proof. 
\end{proof}

Let $\Theta_M$ be the set of all $4$-tuples of hyperplanes $(H_1, H_2, H_3, H_4)$ of $M$ such that $L = H_1 \cap H_2 \cap H_3 \cap H_4$ is a corank $2$ flat with $L = H_i \cap H_j$ for every $i \in \{1, 2\}$ and $j \in \{3,4\}$. Let $\Theta_M^\diamondsuit$ be the set of all {\em non-degenerate} tuples for which $L = H_1 \cap H_2 = H_3 \cap H_4$ also holds. 

\begin{df}
Let $\varphi = \{\varphi_H : E \to P \mid H \in \cH_M\}$ be a $P$-representation of $M$ and let $(H_1,H_2,H_3,H_4)\in\Theta_M$. Let $L=H_1\cap\dotsc\cap H_4$ and choose $a\in H_3-L$ and $b\in H_4-L$, so that $(H_1,H_2,a,b)\in\Xi_M$. The \emph{cross-ratio} $\cross{H_1}{H_2}{H_3}{H_4}{\varphi}$ is defined as
 \[
  \cross{H_1}{H_2}{H_3}{H_4}{\varphi} \ := \ \cross{{H_1}}{{H_2}}ab\varphi.
 \] 
This is well-defined, independent of the choice of $a$ and $b$, by \autoref{lemma: first properties of cross ratios}.

We say that $\cross{H_1}{H_2}{H_3}{H_4}{\varphi}$ is \emph{non-degenerate} if $(H_1,H_2,H_3,H_4) \in \Theta_M^\diamondsuit$; otherwise, it is \emph{degenerate}. Similarly, we say that $\cross{H_1}{H_2}{H_3}{H_4}{\varphi}$ is \emph{non-trivial} if $\cross{H_1}{H_2}{H_3}{H_4}{\varphi} \neq 1$; otherwise, it is \emph{trivial}.
\end{df}

Note that every degenerate cross-ratio is trivial by \autoref{lemma: first properties of cross ratios}. The reverse implication does not hold in general: there are matroids $M$ and $(H_1,H_2,H_3,H_4)\in\Theta_M^\diamondsuit$ such that for every pasture $P$ and every $P$-representation $\varphi$, the non-degenerate cross-ratio $\cross{H_1}{H_2}{H_3}{H_4}{\varphi}$ is trivial; see~\cite[Section A.3.1]{Baker-Lorscheid-Zhang24} for a concrete example.

Cross-ratios are invariants of rescaling classes and behave functorially with respect pasture morphisms, as the following result shows:

\begin{prop}\label{prop: cross ratios under rescaling}
 Let $\varphi = \{\varphi_H\}$ and $\varphi' = \{\varphi_H'\}$ be two $P$-representations of $M$ that are rescaling equivalent. Then for every $(H_1,H_2,H_3,H_4)\in\Theta_M$, we have
 \[
  \cross{H_1}{H_2}{H_3}{H_4}{\varphi} \ = \ \cross{H_1}{H_2}{H_3}{H_4}{\varphi'}.
 \]
 Moreover, let $g:P\to P'$ be a pasture morphism and $g_\ast(\varphi)$ the push-forward of $\varphi$ along $g$. Then
 \[
  \cross{H_1}{H_2}{H_3}{H_4}{g_\ast(\varphi)} \ = \ g\Big( \cross{H_1}{H_2}{H_3}{H_4}{\varphi} \Big).
 \]
\end{prop}

\begin{proof}
Let $(a_H)_{H\in\cH}\in (P^\times)^\cH$ and $(t_e)_{e\in E}\in (P^\times)^E$ be such that $\varphi'_H(e) = a_H \cdot t_e \cdot \varphi_H(e)$ for all $H \in \cH$ and $e\in E$. Take $e \in H_3 - (H_1 \cap H_2)$ and $f \in H_4 - (H_1 \cap H_2)$. Then
\[
\cross{H_1}{H_2}{H_3}{H_4}{\varphi'} = \frac{\varphi'_{H_1}(e) \varphi'_{H_2}(f)}{\varphi'_{H_1}(f) \varphi'_{H_2}(e)} = \frac{a_{H_1}t_e\varphi_{H_1}(e) \cdot a_{H_2}t_f \varphi_{H_2}(f)}{a_{H_1}t_f\varphi_{H_1}(f) \cdot a_{H_2}t_e\varphi_{H_2}(e)} = \cross{H_1}{H_2}{H_3}{H_4}{\varphi}. 
\]

Similarly, since $g: P \to P'$ is a pasture morphism, we have
\[
\cross{H_1}{H_2}{H_3}{H_4}{g_\ast(\varphi)} = \frac{(g \circ \varphi_{H_1})(e) \cdot (g\circ\varphi_{H_2})(f)}{(g\circ\varphi_{H_1})(f) \cdot (g\circ\varphi_{H_2})(e)} = g\Big(\cross{H_1}{H_2}{H_3}{H_4}{\varphi}\Big). \qedhere
\]
\end{proof}

A \emph{universal representation of $M$} is a representation $\varphi = \{\varphi_H : E \to F_M \mid H \in \cH\}$ of $M$ over its foundation $F_M$ whose rescaling equivalence class $[\varphi] \in \upR_M(F_M)$ corresponds to the identity morphism $\id: F_M \to F_M$ under the canonical bijection $\upR_M(F_M) = \Hom(F_M, F_M)$. 

\begin{df}
 Let $\varphi_\univ$ be a universal representation of $M$. The \emph{universal cross-ratio} $\cross{H_1}{H_2}{H_3}{H_4}{}$ is defined as
 \[
  \cross{H_1}{H_2}{H_3}{H_4}{} \ = \ \cross{H_1}{H_2}{H_3}{H_4}{\varphi_\univ} \ \in \ F_M.
 \]
 By \autoref{prop: cross ratios under rescaling}, $\cross{H_1}{H_2}{H_3}{H_4}{}$ does not depend on the choice of $\varphi_\univ$. 
\end{df}

\begin{cor}\label{cor: cross ratios as images of the universal cross ratio}
 Let $\varphi$ be a $P$-representation of $M$ and let $g_\varphi: F_M \to P$ be the pasture morphism corresponding to the rescaling class of $\varphi$. Then
 \[
  \cross{H_1}{H_2}{H_3}{H_4}{\varphi} \ = \ g_\varphi\Big(\cross{H_1}{H_2}{H_3}{H_4}{}\Big)
 \]
 for all $(H_1,H_2,H_3,H_4)\in\Theta_M$.
\end{cor}

\begin{proof}
By \autoref{thm: characterizing property of the foundation} and~\autoref{prop: cross ratios under rescaling}, we have
\[
g_\varphi\Big(\cross{H_1}{H_2}{H_3}{H_4}{}\Big) = g_{\varphi}\Big(\cross{H_1}{H_2}{H_3}{H_4}{\varphi_\univ}\Big) = \cross{H_1}{H_2}{H_3}{H_4}{{g_{\varphi}}_\ast(\varphi_\univ)} = \cross{H_1}{H_2}{H_3}{H_4}{\varphi}. \qedhere
\]
\end{proof}

\begin{rem}\label{rem: rescaling invariant elements are in the image of the foundation}
 The proof of \autoref{cor: cross ratios as images of the universal cross ratio} (and \autoref{prop: cross ratios under rescaling}) only uses the property that cross-ratios are invariant under rescaling the $P$-representation $\varphi$. Thus the previous result applies to every expression $\Pi=\prod_{i\in I}\varphi_{H_i}(a_i)^{\epsilon_i}$ (with $H_i\in\cH$, $a_i\in E$ and $\epsilon_i\in\{\pm1\}$) that is invariant under rescaling. That is, if the degree of $\Pi$ in each $H_i$ and $a_i$ is zero, then $\Pi=\psi(\Pi_\univ)$, where $\psi:F_M\to P$ is the pasture morphism corresponding to $\varphi$ and $\Pi_\univ=\prod_{i\in I}\varphi_{\univ,H_i}(a_i)^{\epsilon_i}$ for a universal $F_M$-representation $\varphi_\univ$.
 
 In fact, Tutte's path theorem implies that every such expression $\Pi$ is a product of cross-ratios (up to a sign); cf.\ \autoref{thm: the foundation is generated by cross ratios}.
\end{rem}

\subsection{A characterization of \texorpdfstring{$P$}{P}-hyperplane representations}
\label{subsection: characterization of P-hyperplane representations}

Our next goal is to describe a presentation for $F_M$ which does not rely on Tutte's homotopy theory and is easy to implement on a computer (indeed, this is the description of $F_M$ used in \cite{Chen-Zhang}). The following result will be central for this:

\begin{thm}\label{thm:iff-modular-triple}
    Let $M$ be a matroid, and let $\varphi = \{\varphi_H\}_{H \in \cH}$ be a family of $P$-hyperplane functions for $M$. Then $\varphi$ is a $P$-representation of $M$ if and only if the following two properties are satisfied: 
    \begin{enumerate}
    \item\label{char2:multiplicative} For every modular triple $(H_1, H_2, H_3)$ of distinct hyperplanes with $L = H_1 \cap H_2 \cap H_3$ and every $e_i \in H_i - L$, we have 
    \[
\frac{\varphi_1(e_2) \cdot \varphi_2(e_3) \cdot \varphi_3(e_1)}{\varphi_1(e_3) \cdot \varphi_2(e_1) \cdot \varphi_3(e_2)} \ = \ -1,
    \]
    where $\varphi_i := \varphi_{H_i}$ is the hyperplane function corresponding to $H_i$.
        \item\label{char3:U24} For every modular quadruple $(H_1, H_2, H_3, H_4)$ of distinct hyperplanes with $L = H_1 \cap H_2 \cap H_3 \cap H_4$ and every $e_i \in H_i - L$, we have
        \[
\cross{H_1}{H_2}{e_3}{e_4}{\varphi} \ + \ \cross{H_1}{H_3}{e_2}{e_4}{\varphi} \ - \ 1 \quad \in \quad N_P.
        \]
    \end{enumerate} 
\end{thm}

\begin{proof}
    Assume first that $\varphi$ is a $P$-representation of $M$. Consider a modular triple $(H_1, H_2, H_3)$ of distinct hyperplanes with $L = H_1 \cap H_2 \cap H_3$ and elements $e_i \in H_i - L$. Since $\varphi$ is a $P$-representation, there exist $a, b, c \in P^\times$ with $a\varphi_1(e) + b\varphi_2(e) + c\varphi_3(e) \in N_P$ for all $e \in E$. Evaluating this linear dependence relation at $e_1, e_2$, and $e_3$, we have 
\[
    b\varphi_2(e_1) + c\varphi_3(e_1), \ 
    a\varphi_1(e_2) + c\varphi_3(e_2), \
    a\varphi_1(e_3) + b\varphi_2(e_3)\ \in \ N_P, 
\]
from which we compute that
\[
\frac{\varphi_1(e_2)\varphi_2(e_3)\varphi_3(e_1)}{\varphi_1(e_3)\varphi_2(e_1)\varphi_3(e_2)} = (-1)^3 \cdot \frac{b}{c} \cdot \frac{c}{a} \cdot \frac{a}{b} = -1. 
\]

If we extend the modular triple $(H_1, H_2, H_3)$ to a modular quadruple $(H_1, H_2, H_3, H_4)$ of distinct hyperplanes with an element $e_4 \in H_4 - L$, then 
\[
a\varphi_1(e_4) + b\varphi_2(e_4) + c\varphi_3(e_4) \in N_P. 
\]
Consequently, 
\[
\frac{\varphi_{1}(e_3) \varphi_{2}(e_4)}{\varphi_{1}(e_4) \varphi_{2}(e_3)} 
 + \frac{\varphi_{1}(e_2) \varphi_{3}(e_4)}{\varphi_{1}(e_4) \varphi_{3}(e_2)} - 1 = -\frac{b\varphi_2(e_4)}{a\varphi_1(e_4)} - \frac{c\varphi_3(e_4)}{a\varphi_1(e_4)} - 1 \in N_P, 
\]
which proves the second property. 

    To show the converse, we consider an arbitrary modular triple $(H_1, H_2, H_3)$ of distinct hyperplanes with $L = H_1 \cap H_2 \cap H_3$. We 
    claim that there exist nonzero $a, c \in P$ such that 
    \[
    a\varphi_1(e) + \varphi_2(e) + c\varphi_3(e) \in N_P
    \]
    for all elements $e \in E$. To see this, pick an arbitrary $e_1 \in H_1 - L$; we know that $c = -\frac{\varphi_2(e_1)}{\varphi_3(e_1)}$. Similarly, fix $e_3 \in H_3 - L$, then $a = -\frac{\varphi_2(e_3)}{\varphi_1(e_3)}$. Therefore, it suffices to show that for every $e \in E$, we have 
    \begin{equation}\label{equa:3-term-hyperplane}
-\frac{\varphi_2(e_3)}{\varphi_1(e_3)} \varphi_1(e) + \varphi_2(e) - \frac{\varphi_2(e_1)}{\varphi_3(e_1)} \varphi_3(e) \in N_P. 
    \end{equation}

    Suppose first that $e \in H_1 - L$. By (\ref{char2:multiplicative}), we have 
    \[
\frac{\varphi_1(e_2) \varphi_2(e_3) \varphi_3(e_1)}{\varphi_1(e_3)\varphi_2(e_1) \varphi_3(e_2)}  \ = -1 \ = \frac{\varphi_1(e_2) \varphi_2(e_3) \varphi_3(e)}{\varphi_1(e_3) \varphi_2(e) \varphi_3(e_2)}, 
    \]
    which implies 
    \[
\frac{\varphi_2(e)\varphi_3(e_1)}{\varphi_2(e_1)\varphi_3(e)} = 1. 
    \]
    Thus, \autoref{equa:3-term-hyperplane} follows. The case when $e \in H_3 - L$ is similar. If $e \in H_2 - L$, then~\autoref{equa:3-term-hyperplane} is equivalent to the multiplicative relation 
    \[
\frac{\varphi_1(e)\varphi_2(e_3)\varphi_3(e_1)}{\varphi_1(e_3)\varphi_2(e_1)\varphi_3(e)} = -1. 
    \]
    Finally, if $e \notin H_1 \cup H_2 \cup H_3$, and if we denote $H_4 = \gen{Le}$, then~\autoref{equa:3-term-hyperplane} follows from (\ref{char3:U24}), which tells us that  
    \[
    \cross{H_2}{H_1}{e_3}{e}{\varphi} + \cross{H_2}{H_3}{e_1}{e}{\varphi} - 1 \in N_P. \qedhere
    \]
\end{proof}

To ease the notation, if $\varphi: \cH \to P^E$ is a family of $P$-hyperplane functions for $M$, $(H_1, H_2, H_3)$ is a modular triple of distinct hyperplanes with $L = H_1 \cap H_2 \cap H_3$, and $e_i \in H_i - L$ for $i = 1,2,3$, we also write 
\[
\tripleratio{H_1}{H_2}{H_3}{e_1}{e_2}{e_3}{\varphi} := 
\frac{\varphi_{H_1}(e_2) \cdot \varphi_{H_2}(e_3) \cdot \varphi_{H_3}(e_1)}{\varphi_{H_1}(e_3) \cdot \varphi_{H_2}(e_1) \cdot \varphi_{H_3}(e_2)}. 
\]

\begin{cor}\label{cor: cross ratios are invariant under row exchange}
 Let $\varphi:\cH\to P^E$ be a $P$-representation and consider $(H_1,H_2,H_3,H_4)\in\Theta_M$ with $L=H_1\cap\dotsc\cap H_4$ and $e_i\in H_i-L$ for $i=1,\dotsc,4$. Then
 \[
  \cross{H_1}{H_2}{e_3}{e_4}{\varphi} \ = \ \cross{H_3}{H_4}{e_1}{e_2}{\varphi}.
 \]
\end{cor}

\begin{proof}
 The case in which $(H_1, H_2, H_3, H_4)$ is degenerate follows from the definition and~\autoref{lemma: first properties of cross ratios}. If $(H_1, H_2, H_3, H_4)$ is non-degenerate, then it follows by the direct verification
 \[
  \cross{H_1}{H_2}{e_3}{e_4}{\varphi} \ = \ \cross{H_1}{H_2}{e_3}{e_4}{\varphi} \ \cdot \ \tripleratio{H_1}{H_2}{H_3}{e_1}{e_2}{e_3}{\varphi} \ \cdot \ \tripleratio{H_4}{H_2}{H_1}{e_4}{e_2}{e_1}{\varphi} \ = \ \cross{H_3}{H_4}{e_1}{e_2}{\varphi},
 \]
 using that $\tripleratio{H_1}{H_2}{H_3}{e_1}{e_2}{e_3}{\varphi}=\tripleratio{H_4}{H_2}{H_1}{e_4}{e_2}{e_1}{\varphi}=-1$ by \autoref{thm:iff-modular-triple}.
\end{proof}

As an application to \autoref{thm:iff-modular-triple} , we give another proof for the excluded minor theorem for binary matroids, which was originally proved by Tutte in \cite{Tutte58b}.

\begin{thm}\label{thm: excluded minors for binary matroids}
 A matroid is binary if and only if it has no minors of type $U_{2,4}$.
\end{thm}

\begin{proof}
Since $U_{2,4}$ is not binary, but all of its proper minors are, $U_{2,4}$ is an excluded minor for the class of binary matroids. Conversely, if a matroid $M$ is without minors of type $U_{2,4}$, then we claim that the family $\varphi = \{\varphi_H\}$ of $\F_2$-hyperplane functions given by $\varphi_H(e) = 1$ if $e \notin H$ and $\varphi_H(e) = 0$ otherwise is a representation of $M$ over $\F_2$. In fact, since there are no modular quadruples of distinct hyperplanes in $M$, \eqref{char3:U24} in \autoref{thm:iff-modular-triple} holds vacuously. Let $(H_1, H_2, H_3)$ be a modular triple of distinct hyperplanes with $L = H_1 \cap H_2 \cap H_3$ and let $e_i \in H_i - L$. Then 
 \[
\frac{\varphi_{H_1}(e_2) \cdot \varphi_{H_2}(e_3) \cdot \varphi_{H_3}(e_1)}{\varphi_{H_1}(e_3) \cdot \varphi_{H_2}(e_1) \cdot \varphi_{H_3}(e_2)} \ = \ 1 \ = \ -1 \ \in \  \F_2 
    \]
and we conclude that $M$ is representable over $\F_2$. 
\end{proof}

\subsection{An algorithm for computing the foundation via the hyperplane incidence graph}
\label{subsection: hyperplane matrix representation}

Let $M$ be a matroid on $E$ with set of hyperplanes $\cH$. Our goal in this section is to give an explicit construction of the foundation $F_M$ in terms of a certain graph $G_M$ associated to $M$. 

\begin{df}
The {\em hyperplane incidence graph of $M$}\footnote{It would perhaps be more accurate to call $G_M$ the ``hyperplane non-incidence graph'' of $M$ or the cocircuit incidence graph of $M$, but we will abuse terminology here in order to avoid awkwardness in our exposition.}
is the bipartite graph $G=G_M$ with vertex set $\cH \cup E$ such that $H \in \cH$ and $a \in E$ are adjacent if and only if $a \notin H$. 
\end{df}

Note that a family $\varphi$ of $P$-hyperplane functions for $M$ can be viewed as an assignment $\varphi: E(G_M) \to P^\times$ of an element in $P^\times$ to every edge in $G_M$. 

\medskip

Let $G$ be an arbitrary bipartite graph. In this section, all graphs considered will be finite. For ease of notation, the bi-partition of the vertex set of $G$ will still be denoted by $\cH \cup E$, even though $G$ is not necessarily the hyperplane incidence graph of a matroid. 

A {\em maximal spanning forest $F$ of $G$} (or just a {\em spanning forest $F$ of $G$}, for short) is a maximal set of edges of $G$ that contains no cycle. Equivalently, $F$ is a subgraph of $G$ consisting of a spanning tree in each connected component. 

If $G$ is the hyperplane incidence graph $G$ of a matroid $M$, then by~\cite[Proposition 4.1.2]{Oxley11}, there is a one-to-one correspondence between the connected components of $G$ and the connected components of $M$. Therefore, the number of edges in any spanning forest $F$ is $$\# E + \# \cH - \#\{\textrm{connected components of $M$}\}.$$ 

\begin{df}
Let $\Gamma$ be an abelian group, written multiplicatively.\footnote{In the rest of this paper, $\Gamma$ is used to denote a modular cut in the lattice of flats of a matroid, but in this subsection we do not use modular cuts so there should hopefully be no risk of confusion.} Two functions $\varphi, \psi: E(G) \to \Gamma$ are {\em rescaling equivalent} if there exist $(a_H)_{H \in \cH} \in \Gamma^{\cH}$ and $(t_e)_{e \in E} \in \Gamma^E$ such that $\varphi(H,e) = a_H \cdot t_e \cdot \psi(H,e)$ for every edge $(H,e) \in E(G)$. 
\end{df}

\begin{lemma}\label{lemma:graph-theoretic-well-defined}
    Let $G$ be a bipartite graph with vertex set $\cH \cup E$ and let $\Gamma$ be an abelian group. Fix a spanning forest $F$ of $G$. Suppose two functions $\varphi, \psi: E(G) \to \Gamma$ agree on $E(F)$. Then $\varphi$ and $\psi$ are identical if and only if they are rescaling equivalent. 
\end{lemma}

\begin{proof}
We need only prove that $\varphi$ and $\psi$ being rescaling equivalent implies $\varphi$ and $\psi$ being identical, since the other direction is trivial. 

Suppose that there exist $(a_H)_{H \in \cH} \in \Gamma^{\cH}$ and $(t_e)_{e \in E} \in \Gamma^E$ such that $\varphi(H,e) = a_H \cdot t_e \cdot \psi(H,e)$ for every edge $(H,e) \in E(G)$. Since $\varphi$ and $\psi$ agree on $E(F)$, we know $a_H \cdot t_e = 1$ for all $(H,e) \in E(F)$. For $\varphi = \psi$, we are left to show $a_H \cdot t_e = 1$ for all $(H,e) \in E(G) - E(F)$. 

    Let $(H_0, e_0)$ be an edge in $G$ that is not in $F$. Since $F$ is a spanning forest, there exists a cycle $(H_0, e_0, H_1, e_1, \dots, H_k, e_k, H_0)$ in $F \cup (H_0, e_0)$. Because $\varphi$ and $\psi$ agree on $E(F)$, we have $1 = a_{H_1}t_{e_0} = a_{H_1}t_{e_1} = a_{H_2}t_{e_1} = \cdots = a_{H_k}t_{e_k} = a_{H_0}t_{e_k}$; thus $a_{H_0} = a_{H_k}$ and $t_{e_0} = t_{e_1} = \cdots = t_{e_k}$, which gives $a_{H_0}t_{e_0} = a_{H_k}t_{e_k} = 1$ as desired. 
\end{proof}

\begin{df}
Let $\Gamma$ be an abelian group, let $G$ be a bipartite graph, let $S$ be a subset of $E(G)$, and let $\psi': S \to  \Gamma$ be a function. An {\em extension of $\psi'$} is a function $\psi: E(G) \to \Gamma$ with $\psi \vert_{S} = \psi'$. 
\end{df}

\begin{prop}\label{prop:graph-theoretic-for-definition}
    Let $G$ be a bipartite graph with vertex set $\cH \cup E$ and let $\Gamma$ be a multiplicatively-written abelian group. Fix a spanning forest $F$ of $G$ and a function $\psi': E(F) \to  \Gamma$. Then, for every function $\varphi: E(G) \to \Gamma$, there exists a unique extension $\psi: E(G) \to \Gamma$ of $\psi'$ which is rescaling equivalent to $\varphi$. 
\end{prop}

\begin{proof}
    We claim that the system of equations
    \[
\varphi(H,e) = a_H \cdot t_e \cdot \psi'(H,e), \ (H, e) \in E(F)
    \]
    has a solution $(a_H)_{H \in \cH} \in \Gamma^{\cH}$ and $(t_e)_{e \in E} \in \Gamma^E$. Assuming the claim, we see that the function $\psi: E(G) \to \Gamma$ defined by $(H,e) \mapsto \frac{\varphi(H,e)}{a_H \cdot t_e}$ is rescaling equivalent to $\varphi$. By~\autoref{lemma:graph-theoretic-well-defined}, $\psi$ is unique. 

    We construct a solution explicitly as follows. Pick a connected component $C$ of $G$ and a vertex $e_0 \in C \cap E$. Set $t_{e_0} = 1$. For every vertex $H$ in $G$ that is adjacent in $F$ to $e_0$, the value of $a_H \in \Gamma$ is determined by the equation $\varphi_H(e_0) = a_H t_{e_0} \psi'_H(e_0)$. Since every vertex in $C$ is connected to $e_0$ via a unique path in $F$, this inductively solves for $a_H$ and $t_e$ for all $H, e \in C$. Repeating the procedure for all connected components of $G$ gives $(a_H)_{H \in \cH} \in \Gamma^{\cH}$ and $(t_e)_{e \in E} \in \Gamma^E$ that solve the original system of equations. 
\end{proof}

\begin{thm}\label{general-bijection-between-equi-classes-and-extensions}
    Let $G$ be a bipartite graph and $\Gamma$ an abelian group. Fix a spanning forest $F$ of $G$ and a function $\psi': E(F) \to \Gamma$. Then there is a bijection 
    \[
        \Phi: \{\textrm{\emph{rescaling equivalence classes of functions $\varphi: E(G) \to \Gamma$}}\} \to \{\textrm{\emph{extensions of $\psi'$}}\}. 
    \]
\end{thm}

\begin{proof}
The extension corresponding to $\varphi:E(G) \to \Gamma$, constructed in~\autoref{prop:graph-theoretic-for-definition}, gives a well-defined surjective map $\Phi$ by~\autoref{lemma:graph-theoretic-well-defined}. If $\Phi(\varphi_1) = \Phi(\varphi_2) = \psi$, then both $\varphi_1$ and $\varphi_2$ are rescaling equivalent to $\psi$. This implies, by \autoref{lemma:graph-theoretic-well-defined}, that $\varphi_1$ and $\varphi_2$ are rescaling equivalent, and hence $\Phi$ is injective, which completes the proof.  
\end{proof}

If the graph $G$ in~\autoref{general-bijection-between-equi-classes-and-extensions} is the hyperplane incidence graph of a matroid $M$, $P$ is a pasture, and we consider only those functions $\varphi:E(G) \to P^\times$ that are $P$-representations of $M$, we get the following algorithm for computing the foundation of $M$:

\begin{enumerate}
    \item Compute the hyperplane incidence graph $G_M$ of $M$.
    \item Choose a spanning forest $F$ of $G_M$. 
    \item Construct a matrix $A = A_M(F)$ with $\#\cH$ rows and $\# E$ columns as follows: if $e \in H$, then the corresponding entry in $A$ is $0$; if $(H, e) \in E(F)$, then the corresponding entry in $A$ is $1$. All remaining entries are left empty. We call $A$ the {\em initial matrix of $M$ with respect to $F$}.
    \item Let $S$ denote the set of all empty entries in the initial matrix. Then the foundation $F_M$ is $F_M = \pastgen{\Funpm(x_i \mid i \in S)}{T}$, where $T$ consists of the two types of relations in~\autoref{thm:iff-modular-triple}. 
\end{enumerate}

Write $x_{H,e}$ for the indeterminate corresponding to an element $(H,e) \in E(G_M)$, with $x_{H,e}=1$ for $(H,e) \in F$. Explicitly, the two types of relations alluded to in (4) are the following:
    
    \begin{enumerate}
    \item[(T1)]\label{rel2:multiplicative} For every modular triple $(H_1, H_2, H_3)$ of distinct hyperplanes  with $L = H_1 \cap H_2 \cap H_3$ and every $a_i \in H_i - L$, we have a relation in $T$ of the form
    \[
\frac{x_{H_1,a_2} x_{H_2,a_3} x_{H_3,a_1}}{x_{H_1,a_3} x_{H_2,a_1} x_{H_3,a_2}} + 1. 
    \]
    \item[(T2)]\label{rel3:U24} For every modular quadruple $(H_1, H_2, H_3, H_4)$ of distinct hyperplanes with $L = H_1 \cap H_2 \cap H_3 \cap H_4$ and every $a_i \in H_i - L$, we have a relation in $T$ of the form
        \[
\frac{x_{H_1,a_3} x_{H_2,a_4}} {x_{H_1,a_4} x_{H_2,a_3}} 
 + \frac{x_{H_1,a_2} x_{H_3,a_4}}{x_{H_1,a_4} x_{H_3,a_2}} - 1.
        \]
    \end{enumerate} 

The validity of the algorithm is implied by the following theorem, which also gives a new proof of the existence of the foundation (\autoref{thm: characterizing property of the foundation}): 

\begin{thm}\label{matroid-bijection-between-equi-classes-and-extensions}
Let $M$ be a matroid with hyperplane incidence graph $G_M$ of $M$. Fix an arbitrary spanning forest $F$ of $G_M$, and let $A$ be the initial matrix of $M$ with respect to $F$. Then for every pasture $P$, there is a bijection 
    \[
        \upR_M(P) \quad \longrightarrow \quad \{\textrm{\emph{pasture morphisms $\pastgen{\Funpm(x_i \mid i \in S)}{T} \to P$}}\} 
    \]
which is functorial in $P$.
\end{thm}

\begin{proof}
    This follows directly from~\autoref{thm:iff-modular-triple} and~\autoref{general-bijection-between-equi-classes-and-extensions}. 
\end{proof}

\subsection{Examples}

We present some examples illustrating \autoref{matroid-bijection-between-equi-classes-and-extensions}.

\begin{ex}\label{ex:foundation-U24}
    Consider the uniform matroid $U_{2,4}$. After choosing a spanning forest in the hyperplane incidence graph, we have an initial matrix as in~\autoref{fig: U24-initial}.

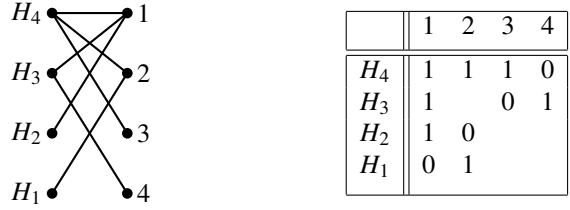
\begin{figure}[t]
\begin{tikzpicture}[x=1.0cm,y=0.8cm, font=\footnotesize,decoration={markings,mark=at position 0.6 with {\arrow{latex}}}]
  \node at (0,0)
  {
  \begin{tikzpicture}[thick]   
   \coordinate[label=left:$H_4$] (H4) at (0,3); 
   \coordinate[label=left:$H_3$] (H3) at (0,2);    
   \coordinate[label=left:$H_2$] (H2) at (0,1); 
   \coordinate[label=left:$H_1$] (H1) at (0,0);    
   \coordinate[label=right:1] (1) at (1,3); 
   \coordinate[label=right:2] (2) at (1,2); 
   \coordinate[label=right:3] (3) at (1,1); 
   \coordinate[label=right:4] (4) at (1,0);
   \fill (H4) circle (2pt);
   \fill (H3) circle (2pt);
   \fill (H2) circle (2pt);
   \fill (H1) circle (2pt);
   \fill (1) circle (2pt);
   \fill (2) circle (2pt);
   \fill (3) circle (2pt);
   \fill (4) circle (2pt);
   \draw (H4) -- (1);
   \draw (H4) -- (2);
   \draw (H4) -- (3);
   \draw (H3) -- (1);
   \draw (H3) -- (4);
   \draw (H2) -- (1);
   \draw (H1) -- (2);
  \end{tikzpicture}};
   \node at (5,0){
 \begin{tabular}{| c || c c c c |} 
 \hline
  & 1 & 2 & 3 & 4 \\ [0.5ex] 
 \hline\hline
 $H_4$ & 1 & 1 & 1 & 0 \\ 
 $H_3$ & 1 &   & 0 & 1 \\
 $H_2$ & 1 & 0 &   &   \\
 $H_1$ & 0 & 1 &   &   \\ [1ex] 
 \hline
 \end{tabular}};
   \end{tikzpicture}
 \caption{A spanning forest for $U_{2,4}$ and the corresponding initial matrix}
 \label{fig: U24-initial}
\end{figure}

By~\autoref{matroid-bijection-between-equi-classes-and-extensions}, every rescaling class of $P$-representations of $U_{2,4}$ corresponds to a unique matrix of the form
\[
A = 
 \begin{pmatrix}
 1 & 1 & 1 & 0 \\ 
 1 & a & 0 & 1 \\
 1 & 0 & b & c \\
 0 & 1 & d & e
 \end{pmatrix}, 
\]
where $a, b, c, d, e \in P^\times$ are chosen so that all relations of the form (T1) and (T2) are satisfied. 

Applying (T2) to the modular quadruple $(H_1, H_2, H_3, H_4)$, we obtain
\[
\frac{dc}{eb} + \frac{1}{ea} - 1 \in N_P. 
\]
Similarly, applying (T1) to the modular triples $(H_1, H_2, H_3)$, $(H_1, H_2, H_4)$, $(H_1, H_3, H_4)$, and $(H_2, H_3, H_4)$, we are forced to have the relations 
\[
\frac{b}{da} = \frac{c}{e} = \frac{d}{e} = \frac{b}{ca} = -1.
\]
Therefore, we obtain an explicit construction of the foundation as 
\[
F_{U_{2,4}} = \pastgen{\Funpm(a,b,c,d,e)}{\frac{dc}{eb} + \frac{1}{ea} - 1, \frac{b}{da} + 1, \frac{c}{e} + 1, \frac{d}{e} + 1, \frac{b}{ca} + 1}. 
\]

Let $x = \frac{cd}{be}$ and $y = \frac{1}{ae}$. There is a pasture isomorphism
\[
F_{U_{2,4}} \to \pastgen{\Funpm(x,y)}{x + y - 1}
\]
defined by  
\[
a \mapsto x^{-1}, \
b \mapsto y^{-1}, \
c \mapsto -xy^{-1}, \
d \mapsto -xy^{-1}, \
e \mapsto xy^{-1}. 
\]
Hence, we also have a presentation of the foundation by the universal cross-ratios $x = \cross{H_1}{H_2}{H_3}{H_4}{}$ and $y = \cross{H_1}{H_3}{H_2}{H_4}{}$ as 
\[
F_{U_{2,4}} = \pastgen{\Funpm(x,y)}{x + y - 1}.
\] 
\end{ex}

\begin{ex}\label{ex:foundation-U25}
    Consider the uniform matroid $U_{2,5}$. After choosing a spanning forest in the hyperplane incidence graph, we have an initial matrix as in~\autoref{fig: U25-initial}.

\begin{figure}[t]
\begin{tikzpicture}[x=1.0cm,y=0.8cm, font=\footnotesize,decoration={markings,mark=at position 0.6 with {\arrow{latex}}}]
  \node at (0,0)
  {
  \begin{tikzpicture}[thick]   
   \coordinate[label=left:$H_5$] (H5) at (0,4); 
   \coordinate[label=left:$H_4$] (H4) at (0,3); 
   \coordinate[label=left:$H_3$] (H3) at (0,2);    
   \coordinate[label=left:$H_2$] (H2) at (0,1); 
   \coordinate[label=left:$H_1$] (H1) at (0,0);    
   \coordinate[label=right:1] (1) at (1,4); 
   \coordinate[label=right:2] (2) at (1,3); 
   \coordinate[label=right:3] (3) at (1,2); 
   \coordinate[label=right:4] (4) at (1,1);
   \coordinate[label=right:5] (5) at (1,0);
   \fill (H5) circle (2pt);
   \fill (H4) circle (2pt);
   \fill (H3) circle (2pt);
   \fill (H2) circle (2pt);
   \fill (H1) circle (2pt);
   \fill (1) circle (2pt);
   \fill (2) circle (2pt);
   \fill (3) circle (2pt);
   \fill (4) circle (2pt);
   \fill (5) circle (2pt);
   \draw (H5) -- (1);
   \draw (H5) -- (2);
   \draw (H5) -- (3);
   \draw (H5) -- (4);
   \draw (H4) -- (1);
   \draw (H4) -- (5); 
   \draw (H3) -- (1);
   \draw (H2) -- (1);
   \draw (H1) -- (2); 
  \end{tikzpicture}};
   \node at (5,0){
 \begin{tabular}{| c || c c c c c |} 
 \hline
  & 1 & 2 & 3 & 4 & 5 \\ [0.5ex] 
 \hline\hline
 $H_5$ & 1 & 1 & 1 & 1 & 0 \\
 $H_4$ & 1 &   &   & 0 & 1 \\ 
 $H_3$ & 1 &   & 0 &   &   \\
 $H_2$ & 1 & 0 &   &   &   \\
 $H_1$ & 0 & 1 &   &   &   \\ [1ex] 
 \hline
 \end{tabular}};
   \end{tikzpicture}
 \caption{A spanning forest for $U_{2,5}$ and the corresponding initial matrix}
 \label{fig: U25-initial}
\end{figure}
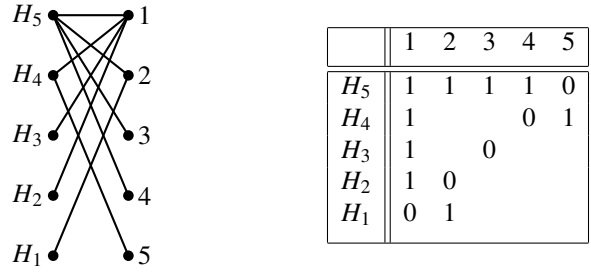

By~\autoref{matroid-bijection-between-equi-classes-and-extensions}, every rescaling class of $P$-representations of $U_{2,5}$ corresponds to a unique matrix of the form
\[
A = \begin{pmatrix}
1 & 1 & 1 & 1 & 0 \\
1 & a & b & 0 & 1 \\
1 & c & 0 & d & e \\
1 & 0 & f & g & h \\
0 & 1 & i & j & k 
\end{pmatrix}, 
\]
where $a, b, c, d, e, f, g, h, i, j, k \in P^\times$ are chosen so that the relations (T1) and (T2) hold.

Our algorithm tells us that the foundation of $U_{2,5}$ is 
\[
F_{U_{2,5}} = \pastgen{\Funpm(a,b,c,d,e,f,g,h,i,j,k)}{T}, 
\]
where $T$ consists of five additive relations
\[
\frac{hd}{ge} + \frac{f}{g} - 1, \
\frac{1}{e} + \frac{dk}{ej} - 1, \
a + \frac{1}{h} - 1, \
\frac{1}{i} + c - 1, \
\frac{jf}{ig} + \frac{b}{ia} - 1 
\]
and ten multiplicative relations 
\[
\frac{f}{ic} = \frac{g}{ja} 
= \frac{h}{k} = \frac{id}{jb}
= \frac{ie}{k} = \frac{j}{k} 
= \frac{fda}{gcb} = \frac{fe}{hc} 
= \frac{g}{ha} = \frac{d}{eb} = -1. 
\]

For $i = 1, 2, 3, 4, 5$, let $x_i$ denote the universal cross-ratio $\cross{H_{i+1}}{H_{i+2}}{H_{i+4}}{H_{i+3}}{}$. 

If we set
\[
x_1 = \frac{hd}{ge}, \
x_2 = \frac{1}{e}, \
x_3 = a, \
x_4 = \frac{1}{i}, \
x_5 = \frac{jf}{ig}, 
\]
then there is a pasture isomorphism
\[
F_{U_{2,5}}
\to
\V = \pastgen{\Funpm(x_1,\dotsc,x_5)}{x_i+x_{i-1}x_{i+1}-1\mid i=1,\dotsc,5}
\]
given by
\begin{align*}
    a \mapsto x_3, \
b \mapsto x_1x_3, \
c \mapsto x_3x_5, \
d \mapsto -x_1x_2^{-1}x_3, \
e \mapsto x_2^{-1}, \
f \mapsto -x_3x_4^{-1}x_5, \\
g \mapsto -x_2^{-1}x_3x_4^{-1}, \
h \mapsto x_2^{-1}x_4^{-1}, \
i \mapsto x_4^{-1}, \
j \mapsto x_2^{-1}x_4^{-1}, \
k \mapsto -x_2^{-1}x_4^{-1}. 
\end{align*}
Thus, once again, we obtain a presentation of the foundation via universal cross-ratios.  
\end{ex}

\begin{ex}\label{ex:foundation-F7}
    Suppose $M$ is the Fano matroid $F_7$. After choosing a spanning forest in the hyperplane incidence graph, we have an initial matrix as in~\autoref{fig: fano-initial}. 
    
\begin{figure}[t]
\begin{tikzpicture}[x=1.0cm,y=0.8cm, font=\footnotesize,decoration={markings,mark=at position 0.6 with {\arrow{latex}}}]
  \node at (0,0)
  {
  \begin{tikzpicture}[thick]
   \coordinate[label=left:567] (567) at (0,6);    
   \coordinate[label=left:347] (347) at (0,5); 
   \coordinate[label=left:245] (245) at (0,4);    
   \coordinate[label=left:236] (236) at (0,3); 
   \coordinate[label=left:146] (146) at (0,2);    
   \coordinate[label=left:135] (135) at (0,1); 
   \coordinate[label=left:127] (127) at (0,0);    
   \coordinate[label=right:1] (1) at (1,6); 
   \coordinate[label=right:2] (2) at (1,5); 
   \coordinate[label=right:3] (3) at (1,4); 
   \coordinate[label=right:4] (4) at (1,3); 
   \coordinate[label=right:5] (5) at (1,2); 
   \coordinate[label=right:6] (6) at (1,1); 
   \coordinate[label=right:7] (7) at (1,0);
   \fill (567) circle (2pt);
   \fill (347) circle (2pt);
   \fill (245) circle (2pt);
   \fill (236) circle (2pt);
   \fill (146) circle (2pt);
   \fill (135) circle (2pt);
   \fill (127) circle (2pt);
   \fill (1) circle (2pt);
   \fill (2) circle (2pt);
   \fill (3) circle (2pt);
   \fill (4) circle (2pt);
   \fill (5) circle (2pt);
   \fill (6) circle (2pt);
   \fill (7) circle (2pt);
   \draw (567) -- (1);
   \draw (567) -- (2);
   \draw (567) -- (3);
   \draw (567) -- (4);
   \draw (347) -- (1);
   \draw (347) -- (5);
   \draw (347) -- (6);
   \draw (245) -- (1);
   \draw (245) -- (7);
   \draw (236) -- (1);
   \draw (146) -- (2);
   \draw (135) -- (2);
   \draw (127) -- (3);
  \end{tikzpicture}};
   \node at (5,0){
 \begin{tabular}{| c || c c c c c c c|} 
 \hline
  & 1 & 2 & 3 & 4 & 5 & 6 & 7 \\ [0.5ex] 
 \hline\hline
 567 & 1 & 1 & 1 & 1 & 0 & 0 & 0 \\ 
 347 & 1 &   & 0 & 0 & 1 & 1 & 0 \\
 245 & 1 & 0 &   & 0 & 0 &   & 1 \\
 236 & 1 & 0 & 0 &   &   & 0 &   \\
 146 & 0 & 1 &   & 0 &   & 0 &   \\
 135 & 0 & 1 & 0 &   & 0 &   &   \\
 127 & 0 & 0 & 1 &   &   &   & 0 \\ [1ex] 
 \hline
 \end{tabular}};
   \end{tikzpicture}
 \caption{A spanning forest for $F_7$ and the corresponding initial matrix}
 \label{fig: fano-initial}
\end{figure}
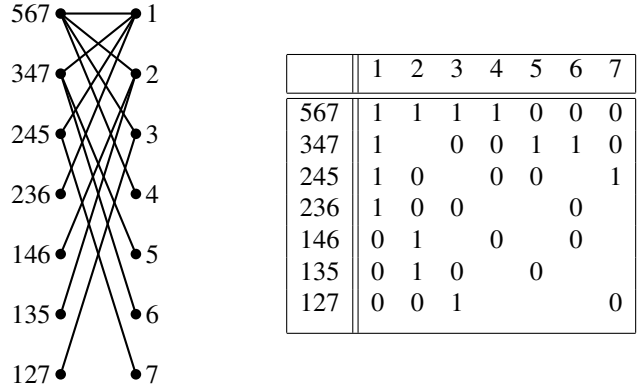

In this case, the only relevant relations are those of the form (T1).
(There are no non-degenerate modular quadruples of hyperplanes in $F_7$.)

\autoref{lemma: first properties of cross ratios} and \autoref{thm:iff-modular-triple} imply that for every ``free variable'' $x_{H,e}$ with $(H,e) \in S$, we have a relation of the form $x_{H,e} = 1$. For example, consider the modular pair of hyperplanes $(H_{567}, H_{347})$ with intersection $L = \{7\}$ and elements $1, 2 \notin H_{567} \cup H_{347}$ with $\gen{L1} = \gen{L2} = H_{127}$. Then \autoref{lemma: first properties of cross ratios} gives
\[
1 = \frac{x_{567,1} x_{347,2}}{x_{567,2} x_{347,1}}. 
\]
Since $x_{567,1} = x_{567,2}=x_{347,1} = 1$, we conclude that $x_{347,2} = 1$. A similar computation shows that all missing entries in the initial matrix must be $1$. By~\autoref{matroid-bijection-between-equi-classes-and-extensions}, there exists at most one morphism $F_M \to P$ for every pasture $P$, which corresponds to the matrix
\[
A = 
\begin{pmatrix}
1 & 1 & 1 & 1 & 0 & 0 & 0 \\ 
1 & 1 & 0 & 0 & 1 & 1 & 0 \\
1 & 0 & 1 & 0 & 0 & 1 & 1 \\
1 & 0 & 0 & 1 & 1 & 0 & 1 \\
0 & 1 & 1 & 0 & 1 & 0 & 1 \\
0 & 1 & 0 & 1 & 0 & 1 & 1 \\
0 & 0 & 1 & 1 & 1 & 1 & 0 
 \end{pmatrix}. 
\]
As a result, the foundation of $F_7$ must be either $\Funpm$ (if $-1 \neq 1$ in $F_M$) or $\F_2$ (if $-1 = 1$ in $F_M$). 
Property (T1), applied to an arbitrary modular triple of distinct hyperplanes, implies that $-1 = 1$ in $F_M$, and hence the foundation of $F_7$ must be $\F_2$. 
\end{ex}

\subsection{Universal cross-ratios as generators for the foundation}
\label{subsection: generators for the foundation}

In the previous section, we saw that in three different examples, the foundation is generated by universal cross-ratios. 
(All universal cross-ratios in the Fano matroid are degenerate.) We now show that this is a general phenomenon which holds in every matroid.

\medskip

Let $P$ be a pasture. We say that a subset $S\subset P^\times$ \emph{generates $P$} if $S\cup\{-1\}$ generates $P^\times$ as a group.

\begin{lemma}\label{lemma: generating set corresponds to epimorphism}
 Let $P$ be a pasture and $S\subset P^\times$. Then $S$ generates $P$ if and only if every pasture morphism $g\colon P\to Q$ is uniquely determined by its restriction $g\vert_S\colon S\to Q$ to $S$.
\end{lemma}

\begin{proof}
The forward implication is evident. The converse implication claims that $S$ generates $P$ as a pasture, i.e., the image of $S$ in the quotient group $P^\times/\{\pm1\}$ is a set of generators. This is the case if and only if for every abelian group $H$, the map $\Hom(P^\times/\{\pm1\},H)\to \Maps(S,H)$ is injective.

In order to deduce this latter condition, we define, for given $H$, the pasture $Q=H\cup\{0\}$ whose null set consists of all two term sums $a+a$ with $a\in Q$ and all sums $a+b+c$ of three nonzero terms $a,b,c\in Q^\times$. Then every group homomorphism $P^\times/\{\pm1\}\to H$ extends (by $-1\mapsto1$ and $0\mapsto0$) to a uniquely determined pasture morphism $P\to Q$, which defines an injective map $\Hom(P^\times/\{\pm1\},H)\to\Hom(P,Q)$ (the latter being the Hom set of pasture morphisms). By assumption, a pasture morphism $P\to Q$ is uniquely determined by its restriction to $S\to H\subset Q$, which shows that also $\Hom(P,Q)\to\Maps(S,H)$ is injective. As the composition of these two injective maps, the map $\Hom(P^\times/\{\pm1\},H)\to \Maps(S,H)$ in question is so, too.
\end{proof}

We now state and prove the main result of this section. The proof makes use, in a crucial way, of Tutte's path theorem (\autoref{thm: Tutte's path theorem}).

\begin{thm}\label{thm: the foundation is generated by cross ratios}
 The foundation $F_M$ of $M$ is generated by the universal cross-ratios of $M$.
\end{thm}

\begin{proof}
By~\autoref{thm: characterizing property of the foundation} and~\autoref{lemma: generating set corresponds to epimorphism}, it suffices to show that, up to rescaling, every $P$-representation $\varphi$ of $M$ is uniquely determined by the values of all cross-ratios. After decomposing $M$ into a direct sum of connected matroids, we can assume (by~\autoref{prop:founadtion of direct sum}) that $M$ is connected. We proceed by induction on the cardinality $n=|E(M)|$ of the ground set. If $n=1$, then $M$ has at most one rescaling equivalence class. If $n>1$, then $M$ has a connected minor on $n-1$ elements (by~\cite[Theorem 4.3.1]{Oxley11}).
After dualizing $M$ if necessary, we can assume (by~\autoref{prop:founadtion of dual}) that $M'=M \minus a$ is connected for some element $a \in E$. 
Let $\cH$ (resp.~$\cH'$) denote the set of hyperplanes of $M$ (resp. $M'$).

We wish to show that it is possible to reconstruct a $P$-representation $\varphi$ of $M$, up to rescaling equivalence, from its cross-ratios. By the induction hypothesis, the rescaling equivalence class $[\varphi']$ of the restriction $\varphi' := \varphi\vert_{M'} = \{\varphi'_{H'}\}_{H'\in\cH'}$ to $M'$ is uniquely determined by its cross-ratios, which are also cross-ratios of $M$. 
It therefore suffices to prove that the rescaling equivalence class of $\varphi$ is uniquely determined by $[\varphi']$ together with the values of all cross-ratios of $M$ which ``involve'' the element $a$. More precisely, we will show that one can reconstruct $[\varphi]$ from $[\varphi']$, all cross-ratios of the form $\cross{H_1}{H_2}{a}{b}{\varphi}$ with $a,b \notin H_1 \cup H_2$, and all cross-ratios of the form $\cross{H_1}{H_2}{b_1}{b_2}{\varphi}$ for which $a \in H_1 \cup H_2$.

Let $H$ be a hyperplane of $M$, for which we wish to reconstruct the values $\varphi_H(e) \in P^\times$ for $e \not\in H$, up to a common rescaling.
Let $\Gamma = \{F \in \Lambda_{M'} \mid a \in \gen{F}_M \}$ be the modular cut in $M'$ consisting of those flats whose closure in $M$ contains $a$. 
We have the following three cases, cf.~\autoref{lemma: hyperplanes of a deletion}.

\medskip\noindent\textbf{Case 1.} 
Suppose $a \in H$ and $H - a$ is a hyperplane of $M'$. In this case, $\varphi_H(a)=0$, and thus $\varphi_H$ is determined by its restriction $\varphi_H\vert_{E'}=\varphi'_{H - a}$ to the ground set $E'=E - a$ of $M'$.

\medskip\noindent\textbf{Case 2.} 
Suppose $a \notin H$. Then $H$ is a hyperplane of $M'$ and there exists a corresponding hyperplane function $\varphi'_H$ for $M'$ in the restriction $\varphi'$. To determine $\varphi_H$, we need only determine $\varphi_H(a)$, since by definition $\varphi_H(e) = \varphi'_H(e)$ for all $e \notin H$ with $e \ne a$. 

Fix a hyperplane $H_0 \in \cH$ with $a \notin H_0$, so that $H_0$ is also a hyperplane of $M'$. Rescaling all $a$-coordinates if necessary, we can assume $\varphi_{H_0}(a) = 1$. By \autoref{thm: Tutte's path theorem} applied to the connected matroid $M'$, there exists a Tutte path $\gamma = (H_0, \dots, H_s = H)$ of hyperplanes in $M'$ such that $H_i \notin \Gamma$, $\cork_{M'}(H_{i-1} \cap H_i) = 2$, and $H_{i-1} \cup H_i \ne E' := E - a$ for each $1 \leqslant i \leqslant s$.
Therefore, it suffices to consider the case where $\cork_{M'}(H_{0} \cap H) = 2$ and $H_{0} \cup H \ne E'$. 

For this, if we pick an arbitrary $b \in E' - (H_0 \cup H)$, then $\varphi_H(a)$ is determined by $\varphi_H(b), \varphi_{H_0}(b)$, and the value of the cross-ratio $\cross{H}{H_0}{a}{b}{\varphi}$ by the formula 
\[
\varphi_H(a) = \cross{H}{H_0}{a}{b}{\varphi} \cdot \frac{\varphi_H(b)}{\varphi_{H_0}(b)} \cdot \varphi_{H_0}(a).
\]
Since $\varphi_H(b) = \varphi'_H(b)$ and $\varphi_{H_0}(b) = \varphi'_{H_0}(b)$ are determined by the restriction $\varphi'$, it follows that $\varphi_H(a)$ is determined by $[\varphi']$ and the values of all cross-ratios of the form $\cross{H_1}{H_2}{a}{b}{\varphi}$ for which $a \notin H_1 \cup H_2$. 

\medskip\noindent\textbf{Case 3.} 
Suppose $a \in H$ and $L := H - a$ is a corank $2$ flat of $M'$. In this case, we need to determine $\varphi_H(b)$ for all $b \in E - H = E' - L$. 

Note that $L$, and all hyperplanes $H \supset L$ in $M'$, are not in $\Gamma$. Let $b_1 \in E'- L$ and let $H_1 = \gen{Lb_1}_{M'}$, which is also a hyperplane of $M$ not containing $a$. Let $H_2$ be a hyperplane of $M'$ such that $L = H_1 \cap H_2$; thus $H_2$ is another hyperplane of $M$ not containing $a$. Rescaling $\varphi_H$ if necessary, we may assume that $\varphi_H(b_1) = 1$. 

If $b \in H_1$, then $\varphi_H(b)$ is determined by $\varphi_{H_2}(b) = \varphi'_{H_2}(b), \varphi_{H_2}(b_1) = \varphi'_{H_2}(b_1)$, and the value of the cross-ratio $\cross{{H_2}}{H}{b_1}{b}{\varphi}$ by the formula 
\[
\varphi_H(b) = \cross{{H_2}}{H}{b_1}{b}{\varphi} \cdot \frac{\varphi_{H_2}(b)}{\varphi_{H_2}(b_1)} \cdot \varphi_H(b_1). 
\]

If $b \in E - (H_1 \cup H) = E' - H_1$, then we consider the hyperplane $H_3 = \gen{Lb}_{M'}$ of $M'$, which can also be considered as a hyperplane of $M$ not containing $a$. Applying~\autoref{thm:iff-modular-triple} to the modular triple $(H_1, H_3, H)$ of distinct hyperplanes in $M$, we obtain the equation
\[
\frac{\varphi_{H_1}(b) \varphi_{H_3}(a) \varphi_H(b_1)}{\varphi_{H_1}(a) \varphi_{H_3}(b_1) \varphi_H(b)} = -1. 
\]
By Case 2, all terms in this equation, except for $\varphi_H(b)$, are uniquely determined by $[\varphi']$ and the values of the cross-ratios of the form $\cross{H_1}{H_2}{a}{b}{\varphi}$ for which $a \notin H_1 \cup H_2$; therefore, so is $\varphi_H(b)$, which is given by the formula 
\[
\varphi_H(b) = -\frac{\varphi_{H_1}(b) \varphi_{H_3}(a) \varphi_H(b_1)}{\varphi_{H_1}(a) \varphi_{H_3}(b_1)}.
\]

We conclude that the hyperplane function $\varphi_H$ is uniquely determined by $[\varphi']$ and the values of all cross-ratios of the form $\cross{H_1}{H_2}{a}{b}{\varphi}$ for which $a \notin H_1 \cup H_2$ and $\cross{H_1}{H_2}{b_1}{b_2}{\varphi}$ for which $a \in H_1 \cup H_2$.
\end{proof}

From this we recover Theorems 7.32 and 7.35 in~\cite{Baker-Lorscheid21b}.

\begin{thm}\label{thm: foundations of regular and binary matroids}
 A matroid is binary if and only if its foundation is either $\Funpm$ or $\F_2$. It is regular if and only if its foundation is $F_M = \Funpm$. 
\end{thm}

\begin{proof}
 Since $\Funpm$ maps into every field, a matroid with foundation $\Funpm$ is regular. Since both $\Funpm$ and $\F_2$ map to $\F_2$, a matroid with foundation $\Funpm$ or $\F_2$ is binary.
 
 If $M$ is binary, then it has no $U_{2,4}$-minor. Therefore, all universal cross-ratios are degenerate and generate the trivial subgroup $\{1\}$ of $F_M^\times$. Consequently, $F_M^\times$ is equal to $\{\pm1\}$ or $\{1\}$. Since $M$ is representable over $\F_2$, the foundation maps to $\F_2$ and therefore the null set of $F_M$ cannot contain any relation with exactly $3$ nonzero terms. The only pastures fitting these criteria are $\Funpm$ and $\F_2$. 
 
 If $M$ is regular, then $M$ is binary. Since $M$ is representable over fields of characteristic different from $2$, its foundation $F_M$ cannot be $\F_2$, which shows that $F_M=\Funpm$.
\end{proof}

\begin{cor}\label{cor: foundation of F7}
 The foundation of the Fano matroid $F_7$ is $\F_2$.
\end{cor}

\begin{proof}
 The Fano matroid is binary, so its foundation is either $\Funpm$ or $\F_2$. Since there is a morphism $\Funpm\to k$ to every field $k$, but no $k$-representation of $F_7$ if the characteristic of $k$ differs from $2$, the foundation of $F_7$ has to be $\F_2$ (cf.~\autoref{ex:foundation-F7}). 
\end{proof}

We now give an explicit isomorphism between $F_{M^\ast}$ and $F_{M}$, following~\cite[Proposition 4.8]{Baker-Lorscheid20}. Let $(H_1, H_2, H_3, H_4) \in \Theta_{M^\ast}$ be a modular quadruple of hyperplanes in $M^\ast$ with $L = H_1 \cap H_2 \cap H_3 \cap H_4$. Choose a set $J \subset L$ that is independent in $M^\ast$ with $\gen{J}_{M^\ast} = L$, and choose $e_i \in H_i - L$. Let $I = E - Je_1e_2e_3e_4$. 

\begin{prop}\label{prop: foundation of dual matroid, explicit}
 The foundation of $M^\ast$ is canonically isomorphic to the foundation of $M$, where the isomorphism $f: F_{M^\ast} \to F_M$ is determined by 
 \[
\cross{H_1}{H_2}{H_3}{H_4}{} = \cross{\gen{Je_1}_{M^\ast}}{\gen{Je_2}_{M^\ast}}{\gen{Je_3}_{M^\ast}}{\gen{Je_4}_{M^\ast}}{}\
\mapsto \
\cross{\gen{Ie_1}_{M}}{\gen{Ie_2}_{M}}{\gen{Ie_3}_{M}}{\gen{Ie_4}_{M}}{}. 
 \]
\end{prop}

\begin{ex}\label{ex:foundation-F7*}
  By~\autoref{cor: foundation of F7} and~\autoref{prop: foundation of dual matroid, explicit}, the foundation of the dual $F_7^\ast$ of the Fano matroid is $\F_2$. 
\end{ex}

\begin{ex}\label{ex:foundation-U35}
   By~\autoref{ex:foundation-U25} and~\autoref{prop: foundation of dual matroid, explicit}, the foundation of the uniform matroid $U_{3,5}$ is isomorphic to $\V = \pastgen{\Funpm(x_1,\dotsc,x_5)}{x_i+x_{i-1}x_{i+1}-1\mid i=1,\dotsc,5}$, where the isomorphism is given by 
   \[
\cross{H_{i, i+1}}{H_{i, i+2}}{H_{i, i+4}}{H_{i, i+3}}{} \mapsto x_i. 
   \]
   Here, $H_{i,j}$ denotes the hyperplane $\{i, j\}$ in $U_{3,5}$, and all subscripts are read modulo $5$. 
\end{ex}

\subsection{The fundamental presentation}
\label{subsection: fundamental presentation by small minors}

Let $\Lambda$ be a geometric lattice of type $M$ and $\Lambda'$ an upper sublattice of type $M'$ of $\Lambda$. Then the lattice inclusion $\Lambda'\hookrightarrow\Lambda$ restricts to an inclusion $\cH_{M'}\hookrightarrow\cH_M$. Therefore, the restriction of a $P$-representation $\varphi: \cH \to P^E$ of $M$ to $\cH'$ is a $P$-representation of $M'$, where $P$ is an arbitrary pasture. This restriction evidently commutes with rescaling equivalence, and thus defines a map $\upR_{M}(P)\to\upR_{M'}(P)$ between the corresponding realization spaces which is functorial in $P$. By \autoref{thm: characterizing property of the foundation}, the realization space $\upR_M$ is represented by the foundation $F_M$, and by the Yoneda lemma, the functorial map $\upR_M(-)\to\upR_{M'}(-)$ is induced by a pasture morphism $F_{M'}\to F_M$.

The morphism $F_{M'}\to F_M$ maps the universal cross-ratio $\cross{H_1}{H_2}{H_3}{H_4}{}\in F_{M'}$ of $M'$ to the universal cross-ratio $\cross{H_1}{H_2}{H_3}{H_4}{}\in F_{M}$ of $M$. Since the foundation is generated by the universal cross-ratios (by \autoref{thm: the foundation is generated by cross ratios}), this determines the map $F_{M'}\to F_M$. For more details, see~\cite[Proposition 4.9]{Baker-Lorscheid20}.

\begin{prop}\label{prop: foundation of embedded minors}
 Let $M'=M\minor JI$ be an embedded minor of $M$ whose associated lattice inclusion $\Lambda'\hookrightarrow\Lambda$ is a bijection, i.e., $M'$ and $M$ have the same simplification. Then the morphism $F_{M'}\to F_M$ is an isomorphism.
\end{prop}

\begin{proof}
 This follows from the fact that the foundation depends only on the lattice of flats of $M$, which equals that of $M'$.
\end{proof}

\begin{ex}\label{ex:foundation-C5}
    Consider the matroid $C_5$ on $E = \{1,2,3,4,5\}$ whose set of bases is $\binom{E}{3} - \{123\}$. The set of hyperplanes of $C_5^\ast$ is $\cH(C_5^\ast) = \{1,2,3,45\}$. Therefore, $C_5^\ast$ has the same lattice of flats as $U_{2,4}$. By~\autoref{ex:foundation-U24},~\autoref{prop: foundation of dual matroid, explicit}, and~\autoref{prop: foundation of embedded minors}, the foundation of $C_5$ is isomorphic to $\pastgen{\Funpm(x,y)}{x + y - 1}$, where $x = \cross{H_{15}}{H_{25}}{H_{35}}{H_{45}}{}$ and $y = \cross{H_{15}}{H_{35}}{H_{25}}{H_{45}}{}$. 
\end{ex}

Let $\cL_M$ be the diagram of all upper sublattices of $\Lambda$ of types $U_{2,4}$, $U_{2,5}$, $U_{3,5}$, $C_5$, $F_7$, and $F_7^\ast$, together with all lattice embeddings. Let $\cF_M$ be the associated diagram of the foundations of these upper sublattices, together with the induced morphisms. 

\begin{thm}[Fundamental presentation]\label{thm: fundamental presentation of the foundation}
 The canonical morphism $\colim {\cF}_M \to F_M$ is an isomorphism.
\end{thm}

\begin{proof}
 By Yoneda's Lemma and \autoref{thm: characterizing property of the foundation}, it suffices to show that the natural map $\Phi_{M,P}: \Hom(F_M, P) \to \Hom(\colim\cF_M, P)$ is a bijection for every pasture $P$. Since $\cL_M$ contains all upper sublattices of type $U_{2,4}$, it follows from \autoref{thm: the foundation is generated by cross ratios} that $\Phi_{M,P}$ is injective. 

 The proof of surjectivity is more complicated and rests on an application of the extended version \autoref{thm:extended-homotopy-theorem} of Tutte's homotopy theorem. We establish this claim by induction on the size $n=\# E$ of $M$.

 The claim is evident for regular matroids: in this case, $F_M=\Funpm$ and $\cF_M$ is empty, i.e., $\colim\cF_M$ is the initial object $\Funpm$. This establishes the base case, since all matroids of size $n\leq3$ are regular.

 Assume that $n\geq4$ and consider a morphism $\psi:\colim\cF_M\to P$. We aim to show that $\psi$ is the image of a morphism $\hat\psi:F_M\to P$ under $\Phi_{M,P}$.
 
 As a first step, we note that we can assume without loss of generality that $M$ is connected, since the result for $M=M_1\oplus M_2$ follows from the fundamental presentations of $F_{M_1}$ and $F_{M_2}$ in terms of the canonical isomorphisms
 \begin{multline*}
  F_{M} \ \simeq \ F_{M_1}\otimes F_{M_2} \ \simeq \ (\colim\cF_{M_1})\otimes(\colim\cF_{M_2}) \ \simeq \ \colim(\cF_{M_1}\sqcup\cF_{M_2}) \ = \ \colim \cF_{M}, 
 \end{multline*}
 where the first isomorphism follows from \autoref{prop:founadtion of direct sum} and the identity $\cF_M=\cF_{M_1}\sqcup\cF_{M_2}$ follows from the fact that every upper sublattice in $\cL_M$ is indecomposable, and thus must belong to either $\cL_{M_1}$ or $\cL_{M_2}$.

 By \cite[Theorem 4.3.1]{Oxley11}, there is an element $a$ of $M$ such that either $M/a$ or $M\minus a$ is connected. Since the collection of upper sublattices over which the colimit is taken is closed under duality (note that the lattice of $C_5^\ast$ equals that of $U_{2,4}$; cf.\ \autoref{ex:foundation-C5}), we may assume that $M'=M\minus a$ is connected and of the same rank as $M$ (i.e., $a$ is not a coloop of $M$).
 
 The lattice $\Lambda_{M'}$ is embedded as an upper sublattice of $\Lambda_M$, which induces a morphism $\iota_{M'}:F_{M'}\to F_M$ as well as an inclusion of the fundamental diagram $\cL_{M'}$ of $M'$ as a subdiagram of $\cL_M$, and thus a morphism $\eta_{M'}:\colim\cF_{M'}\to\colim\cF_M$. Let $\psi'=\psi\circ\eta_{M'}$ be the restriction of $\psi$ to $\colim\cF_{M'}$. The inductive hypothesis applies to $M'$ and identifies $\psi'$ with the image of a morphism $\hat\psi':F_{M'}\to P$. By the universal property of the foundation (\autoref{thm: characterizing property of the foundation}), $\hat\psi'$ corresponds to the rescaling class of a $P$-representation $\varphi':\cH'\to P^{E'}$, where $E'=E-a$ is the ground set of $M'$ and $\cH'$ is the collection of hyperplanes of $M'$.

 To simplify the following arguments, we include the upper sublattices of all regular minors of $M$ on up to $5$ elements in $\cL_M$, which does not change the colimit $\colim\cF_M$ since regular matroids have foundation $\Funpm$, which maps uniquely into any other pasture. In particular, this means that $\cL_M$ contains all geometric lattices with up to $5$ atoms, with the unique exception of type $U_{2,4}\oplus U_{1,1}$.
 
 In the following, we use the notation
 \[
  \cross{H_1}{H_2}{e_3}{e_4}{\psi} \ := \ \cross{H_1}{H_2}{H_3}{H_4}{\psi} \ := \ \psi\circ\iota_\xi\big(\cross{H_1}{H_2}{H_3}{H_4}{}\big)
 \]
 for $\xi=(H_1,H_2,e_3,e_4)\in\Xi_M$, where $H_3=\gen{Le_3}$ and $H_4=\gen{Le_4}$ for $L=H_1\cap H_2$, and where $\iota_\xi:F_{\Lambda_\xi}\to\colim \cF_M$ is the canonical inclusion for the upper sublattice $\Lambda_\xi=\{L,H_1,H_2,H_3,H_4,E\}$ of $\Lambda_M$, which is of type $U_{2,3}$ (if $H_3=H_4$) or type $U_{2,4}$ (if $H_3\neq H_4$). 

 The rest of the proof proceeds in three major steps:
 \begin{enumerate}
  \item[\textbf{Step 1}] By reverse-engineering the proof of \autoref{thm: the foundation is generated by cross ratios}, we extend $\varphi'$ to a map $\varphi:\cH\to P^E$, which satisfies $H=\{e\in E\mid \varphi_H(e)=0\}$ for all hyperplanes $H\in\cH$ of $M$.
  
  \item[\textbf{Step 2}] Even though we do not know at this point that $\varphi$ is a $P$-representation, we can define the cross-ratios
  \[
   \cross{H_1}{H_2}{e_3}{e_4}{\varphi} \ = \ \frac{\varphi_{H_1}(e_3)\cdot\varphi_{H_2}(e_4)}{\varphi_{H_1}(e_4)\cdot\varphi_{H_2}(e_3)} \quad \in \quad P^\times
  \]
  for $(H_1,H_2,e_3,e_4)\in\Xi_M$. In this step, we verify that $\cross{H_1}{H_2}{e_3}{e_4}{\varphi}=\cross{H_1}{H_2}{e_3}{e_4}{\psi}$ for all $(H_1,H_2,e_3,e_4)\in\Xi_M$.
  \item[\textbf{Step 3}] We verify that $\varphi$ is indeed a $P$-representation. By \autoref{thm:iff-modular-triple}, it suffices to show that
  \begin{equation}\label{eq:tripleratio}
   \tripleratio{H_1}{H_2}{H_3}{e_1}{e_2}{e_3}{\varphi} \ = \ \frac{\varphi_{H_1}(e_2)\varphi_{H_2}(e_3)\varphi_{H_3}(e_1)}{\varphi_{H_1}(e_3)\varphi_{H_2}(e_1)\varphi_{H_3}(e_2)} = -1
  \end{equation}
  for every triple of distinct hyperplanes $H_1$, $H_2$, $H_3$ that intersect in a corank $2$ flat $L$ and elements $e_i\in H_i-L$ for $i=1,2,3$. 
  
  Note that the second condition of \autoref{thm:iff-modular-triple} is automatically satisfied since $\cross{H_1}{H_i}{e_j}{e_4}{\varphi}=\cross{H_1}{H_i}{e_j}{e_4}{\psi}$ for $\{i,j\}=\{2,3\}$ by Step 2 and since the corresponding relation holds for cross-ratios of $U_{2,4}$.
 \end{enumerate}
 Once these claims are established, we conclude that $\varphi$ is a $P$-representation whose rescaling class corresponds to a morphism $\hat\psi:F_M\to P$. By \autoref{cor: cross ratios as images of the universal cross ratio}, $\hat\psi$ maps the universal cross-ratios to $\cross{H_1}{H_2}{e_3}{e_4}{\varphi}=\psi\circ\iota_\xi\big(\cross{H_1}{H_2}{H_3}{H_4}{}\big)$, which proves that $\Phi_{M,P}(\hat\psi)=\psi$ and establishes the surjectivity of $\Phi_{M,P}$.

 We verify each of the steps according to the following ordering of cases (cf.~\autoref{lemma: hyperplanes of a deletion}):
 \\[5pt]
 \begin{tabular}{ll}
  \textbf{Case A} & $H_i-a\in\cH'$ for all $i$ and $a\neq e_k$ for all $k$; \\
  \textbf{Case B} & $H_i-a\in\cH'$ for all $i$ and $a=e_k$ for some $k$; \\
  \textbf{Case C} & $H_i-a\notin\cH'$ for some $i$, and $H_i-a\subset H_k$ for all $k$; \\
  \textbf{Case D} & $H_i-a\notin\cH'$ for some $i$, and $H_i-a\not\subset H_k$ for some $k$. \\
 \end{tabular}
 \\
 
 Note that Case D does not occur in Step 1, since $H - a \subseteq H$ holds for every hyperplane $H$. For technical reasons, we will verify 1A through 1C, then 2A, 3A, 2B, 3B, 2C, and 3C (in that order), followed by 3D and then 2D.
  
 Before we explain the proof of each case in detail, we point out that Tutte's homotopy theorem (or, more precisely, its ``extended'' version \autoref{thm:extended-homotopy-theorem}) enters the proof only in Case 2B, which can therefore be regarded as the deepest part of the entire argument.

 \medskip\noindent\textbf{Step 1.}
 We define $\varphi:\cH\to P^E$ in terms of the values $\varphi_H(e)$. The requirement that the functions $\varphi_H:E\to P$ are hyperplane functions leads to the definition $\varphi_H(e)=0$ for $e\in H$. The equality $H=\{e\in E\mid\varphi_H(e)=0\}$ will follow from the fact that we will define the value of $\varphi_H$ for all $e\notin H$ as an element of $P^\times$. 

 \medskip\noindent\textbf{Case 1A.}
 If $H'=H-a\in\cH'$ and $e\in E'-H'$, then we define $\varphi_H(e)=\varphi'_{H'}(e)$. Note that this guarantees that $\varphi'$ is the restriction of $\varphi$ to $\cH'$ and $E'$.
 
 \medskip\noindent\textbf{Case 1B.}
 If $a\notin H$, then necessarily $H=H-a\in\cH'$. Let $\Gamma=\{F\in\Lambda_{M'}\mid a\in\gen{F}_M\}$ be the modular cut of $M'$ determined by the single-element extension $M$. Let $G=G_{M',\Gamma}$ be the graph whose vertices are hyperplanes $H\in\cH'\setminus\Gamma$ of $M'$ off $\Gamma$ and whose edges $(H,H')$ are pairs of hyperplanes whose intersection $L=H\cap H'$ is an indecomposable flat of corank $2$. As explained in \autoref{cor: graph version of the path theorem}, Tutte's path theorem shows that $G$ is connected.
 
 In order to define $\varphi_H(a)$ in this case, we choose (arbitrarily) a spanning tree $T$ of $G$, a root $H_0$ of $T$, and a value $\varphi_{H_0}(a)\in P^\times$ and an element $b_{H,H'}\in E'-(H\cup H')$ for every edge $(H,H')\in G$, with $b_{H,H'}=b_{H',H}$. We define
 \[
  \varphi_H(a) \ = \ \cross H{H'}a{b_{H,H'}}\psi \ \cdot \ \frac{\varphi_{H}(b_{H,H'})}{\varphi_{H'}(b_{H,H'})} \ \cdot \ \varphi_{H'}(a)
 \]
 recursively over the tree distance from $H_0$, where $(H,H')$ is an edge in $T$ with $H'$ closer to $H_0$ than $H$.

 \medskip\noindent\textbf{Case 1C.}
 If $L=H-a\notin\cH'$, then $L$ is a flat of $M'$ of corank $2$. We choose (arbitrarily) an element $b_0\in E'-L$ and a value $\varphi_H(b_0)\in P^\times$ where $H_{b_0}=\gen{Lb_0}$, as well as a hyperplane $H'$of $M$ that contains $L$ and differs from both $H$ and $H_{b_0}$ (note that $L$ is contained in at least 3 hyperplanes).
 
 For $b\in H_{b_0}-L$, we define
 \[
  \varphi_{H}(b) \ = \ \frac{\varphi_{H'}(b)}{\varphi_{H'}(b_0)} \ \cdot \ \varphi_{H}(b_0).
 \]
 For $b\in E'-H_{b_0}$ and $H_b=\gen{Lb}$, we define
 \[
  \varphi_{H}(b) \ = \ - \ \frac{\varphi_{H_{b_0}}(b) \cdot \varphi_{H_b}(a)}{\varphi_{H_{b_0}}(a) \cdot \varphi_{H_b}(b_0)} \ \cdot \ \varphi_{H}(b_0).
 \]

 \medskip\noindent\textbf{Steps 2 and 3.}
 We keep the notation from before. That is, when we verify that $\cross{H_1}{H_2}{e_3}{e_4}{\varphi}=\cross{H_1}{H_2}{e_3}{e_4}{\psi}$ (Step 2), we assume that $\xi=(H_1,H_2,e_3,e_4)\in\Xi_M$, i.e., $L=H_1\cap H_2$ is a corank $2$ flat of $M$ and $e_3,e_4\notin H_1\cup H_2$, and when we verify that $\tripleratio{H_1}{H_2}{H_3}{e_1}{e_2}{e_3}{\varphi}=-1$ (Step 3), we assume that $L=H_1\cap H_2\cap H_3$ is a corank $2$ flat of $M$ and $e_i\in H_i-L$ for $i=1,2,3$.
 
 \medskip\noindent\textbf{Case 2A.}
 Assume that $H_1'=H_1-a$ and $H_2'=H_2-a$ are in $\cH'$ and $a\notin\{e_3,e_4\}$. Then 
 \[
  \cross{H_1}{H_2}{e_3}{e_4}{\varphi} \ = \ \frac{\varphi'_{H'_1}(e_3)\cdot\varphi'_{H'_2}(e_4)}{\varphi'_{H'_1}(e_4)\cdot\varphi'_{H'_2}(e_3)} 
 \]
 by the definition of $\varphi$. Since $\Lambda_\xi=\{L,H_1,H_2,H_3,H_4,E\}$ is contained in $\Lambda_{M'}$, the lattice of flats of $M'$, the morphism $\iota_\xi:F_{\Lambda_\xi}\to \colim\cF_M$ factors through $\iota'_\xi:F_{\Lambda_\xi}\to\colim\cF_{M'}$. Since $\varphi'$ is a $P$-representation whose rescaling class corresponds to $\psi'$, we have  
 \[
  \frac{\varphi'_{H'_1}(e_3)\cdot\varphi'_{H'_2}(e_4)}{\varphi'_{H'_1}(e_4)\cdot\varphi'_{H'_2}(e_3)} \ = \ \psi'\circ\iota'_\xi\big(\cross{H_1}{H_2}{H_3}{H_4}{}\big) \ = \ \cross{H_1}{H_2}{e_3}{e_4}{\psi}
 \]
 as desired.

 \medskip\noindent\textbf{Case 3A.}
 Assume that $H_i'=H_i-a$ is a hyperplane of $M'$ for $i=1,2,3$ and $a\notin\{e_1,e_2,e_3\}$. Then $\varphi_{H_i}(e_k)=\varphi'_{H_i'}(e_k)$ for all terms appearing in $\tripleratio{H_1}{H_2}{H_3}{e_1}{e_2}{e_3}{\varphi}$, and \eqref{eq:tripleratio} follows from the assumption that $\varphi'$ is a $P$-representation. 
 
 \medskip\noindent\textbf{Case 2B.}
 Assume that $H_1'=H_1-a$ and $H_2'=H_2-a$ are in $\cH'$ and that $a\in\{e_3,e_4\}$. Since $\cross{H_1}{H_2}{a}{e_4}{\varphi}=\cross{H_1}{H_2}{a}{e_4}{\psi}$ implies that
 \[
  \cross{H_1}{H_2}{e_4}{a}{\varphi} \ = \ \crossinv{H_1}{H_2}{a}{e_4}{\varphi} \ = \ \crossinv{H_1}{H_2}{a}{e_4}{\psi} \ = \ \cross{H_1}{H_2}{e_4}{a}{\psi},
 \]
 we can assume without loss of generality that $e_3=a$. Moreover, if we know that $\cross{H_1}{H_2}{a}{b}{\varphi}=\cross{H_1}{H_2}{a}{b}{\psi}$ for some $b\notin H_1\cup H_2\cup\{a\}$, then 
 \begin{multline}\label{eq: cross ratio equality in case 2B}
  \cross{H_1}{H_2}{a}{e_4}{\varphi} \ = \ \cross{H_1}{H_2}{a}{b}{\varphi} \ \cdot \ \cross{H_1}{H_2}{b}{e_4}{\varphi} \\ = \ \cross{H_1}{H_2}{a}{b}{\psi} \ \cdot \ \cross{H_1}{H_2}{b}{e_4}{\psi} \ = \ \cross{H_1}{H_2}{a}{e_4}{\psi},
 \end{multline}
 where the first equality follows from the definition of the cross-ratios, the second equality follows from our assumption on $\cross{H_1}{H_2}{a}{b}{\varphi}$ and Case 2A, and the third equality holds since the upper sublattice $\Lambda'=\{L,H_1,H_2,\gen{La},\gen{Lb},\gen{Le_4},E\}$ of $\Lambda_M$ is of type $U_{2,3}$, $U_{2,4}$, or $U_{2,5}$, depending on the cardinality of $\{\gen{La},\gen{Lb},\gen{Le_4}\}$, which are all contained in $\cL_M$.

 Since $a=e_3\notin H_1\cup H_2$, it follows that $H_1,H_2\in\cH'$ are hyperplanes of $M'$ off $\Gamma$. In other words, $H_1$ and $H_2$ are vertices of the graph $G=G_{M',\Gamma}$ (as defined in the context of \autoref{cor: graph version of the path theorem}) and $(H_1,H_2)$ is an (oriented) edge of $G$ (note that $L$ is indecomposable, since it is contained in at least $3$ hyperplanes $H_1$, $H_2$ and $\gen{La}$).
 
 The equality $\cross{H_1}{H_2}{a}{b}{\varphi}=\cross{H_1}{H_2}{a}{b}{\psi}$ follows from the definition of $\varphi_{H_1}(a)$ if $(H_1,H_2)$ is an edge of the spanning tree $T$ of $G$ for which $H_2$ is closer to the root $H_0$ of $T$ than $H_1$ and $b=b_{H_1,H_2}$ is the element that appears in the definition of $\varphi_{H_1}(a)$. If $(H_1,H_2)$ is an edge of $T$, but $H_1$ is closer to $H_0$ than $H_2$ is, then
 \[
  \cross{H_1}{H_2}{a}{b}{\varphi} \ = \ \crossinv{H_2}{H_1}{a}{b}{\varphi}  \ = \ \crossinv{H_2}{H_1}{a}{b}{\psi}  \ = \ \cross{H_1}{H_2}{a}{b}{\psi},
 \]
 where the first and third equality follow directly from the definitions and the middle equality follows from the definition of $\varphi_{H_2}(a)$. Together with Equation \eqref{eq: cross ratio equality in case 2B}, this establishes $\cross{H_1}{H_2}{e_3}{e_4}{\varphi}=\cross{H_1}{H_2}{e_3}{e_4}{\psi}$ whenever $(H_1,H_2)$ is an edge of $T$.
 
 Every edge $\ell=(H_1,H_2)$ of $G-T$ is part of a cycle $\gamma$ in $T\cup\ell$, since $T$ is a spanning tree. As shorthand, we write ``$\ell\in\gamma$'' for this. The extended version \autoref{thm:extended-homotopy-theorem} of Tutte's homotopy theorem states that $\gamma$ can be decomposed into elementary Tutte paths $\pi_1,\dotsc,\pi_t$ (of types 1--9) by concatenation. We choose a fixed $b_{\ell'}\in E'-(H\cup H')$ for every edge $\ell'=(H,H')$ that appears in one of the $\pi_j$ such that $b_{\ell'}=b_{H,H'}$ (as chosen in Case 1B) if $\ell'\in T$ and $b_{\ell}=e_4$.
 
 Let $(\Lambda'_j,\Gamma'_j)$ be the subconstellation of $(\Lambda_{M'},\Gamma)$ that pertains to $\pi_j$, and let $\Lambda_j$ be the lattice of flats of the single-element extension $\hat N$ determined by $(\Lambda'_j,\Gamma)$. Then $\Lambda_j$ is in $\cL_M$. In particular, the rescaling class corresponding to $\psi_{\Lambda_j}=\psi\circ\iota_{\Lambda_j}:F_{\Lambda_j}\to P$ is represented by a $P$-representation $\tilde\varphi_j$ of $\Lambda_j$, i.e., $\cross{H}{H'}{e}{e'}{\psi}=\cross{H}{H'}{e}{e'}{\tilde\varphi_j}$ for $H$ and $H'$ in $\Lambda_j$, by \autoref{cor: cross ratios as images of the universal cross ratio}. Since any two $P$-representations representing the rescaling class $\psi\circ\iota_{\Lambda_j'}$ of $\Lambda_j'$ are rescaling equivalent, we can rescale $\tilde\varphi_j$ so that its restriction to $\Lambda_j'$ agrees with the restriction of $\varphi'$ to $\Lambda_j'$, i.e., $\varphi'_H=\tilde\varphi_{j,H}$ for all hyperplanes $H$ in $\Lambda_j'$.
    
 The product
 \[
  \prod_{\ell'=(H,H')\in\gamma} \cross{H}{H'}{a}{b_{\ell'}}\psi \ \cdot \ \frac{\varphi'_{H}(b_{\ell'})}{\varphi'_{H'}(b_{\ell'})}
 \]
 is multiplicative in the concatenation of paths and allows for the removal of backtrackings of the form $(H,H',H)$ since
 \[
  \bigg(\cross{H}{H'}{a}{b}\psi \ \cdot \ \frac{\varphi'_{H}(b)}{\varphi'_{H'}(b)}\bigg) \ \cdot \ \bigg(\cross{H'}{H}{a}{b}\psi \ \cdot \ \frac{\varphi'_{H'}(b)}{\varphi'_{H}(b)}\bigg) \ = \ 1
 \]
 for $b=b_{H,H'}=b_{H',H}$. Thus the decomposition of $\gamma$ into the elementary Tutte paths $\pi_1,\dotsc,\pi_t$ yields a decomposition of the above product into smaller products indexed by the edges of the $\pi_j$, which leads to
 \begin{multline*}
  \prod_{\ell'=(H,H')\in\gamma} \cross{H}{H'}{a}{b_{\ell'}}\psi \ \cdot \ \frac{\varphi'_{H}(b_{\ell'})}{\varphi'_{H'}(b_{\ell'})}
  \ = \ \prod_{j=1}^t \bigg( \prod_{\ell'=(H,H')\in\pi_j} \cross{H}{H'}{a}{b_{\ell'}}{\psi} \ \cdot \ \frac{\varphi'_{H}(b_{\ell'})}{\varphi'_{H'}(b_{\ell'})} \bigg) \\
  = \ \prod_{j=1}^t \bigg( \prod_{\ell'=(H,H')\in\pi_j} \hspace{-35pt} \underbrace{\cross{H}{H'}{a}{b_{\ell'}}{\tilde\varphi_{j}}}_{\ \hspace{30pt}= \ \dfrac{\tilde\varphi_{j,H}(a)\cdot\tilde\varphi_{j,H'}(b_{\ell'})}{\tilde\varphi_{j,H}(b_{\ell'})\cdot\tilde\varphi_{j,H'}(a)}} \hspace{-40pt} \ \cdot \ \frac{\tilde\varphi_{j,H}(b_{\ell'})}{\tilde\varphi_{j,H'}(b_{\ell'})} \bigg)
  \ = \ \prod_{j=1}^t \bigg( \underbrace{\prod_{\ell'=(H,H')\in\pi_j} \frac{\tilde\varphi_{j,H}(a)}{\tilde\varphi_{j,H'}(a)}}_{= \ 1 \ \text{(since $\pi_j$ is closed)}} \bigg) \ = \ 1.
 \end{multline*}
 We conclude that
 \begin{align*}
  \cross{H_1}{H_2}{a}{e_4}{\varphi} \ 
                                      &= \bigg(\prod_{\substack{\ell'=(H,H')\in\gamma\\\ell'\neq(H_1,H_2)}} \ \cross{H}{H'}{a}{b_{\ell'}}{\varphi}\bigg)^{-1} \ \cdot \ {\bigg(\prod_{\ell'=(H,H')\in\gamma} \ \frac{\varphi_{H}(b_{\ell'})}{\varphi_{H'}(b_{\ell'})} \bigg)^{-1}} \\
                                      &= \bigg(\prod_{\substack{\ell'=(H,H')\in\gamma\\\ell'\neq(H_1,H_2)}} \ \cross{H}{H'}{a}{b_{\ell'}}{\psi}\bigg)^{-1} \ \cdot \ {\bigg(\prod_{\ell'=(H,H')\in\gamma} \ \frac{\varphi'_{H}(b_{\ell'})}{\varphi'_{H'}(b_{\ell'})} \bigg)^{-1}} \\
                                      &= \cross{H_1}{H_2}{a}{e_4}{\psi} \ \cdot \ \bigg( \prod_{\ell'=(H,H')\in\gamma} \cross{H}{H'}{a}{b_{\ell'}}\psi \ \cdot \ \frac{\varphi'_{H}(b_{\ell'})}{\varphi'_{H'}(b_{\ell'})} \bigg)^{-1} \\
                                      &= \cross{H_1}{H_2}{a}{e_4}{\psi}
 \end{align*}
 as desired.

\medskip\noindent\textbf{Case 3B.}
 Assume that $H_i'=H_i-a$ is in $\cH'$ for $i=1,2,3$ and that $a\in\{e_1,e_2,e_3\}$, say $e_3=a$. Then there is an element $b\in H_3-L$ with $b \ne a$ and
 \[
  \tripleratio{H_1}{H_2}{H_3}{e_1}{e_2}{a}{\varphi} \ = \ \tripleratio{H_1}{H_2}{H_3}{e_1}{e_2}{b}{\varphi} \ \cdot \  \cross{H_1}{H_2}{b}{a}{\varphi} \ = \ -1,
 \]
 since $\cross{H_1}{H_2}{b}{a}{\varphi}=\cross{H_1}{H_2}{b}{a}{\psi}=1$ is degenerate. (Here we use Case 2A and apply Case 3A to $\tripleratio{H_1}{H_2}{H_3}{e_1}{e_2}{b}{\varphi}=-1$.)

\medskip\noindent\textbf{Strategy for 2C and 2D.}
 Assume that $L_i=H_i-a$ is not a hyperplane of $M'$ for some $i$, say $i=1$. Then $L_1$ is a corank $2$ flat of both $M'$ and $M$. Let $L=H_1\cap H_2$, and choose $e_i\in H_i-L$ for $i=1,2$. Let $H_{1k}=\gen{L_1e_k}_M$ for $k=2,3,4$. Expanding definitions from Case 1C identifies the ratio $\varphi_{H_1}(e_3)/\varphi_{H_1}(e_4)$ with a product of $\varphi_{H_{13}}(a)/\varphi_{H_{14}}(a)$ and terms that do not contain $a$.
 
 Assuming that there exists an element $e'\in E'-(H_{13}\cup H_{14})$, we can eliminate $a$ using the identity
 \[
  \frac{\varphi_{H_{13}}(a)}{\varphi_{H_{14}}(a)} \ = \ \cross{H_{13}}{H_{14}}{a}{e'}{\varphi} \ \cdot \ \frac{\varphi_{H_{13}}(e')}{\varphi_{H_{14}}(e')},
 \]
 since $\cross{H_{13}}{H_{14}}{a}{e'}{\varphi}=\cross{H_{13}}{H_{14}}{a}{e'}{\psi}$ by Case 2B. Analogous reasoning for $H_2$, in case that $H_2-a\notin\cH'$, identifies $\cross{H_{1}}{H_{2}}{e_3}{e_4}{\varphi}$ with a product of cross-ratios from Case 2B and a product $\Pi$ of terms of the form $\varphi'_H(e)^\epsilon$ with $H\in\cH'$, $e\in E'$, and $\epsilon=\pm 1$. One checks easily that the degree of $\Pi$ in each $H\in\cH'$ and $e\in E'$ is $0$. As explained in \autoref{rem: rescaling invariant elements are in the image of the foundation}, this means that $\Pi$ is contained in the image of $F_{M'}\to P$. This allows us to apply \autoref{thm: the foundation is generated by cross ratios} to the smallest upper sublattice $\Lambda'$ that contains all of the relevant hyperplanes, i.e., $\Pi$ is an expression in values of the restriction of $\varphi$ to $\Lambda'$. Up to a sign, this identifies $\Pi$ with a product of cross-ratios for $\Lambda'$. These cross-ratios are identified with cross-ratios for $\psi$ by Case 3A, and then we need to show that each identification we have made only involves hyperplanes of an upper sublattice in $\cL_M$.
 
\medskip\noindent\textbf{Case 2C.}
 Assume that $L_1\subset H_2$. Then $L_1=L=H_1\cap H_2$ and $H_{1k}=H_k$ for $k=3,4$. Since $e'=e_2\in E'-(H_{13}\cup H_{14})$, we can apply the previously explained strategy. Expanding the definition of $\varphi_{H_1}$ involves only hyperplanes that cover $L$, which are $H_1$, $H_2$, $H_3=\gen{Le_3}_M$, $H_4=\gen{Le_4}_M$, and $H_{b_0}$ (as chosen in Case 1C), as well as $H'$ in the cases where $H_3=H_{b_0}$ or $H_4=H_{b_0}$. Thus the upper sublattice $\Lambda'$ generated by all these hyperplanes has at most $5$ atoms and is contained in $\cL_M$, which establishes $\Pi=\cross{H_1}{H_2}{e_3}{e_4}{\psi}$ in this case.
 
\medskip\noindent\textbf{Case 3C.}
 Assume that $L_1\subset H_k$ for all $k$. Then $L_1=L=H_1\cap H_2\cap H_3$ and the identity
 \[
  \tripleratio{H_1}{H_{2}}{H_{3}}{a}{e_2}{e_3}{\varphi} \ = \ -1
 \]
 follows from either Case 2A or 3A (depending on whether $b_0\in H_2\cup H_3$ or not).
 
 \medskip
 As a consequence, the restriction of $\varphi$ to the hyperplanes containing $L$ is a $P$-representation of $M/L$, which shows that the defining relations in 1C are satisfied for any choice of $b_0$ and $H'$. This allows us to assume particular choices for $b_0$ and $H'$ in the following.

\medskip\noindent\textbf{Case 3D.}
 Assume that $L_1\not\subset H_2$, i.e.,\ $L_1\neq L=H_1\cap H_2$. Let $F = L\cap L_1=L-a$, which is a corank $3$ flat of both $M'$ and $M$. Let $L_i=\gen{Fe_i}_M$ and $H_{ij}=\gen{Fe_ie_j}_M$ for distinct $i,j\in\{1,\dotsc,4\}$, which agrees with the previous definitions of $L_1$, $H_{13}$, and $H_{14}$. 

 If $H_{12}$, $H_{13}$, and $H_{23}$ are pairwise distinct, then
 \[
  \tripleratio{H_1}{H_{2}}{H_{3}}{e_1}{e_2}{e_3}{\varphi} \ = \ \tripleratio{H_1}{H_{12}}{H_{13}}{a}{e_2}{e_3}{\varphi} \ \cdot \ \tripleratio{H_{12}}{H_{2}}{H_{23}}{e_1}{a}{e_3}{\varphi} \ \cdot \ \tripleratio{H_{13}}{H_{23}}{H_{3}}{e_1}{e_2}{a}{\varphi} \ = \ -1,
 \]
 where the first equality follows from expanding definitions and the second follows from Cases 3B and 3C, depending on whether $H_i-a$ is in $\cH'$ or not.
 
 If two of $H_{12}$, $H_{13}$, and $H_{23}$ agree, then $H_{12}=H_{13}=H_{23}$ and there is an element $e_4\in E'-H_{12}$, since $L_i$ is contained in at least $2$ hyperplanes of $M'$. Let $L_4=\gen{Fe_4}_M$ and $H_{i4}=\gen{Fe_ie_4}_M$ for $i=1,2,3$. Then
 \[
  \tripleratio{H_1}{H_{2}}{H_{3}}{e_1}{e_2}{e_3}{\varphi} \ = \ \cross{H_1}{H_{14}}{e_2}{e_3}{\varphi} \cdot \cross{H_2}{H_{24}}{e_3}{e_1}{\varphi} \cdot \cross{H_3}{H_{34}}{e_1}{e_2}{\varphi} \cdot \tripleratio{H_{14}}{H_{24}}{H_{34}}{e_1}{e_2}{e_3}{\varphi} \ = \ -1,
 \]
 where the first equality follows from expanding definitions and the second equality follows since all involved cross-ratios are degenerate and thus equal to $1$ and since Case 3A (if $H_i\neq H_{i4}$ for all $i=1,2,3$) and Case 3C (if $H_i=H_{i4}$ for some $i\in\{1,2,3\}$) identify the last term with $-1$.

\medskip\noindent\textbf{Case 2D.}
 Let $F$, $L_i$, and $H_{ij}$ be as in Case 3D. If $e_3,e_4\notin H_{12}$, then $e_2\in E'-(H_{13}\cup H_{14})$, which allows us to apply the general strategy (explained before Case 2C) with $e'=e_2$. In this case we're done, since all involved hyperplanes are contained in an upper sublattice in $\cL_M$ (it has at most $5$ atoms and is not of type $U_{2,4}\oplus U_{1,1}$).

 If exactly one of $e_3$ and $e_4$ is contained in $H_{12}$, say $e_3\in H_{12}$, then $H_{12}=H_{13}=H_{23}$ is distinct from $H_{14}$ and $H_{24}$, and $a,e_3\notin H_{14}\cup H_{24}$. Thus all expressions in
 \begin{multline*}
  \cross{H_1}{H_2}{e_3}{e_4}{\varphi} \ = \ \tripleratio{H_1}{H_{13}}{H_{14}}{a}{e_3}{e_4}{\varphi} \ \cdot \ \tripleratio{H_2}{H_{24}}{H_{23}}{a}{e_4}{e_3}{\varphi} \ \cdot \ \cross{H_{14}}{H_{24}}{e_3}{a}{\varphi} \\ = \ (-1)^2 \ \cdot \ \cross{H_{14}}{H_{24}}{e_3}{a}{\psi} \ = \ \cross{H_{1}}{H_{2}}{e_3}{e_4}{\psi}
 \end{multline*}
 are defined, where the first equality follows from expanding the definitions, the second follows from Cases 2B, 3B, and 3C, and the third follows by reading the equation backwards for $\psi$ instead of $\varphi$, which we can do since the lattice $\Lambda'$ generated by $L,L_1,\dotsc,L_4$ over $F$ is in $\cL_M$ (it has at most $5$ atoms and is not of type $U_{2,4}\oplus U_{1,1}$).
  
 If $e_3,e_4\in H_{12}$, then $H_{34}=H_{12}\neq H_3$ and thus also $H_3\neq H_4$. Therefore all the quantities in the equation 
 \[
  \cross{H_1}{H_2}{e_3}{e_4}{\varphi} \ = \ \tripleratio{H_1}{H_3}{H_4}{e_1}{e_3}{e_4}{\varphi} \ \cdot \ \tripleratio{H_4}{H_3}{H_2}{e_4}{e_3}{e_2}{\varphi} \ \cdot \ \cross{H_3}{H_4}{e_1}{e_2}{\varphi} \ = \ \cross{H_3}{H_4}{e_1}{e_2}{\psi} 
 \]
 are defined, where the first equality follows at once from expanding the definitions and the second follows by Cases 3D and 2A. By \autoref{cor: cross ratios are invariant under row exchange}, the last term equals $\cross{H_1}{H_2}{e_3}{e_4}{\psi}$, as desired.
\end{proof}

\subsection{Presentation by generators and relations}
\label{subsection: presentation by generators and relations}

Recall that $\Theta_M$ denotes the set of all tuples of hyperplanes $(H_1, H_2, H_3, H_4)$ of $M$ such that $L = H_1 \cap H_2 \cap H_3 \cap H_4 \in \Lambda_M^{(2)}$ and such that $L = H_i \cap H_j$ for every $i \in \{1, 2\}$ and $j \in \{3,4\}$. Similarly, $\Theta_M^\diamondsuit$ denotes the set of all non-degenerate tuples for which $L = H_1 \cap H_2 = H_3 \cap H_4$ also holds. The following result characterizes all relations between universal cross-ratios, which generate $F_M$ by~\autoref{thm: the foundation is generated by cross ratios}, thus yielding a presentation of $F_M$ by generators and relations which is more natural and theoretically useful than the presentation given by \autoref{matroid-bijection-between-equi-classes-and-extensions}.

In the following description of the foundation $F_M$ in terms of generators and a nullset $R_M$, we write ``$a=b$'' to indicate that $R_M$ contains $a-b$ (which implies that $a$ and $b$ are indeed equal as elements of $F_M$).

\begin{thm}\label{thm:relations}
Let $M$ be a matroid with foundation $F_M$. Then 
\[\textstyle
 F_M \cong \pastgen{\F_1^\pm\big(\cross{H_1}{H_2}{H_3}{H_4}{} \mid (H_1, H_2, H_3, H_4) \in \Theta_M\big)}{R_M}, 
\]
where $R_M$ consists of the multiplicative relations
\begin{equation*}
\tag{R$-$}
\label{relation:R-}
    -1 = 1
\end{equation*}
if the Fano matroid $F_7$ or its dual $F_7^\ast$ is a minor of $M$, 
\begin{equation*}
\tag{R$\sigma$}
\label{relation:sigma}
\cross{H_1}{H_2}{H_3}{H_4}{} = 
\cross{H_2}{H_1}{H_4}{H_3}{} = 
\cross{H_3}{H_4}{H_1}{H_2}{} = 
\cross{H_4}{H_3}{H_2}{H_1}{}
\end{equation*}
for all $(H_1, H_2, H_3, H_4) \in \Theta_M^\diamondsuit$, 
\begin{equation*}
\tag{R0}
\label{relation:degenerate}
\cross{H_1}{H_2}{H_3}{H_4}{} = 1
\end{equation*}
for all degenerate $(H_1, H_2, H_3, H_4) \in \Theta_M$, 
\begin{equation*}
\tag{R1}
\label{relation:R1}
\cross{H_1}{H_2}{H_4}{H_3}{} = 
\crossinv{H_1}{H_2}{H_3}{H_4}{}
\end{equation*}
for all $(H_1, H_2, H_3, H_4) \in \Theta_M^\diamondsuit$,  
\begin{equation*}
\tag{R2}
\label{relation:R2}
\cross{H_1}{H_2}{H_3}{H_4}{} \cdot
\cross{H_1}{H_3}{H_4}{H_2}{} \cdot 
\cross{H_1}{H_4}{H_2}{H_3}{} = -1
\end{equation*}
for all $(H_1, H_2, H_3, H_4) \in \Theta_M^\diamondsuit$, 
\begin{equation*}
\tag{R3}
\label{relation:R3}
\cross{H_1}{H_2}{H_3}{H_4}{} \cdot
\cross{H_1}{H_2}{H_4}{H_5}{} \cdot
\cross{H_1}{H_2}{H_5}{H_3}{} = 1
\end{equation*}
for all $(H_1, H_2, H_3, H_4), (H_1, H_2, H_4, H_5), (H_1, H_2, H_5, H_3) \in \Theta_M^\diamondsuit$, 
\begin{equation*}
\tag{R4}
\label{relation:R4}
\cross{H_{13}}{H_{23}}{H_{34}}{H_{35}}{} \ \cdot \
\cross{H_{14}}{H_{24}}{H_{45}}{H_{34}}{} \ \cdot \
\cross{H_{15}}{H_{25}}{H_{35}}{H_{45}}{} \ = \ 1
\end{equation*}
for all $(H_{13}, H_{23}, H_{34}, H_{35}), (H_{14}, H_{24}, H_{45}, H_{34}), (H_{15}, H_{25}, H_{35}, H_{45}) \in \Theta_M$ such that the common intersection of all involved hyperplanes is a corank $3$ flat $L$ and such that $H_{i,j} = \gen{Lij}$ for all involved pairs of indices $ij$, where $1$, $2$, $3$, $4$ and $5$ are suitable atoms of $\Lambda$,
as well as the additive Pl\"ucker relations
\begin{equation*}
\tag{R$+$}
\label{relation:R+}
\cross{H_1}{H_2}{H_3}{H_4}{} + 
\cross{H_1}{H_3}{H_2}{H_4}{} = 1
\end{equation*}
for all $(H_1, H_2, H_3, H_4) \in \Theta_M^\diamondsuit$. 
\end{thm}

\begin{proof}
Write $F_M'$ for $\pastgen{\F_1^\pm\big(\cross{H_1}{H_2}{H_3}{H_4}{} \mid (H_1, H_2, H_3, H_4) \in \Theta_M\big)}{R_M}$. In~\autoref{thm: fundamental presentation of the foundation}, we proved that $F_M \cong \colim \cF_M$, where $\cF_M$ is the diagram of the foundations of the upper sublattices of $\Lambda_M$ of types $U_{2,4}, U_{2,5}, U_{3,5}, C_5, F_7$, and $F_7^\ast$, together with all the induced morphisms. 
Hence, we need only verify that $F_M' \cong F_M$ for $M = U_{2,4}, U_{2,5}, U_{3,5}, C_5, F_7$, and $F_7^\ast$. 

By~\autoref{cor: foundation of F7} and~\autoref{ex:foundation-F7*}, the foundation of both $F_7$ and $F_7^\ast$ is $\F_2 \cong \pastgen{\Funpm}{1 + 1}$. The presentation generated by all universal cross-ratios follows once we include all degenerate cross-ratios $\cross{H_1}{H_2}{H_3}{H_4}{} = 1$ as well.

For $M = U_{2,4}$, we computed in~\autoref{ex:foundation-U24} that $F_M \cong \pastgen{\Funpm(x,y)}{x+y-1}$, where $x = \cross{H_1}{H_2}{H_3}{H_4}{}$ and $y = \cross{H_1}{H_3}{H_2}{H_4}{}$. It follows easily from \autoref{thm:iff-modular-triple} that the relations~(\ref{relation:sigma}),~(\ref{relation:R1}) and~(\ref{relation:R2}) hold in $F_M$, and~\autoref{thm:iff-modular-triple} shows that the relation~(\ref{relation:R+}) holds as well. We claim that all other non-degenerate generators from the presentation $F_M'$ can be generated by $x$ and $y$ using relations~(\ref{relation:sigma}),~(\ref{relation:R1}) and~(\ref{relation:R2}). Note that up to the relations~(\ref{relation:sigma}) and~(\ref{relation:R1}), we only need to consider $\cross{H_1}{H_4}{H_2}{H_3}{}$. It is generated by $x$ and $y$ by~(\ref{relation:R2}): we have $\cross{H_1}{H_4}{H_2}{H_3}{} = -x^{-1}y$. Hence, $F_M' \cong F_M$. 

Consider $M = U_{2,5}$ with $M^\ast = U_{3,5}$. Since the relation~(\ref{relation:R4}) in $M^\ast$ is the image of~(\ref{relation:R3}) in $M$ under the canonical isomorphism $F_{M^\ast} \to F_M$, as defined in~\autoref{prop: foundation of dual matroid, explicit}, we shall only provide a proof for $M = U_{2,5}$ and skip the corresponding argument for $U_{3,5}$. A direct computation shows that the relation~(\ref{relation:R3}) holds in $F_M \cong \pastgen{\Funpm(x_1,\dotsc,x_5)}{x_i+x_{i-1}x_{i+1}-1\mid i=1,\dotsc,5}$, where $x_i = \cross{H_{i+1}}{H_{i+2}}{H_{i+4}}{H_{i+3}}{}$. Conversely, we wish to show that every non-degenerate universal cross-ratio from the presentation $F_M'$ can be generated by $x_1, \dotsc, x_5$. For each $i = 1, \dotsc, 5$, using relations~(\ref{relation:sigma}), (\ref{relation:R1}), and (\ref{relation:R2}), we only need to consider $\cross{H_{i+1}}{H_{i+2}}{H_{i+4}}{H_{i+3}}{}$ and $\cross{H_{i+1}}{H_{i+4}}{H_{i+2}}{H_{i+3}}{}$. The former is given by $\cross{H_{i+1}}{H_{i+2}}{H_{i+4}}{H_{i+3}}{} = x_i$. By~(\ref{relation:R3}), we have 
\[
\cross{H_{i+2}}{H_{i+3}}{H_{i+4}}{H_i}{} \cdot
\cross{H_{i+2}}{H_{i+3}}{H_{i}}{H_{i+1}}{} \cdot
\cross{H_{i+2}}{H_{i+3}}{H_{i+1}}{H_{i+4}}{} = 1. 
\]
Hence, we conclude 
\[
\cross{H_{i+1}}{H_{i+4}}{H_{i+2}}{H_{i+3}}{} = 
\cross{H_{i+2}}{H_{i+3}}{H_{i+1}}{H_{i+4}}{} = \crossinv{H_{i+2}}{H_{i+3}}{H_{i+4}}{H_i}{} \cdot \crossinv{H_{i+2}}{H_{i+3}}{H_{i}}{H_{i+1}}{} = 
x_{i-1}x_{i+1}, 
\]
which implies that $F_M' \cong F_M$ for $M = U_{2,5}$. 

To examine the last case $M = C_5$, we recall from~\autoref{ex:foundation-C5} that $F_M \cong \pastgen{\Funpm(x,y)}{x + y - 1}$, where $x = \cross{H_{15}}{H_{25}}{H_{35}}{H_{45}}{}$ and $y = \cross{H_{15}}{H_{35}}{H_{25}}{H_{45}}{}$. The only non-degenerate quadruples of hyperplanes in $C_5$, up to permutations, are $(H_{14}, H_{24}, H_{34}, H_{45})$ and $(H_{15}, H_{25}, H_{35}, H_{45})$. By~(\ref{relation:R4}), we have 
\[
\cross{H_{14}}{H_{24}}{H_{34}}{H_{45}}{} = x, \ 
\cross{H_{14}}{H_{34}}{H_{24}}{H_{45}}{} = y. 
\]
Thus, $F_M' \cong F_M$ for $M = C_5$. 
\end{proof}

\begin{rem}
 The concept of a foundation as a pasture that characterizes all matroid representations (up to rescaling) is generalized to polymatroids in \cite{BHKKL25}. This paper also contains a generalization of \autoref{thm:relations} to the foundation of a polymatroid: it is generated by cross-ratios, and a complete set of relations is given by the type of relations from \autoref{thm:relations} (suitably adapted to the formalism of polymatroids), together with a single additional relation, which is $-1=1$ for polymatroids that are not matroids.
\end{rem}


\section{Applications}
\label{section: applications}

\subsection{Excluded minors for regular and ternary matroids}

The fundamental presentation $F_M=\colim\cF_M$ of the foundation of a matroid $M$ lies at the heart of several important results in matroid theory. As first consequences, we recover the excluded minor characterizations of regular~\cite{Tutte58b}, binary~\cite{Tutte58b}, and ternary~\cite{Bixby79,Seymour79} matroids (the first of which was the original motivation for Tutte to develop his homotopy theorem).

\autoref{table: matroids of the fundamental representation} displays the types of minors that appear in the fundamental presentation $\cF_M$ and their respective foundations. Note that $U_{2,4}$ is a minor of $C_5$, of $U_{2,5}$, and of $U_{3,5}$.

\begin{table}[htb]
 \caption{Matroids of the fundamental presentation and their foundations}
 \label{table: matroids of the fundamental representation}
 \begin{tabular}{c||c|c|c|c|c|c}
  matroid    & $U_{2,4}$ & $C_5$     & $U_{2,5}$ & $U_{3,5}$ & $F_7$     & $F_7^\ast$ \\
  \hline
  foundation & $\U$      & $\U$      & $\V$      & $\V$      & $\F_2$    & $\F_2$    
 \end{tabular}
\end{table}

The following result was originally proved by Tutte \cite{Tutte58b} using his homotopy theorem, and later reproved by Gerards \cite{Gerards89} without the use of this tool.

\begin{thm}\label{thm: excluded minors for regular matroids}
 A matroid is regular if and only if it has no minors of type $U_{2,4}$, $F_7$, or $F_7^\ast$.
\end{thm}

\begin{proof}
 Since none of $U_{2,4}$, $F_7$, and $F_7^\ast$ is regular, but all of their proper minors are, each is an excluded minor for the class of regular matroids. Conversely, if a matroid $M$ has no minor of type $U_{2,4}$, $F_7$, or $F_7^\ast$, then its fundamental presentation $\cF_M$ is empty; thus $F_M=\colim\cF_M=\Funpm$. Since $\Funpm$ maps to every field, $M$ is representable over every field and therefore regular. This shows that $U_{2,4}$, $F_7$, and $F_7^\ast$ form a complete set of excluded minors for the class of regular matroids.
\end{proof}

Recall that the excluded minor characterization for binary matroids (\autoref{thm: excluded minors for binary matroids}) can be deduced from \autoref{thm:iff-modular-triple}, which is central for the presentation of foundations via the hyperplane incidence graph. One can also prove \autoref{thm: excluded minors for binary matroids} using the fundamental presentation. We leave the details to the interested reader.

The fundamental presentation $F_M=\colim\cF_M$ also allows us to determine the list of excluded minors for ternary matroids, but the argument is slightly more involved, since $\cF_M$ might have non-trivial ``monodromy'' in this case. We derive the set of excluded minors for ternary matroids from the following more general result.

We say that a matroid is \emph{without large uniform minors} if it does not have any minors of type $U_{2,5}$ or $U_{3,5}$. Examples are binary and ternary matroids.

\begin{thm}[Theorem 5.9 in \cite{Baker-Lorscheid20}]\label{thm: structure theorem for foundations of wlum matroids}
 Let $M$ be a matroid without large uniform minors. Then $F_M\simeq\bigotimes_{i=1}^s F_i$ for some $s\geq0$ and $F_1,\dotsc,F_s\in\{\F_2,\F_3,\H,\D,\U\}$.
\end{thm}

\begin{proof}[Proof sketch]
 The colimit of $\cF_M$ is the tensor product of the colimits of its connected components. Thus it suffices to show that the colimit of each connected component $C$ of the fundamental diagram is one of $\F_2$, $\F_3$, $\H$,\ $\D$ and $\U$. Since $M$ is without large uniform minors, $\cF_M$ consists of copies of $\U$ and $\F_2$ only. The only non-identity morphisms in $\cF_M$ are isomorphisms $\U\to\U$, which are induced by the minor embedding $U_{2,4}\to C_5$. In particular, $\F_2$ is isolated in $\cF_M$, and thus a component with $\F_2$ has colimit $\F_2$.
 
 Thus we are left with the connected components of $\cF_M$ that consist entirely of isomorphisms between copies of $\U$. We can choose a spanning tree of isomorphisms and contract it without changing the colimit, which replaces the connected component $C$ by a diagram consisting of a single object $\U$ together with a set of automorphisms. Thus, the colimit of $C$ is a quotient of $\U$ by a group of automorphisms.
 
 The theorem follows once we have proven that every quotient of $\U$ by a symmetry group (i.e., a group consisting of automorphisms of $\U$) is isomorphic to either $\F_3$, $\H$, $\D$, or $\U$. This follows from elementary considerations, which we outline in the following.
 
 The automorphism group of $\U$ is $\Aut(\U)\simeq S_3$, which can be seen by studying its (simply transitive) action on the six \emph{fundamental elements} $x,\ y,\ y^{-1},\ -xy^{-1},\ -x^{-1}y,\ x^{-1}$ of $\U$, which are those elements $z$ for which there is a $t$ with $z+t-1\in N_\U$. The quotient of $\U$ by a subgroup $H$ of $\Aut(\U)$ can be determined by identifying the generators $x$ and $y$ of $\U$ with their respective images under the action of $H$. This yields $\U/\Aut(\U)\simeq\F_3$, $\U/\gen{\rho}\simeq\H$ if $\ord\rho=3$, $\U/\gen{\sigma}\simeq\D$ if $\ord\sigma=2$, and $\U/\gen{e}\simeq\U$.
\end{proof}

Since the representation theory of a matroid is controlled by its foundation, \autoref{thm: structure theorem for foundations of wlum matroids} has far-reaching consequences for the class of matroids without large uniform minors. We present a few sample results from \cite{Baker-Lorscheid20,Baker-Lorscheid21} in the remainder of this section, starting with the classification of excluded minors for ternary matroids.

The following theorem was originally proved by Bixby and Reid \cite{Bixby79} using Tutte's homotopy theory, and later reproved by Seymour without the use of Tutte's theory \cite{Seymour79}. 

\begin{thm}\label{thm: excluded minors for ternary matroids}
 A matroid is ternary if and only if it has no minors of type $U_{2,5}$, $U_{3,5}$, $F_7$, or $F_7^\ast$.
\end{thm}

\begin{proof}
 Since none of $U_{2,5}$, $U_{3,5}$, $F_7$, and $F_7^\ast$ is ternary, but all of their proper minors are, each is an excluded minor for the class of ternary matroids.  Conversely, if a matroid $M$ is without minors of type $U_{2,5}$, $U_{3,5}$, $F_7$, or $F_7^\ast$, then it is, by definition, without large uniform minors and thus, by \autoref{thm: structure theorem for foundations of wlum matroids}, $F_M\simeq F_1\otimes\dotsc\otimes F_s$ for certain $F_1,\dotsc,F_s\in\{\F_3,\ \H,\ \D,\ \U\}$. (Note that $\F_2$ does not appear, since $M$ is without $F_7$ and $F_7^\ast$ minors.) Each of $\F_3$, $\H$, $\D$, and $\U$ maps to $\F_3$, and therefore so does the tensor product $F_M\simeq F_1\otimes\dotsc\otimes F_s$. Thus $M$ is ternary and the list of excluded minors is complete.
\end{proof}

\subsection{Realization spaces of ternary matroids over certain finite fields}

The following result was first proved as a special case of \cite[Theorem 5.8]{Baker-Lorscheid21}.

\begin{thm}\label{thm: bijection between realization spaces of ternary matroids}
 Let $M$ be ternary. Then there is a bijection $\upR_M(\F_8)\simeq\upR_M(\F_4)\times\upR_M(\F_5)$.
\end{thm}

\begin{proof}[Sketch of proof]
 By \autoref{thm: structure theorem for foundations of wlum matroids}, the foundation of $M$ is of the form $F_M\simeq F_1\otimes\dotsc\otimes F_s$ for certain $F_1,\dotsc,F_s\in\{\F_3,\ \H,\ \D,\ \U\}$ (note that $\F_2$ does not map to $\F_3$). By \autoref{thm: characterizing property of the foundation} and the universal property of the tensor product of pastures, we have
 \[
  \upR_M(\F_q) \ = \ \Hom(F_M,\F_q) \ = \ \prod_{i=1}^s \ \Hom(F_i,\F_q)
 \]
 for every prime power $q$. Since (by elementary considerations; we omit the details) there is a bijection
 \[
  \Hom(F,\F_8) \ \to \Hom(F,\F_4) \ \times \ \Hom(F,\F_5) 
 \]
 for every $F\in\{\F_3,\ \H,\ \D,\ \U\}$, the theorem follows.
\end{proof}

\begin{rem}
In \cite[Theorem 5.8]{Baker-Lorscheid21}, the following more general result is established. 
Let $p_1$ and $p_2$ be prime powers such that $q = (p_1 - 2)(p_2 - 2) + 2$ is a
prime power with $3 \nmid q$. Then for every ternary matroid $M$, there is a bijection
$\upR_M(\F_q)\simeq\upR_M(\F_{p_1})\times\upR_M(\F_{p_2})$.
\end{rem}

\subsection{Orientable matroids without large uniform minors}

The following result, originally proved in \cite[Theorems 6.9 and 6.15]{Baker-Lorscheid20}, furnishes a new proof and generalization of a theorem of Lee and Scobee \cite{Lee-Scobee99}, as well as a new proof and generalization of a special case of a theorem of Ardila--Rincon--Williams \cite{Ardila-Rincon-Williams17}. 

\begin{thm}\label{thm: lifting orientations of wlum matroids}
Let $M$ be without large uniform minors. If $M$ is orientable, then $M$ is representable over $\D$. If $M$ is positively orientable, then $M$ is representable over $\U$.
\end{thm}

\begin{proof}[Sketch of proof]
 By \autoref{thm: structure theorem for foundations of wlum matroids}, the foundation of $M$ is of the form $F_M\simeq F_1\otimes\dotsc\otimes F_s$ for certain $F_1,\dotsc,F_s\in\{\F_2,\ \F_3,\ \H,\ \D,\ \U\}$. None of $\F_2$, $\F_3$, or $\H$ maps to $\S$, so if $M$ is orientable then $F_1,\dotsc,F_s\in\{\D,\ \U\}$. Both $\D$ and $\U$ map to $\D$, and therefore so does the tensor product $F_M\simeq F_1\otimes\dotsc\otimes F_s$, which shows that $M$ is representable over $\D$. If $M$ is positively orientable, then $\D$ cannot occur as a factor of $F_M$ (see \cite[Lemma 6.14]{Baker-Lorscheid20}, details omitted). Thus $F_M\simeq \U^{\otimes s}$, which maps to $\U$.
\end{proof}

\subsection{Dressians of matroids without large uniform minors}

The {\em Dressian} $\Dr(M)$ of a matroid $M$ is the set of all valuated matroids (i.e., $\T$-matroids) with underlying matroid $M$. The Pl\"ucker coordinates of the corresponding Grassmann-Pl\"ucker functions realize $\Dr(M)$ as a subset of a tropical projective space and endow $\Dr(M)$ with a topology.

We define the {\em tropical line} $L$ in the tropical plane as the ``$Y$-shaped'' set of all solutions to $x+y+1 \in N_{\T}$, i.e., all points $(a,b)\in(\T^\times)^2$ for which the maximum of $a$, $b$, and $1$ occurs at least twice. The following theorem is a consequence of the more general result \cite[Cor.\ 12.5]{BHKKL25}.

\begin{thm}\label{thm: Dressians for wlum matroids}
 Let $M$ be without large uniform minors. Then the Dressian $\Dr(M)$ of $M$ is homeomorphic to $\R^n\times [0,\infty)^m\times L^p$ for some $n,m,p\geq0$, where $L$ is the tropical line.
\end{thm}

\begin{proof}[Sketch of proof]
 The Dressian $\Dr(M)$ maps to $\upR_M(\T)$ by sending valuated matroids to their rescaling classes. The kernel of this map is a real vector space $\R^n$ (called the \emph{lineality space} of $\Dr(M)$), and in fact, $\Dr(M)\simeq \R^n\times\upR_M(\T)$. By \autoref{thm: structure theorem for foundations of wlum matroids}, the foundation of $M$ is of the form $F_M\simeq F_1\otimes\dotsc\otimes F_s$ for certain $F_1,\dotsc,F_s\in\{\F_2,\ \F_3,\ \H,\ \D,\ \U\}$, and thus
 \[
  \upR_M(\T) \ = \ \Hom(F_M,\T) \ \simeq \ \prod_{i=1}^s \ \Hom(F_i,\T),
 \]
 where the $\Hom$-sets are topologized with the compact-open topology with respect to the discrete topologies for $F_M$ and the $F_i$ and the natural topology for $\T=\R_{\geq0}$. The factors of the product are homeomorphic to one of
 \[
  \Hom(\F_2,\ \T) \ = \ \Hom(\F_3,\ \T) \ = \ \Hom(\H,\ \T) \ = \ \{ \textrm{point} \},
 \]
 \[
  \Hom(\D,\ \T) \ \simeq \ [0,\infty), \qquad \textrm{or} \qquad \Hom(\U,\ \T) \ \simeq \ L,
 \]
 which verifies the claim of the theorem.
\end{proof}

The papers \cite{Baker-Lorscheid20} and \cite{Baker-Lorscheid21} contain further applications of the theory of foundations to the representation theory of matroids without large uniform minors. Chen and Zhang developed a Macaulay2 package to compute the foundation of a matroid; see~\cite{Chen-Zhang}. The appendix in \cite{Baker-Lorscheid-Zhang24} contains a comprehensive list of interesting foundations that the authors have found with help of this computer software.


\section{Towards higher homotopy theorems}
\label{section: towards higher homotopy theorems}

Let $M$ be a matroid, and let $\tau=(\Lambda,\Gamma)$ be a Tutte constellation of type $M$. We say $\tau$ is {\em indecomposable} if $M$ is a connected matroid; otherwise, it is {\em decomposable}. 

Let $\cX^\tau$ be the poset of all subconstellations of $\tau$, ordered by lattice inclusion. In the following, we define subposets $\cX^\tau_0\subset\cX^\tau_1\subset\cX^\tau_2$ of $\cX^\tau$ that allow for a topological reformulation of Tutte's path theorem and Tutte's homotopy theorem. This formulation has an obvious generalization to the vanishing of higher homotopy groups or to the vanishing of certain homology groups.

We will frequently make use of the \emph{order complex} of a poset $\cX$, which is a simplicial complex having $\cX$ as its vertices and all finite chains of $\cX$ as its faces.

\subsection{A topological interpretation of the path theorem}
\label{subsection: topological interpretation of the path theorem}

Let $E$ be the ground set of $M$, which is the top element of $\Lambda$. Consider the following two classes of subconstellations of $\tau$: 

\medskip\noindent\textbf{Class 0.} A \emph{subconstellation of class 0} is a subconstellation $\sigma$ of type $U_{1,1}$ with $\Gamma_\sigma=\{E\}$. 

\medskip\noindent\textbf{Class 1.} A \emph{subconstellation of class 1} is a subconstellation $\sigma$ of type $U_{2,2}$ with $\Gamma_\sigma=\{E\}$, and such that the bottom element of $\Lambda_\sigma$ is an indecomposable flat of $\Lambda$. 

\begin{df}
 The {\em zeroth subposet $\cX^\tau_0$ of $\cX^\tau$} consists of all subconstellations of $\tau$ of class 0. We denote the order complex of $\cX^\tau_0$ by $\Sigma^\tau_0$.
 
 The {\em first subposet $\cX^\tau_1$ of $\cX^\tau$} consists of all subconstellations of $\tau$ of class 0 or 1. We denote the order complex of $\cX^\tau_1$ by $\Sigma^\tau_1$.
\end{df}

The subconstellations of class 0 correspond to the hyperplanes of $\Lambda$ off $\Gamma$. Thus the order complex $\Sigma^\tau_0$ is the (discrete) set of hyperplanes of $\Lambda$ off $\Gamma$. The order complex $\Sigma^\tau_1$ is $1$-dimensional. We illustrate the part of $\Sigma^\tau_1$ that stems from a subconstellation $\sigma$ of class 1 in \autoref{fig: Sigma of U22} (where the dotted line indicates that the bottom of $\Lambda_\sigma$ is indecomposable in $\tau$, i.e.,\ is contained in a third hyperplane of $\tau$ that is not in $\Lambda_\sigma$). Note that after identifying the two class 0 subconstellations $\Lambda_\sigma/1$ and $\Lambda_\sigma/2$ in \autoref{fig: Sigma of U22} with the corresponding hyperplanes $H_1$ and $H_2$, the path between these two subconstellations corresponds exactly to two consecutive entries in a Tutte path $(\dots,H_1,H_{2},\dotsc)$ in $M$ off $\Gamma$.

\begin{figure}[t]
\begin{tikzpicture}[x=1.0cm,y=0.8cm, font=\footnotesize,decoration={markings,mark=at position 0.6 with {\arrow{latex}}}]
   \node at (0,0)
  {
   \begin{tikzpicture}
   \node at (-0.5,2) {$\Lambda_\sigma = \Lambda_{U_{2,2}}$};  
   \filldraw[fill=green!30!white,draw=green!80!black,rounded corners=2pt] (0.7,2.3) -- (0.7,1.7) -- (1.3,1.7) -- (1.3,2.3) -- cycle;
   \node (0) at (1,0) {$\emptyset$};  
   \node (1) at (0.5,1) {$1$};  
   \node (2) at (1.5,1) {$2$};  
   \node (3) at (2.5,1) {$\ $};  
   \node (E) at (1,2) {$E$};  
   \draw (0) to (1);
   \draw (0) to (2);
   \draw[dashed] (0) to (3);
   \draw (1) to (E);
   \draw (2) to (E);
   \end{tikzpicture} 
  };
  \node at (4.5,0)
  {
  \begin{tikzpicture}
   \node at (-1,1) {$\cX^\sigma_1$};  
   \node (0) at (1,-1) {$\Lambda_\sigma$};  
   \node (1) at (0.3,0.5) {$\Lambda_\sigma/1$};  
   \node (2) at (1.7,0.5) {$\Lambda_\sigma/2$};  
   \draw (0) to (1);
   \draw (0) to (2);
  \end{tikzpicture} 
 };
 \node at (10,0)
 {
 \begin{tikzpicture}
  \node at (3,2.5) {$\Sigma^\sigma_1$};  
  \draw (4,2) -- (5,1) -- (6,2);
  \filldraw (4,2) circle (2pt);  
  \node at (3.5,1.5) {$\Lambda_\sigma/1$};  
  \filldraw (5,1) circle (2pt);  
  \node at (5,0.5) {$\Lambda_\sigma$};  
  \filldraw (6,2) circle (2pt);  
  \node at (6.5,1.5) {$\Lambda_\sigma/2$};  
   \end{tikzpicture} 
  };
\end{tikzpicture}
 \caption{Subconstellation of class 1 and its associated poset and order complex}
 \label{fig: Sigma of U22}
\end{figure}
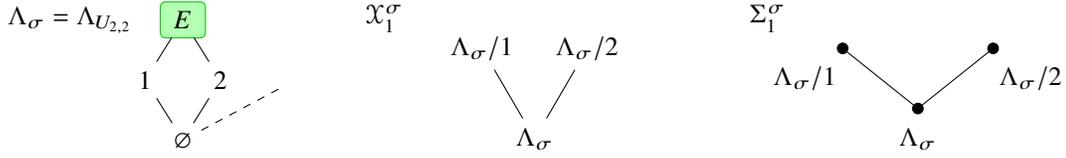 

We denote by $H_i(X,\Z)$ the singular homology of a topological space $X$ with integral coefficients.

\begin{thm}[Topological path theorem]\label{thm: topological path theorem}
 Let $\tau=(\Lambda,\Gamma)$ be an indecomposable Tutte constellation with $\Gamma\neq\Lambda$. Then $H_0(\Sigma^\tau_1,\Z)\simeq\Z$.
\end{thm}

\begin{proof}
 Note that $\Sigma^\tau_1$ is non-empty since $\Gamma\neq\Lambda$. Thus $H_0(\Sigma^\tau_1,\Z)=\Z$ if $\Sigma^\tau_1$ is connected. As a simplicial complex, $\Sigma^\tau_1$ is connected if it is path connected, which we prove in the following.

 Since every subconstellation $\sigma$ of class 1 contains two subconstellations $\eta_1$ and $\eta_2$ of class 0, every class 1 vertex $\sigma$ of $\Sigma^\tau_1$ is connected to exactly two class 0 vertices $\eta_1$ and $\eta_2$ of $\Sigma^\tau_1$ by line segments (see\ \autoref{fig: Sigma of U22}). Thus we only need to show that any two class 0 vertices $\eta$ and $\eta'$ of $\Sigma^\tau_1$ can be connected by a path.
 
 The class 0 subconstellations $\eta$ and $\eta'$ correspond to hyperplanes $H$ and $H'$, respectively, off $\Gamma$. By Tutte's path theorem (\autoref{thm: Tutte's path theorem}), there is a Tutte path $(H_1,\dotsc,H_n)$ off $\Gamma$ with $H_1=H$ and $H_n=H'$. Let $\eta_i$ be the class 0 subconstellation corresponding to $H_i$. By the definition of a Tutte path, $H_i$ and $H_{i+1}$ intersect in an indecomposable corank $2$ flat $F_i$, which corresponds to a class 1 subconstellation $\sigma_i$, which is path-connected to both $\eta_i$ and $\eta_{i+1}$ in $\Sigma^\tau_1$ by our previous observations. This shows that $\eta$ and $\eta'$ can be connected by a path in $\Sigma^\tau_1$, concluding the proof.
\end{proof}

\begin{rem}
By the Hurewicz theorem, $H_0(\Sigma^\tau_1,\Z)\simeq\Z$ is equivalent to $\pi_0(\Sigma^\tau_1)= \{ 0 \}$, i.e., to the statement that $\Sigma^\tau_1$ is path-connected. Since homotopy groups depend, \emph{a priori}, on the choice of a base point, but there is no such canonical choice for a Tutte constellation, it is more natural to consider homology. 
\end{rem}

The following more general version of \autoref{thm: topological path theorem} allows us to relax the hypothesis that $\Lambda$ is indecomposable. Note that every subconstellation $\sigma$ of type $U_{1,1}$ with $\Gamma_\sigma=\{E\}$ corresponds to a point of $\Sigma^\tau_1$ and thus defines a class $[\sigma]$ in $H_0(\Sigma_1^\tau,\Z)$. Note further that the lattice of flats of $M=M_1\oplus\dotsb\oplus M_r$ is the product $\Lambda=\Lambda_{M_1}\times\dotsb\times\Lambda_{M_r}$ of the respective lattices of flats of the $M_i$, and that $\Lambda_{M_i}$ can be considered naturally as an upper sublattice of $\Lambda$.

\begin{thm}\label{thm: topological path theorem for direct sums}
 Let $\tau$ be a Tutte constellation of type $M$ and let $M=M_1\oplus\dotsb\oplus M_r$ be the decomposition of $M$ into connected components. Assume that there exists some $s \leq r$ such that $\Lambda_{M_i}\not\subset\Gamma_\tau$ if and only if $1 \leq i \leq s$, i.e., suppose that for $i=1,\dotsc,s$, there is a subconstellation $\sigma_i$ of $\tau$ of class 0 with $\Lambda_{\sigma_i}\subset\Lambda_{M_i}$, and for $i = s+1, \dots, r$ we have $\Lambda_{M_i} \subset\Gamma_\tau$. Then $H_0(\Sigma^\tau_1,\Z)\simeq\Z^s$, and the classes $[\sigma_1],\dotsc,[\sigma_s]$ form a basis of $H_0(\Sigma^\tau_1,\Z)$.
\end{thm}

\begin{proof}
Let $\Gamma_i$ be the intersection of $\Gamma$ with $\Lambda_{M_i}$, which defines the subconstellation $\tau_i=(\Lambda_{M_i},\Gamma_i)$ of $\tau$. If $\Gamma_i=\Lambda_{M_i}$, then $\Sigma^{\tau_i}_1$ is empty and thus $H_0(\Sigma^{\tau_i}_1,\Z)=0$. If $\Gamma_i\subsetneq\Lambda_{M_i}$, then $H_0(\Sigma^{\tau_i}_1,\Z)=\Z$ by \autoref{thm: topological path theorem}, and the class of any vertex generates $H_0(\Sigma^{\tau_i}_1,\Z)$.
 
 The set of hyperplanes of $M$ is a disjoint union of the individual sets of hyperplanes of the direct summands $M_i$ (where a hyperplane $H$ of $M_i$ corresponds to the hyperplane $H_i\cup\bigcup_{j\neq i} E(M_j)$), and any two hyperplanes $H$ and $H'$ that belong to different components $M_i$ and $M_j$ intersect in a decomposable corank $2$ flat. This shows that $\Sigma^\tau_1$ is the disjoint union of the order complexes $\Sigma^{\tau_i}_1$, and thus
 \[
  H_0(\Sigma^{\tau}_1,\Z) \ \ = \ \ \bigoplus_{i=1}^r \ H_0(\Sigma^{\tau_i}_1,\Z).
 \]
 as desired.
\end{proof}

\subsection{A topological interpretation of the homotopy theorem}
\label{subsection: topological interpretation of the homotopy theorem}

Consider the following classes of subconstellations of $\tau$:

\medskip\noindent\textbf{Class 2a.} A \emph{subconstellation of class 2a} is a subconstellation $\sigma$ of type $U_{2,3}$ with $\Gamma_\sigma=\{E\}$ (see~\autoref{fig: kind2}). 

\medskip\noindent\textbf{Class 2b.} A \emph{subconstellation of class 2b} is a subconstellation $\sigma$ of type $U_{3,3}$ with $\Gamma_\sigma=\{E\}$ such that each corank $2$ flat of $\sigma$ is indecomposable in $\tau$ (see~\autoref{fig: kind2}). 

\medskip\noindent\textbf{Class 2c.} A \emph{subconstellation of class 2c} is a subconstellation $\sigma$ of type $U_{3,4}$ with $\Gamma_\sigma=\{H_1,\ H_2,\ E\}$, where $H_1$ and $H_2$ are two hyperplanes 
whose meet in $\Lambda_\sigma$ is its bottom element (see~\autoref{fig: kind3}).

\medskip\noindent\textbf{Class 2d.} A \emph{subconstellation of class 2d} is a subconstellation $\sigma$ of type $M(K_{2,3})$ with $\Gamma=\{H_1,\dotsc,H_4,\ E\}$ such that the pairwise intersections $H_i\cap H_j$ (for $i\neq j$) correspond to six corank $3$ flats of $\Lambda$, and such that its three decomposable corank $2$ flats remain decomposable in $\tau$ (see\ \autoref{fig: Tutte constellation of kind4}).

\begin{df}
 The {\em second subposet $\cX^\tau_2$} is the union of $\cX^\tau_1$ with all subconstellations of $\tau$ of classes 2a--2d. We denote the order complex of $\cX^\tau_2$ by $\Sigma^\tau_2$.
\end{df}

Tutte's homotopy theorem (\autoref{thm: Tutte's homotopy theorem}) asserts that we do not need to add any further cells to make the first homology of $\Sigma^\tau_2$ vanish, as made precise in the following result.

\begin{thm}[Topological homotopy theorem]\label{thm: topological homotopy theorem absolute}
 Let $\tau$ be a Tutte constellation. Then $H_1(\Sigma^\tau_2,\Z)=0$.
\end{thm}

\begin{proof}
 Let $M$ be the type of $\tau$. As a first step, we reduce the problem to the case where $M$ is connected. Let $M=M_1\oplus\dotsb\oplus M_r$ be the decomposition of $M$ into its connected components and let $\Lambda=\Lambda_{M_1}\times\dotsb\times\Lambda_{M_r}$ be the corresponding decomposition of $\Lambda$ into upper sublattices of $\Lambda$. Let $\Gamma_i=\Gamma\cap\Lambda_{M_i}$ and $\tau_i=(\Lambda_{M_i},\Gamma_i)$. Each subconstellation $\sigma$ of $\tau$ of class 2a--2d is contained in some $\tau_i$. (If $\sigma$ is of class 2a, 2c or 2d, this follows from the fact that $\sigma$ is indecomposable, and if $\sigma$ is of class 2b, this follows from the fact that all corank $2$ flats of $\sigma$ are indecomposable.) This shows that $\Sigma^\tau_2$ is the disjoint union of the $\Sigma^{\tau_i}_2$, which reduces the proof to the case that $M$ is connected.

 The first homology group of $\Sigma^\tau_2$ is generated by the classes of primitive closed $1$-chains (i.e., they cannot be written as the class of more than one closed path in $\Sigma^\tau_2$), and we aim to show that each such class is trivial, i.e., each closed $1$-chain is the boundary of a $2$-chain. A primitive closed $1$-chain is a sequence of oriented $1$-simplices, which can be represented as the sequence $(\sigma_1,\dotsc,\sigma_\ell,\sigma_1)$ of consecutive end vertices $\sigma_i$ of the $1$-simplices in the $1$-chain.
 
 As a first reduction step, we insert subconstellations of class 0 at every second position. Observe that for each $i=1,\dotsc,\ell$, we have either $\sigma_i\subset\sigma_{i+1}$ or $\sigma_{i+1}\subset\sigma_{i}$ (where $\sigma_{\ell+1}=\sigma_1$). If none of $\sigma_i$ and $\sigma_{i+1}$ is of class 0, then we can choose a subconstellation $\sigma_{i+\epsilon}\subset\sigma_i\cap\sigma_{i+1}$ of class 0 and add the boundary $(\sigma_i,\sigma_{i+\epsilon},\sigma_{i+1},\sigma)$ of the $2$-simplex $(\sigma_{i+\epsilon}\subset\sigma_i\subset\sigma_{i+1})$ (resp.\ $(\sigma_{i+\epsilon}\subset\sigma_{i+1}\subset\sigma_{i})$, depending on the containment relation between $\sigma_i$ and $\sigma_{i+1}$) to the $1$-chain, which replaces the $1$-simplex between $\sigma_i$ and $\sigma_{i+1}$ by the sequence of $1$-simplices $(\sigma_i,\sigma_{i+\epsilon},\sigma_{i+1})$. Therefore, we can assume without loss of generality that $\ell$ is even and that $\sigma_i$ is of class 0 for $i$ odd.
 
 As a second reduction step, we replace subconstellations of classes 2a--2d by sequences of subconstellations of classes 0 and 1. Consider, for odd $i$, the sequence $(\sigma_i,\sigma_{i+1},\sigma_{i+2})$, where $\sigma_i$ and $\sigma_{i+2}$ are of class 0 by assumption. If $\sigma_{i+1}$ is of class 2a--2d, then we find a $2$-chain of the form
 \[
  (\sigma_i\subset\sigma_{i+\epsilon}\subset\sigma_{i+1}) \ + \ (\sigma_{i+2\epsilon}\subset\sigma_{i+\epsilon}\subset\sigma_{i+1}) \ + \ \dotsb \ + \ (\sigma_{i+2}\subset\sigma_{i+s\epsilon}\subset\sigma_{i+1}),
 \]
 where the $\sigma_{i+k\epsilon}$ are subconstellations of $\sigma_{i+1}$ of classes 0 or 1 (depending on the parity of $k$). Adding its boundary $(\sigma_{i},\sigma_{i+\epsilon},\dotsc,\sigma_{i+s\epsilon},\sigma_{i+2},\sigma_i)$ to our chain replaces the sequence $(\sigma_i,\sigma_{i+1},\sigma_{i+2})$ by the sequence $(\sigma_{i},\sigma_{i+\epsilon},\dotsc,\sigma_{i+s\epsilon},\sigma_{i+2})$. Thus, we can assume without loss of generality that $\sigma_i$ is of class 1 for $i$ even.
 
 Let $H_1,H_3,\dotsc,H_{\ell-1}$ be the hyperplanes corresponding to the class 0 subconstellations $\sigma_1,\sigma_3,\dotsc,\sigma_{\ell-1}$. Since for even $i$ the corank $2$ flat of $\sigma_i$ is connected and contained in $H_{i-1}$ and $H_{i+1}$, the sequence $(H_1,H_3,\dotsc,H_{\ell-1},H_1)$ is a closed Tutte path. By Tutte's homotopy theorem (\autoref{thm: Tutte's homotopy theorem}), this closed path is null-homotopic, i.e., it can be deformed into a combination of elementary Tutte paths $(H_{k,1},H_{k,3},\dotsc,H_{k,\ell_k},H_{k,1})$, which themselves correspond to $1$-chains $(\sigma_{k,1},\sigma_{k,2},\dotsc,\sigma_{k,\ell_k},\sigma_{k,1})$ in $\Sigma^\tau_2$ where $\sigma_{k,i}$ is of class 0 for odd $i$ and of class $1$ for even $i$. This means that 
 \[
  \big[(\sigma_{1},\dotsc,\sigma_{\ell},\sigma_{1})\big] \ \ = \ \ \sum_k \ \big[(\sigma_{k,1},\sigma_{k,2},\dotsc,\sigma_{k,\ell_k},\sigma_{k,1})\big]
 \]
 as classes in $H_1(\Sigma^\tau_2,\Z)$. Each of the $1$-chains $(\sigma_{k,1},\sigma_{k,2},\dotsc,\sigma_{k,\ell_k},\sigma_{k,1})$ is contained in a contractible subcomplex $\Sigma^{\sigma_{k,0}}_2$ of $\Sigma^\tau_2$, where $\sigma_{k,0}$ is a subconstellation of class 2a--2d, which shows that the class of $(\sigma_{k,1},\sigma_{k,2},\dotsc,\sigma_{k,\ell_k},\sigma_{k,1})$ is trivial in $H_1(\Sigma^\tau_2,\Z)$. This shows that the class of $(\sigma_{1},\dotsc,\sigma_{\ell},\sigma_{1})$ is trivial and thus $H_1(\Sigma^\tau_2,\Z)=0$, as claimed.
\end{proof}

\begin{rem}\label{rem: subconstellations for topological homotopy theorem}
If we remove any of the classes 2a--2d in the definition of $\Sigma^\tau_2$, then \autoref{thm: topological homotopy theorem absolute} no longer holds. This can be seen as follows. For notational purposes, define $\cX^\tau_{2,0} := \cX^\tau_1$. 
 
 We begin with the order complex $\Sigma^\sigma_{2,0}$ of a subconstellation $\sigma$ of class 2a, which is illustrated in \autoref{fig: class 2a} and homeomorphic to a $1$-sphere. Thus $H_1(\Sigma^\sigma_{2,0},\Z)=\Z$, which means that we need to include subconstellations of class 2a in order to make the topological homotopy theorem true. We define $\cX^\tau_{2,1}$ as the union of $\cX^\tau_{2,0}$ with all subconstellations of class 2a.

 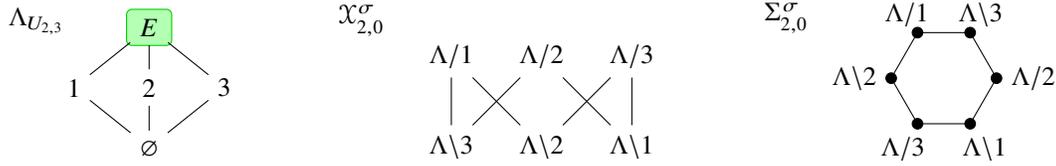
\begin{figure}[t]
\begin{tikzpicture}[x=1.0cm,y=0.8cm, font=\footnotesize,decoration={markings,mark=at position 0.6 with {\arrow{latex}}}]
   \node at (-0.5,0)
  {
   \begin{tikzpicture}
   \node at (-0.5,2.1) {$\Lambda_{U_{2,3}}$};  
   \filldraw[fill=green!30!white,draw=green!80!black,rounded corners=2pt] (0.7,2.3) -- (0.7,1.7) -- (1.3,1.7) -- (1.3,2.3) -- cycle;
   \node (0) at (1,0) {$\emptyset$};  
   \node (1) at (0,1) {$1$};  
   \node (2) at (1,1) {$2$};  
   \node (3) at (2,1) {$3$};  
   \node (E) at (1,2) {$E$};  
   \draw (0) to (1);
   \draw (0) to (2);
   \draw (0) to (3);
   \draw (1) to (E);
   \draw (2) to (E);
   \draw (3) to (E);
   \end{tikzpicture} 
  };
  \node at (4.5,0)
  {
  \begin{tikzpicture}[x=1.2cm,y=0.8cm]
   \node at (-1,2.1) {$\cX^\sigma_{2,0}$};  
   \node (12) at (0,0) {$\Lambda\minus3$};  
   \node (13) at (1,0) {$\Lambda\minus2$};  
   \node (23) at (2,0) {$\Lambda\minus1$};  
   \node (1) at (0,1.5) {$\Lambda/1$};  
   \node (2) at (1,1.5) {$\Lambda/2$};  
   \node (3) at (2,1.5) {$\Lambda/3$};  
   \draw (12) to (1);
   \draw (12) to (2);
   \draw (13) to (1);
   \draw (13) to (3);
   \draw (23) to (2);
   \draw (23) to (3);
  \end{tikzpicture} 
 };
 \node at (10,0)
 {
 \begin{tikzpicture}[x=0.7cm,y=0.7cm]
  \node at (-3,1.15) {$\Sigma^\sigma_{2,0}$};  
  \foreach \a in {1,...,6}
  {
   \filldraw (60*\a:1) circle (2pt);  
   \draw (60*\a:1) -- (60+60*\a:1);
  };
  \node at (0:1.7) {$\Lambda/2$};  
  \node at (60:1.4) {$\Lambda\minus3$};  
  \node at (120:1.4) {$\Lambda/1$};  
  \node at (180:1.7) {$\Lambda\minus2$};  
  \node at (240:1.55) {$\Lambda/3$};  
  \node at (300:1.55) {$\Lambda\minus1$};  
   \end{tikzpicture} 
  };
\end{tikzpicture}
 \caption{$\Sigma^\sigma_{2,0}$ for a subconstellation $\sigma$ of class 2a}
 \label{fig: class 2a}
\end{figure}

Next we consider $\cX^\sigma_{2,1}$ for $\sigma$ of class 2b, whose order complex $\Sigma^\sigma_{2,1}$ is also homeomorphic to a $1$-sphere, as illustrated in \autoref{fig: class 2b}. Thus $H_1(\Sigma^\sigma_{2,1},\Z)=\Z$, which means that we need to include subconstellations of class 2b in order for the topological homotopy theorem to be true. We define $\cX^\tau_{2,2}$ as the union of $\cX^\tau_{2,1}$ with all subconstellations of class 2b.

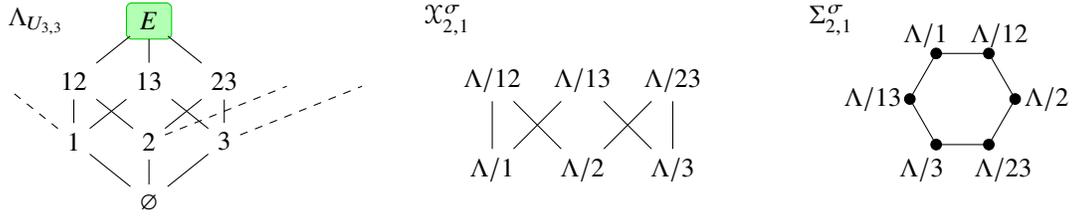
\begin{figure}[t]
\begin{tikzpicture}[x=1.0cm,y=0.8cm, font=\footnotesize,decoration={markings,mark=at position 0.6 with {\arrow{latex}}}]
   \node at (-0.5,0)
  {
   \begin{tikzpicture}
   \node at (-0.5,3) {$\Lambda_{U_{3,3}}$};  
   \filldraw[fill=green!30!white,draw=green!80!black,rounded corners=2pt] (0.7,3.3) -- (0.7,2.7) -- (1.3,2.7) -- (1.3,3.3) -- cycle;
   \node (0) at (1,0) {$\emptyset$};  
   \node (1) at (0,1) {$1$};  
   \node (2) at (1,1) {$2$};  
   \node (3) at (2,1) {$3$};  
   \node (12) at (0,2) {$12$};  
   \node (13) at (1,2) {$13$};  
   \node (23) at (2,2) {$23$};  
   \node (14) at (-1,2) {$\ $};  
   \node (24) at (3,2) {$\ $};  
   \node (34) at (4,2) {$\ $};  
   \node (E) at (1,3) {$E$};  
   \draw (0) to (1);
   \draw (0) to (2);
   \draw (0) to (3);
   \draw (1) to (12);
   \draw (1) to (13);
   \draw (2) to (12);
   \draw (2) to (23);
   \draw (3) to (13);
   \draw (3) to (23);
   \draw[dashed] (1) to (14);
   \draw[dashed] (2) to (24);
   \draw[dashed] (3) to (34);
   \draw (12) to (E);
   \draw (13) to (E);
   \draw (23) to (E);
   \end{tikzpicture} 
  };
  \node at (4.5,0.25)
  {
  \begin{tikzpicture}[x=1.2cm,y=0.8cm]
   \node at (-0.5,3) {$\cX^\sigma_{2,1}$};  
   \node (1) at (0,0.5) {$\Lambda/1$};  
   \node (2) at (1,0.5) {$\Lambda/2$};  
   \node (3) at (2,0.5) {$\Lambda/3$};  
   \node (12) at (0,2) {$\Lambda/{12}$};  
   \node (13) at (1,2) {$\Lambda/{13}$};  
   \node (23) at (2,2) {$\Lambda/{23}$};  
   \draw (12) to (1);
   \draw (12) to (2);
   \draw (13) to (1);
   \draw (13) to (3);
   \draw (23) to (2);
   \draw (23) to (3);
  \end{tikzpicture} 
 };
 \node at (9.5,0.25)
 {
 \begin{tikzpicture}[x=0.7cm,y=0.7cm]
  \node at (-2.5,1.5) {$\Sigma^\sigma_{2,1}$};  
  \foreach \a in {1,...,6}
  {
   \filldraw (60*\a:1) circle (2pt);  
   \draw (60*\a:1) -- (60+60*\a:1);
  };
  \node at (0:1.6) {$\Lambda/2$};  
  \node at (60:1.4) {$\Lambda/{12}$};  
  \node at (120:1.4) {$\Lambda/1$};  
  \node at (180:1.7) {$\Lambda/{13}$};  
  \node at (240:1.55) {$\Lambda/3$};  
  \node at (300:1.55) {$\Lambda/{23}$};   
   \end{tikzpicture} 
  };
\end{tikzpicture}
 \caption{$\Sigma^\sigma_{2,1}$ for a subconstellation $\sigma$ of class 2b}
 \label{fig: class 2b}
\end{figure}

Next we consider $\cX^\sigma_{2,2}$ for $\sigma$ of class 2c, whose order complex $\Sigma^\sigma_{2,2}$ is also homeomorphic to a $1$-sphere, as illustrated in \autoref{fig: class 2c}. Thus $H_1(\Sigma^\sigma_{2,2},\Z)=\Z$, which means that we need to include subconstellations of class 2c in order for the topological homotopy theorem to be true. We define $\cX^\tau_{2,3}$ as the union of $\cX^\tau_{2,2}$ with all subconstellations of class 2c.

\begin{figure}[t]
\begin{tikzpicture}[x=1.0cm,y=0.8cm, font=\footnotesize,decoration={markings,mark=at position 0.6 with {\arrow{latex}}}]
   \node at (0,0)
  {
   \begin{tikzpicture}[x=0.8cm,y=0.8cm]
   \node at (0,3) {$\Lambda_{U_{3,4}}$};  
   \filldraw[fill=green!20!white,draw=green!80!black,rounded corners=2pt] (2.15,3.35) -- (1.75,2.3) -- (1.75,1.7) -- (3.25,1.7) -- (3.25,2.3) -- (2.85,3.35) -- cycle;
   \node (0) at (2.5,0) {$\emptyset$};  
   \node (1) at (1,1) {$1$};  
   \node (2) at (2,1) {$2$};  
   \node (3) at (3,1) {$3$};  
   \node (4) at (4,1) {$4$};  
   \node (12) at (0,2) {$12$};  
   \node (13) at (1,2) {$13$};  
   \node (23) at (2,2) {$23$};  
   \node (14) at (3,2) {$14$};  
   \node (24) at (4,2) {$24$};  
   \node (34) at (5,2) {$34$};  
   \node (1234) at (2.5,3) {$1234$};  
   \draw (0) to (1);
   \draw (0) to (2);
   \draw (0) to (3);
   \draw (0) to (4);
   \draw (1) to (12);
   \draw (1) to (13);
   \draw (1) to (14);
   \draw (2) to (12);
   \draw (2) to (23);
   \draw (2) to (24);
   \draw (3) to (13);
   \draw (3) to (23);
   \draw (3) to (34);
   \draw (4) to (14);
   \draw (4) to (24);
   \draw (4) to (34);
   \draw (12) to (1234);
   \draw (13) to (1234);
   \draw (14) to (1234);
   \draw (23) to (1234);
   \draw (24) to (1234);
   \draw (34) to (1234);
   \end{tikzpicture} 
  };
  \node at (5.4,0.25)
  {
  \begin{tikzpicture}[x=1.2cm,y=0.8cm]
   \node at (-0.3,3) {$\cX^\sigma_{2,2}$};  
   \node (12) at (0,2) {$\Lambda/{12}$};  
   \node (13) at (1,2) {$\Lambda/{13}$};  
   \node (24) at (2,2) {$\Lambda/{24}$};  
   \node (34) at (3,2) {$\Lambda/{34}$};  
   \node (1) at (0,0.5) {$\Lambda\minor 41$};  
   \node (2) at (1,0.5) {$\Lambda\minor32$};  
   \node (3) at (2,0.5) {$\Lambda\minor23$};  
   \node (4) at (3,0.5) {$\Lambda\minor14$};  
   \draw (1) to (12);
   \draw (1) to (13);
   \draw (2) to (12);
   \draw (2) to (24);
   \draw (3) to (13);
   \draw (3) to (34);
   \draw (4) to (24);
   \draw (4) to (34);
  \end{tikzpicture} 
 };
 \node at (10.5,0.1)
 {
 \begin{tikzpicture}[x=0.8cm,y=0.8cm]
  \node at (-2.2,1.5) {$\Sigma^\sigma_{2,2}$};  
  \foreach \a in {1,...,8}
  {
   \filldraw (22.5+45*\a:1) circle (2pt);  
   \draw (-22.5+45*\a:1) -- (22.5+45*\a:1);
  };
  \node at (1.7,0.37) {$\Lambda\minor14$};  
  \node at (0.8,1.3) {$\Lambda/{24}$};  
  \node at (-0.8,1.3) {$\Lambda\minor32$};  
  \node at (-1.6,0.37) {$\Lambda/{12}$};  
  \node at (-1.7,-0.4) {$\Lambda\minor41$};  
  \node at (-0.8,-1.32) {$\Lambda/{13}$};  
  \node at (0.8,-1.32) {$\Lambda\minor23$};  
  \node at (1.6,-0.4) {$\Lambda/{34}$};   
   \end{tikzpicture} 
  };
\end{tikzpicture}
 \caption{$\Sigma^\sigma_{2,2}$ for a subconstellation $\sigma$ of class 2c}
 \label{fig: class 2c}
\end{figure}
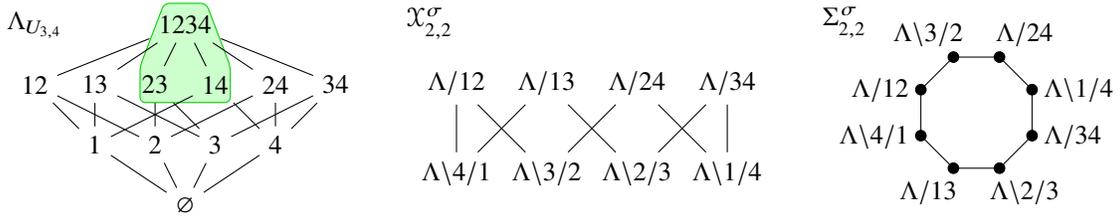

Finally, we consider $\cX^\sigma_{2,3}$ for $\sigma$ of class 2d, as illustrated in \autoref{fig: Tutte constellation of kind4} and \autoref{fig: kind4}. We use the same representation of $\sigma$ in this discussion. The poset $\cX^\sigma_{2,3}$ consists of subconstellations of classes $0$, $1$, and $2c$. In fact, every subconstellation of class $0$ and $1$ is contained in one of the six class $2c$ constellations with respective sublattices
\[
 \Lambda\minor41, \qquad \Lambda\minor52, \qquad \Lambda\minor63, \qquad \Lambda\minor14, \qquad \Lambda\minor25, \qquad \Lambda\minor36;
\]
see\ \autoref{fig: subconstellation Lambda41 in class 2d} for an illustration of the order complex of $\Lambda\minor41$, which is homeomorphic to a closed disc (note that we omit ``$\Lambda\minus$'' for better readability in the illustration of $\Sigma^{\sigma\minor41}_{2,3}$). 

\begin{figure}[t]
\begin{tikzpicture}[x=1.0cm,y=0.8cm, font=\footnotesize,decoration={markings,mark=at position 0.6 with {\arrow{latex}}}]
   \node at (0,0)
  {
   \begin{tikzpicture}[x=0.8cm,y=0.8cm]
   \node at (0.15,3) {$\Lambda\minor41$};  
   \filldraw[fill=green!20!white,draw=green!80!black,rounded corners=2pt] (2.25,3.35) -- (1.55,2.3) -- (1.55,1.7) -- (3.45,1.7) -- (3.45,2.3) -- (2.75,3.35) -- cycle;
   \node (1) at (2.5,0) {$1$};  
   \node (12) at (1,1) {$12$};  
   \node (13) at (2,1) {$13$};  
   \node (15) at (3,1) {$15$};  
   \node (16) at (4,1) {$16$};  
   \node (1245) at (0,2) {$1245$};  
   \node (126) at (1,2) {$126$};  
   \node (123) at (2,2) {$123$};  
   \node (156) at (3,2) {$156$};  
   \node (135) at (4,2) {$135$};  
   \node (1346) at (5,2) {$1346$};  
   \node (E) at (2.5,3) {$E$};  
   \draw (1) to (12);
   \draw (1) to (13);
   \draw (1) to (15);
   \draw (1) to (16);
   \draw (12) to (1245);
   \draw (12) to (126);
   \draw (12) to (123);
   \draw (13) to (123);
   \draw (13) to (135);
   \draw (13) to (1346);
   \draw (15) to (1245);
   \draw (15) to (156);
   \draw (15) to (135);
   \draw (16) to (126);
   \draw (16) to (156);
   \draw (16) to (1346);
   \draw (1245) to (E);
   \draw (126) to (E);
   \draw (123) to (E);
   \draw (156) to (E);
   \draw (135) to (E);
   \draw (1346) to (E);
   \end{tikzpicture} 
  };
  \node at (5.8,0)
  {
  \begin{tikzpicture}[x=1.4cm,y=0.8cm]
   \node at (-0.0,3) {$\cX^{\sigma\minor41}_{2,2}$};  
   \node (12) at (0,2) {$\Lambda/{1245}$};  
   \node (13) at (1,2) {$\Lambda/{135}$};  
   \node (24) at (2,2) {$\Lambda/{126}$};  
   \node (34) at (3,2) {$\Lambda/{1346}$};  
   \node (1) at (0,1) {$\Lambda\minor615$};  
   \node (2) at (1,1) {$\Lambda\minor312$};  
   \node (3) at (2,1) {$\Lambda\minor213$};  
   \node (4) at (3,1) {$\Lambda\minor516$};  
   \node (41) at (1.5,0) {$\Lambda\minor41$};
   \draw (41) to (1);
   \draw (41) to (2);
   \draw (41) to (3);
   \draw (41) to (4);
   \draw (1) to (12);
   \draw (1) to (13);
   \draw (2) to (12);
   \draw (2) to (24);
   \draw (3) to (13);
   \draw (3) to (34);
   \draw (4) to (24);
   \draw (4) to (34);
  \end{tikzpicture} 
 };
 \node at (10.6,0.1)
 {
 \begin{tikzpicture}[x=1.3cm,y=1.3cm,font=\bf]
  \node at (-1.4,0.85) {$\Sigma^{\sigma\minor41}_{2,2}$};  
  \filldraw[fill=blue!20!white,draw=blue!40!white] (22.5:1) -- (67.5:1) -- (112.5:1) -- (157.5:1) -- (202.5:1) -- (247.5:1) -- (292.5:1) -- (337.5:1) -- cycle;
  \foreach \a in {1,...,8}
  {
   \draw[draw=blue!40!white] (0:0) -- (22.5+45*\a:1);
   \draw[draw=blue!40!white]  (-22.5+45*\a:1) -- (22.5+45*\a:1);
  };
  \node at (0,0) {4/1};  
  \node at (22.5:1) {5/16};  
  \node at (67.5:1) {/126}; 
  \node at (112.5:1) {3/12};
  \node at (157.5:1) {/1245};
  \node at (202.5:1) {6/15};
  \node at (247.5:1) {/135};
  \node at (292.5:1) {2/13};
  \node at (337.5:1) {/1346};
   \end{tikzpicture} 
  };
\end{tikzpicture}
\caption{$\Sigma^{\sigma\minor41}_{2,3}$ for a subconstellation $\sigma$ of class 2d}
 \label{fig: subconstellation Lambda41 in class 2d}
\end{figure}

The order complex $\Sigma^\sigma_{2,3}$ is the union of the six discs corresponding to the six class $2c$ subconstellations, and is homeomorphic to a closed disc whose boundary points are identified with their antipodes, as illustrated in \autoref{fig: class 2d}. Thus $\Sigma^\sigma_{2,3}$ is a real projective plane and $H_1(\Sigma^\sigma_{2,3},\Z)=\Z/2\Z$. This means that we need to include subconstellations of class 2d in order for the homotopy theorem to be true. 

\begin{figure}[t]
\begin{tikzpicture}[x=1.3cm,y=1.3cm,font=\bf,line width=1pt]
  \filldraw[fill=blue!20!white,draw=blue!40!white] (0:3) -- (60:3) -- (120:3) -- (180:3) -- (240:3) -- (300:3) -- cycle;
  \foreach \a in {1,...,6}
  {
   \draw[draw=blue!40!white] (0:0) -- (60*\a:3);
   \draw[draw=blue!40!white]  (60*\a:1) -- (60+60*\a:1);
   \draw[draw=blue!40!white]  (60*\a:2) -- (60+60*\a:2);
   \draw[draw=blue!40!white]  (60*\a:3) -- (60+60*\a:3);
   \draw[draw=blue!40!white]  (60*\a:2) -- (30+60*\a:2.598);
   \draw[draw=blue!40!white]  (30+60*\a:1.732) -- (30+60*\a:2.598);
   \draw[draw=blue!40!white]  (60+60*\a:2) -- (30+60*\a:2.598);
  };
  \foreach \a in {1,...,3}
  {
   \draw[draw=blue!40!white]  (120*\a:2) -- (120+120*\a:2);
   \draw[draw=blue!40!white]  (60+120*\a:1) -- (30+120*\a:1.732);
   \draw[draw=blue!40!white]  (60+120*\a:1) -- (90+120*\a:1.732);
  }; 
  \node at (0:0) {/126};  
  \node at (0:1) {3/12};  
  \node at (60:1) {5/2};  
  \node at (120:1) {4/26};  
  \node at (180:1) {1/6};  
  \node at (240:1) {5/16};  
  \node at (300:1) {4/1};  
  \node at (  0:2) {/1245};  
  \node at ( 60:2) {/234};  
  \node at (120:2) {/2356};  
  \node at (180:2) {/456};  
  \node at (240:2) {/1346};  
  \node at (300:2) {/135};  
  \node at (  0:3) {3/45};  
  \node at ( 60:3) {5/34};  
  \node at (120:3) {4/35};  
  \node at (180:3) {3/45};  
  \node at (240:3) {5/34};  
  \node at (300:3) {4/35};  
  \node at ( 30:1.732) {3/24};  
  \node at ( 90:1.732) {4/23};  
  \node at (150:1.732) {1/56};  
  \node at (210:1.732) {2/46};  
  \node at (270:1.732) {2/13};  
  \node at (330:1.732) {6/15};  
  \node at ( 30:2.598) {1/4};  
  \node at ( 90:2.598) {6/3};  
  \node at (150:2.598) {2/5};  
  \node at (210:2.598) {1/4};  
  \node at (270:2.598) {6/3};  
  \node at (330:2.598) {2/5};  
\end{tikzpicture}
 \caption{$\Sigma^\sigma_{2,3}$ for a subconstellation $\sigma$ of class 2d}
 \label{fig: class 2d}
\end{figure}
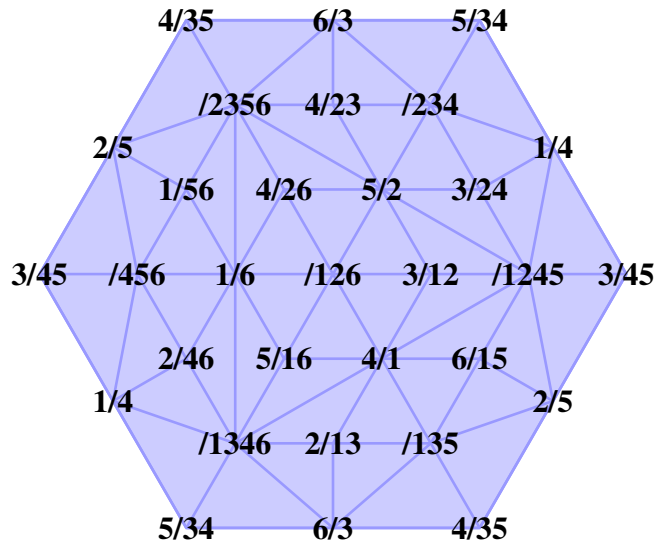
\end{rem}

\subsection{Towards a second homotopy theorem}
\label{subsection: towards a second homotopy theorem}

From a high level point of view (and deliberately oversimplifying matters somewhat), one might say that Tutte's proof of the homotopy theorem consists of finding an upper bound on the size of the subconstellations needed in the definition of $\cX_2^\tau$, together with an exhaustive search for all necessary subconstellations up to this bound, as described in \autoref{subsection: topological interpretation of the homotopy theorem}.

Establishing an upper bound on the size of subconstellations which must be included in $\cX_2^\tau$ is the more difficult part of the theorem, and Tutte's argument involves some ingenious ideas. At the time of writing, we do not know of a corresponding upper bound for the size of subconstellations which must be included in $\cX_3^\tau$ in order for a conjectural second homotopy theorem to be true. Indeed, we do not know if such an upper bound exists at all. Nevertheless, we can search for classes of subconstellations which would need to be included by testing whether the second homology of order complexes of various Tutte constellations vanishes or not. We make this process explicit in the following, and exhibit a few first such necessary subconstellations.

A \emph{marked constellation} is a Tutte constellation $\sigma=(\Lambda,\Gamma)$, together with a collection $\Theta$ of decomposable corank $2$ flats of $\Lambda-\Gamma$. If $\sigma$ appears as a subconstellation of a Tutte constellation $\tau$, then we require that $\Theta$ consists of precisely those decomposable corank $2$ flats of $\Lambda-\Gamma$ that are indecomposable in $\tau$. 

By abuse of notation, we use the same symbol $\sigma$ for a marked constellation $(\Lambda,\Gamma,\Theta)$. The \emph{isomorphism class of $\sigma$} is the class of all marked constellations $\sigma'=(\Lambda',\Gamma',\Theta')$ for which there is a lattice isomorphism $\Lambda\simeq\Lambda'$ that restricts to bijections $\Gamma\to\Gamma'$ and $\Theta\to\Theta'$. A \emph{class of subconstellations} is an isomorphism class of marked constellations. This gives a precise meaning to the notion of subconstellations of classes 0, 1, and 2a--2d.

Such classes of subconstellations are partially ordered by the relation $\sigma\leq\tau$ if $\sigma$ is isomorphic to a marked subconstellation of $\tau$ such that $\Theta_\sigma$ consists of exactly those decomposable corank $2$ flats in $\Lambda_\sigma-\Gamma_\sigma$ that are either in $\Theta_\tau$ or indecomposable in $\tau$. 
This allows us to search recursively over the poset of classes of subconstellations for those classes that are necessary for the second homotopy theorem to hold. 

Namely, we begin with the list $\cL_3^{\sigma_0}$ of classes 0, 1, 2a, and 2b,\footnote{First experimental data suggests that the classes 2c and 2d behave like exceptional cases and are better omitted. A more profound explanation for why we have to omit the classes 2c and 2d awaits further investigation.}, where $\sigma_{0}=\big(\{\emptyset\},\{\emptyset\},\emptyset\big)$ is the trivial marked constellation. Given a marked constellation $\tau$ such that $\cL_3^\sigma$ is defined for all marked subconstellations $\sigma$ of $\tau$, we define $\cL_3^{<\tau}$ as the union of all $\cL_3^\sigma$ with $[\sigma]<[\tau]$. Let $\cX^{<\tau}_3$ be the poset of all marked subconstellations of $\tau$ whose class is in $\cL_3^{<\tau}$, and let $\Sigma^{<\tau}_3$ be its order complex. If $H_2(\Sigma^{<\tau}_3,\Z)=0$, then we define $\cL_3^\tau=\cL_3^{<\tau}$; otherwise, we define $\cL_3^\tau=\cL_3^{<\tau}\cup\{[\tau]\}$. We denote by $\cL_3$ the union of all $\cL^\tau_3$, where $\tau$ varies over all classes of subconstellations.

In particular, if we consider a Tutte constellation $\tau$ as a trivially marked constellation (i.e., $\Theta=\emptyset$), then this definition yields a poset $\cX^\tau_3$ and the associated order complex $\Sigma^\tau_3$. By construction, we have $H_2(\Sigma^\tau_3,\Z)=0$ for all $\tau$, which could be regarded as a ``second homotopy theorem.''

But of course, such a second homotopy theorem would only be useful if we could describe the list $\cL_3$ explicitly. With this in mind, we pose the following tantalizing, but perhaps difficult, problem:

\begin{problem} \label{question:second homotopy conjecture}
 Is $\cL_3$ finite? If so, can we find an explicit description of $\cL_3$ and/or a marked constellation $\tau$ such that $\cL_3=\cL_3^\tau$? 
\end{problem}

Of course, if an affirmative answer to \autoref{question:second homotopy conjecture} is established, then one could also ask the analogous questions about $\cL_k$ for $k \geq 4$.

\subsection{First subconstellations for a conjectural second homotopy theorem}
\label{subsection: first subconstellations for the second homotopy theorem}

An easy case analysis shows that the classes 0, 1, 2a, and 2b are the only classes in $\cL_3$ whose lattice $\Lambda$ has fewer than $4$ atoms. The first new classes have lattices with $4$ atoms and are, besides class 2c, the following:

\medskip\noindent\textbf{Class 3a.} 
A \emph{subconstellation of class 3a} is a subconstellation $\sigma$ of type $U_{2,4}$ with $\Gamma_\sigma=\{E\}$ and $\Theta=\emptyset$. 

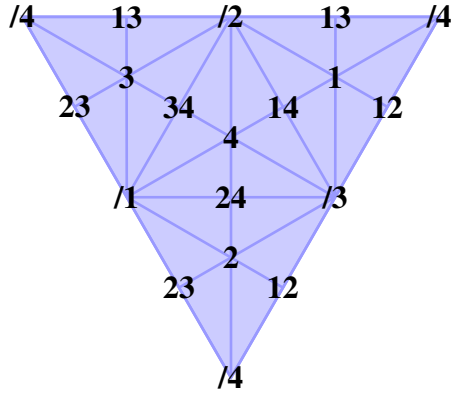
\begin{figure}[t]
\begin{tikzpicture}[x=0.8cm,y=0.8cm,font=\bf,line width=1pt]
  \filldraw[fill=blue!20!white,draw=blue!40!white] (30:4) -- (150:4) -- (270:4) -- cycle;
  \foreach \a in {1,...,3}
  {
   \draw[draw=blue!40!white] (30+120*\a:4) -- (150+120*\a:4);
   \draw[draw=blue!40!white] (30+120*\a:4) -- (210+120*\a:2);
   \draw[draw=blue!40!white]  (90+120*\a:2) -- (169+120*\a:2.646);
   \draw[draw=blue!40!white]  (11+120*\a:2.646) -- (90+120*\a:2);
   \draw[draw=blue!40!white]  (90+120*\a:2) -- (210+120*\a:2);
  };
  \node at (  0:0) {4};  
  \node at ( 90:2) {/2};  
  \node at (210:2) {/1};  
  \node at (330:2) {/3};  
  \node at ( 30:1) {14};  
  \node at (150:1) {34};  
  \node at (270:1) {24};  
  \node at ( 30:2) {1};  
  \node at (150:2) {3};  
  \node at (270:2) {2};  
  \node at ( 11:2.646) {12};  
  \node at ( 49:2.646) {13};  
  \node at (131:2.646) {13};  
  \node at (169:2.646) {23};  
  \node at (251:2.646) {23};  
  \node at (289:2.646) {12};  
  \node at ( 30:4) {/4};  
  \node at (150:4) {/4};  
  \node at (270:4) {/4};  
\end{tikzpicture}
 \caption{The order complex $\Sigma^{<\sigma}_3$ for $\sigma$ of class 3a}
 \label{fig: class 3a}
\end{figure}

\medskip\noindent\textbf{Class 3b.} 
A \emph{subconstellation of class 3b} is a subconstellation $\sigma$ of type $U_{2,3}\oplus U_{1,1}$ with $\Gamma_\sigma=\{E\}$ and $\Theta=\{1,2,3\}$. 

\begin{figure}[t]
\begin{tikzpicture}[x=0.8cm,y=0.8cm,font=\bf,line width=1pt]
  \filldraw[fill=blue!20!white,draw=blue!40!white] (30:4) -- (150:4) -- (270:4) -- cycle;
  \foreach \a in {1,...,3}
  {
   \draw[draw=blue!40!white] (30+120*\a:4) -- (150+120*\a:4);
   \draw[draw=blue!40!white] (30+120*\a:4) -- (210+120*\a:2);
   \draw[draw=blue!40!white]  (90+120*\a:2) -- (169+120*\a:2.646);
   \draw[draw=blue!40!white]  (11+120*\a:2.646) -- (90+120*\a:2);
   \draw[draw=blue!40!white]  (90+120*\a:2) -- (210+120*\a:2);
  };
  \node at (  0:0) {/4};  
  \node at ( 90:2) {34};  
  \node at (210:2) {24};  
  \node at (330:2) {14};  
  \node at ( 30:1) {2/4};  
  \node at (150:1) {1/4};  
  \node at (270:1) {3/4};  
  \node at ( 30:2) {2};  
  \node at (150:2) {1};  
  \node at (270:2) {3};  
  \node at ( 11:2.646) {/1};  
  \node at ( 49:2.646) {/3};  
  \node at (131:2.646) {/3};  
  \node at (169:2.646) {/2};  
  \node at (251:2.646) {/2};  
  \node at (289:2.646) {/1};  
  \node at ( 30:4) {123};  
  \node at (150:4) {123};  
  \node at (270:4) {123};  
\end{tikzpicture}
 \caption{The order complex $\Sigma^{<\sigma}_3$ for $\sigma$ of class 3b}
 \label{fig: class 3b}
\end{figure}
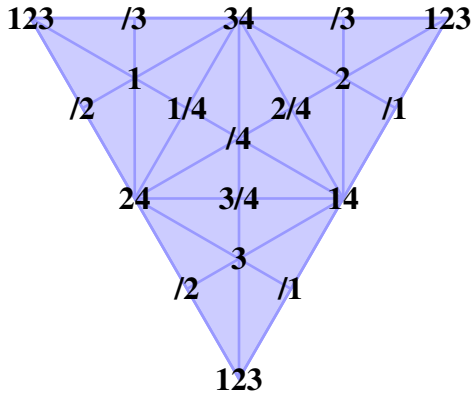

\medskip\noindent\textbf{Class 3c.} 
A \emph{subconstellation of class 3c} is a subconstellation $\sigma$ of type $U_{3,4}$ with $\Gamma_\sigma=\{E\}$ and $\Theta=\emptyset$. 

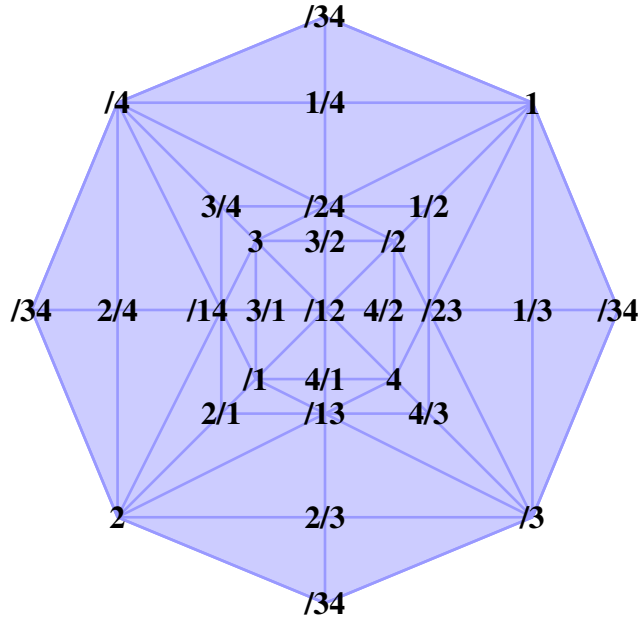
\begin{figure}[t]
\begin{tikzpicture}[x=1.3cm,y=1.3cm,font=\bf,line width=1pt]
  \filldraw[fill=blue!20!white,draw=blue!40!white] (0:3) -- (45:3) -- (90:3) -- (135:3) -- (180:3) -- (225:3) -- (270:3) -- (315:3) -- cycle;
  \foreach \a in {1,...,2}
  {
  };
  \foreach \a in {1,...,4}
  {
   \draw[draw=blue!40!white] (0+45*\a:3) -- (180+45*\a:3);
   \draw[draw=blue!40!white] (45+90*\a:3) -- (135+90*\a:3);
   \draw[draw=blue!40!white] (45+90*\a:3) -- (90+90*\a:3) -- (135+90*\a:3);
   \draw[draw=blue!40!white] (45+90*\a:1.5) -- (135+90*\a:1.5);
   \draw[draw=blue!40!white] (45+90*\a:1) -- (135+90*\a:1);
   \draw[draw=blue!40!white] (45+90*\a:1) -- (135+90*\a:3);
   \draw[draw=blue!40!white] (45+90*\a:3) -- (135+90*\a:1);
  };
  \node at (0:0) {/12};  
  \node at (  0:0.6) {4/2};  
  \node at ( 90:0.707) {3/2};  
  \node at (180:0.6) {3/1};  
  \node at (270:0.707) {4/1};  
  \node at ( 45:1) {/2};  
  \node at (135:1) {3};  
  \node at (225:1) {/1};  
  \node at (315:1) {4};  
  \node at (  0:1.2) {/23};  
  \node at ( 90:1.061) {/24};  
  \node at (180:1.2) {/14};  
  \node at (270:1.061) {/13};  
  \node at ( 45:1.5) {1/2};  
  \node at (135:1.5) {3/4};  
  \node at (225:1.5) {2/1};  
  \node at (315:1.5) {4/3};  
  \node at (  0:2.121) {1/3};  
  \node at ( 90:2.121) {1/4};  
  \node at (180:2.121) {2/4};  
  \node at (270:2.121) {2/3};  
  \node at (  0:3) {/34};  
  \node at ( 45:3) {1};  
  \node at ( 90:3) {/34};  
  \node at (135:3) {/4};  
  \node at (180:3) {/34};  
  \node at (225:3) {2};  
  \node at (270:3) {/34};  
  \node at (315:3) {/3};  
\end{tikzpicture}
 \caption{The order complex $\Sigma^{<\sigma}_3$ for $\sigma$ of class 3c}
 \label{fig: class 3c}
\end{figure}

\medskip\noindent\textbf{Class 3d.} 
A \emph{subconstellation of class 3d} is a subconstellation $\sigma$ of type $U_{4,4}$ with $\Gamma_\sigma=\{E\}$ for which all corank $2$ flats are in $\Theta$. 

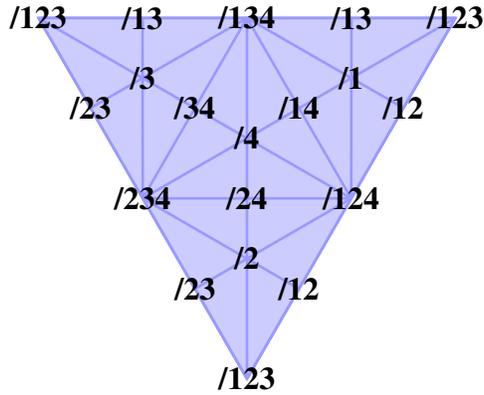
\begin{figure}[t]
\begin{tikzpicture}[x=0.8cm,y=0.8cm,font=\bf,line width=1pt]
  \filldraw[fill=blue!20!white,draw=blue!40!white] (30:4) -- (150:4) -- (270:4) -- cycle;
  \foreach \a in {1,...,3}
  {
   \draw[draw=blue!40!white] (30+120*\a:4) -- (150+120*\a:4);
   \draw[draw=blue!40!white] (30+120*\a:4) -- (210+120*\a:2);
   \draw[draw=blue!40!white]  (90+120*\a:2) -- (169+120*\a:2.646);
   \draw[draw=blue!40!white]  (11+120*\a:2.646) -- (90+120*\a:2);
   \draw[draw=blue!40!white]  (90+120*\a:2) -- (210+120*\a:2);
  };
  \node at (  0:0) {/4};  
  \node at ( 90:2) {/134};  
  \node at (210:2) {/234};  
  \node at (330:2) {/124};  
  \node at ( 30:1) {/14};  
  \node at (150:1) {/34};  
  \node at (270:1) {/24};  
  \node at ( 30:2) {/1};  
  \node at (150:2) {/3};  
  \node at (270:2) {/2};  
  \node at ( 11:2.646) {/12};  
  \node at ( 49:2.646) {/13};  
  \node at (131:2.646) {/13};  
  \node at (169:2.646) {/23};  
  \node at (251:2.646) {/23};  
  \node at (289:2.646) {/12};  
  \node at ( 30:4) {/123};  
  \node at (150:4) {/123};  
  \node at (270:4) {/123};  
\end{tikzpicture}
 \caption{The order complex $\Sigma^{<\sigma}_3$ for $\sigma$ of class 3d}
 \label{fig: class 3d}
\end{figure}

\medskip
For a marked constellation $\sigma$ of any of these classes, the order complex $\Sigma^{<\sigma}_3$ is homeomorphic to a $2$-sphere, as illustrated in \autoref{fig: class 3a}, \autoref{fig: class 3b}, \autoref{fig: class 3c}, and \autoref{fig: class 3d}, where we once again omit ``$\Lambda\minus$'' from the notation of the subconstellations for better readability. Thus $H_2(\Sigma^{<\sigma}_3,\Z)=\Z$ in all four cases, which shows that all four classes 3a--3d belong to $\cL_3$. It shows, moreover, that we cannot omit any cell from $\Sigma^{<\sigma}_3$ in each case. 

There are no other marked constellations on $4$ elements in $\cL_3$, with the exception of class $2c$, which enters the list because of a non-vanishing homology.


\appendix


\section{Lemmas that enter the proof of the path theorem}\label{appendix:flats}

In this section, we provide proofs of all statements that enter the proof of the path theorem (\autoref{thm: Tutte's path theorem}).\footnote{\label{footnote:LeanPath}As mentioned in the Introduction, we have formalized the proof of both the path theorem (\autoref{thm: Tutte's path theorem}) and homotopy theorem (\autoref{thm: Tutte's homotopy theorem}) in Lean~4. Although the Lean proofs generally follow the exposition in this paper fairly closely, in some cases they diverge. For example, the Lean proofs of \autoref{lemma:connected-contraction} and \autoref{prop:connected-intersection} differ from the proofs in this manuscript.}

Recall that a flat $F$ is called {\em indecomposable} if the contraction $M / F$ is connected, and {\em decomposable} otherwise. We use the rank characterization of connectedness: a finite
matroid $N$ is disconnected precisely when its ground set
admits a partition $E(N)=A\sqcup B$ into nonempty sets with
$\rk_N(A)+\rk_N(B)=\rk(N)$. Thus the empty matroid and
one-element matroids are connected under our convention.

The following lemma gives a set-theoretic description of decomposable flats. 

\begin{lemma}\label{lemma:connected-contraction}
A flat $F$ of $M$ is decomposable if and only if it can be written as $F = X_1 \cap X_2$ with $X_1 \cup X_2 = E$, such that neither $X_1$ nor $X_2$ equals $F$, and such that for every hyperplane $H \supseteq F$, either $H \supseteq X_1$ or $H \supseteq X_2$. 
\end{lemma}

\begin{proof}
If the contraction $M/F \cong M_1 \oplus M_2$ is disconnected, where the ground sets of $M_1$ and $M_2$ are $E_1$ and $E_2$, respectively, then the set of hyperplanes of $M/F$ is $\{H_1 \cup E_2 \mid H_1 \in \cH(M_1)\} \cup \{E_1 \cup H_2 \mid H_2 \in \cH(M_2)\} = \{X \subseteq E - F \mid X \cup F \in \cH(M)\}$. When we take $X_1$ to be $E_1 \cup F$ and $X_2$ to be $E_2 \cup F$, every hyperplane $H \supseteq F$ contains a hyperplane of $M/F$, and therefore must contain either $E_1$ or $E_2$. 

Conversely, assume that we can write $F$ as the intersection of $X_1$ and $X_2$ with all conditions of the lemma satisfied. Take two elements $a \in X_1 - F$ and $b \in X_2 - F$. Every hyperplane $H \supseteq F$ contains either $a$ or $b$. Therefore, there is no hyperplane of $M/F$ that avoids both $a$ and $b$; we conclude that $M/F$ is decomposable.    
\end{proof}

\begin{rem} \label{rem:connected-contraction}
If $F$ is decomposable, we call sets $X_1$ and $X_2$ as in \autoref{lemma:connected-contraction} a \emph{separation} of $F$.
\end{rem}

\begin{prop}\label{prop:connected-intersection}
If $F_1$ and $F_2$ are indecomposable flats with $F_1 \cup F_2 \ne E$, then $F_1 \cap F_2$ is an indecomposable flat. 
\end{prop}

\begin{proof}
Given distinct elements $a, b \notin F$, we wish to find a hyperplane in $M/F$ avoiding both. If $a, b \notin F_1$, then since $M/F_1$ is connected, there exists a hyperplane $H' \in \cH(M/F_1)$ avoiding both $a$ and $b$. Passing to $M/F$, we obtain a desired hyperplane $H = H' \cup (F_1 - F_2)$ of $M/F$ that does not contain $a$ or $b$. The case where $a, b \notin F_2$ is similar. 

The remaining case is when both $a \in F_1 - F_2$ and $b \in F_2 - F_1$ hold. Take an element $e \notin F_1 \cup F_2$. By the argument above, $M/F$ has a hyperplane $H_1$ that avoids $a$ and $e$, and a hyperplane $H_2$ that avoids $b$ and $e$. This implies that in the dual matroid $(M/F)^\ast$, there is a circuit $C_1$ containing $a$ and $e$, and a circuit $C_2$ containing $b$ and $e$. From this, we deduce the existence of a third circuit $C$ containing $a$ and $b$~\cite[Proposition 4.1.2]{Oxley11}, which gives a hyperplane in $M/F$ avoiding $a$ and $b$. 
\end{proof}

\begin{lemma}\label{lemma:flat-complement}
Let $S \supseteq T$ be flats of the matroid $M$. Then there exists a flat $U$ of $M$ such that $U \supseteq T$, $U \cap S = T$, $U \lor S = E$, and $\cork(U) = \rk(S) - \rk(T)$. 
\end{lemma}

\begin{proof}
Write $S_0 = S$ and $T_0 = T$. For $i = 1, 2, \dots, \cork(S)$, take an element $a_i \in E - S_{i-1}$, and form the two larger flats $S_i = \gen{S_{i-1}a_i}$ and $T_i = \gen{T_{i-1}a_i}$. Since $S \supseteq T$, we know that $\rk(S_i) - \rk(S) = \rk(T_i) - \rk(T) = i$, and hence $S_{\cork(S)} = E$. Consider $U := T_{\cork(S)}$. Then $U \supseteq T$, $U \lor S = T \lor \{a_1, \dots, a_{\cork(S)}\} \lor S = T \lor S_{\cork(S)} = E$, and $\cork(U) = \rk(M) - \rk(U) = \rk(M) - (\rk(T) + \cork(S)) = \rk(S) - \rk(T)$. The last equality implies that $\rk(U \cap S) \leqslant \rk(U) + \rk(S) - \rk(U \vee S) = \rk(T)$. Since $U \cap S \supseteq T$, we also have $U \cap S = T$. 
\end{proof}

\begin{rem}
    The construction in the proof of \autoref{lemma:flat-complement} shows that the lattice of flats of a matroid is {\em relatively complemented}~\cite[Exercise 1.7.7]{Oxley11}. It also shows that we can always choose a relative complement $U$ of $S$ with respect to $T$ such that $(S, U)$ is a modular pair of flats intersecting in $T$. 
\end{rem}

\begin{prop}\label{prop:conn-path}
Let $S \supset T$ be indecomposable flats. Then there exists an indecomposable flat $U$ such that $S \supseteq U \supseteq T$ and $\rk(U) = \rk(S) - 1$. 
\end{prop}

\begin{proof}
The case where $S = E$ is trivial. Now we assume that $S \ne E$. Since $T = S \cap (T \cup (E - S))$ is indecomposable, \autoref{lemma:connected-contraction} implies that the set $X$ of hyperplanes $H \supseteq T$ for which there exist distinct elements $s_H \in S - H$ and $t_H \in (T \cup (E - S)) - H = E - (S \cup H)$ is non-empty. Choose $H$ so that the quantity $|H \cap S|$ attains its maximum among all $H \in X$. We claim that $U := H \cap S$ is the desired indecomposable flat. 

By \autoref{prop:connected-intersection}, $U$ is indecomposable. We need to show that $\rk(U) = \rk(S) - 1$. Suppose for the sake of contradiction that $\rk(U) \leqslant \rk(S) - 2$. Then since $s_H \in S - H = S - U$, we have $\rk(\gen{Us_H}) = \rk(U) + 1 \leqslant \rk(S) - 1$. Let $U' := \gen{Us_H}$ (see\ \autoref{fig: subposet structure on the proof of prop:conn-path} for an illustration). Applying~\autoref{lemma:flat-complement} to the flats $S \supset U'$, and extending the resulting flat if necessary, there exists a hyperplane $H' \supseteq U' \supseteq T$ with $H' \lor S = E$ and hence $S - H' \ne \emptyset$. Then $(T \cup (E-S)) - H' = E - (S \cup H') \ne \emptyset$, since otherwise we would have $H' \supseteq (H \cap (E - S)) \cup (H \cap S) = H$ and $s_H \in H' - H$, a contradiction. 
Thus $H' \in X$. On the other hand, we have $|H' \cap S| \geqslant |H' \cap U'| \geqslant |U| + 1 = |H \cap S| + 1$, contradicting the maximality of $|H \cap S|$.
\end{proof}

\begin{figure}[t]
\centering
\begin{tikzpicture}[scale=1, vertices/.style={draw, fill=black, circle, inner sep=0pt},font=\footnotesize]
	\node (1) at (0, 1){$T$};
	\node (2) at (0, 2){$U$};
	\node (3) at (0, 3){$U'$};
	\node (4) at (0, 3.8){$S$};
	\node (5) at (1, 4.8){$H'$};
	\node (6) at (2, 4.8){$H$};
	
\foreach \to/\from in {2/3}
\draw [-] (\to)--(\from);
\draw[dotted] (1) -- (2);
\draw[dotted] (3) -- (4);
\draw[dotted] (3) -- (5);
\draw[dotted] (1) -- (6);
\draw[dotted] (2) -- (6);
\end{tikzpicture}
\caption{Subposet structure in \autoref{prop:conn-path}}
\label{fig: subposet structure on the proof of prop:conn-path}
\end{figure}

\begin{cor}[Indecomposable Chain Property]\label{cor:conn-chain}
    Let $S \supset T$ be indecomposable flats. Then there exists a chain of indecomposable flats $S = U_0 \supset U_1 \supset \cdots \supset U_k = T$ such that $\rk(U_i) - \rk(U_{i+1}) = 1$ for all $i = 0, 1, \dots, k-1$. 
\end{cor}

\begin{proof}
    Applying \autoref{prop:conn-path} repeatedly to flats $U_i \supset T$ gives the chain of indecomposable flats $S = U_0 \supset U_1 \supset \cdots \supset U_k = T$. 
\end{proof}

\begin{prop}[Indecomposable Diamond Property]\label{prop:conn-diamond}
Let $S \supseteq T$ be indecomposable flats with $\rk(S) = \rk(T) + 2$. Then there exist indecomposable flats $S \supseteq U \ne V \supseteq T$ with $\rk(U) = \rk(V) = \rk(S) - 1$. 
\end{prop}

\begin{proof}
We first choose $U$ as in \autoref{prop:conn-path}. Pick $a \in S - U$ and write $W = \gen{Ta}$, which is another flat of the same rank as $U$ satisfying $S \supseteq W \supseteq T$. By \autoref{lemma:flat-complement}, there exists a corank $2$ flat $L$ such that $L \supseteq T$, $L \cap S = T$, and $L \lor S = E$. We deduce that $U \cap L = W \cap L = T$. Consider the hyperplanes $H_1 = U \lor L$ and $H_2 = W \lor L$. We have $S \cap H_2 = W$, and we may assume without loss of generality that $W$ is decomposable. \autoref{prop:connected-intersection} implies that $S \cup H_2 = E$. 

We claim that $U \cup H_2 \ne E$. Otherwise, the ground set of the contraction matroid $M/T$ is partitioned as $E - T = (U-T) \cup (H_2 - T)$, but then $\rk_{M/T}(U-T) + \rk_{M/T}(H_2 - T) = 1 + (\rk(M) - 1 - \rk_M(T)) = \rk(M/T)$, implying that $M/T$ is disconnected, a contradiction. Similarly, we have $S \cup H_1 \ne E$ from the indecomposability of the flat $U$.

Pick $b \notin U \cup H_2$ and $c \notin S \cup H_1$. Consider $V = \gen{Tb}$. Since $b \notin H_2$ and $S \cup H_2 = E$, we have $b \in S$, and hence $V \subseteq S$. Moreover, $V \ne U$ because $b \notin U$. By submodularity, $H = V \lor L$ is a hyperplane and $S \cap H=V$. Since $b \in H - H_2$, the hyperplanes $H, H_2$ are distinct; both contain the corank $2$ flat $L$, so $H \cap H_2 = L$. Now $c \notin S$ and $S \cup H_2 = E$ imply $c \in H_2$. If $c \in H$, then $c \in H \cap H_2 = L \subseteq H_1$, contrary to the choice of $c$. Thus $S \cup H \ne E$. By \autoref{prop:connected-intersection}, $V = S \cap H$ is indecomposable, which completes the proof. See \autoref{fig: subposet structure in the proof of prop:conn-diamond} for an illustration.
\end{proof}

\begin{figure}[t]
\centering
\begin{tikzpicture}[scale=1, vertices/.style={draw, fill=black, circle, inner sep=0pt},font=\footnotesize]
	\node (1) at (0, 0){$T$};
	\node (2) at (-1, 1){$U$};
	\node (3) at (0, 1){$W$};
	\node (4) at (1, 1){$V$};
	\node (5) at (0, 2){$S$};
	\node (6) at (3.5, 3){$L$};
        \node (7) at (2.5, 4){$H_1$};
        \node (8) at (3.5, 4){$H_2$};
        \node (9) at (4.5, 4){$H$};
	
\foreach \to/\from in {1/2, 1/3, 1/4, 2/5, 3/5, 4/5, 6/7, 6/8, 6/9}
\draw [-] (\to)--(\from);
\draw[dotted] (1) -- (2);
\draw[dotted] (3) -- (8);
\draw[dotted] (1) to[out=10,in=-120] (6);
\draw[dotted] (2) to[out=90,in=170] (7);
\draw[dotted] (4) to[out=10,in=-90] (9);
\end{tikzpicture}
\caption{Subposet structure in \autoref{prop:conn-diamond}}
\label{fig: subposet structure in the proof of prop:conn-diamond}
\end{figure}

\begin{prop}[Indecomposable Complement Property]\label{prop:conn-complement}
Let $S$, $T$, and $U$ be flats. Assume that $S$ and $T$ are indecomposable, $S \cap U \supseteq T$, and $S \lor U = E$. Then there exists an indecomposable flat $R$ such that $S \supseteq R \supseteq T$, $R \lor U = E$, and $\cork(R) = \rk(U) - \rk(T)$. 
\end{prop}

\begin{proof}
We use induction on the corank of $U$. The case where $U = E$ is trivial, since we can always choose $R = T$. Now suppose that the flat $U$ has corank $c \geqslant 1$. Let $W$ be an indecomposable flat of the least possible rank such that $S \supseteq W \supseteq T$ and $W \not\subseteq U$. Then $\rk(W) = \rk(T) + 1$ by \autoref{cor:conn-chain} and \autoref{prop:conn-diamond}. By the submodularity of the rank function, $\rk(U \lor W) - \rk(U) \leqslant \rk(W) - \rk(T) = 1$. However, $U \lor W$ properly contains $U$, and therefore the corank of $U \lor W$ is $c-1$. By induction, there exists an indecomposable flat $R$ such that $S \supseteq R \supseteq W \supsetneq T$,  $R \lor (U \lor W)= E$, and $\cork(R) = \rk(U \lor W) - \rk(W) = (\rk(U) + 1) - (\rk(T) + 1) = \rk(U) - \rk(T)$. 
\end{proof}


\section{Proof of the homotopy theorem (by Ju\v s Kocutar)}\label{appendix: proof of the homotopy theorem}

In this section, we present Tutte's proof of the homotopy theorem in the language developed in this paper. In particular, we replace circuits by hyperplanes, in contrast to Tutte's original account. We also explain some details which Tutte, with his condensed style, leaves out. The proof follows the outline from \autoref{sec:tutte-homotopy-outline}. When there is a statement which exactly corresponds to a statement from Tutte's paper, we give a reference to it.

Throughout this Appendix, the matroid $M$ is always connected.  

\subsection{Preliminaries}\label{subsubsection: preliminaries}We need the next lemma for~\autoref{theorem: disconnected corank2 on connected corank 3}. 

\begin{lemma}\label{theorem: cap of points is disconnected.}
    Let $H_1$ and $ H_2$ be distinct hyperplanes such that $H_1 \cap H_2$ is a decomposable flat. Then $\cork (H_1 \cap H_2) =2.$
\end{lemma}

\begin{proof}
 Let $\{X_1, X_2\}$ be a separation for $H_1 \cap H_2$ such that $H_i \supseteq X_i$ (if both $H_1$ and $H_2$ contain the same set $X_j$ of the separation, then
$H_1 \cap H_2 = X_j$, contradicting the definition of separation). Assume, for the sake of contradiction, that there exists $a\in H_1-X_1$. Then $a \in X_2 - X_1$, implying that $a \in H_2.$ Therefore $a \in (H_1 \cap H_2) - X_1$, which is impossible. It follows that $X_i = H_i,$ implying that the only hyperplanes above $H_1 \cap H_2 $ are $H_1$ and $H_2$, because any such hyperplane is either equal to $X_1$ or $X_2$. Therefore $H_1\cap H_2$ is a decomposable and corank 2 flat.
\end{proof}

We use $[S,U,T]$ to denote a triple of flats $S, U$, and $T$ satisfying the three assumptions in \autoref{prop:conn-complement}, i.e., $S$ and $T$ are indecomposable, $S \cap U \supseteq T$, and $S \lor U = E$. When we say that a flat $G_1$ is \textit{above} or \textit{below} a flat $G_2$, we mean $G_1 \supseteq G_2$ or $G_1 \subseteq G_2$ respectively.

\begin{lemma}\cite[(4.2)]{Tutte58a}\label{theorem: disconnected corank2 on connected corank 3}
        Let $L$ be a decomposable corank 2 flat, and let $S$ be an indecomposable flat with  $L \supsetneq S$. Then there exists an indecomposable corank 3 flat $P$ with $L \supsetneq P\supseteq S$. 
\end{lemma}

\begin{proof}
Let $P$ be an indecomposable flat of minimal corank such that $L \supsetneq P \supseteq S$. Assume for the sake of contradiction that $\cork(P)>3.$

Denote the two distinct hyperplanes above $L$ by $X$ and $Y.$ We first assume, for the sake of contradiction, that there exists a decomposable corank 2 flat $L' \neq L$ such that $X \supsetneq L' \supsetneq P.$ By \autoref{prop:conn-complement} applied to $[Y, L', P]$, there exists an indecomposable flat $U$ such that $\cork(U) = \rk(L')-\rk(P) = \cork(P)-2,$ $U \lor L' = E$, and $Y \supseteq U\supseteq P.$ By \autoref{prop:conn-diamond}, there exist indecomposable flats $V$ and $W$ such that $V \lor W = U$ and $V \cap W = P.$ We also observe that $V\cap L' = P$, because of $P \subset V \cap L' \subset U \cap L'$ and the submodular inequality
$$\rk(U\cap L')\leq \rk(U)+\rk(L')-\rk(U\lor L') = \rk(U)-2 = \rk(P).$$
Using the submodular inequality again, we get $$\rk(V\lor L') \leq \rk(V)+\rk(L')-\rk(V\cap L') = \rk(L')+1,$$ showing that $V\lor L'$ and $W \lor L'$ are distinct hyperplanes (since $V = (V\lor L') \cap U$ and $W = (W\lor L') \cap U$) above $L'$. Assume without loss of generality that $V \lor L' = X.$ Then $V$ is an indecomposable flat below $X$ and $Y$, implying that $L \supseteq V$ and $\cork(V) = \cork(P)-1,$ which contradicts the minimality of $P.$ Therefore, $L$ is the only decomposable corank 2 flat above $P$ and below $X$. By symmetry $L$ is the only decomposable corank 2 flat above $P$ and below $Y$. 

Therefore $L$ is the unique decomposable corank 2 flat which is above $P$ and below $X$ or below $Y$. Pick $a \in L - P$. The flat $P \lor a$ is a decomposable corank $\cork(P)-1$ flat, by the minimality of $P$, and $L$ is the unique decomposable corank 2 flat above $P \lor a$ and below $X$ or $Y$. Let $\{X_1,X_2\}$ be a separation of $P\lor a.$ In the rest of the proof, we show that the only hyperplanes above $P \lor a$ are $X$ and $Y$. 

Without loss of generality we may assume that $X \supseteq X_1$. We claim that if $H_2$ is a hyperplane above $P\lor a$ such that $H_2 \supseteq X_2$, then $X \cap H_2$ is a decomposable corank 2 flat. To see this, suppose for the sake of contradiction that $H_3$ is a third hyperplane above $X \cap H_2.$ Assume first that $H_3 \supseteq X_1.$ Then if $b \in X-H_2$, we have $b \notin X_2$, so $b \in X_1$, implying $X_1 \supseteq X - H_2.$
Therefore $H_3 \supseteq (X\cap H_2) \cup X_1 \supseteq (X\cap H_2) \cup( X - H_2) = X$, a contradiction. Second, if $H_3 \supseteq X_2$, taking any $c\in H_2-X$ we get $c\in X_2$, and therefore $H_3 \supseteq (X \cap H_2)\cup X_2 \supseteq (X\cap H_2)\cup(H_2-X) = H_2$, again a contradiction. Hence $X \cap H_2$ is a decomposable corank 2 flat. The flat $X\cap H_2$ is thus a decomposable corank 2 flat above $P$ and below $X$, therefore it is equal to $L$. In particular, $H_2 = Y$.  Similarly, any hyperplane above $P \lor a$ containing $X_1$ must equal $X$. The only hyperplanes above the decomposable flat $P\lor a$ are thus $X$ and $Y$, implying that $P\lor a = X \cap Y = L$ and $\cork(P) = 3$, a contradiction. 
\end{proof}

\begin{lemma}\cite[(4.3) and (4.4)]{Tutte58a}\label{theorem: connected corank 3 flats}
    Let $P$ be an indecomposable corank 3 flat and $L \supsetneq P$ a decomposable corank $2$ flat. Let $X$ and $Y$ be the two distinct hyperplanes containing $L$. Then for every other hyperplane $Z$ containing $P$, the only corank 2 flats between $Z$ and $P$ are $Z\cap X$ and $Z\cap Y$, which are both indecomposable. As a consequence, $L$ is the only decomposable corank $2$ flat containing $P$. 
\end{lemma}
\begin{proof}

Let $Z$ be a hyperplane above $P$ distinct from $X$ or $Y$, and let $L'$ be a corank 2 flat such that $Z \supsetneq L' \supsetneq P.$ By the submodular inequality, the flat $L \lor L'$ is a hyperplane and hence is equal to either $X$ or $Y.$ Therefore, the only corank 2 flats between $Z$ and $P$ are $X \cap Z$ and $Y \cap Z$. By \autoref{prop:conn-diamond}, $X\cap Z$ and $Y \cap Z$ must be indecomposable.

Let $L'' \ne L$ be any other corank 2 flat above $P$. Then $L'' \lor L$ is a hyperplane, so is equal to either $X$ or $Y$. But there is also a hyperplane $Z \supseteq L'' \supseteq P$ which is equal to neither $X$ nor $Y.$ We conclude that $L''$ is indecomposable. 
\end{proof}

 Let $\gamma = (H_0, \ldots, H_k)$ be a Tutte path. We denote by $F(\gamma)$ the flat $H_0 \cap \cdots \cap H_k.$ By the corank and the rank of $\gamma$, we mean $\cork(F(\gamma))$ and $\rk (F(\gamma))$, respectively.

\begin{lemma}
\label{theorem: carrier of path connected}
The flat $F(\gamma)$ is indecomposable for every Tutte path $\gamma = (H_0, \ldots, H_k)$.
\end{lemma}

\begin{proof}
    The case where $k = 0$ is trivial. Now we suppose $k \geq 1$. Because $H_i\cap H_{i+1}$ is indecomposable for all $i$, we have $H_i \cup H_{i+1} \neq E$ by \autoref{lemma:connected-contraction}. Define, for all $0 \leq i \leq k$, the flat $F_i = \bigcap_{j = 0}^{i}H_j.$ The flat $F_1 = H_0 \cap H_1$ is indecomposable by hypothesis, and it follows from \autoref{prop:connected-intersection} that $F_2$ is indecomposable. Indeed, $F_2 = F_1 \cap H_2$ with $F_1$ and $H_2$ indecomposable, and $H_2 \cup F_1 \subseteq H_2 \cup H_1 \neq E.$ It follows by similar reasoning that $F_k = F(\gamma)$ is indecomposable.
\end{proof}

\subsection{Statement of the homotopy theorem}\label{subsubsection: statement of homotopy theorem}

We recall the statement of Tutte's homotopy theorem.

\begin{thm}\cite[(6.1), Tutte's Homotopy theorem]{Tutte58a} \label{theorem: Tutte's homotopy theorem}
Let $M$ be a connected matroid and let $\Gamma$ be a modular cut in $M$. Then every closed Tutte path off $\Gamma$ is null-homotopic.
\end{thm}

\begin{rem}\label{remark: tutte version of fourth path}
    The elementary path of the fourth kind is described in a different way in \cite{Tutte58a} than in \autoref{subsubsection: constellations}. We summarize Tutte's original point of view as follows. The starting point is a corank 4 flat $D$ above which there are three hyperplanes $A,$ $B$, and $C$ such that $A\cap B,$ $B\cap C$, and $A \cap C$ are all decomposable corank 2 flats. Above $D$ there are exactly six indecomposable corank 3 flats, such that each decomposable corank 2 flat as described above lies above exactly two of the six indecomposable corank 3 flats. The flats $A,$ $B$, and $C$ are not in $\Gamma$, and there are exactly two hyperplanes of $\Gamma$ above each of the six indecomposable corank 3 flats. We define a path of the form $\delta = (A,X,B,Y,A),$ where $X,Y\notin \Gamma$ lie above distinct indecomposable corank 3 flats below $A\cap B,$ to be an elementary path of the fourth kind with respect to $\Gamma.$ An explicit description of all flats generated as joins of the six indecomposable corank 3 flats is given in \cite{Tutte58a}. One can check\footnote{Our Lean formalization of Tutte's homotopy theorem, which was mentioned in the Introduction, implements a proof of this assertion.} using \autoref{theorem: connected corank 3 flats} and \autoref{theorem: carrier of path connected}, as well as some of the lemmas from \autoref{appendix:flats}, that the description in \cite{Tutte58a} gives the same lattice as the lattice of flats of $M(K_{2,3})$, which is what we used in \autoref{subsubsection: constellations}. Our proof of the homotopy theorem will use Tutte's original description of elementary paths of the fourth kind, following \cite{Tutte58a}.
\end{rem}

\subsection{The special lemma}\label{subsubsection: special lemma}

For the remainder of \autoref{appendix: proof of the homotopy theorem}, we closely follow the proof in \cite{Tutte58a}, except for the fact that we  replace circuits with hyperplanes of the dual matroid. 

The homotopy theorem is true for any path $\gamma$ with $\cork (F(\gamma)) = 1$. We assume that \autoref{theorem: Tutte's homotopy theorem} is true for all closed Tutte paths $\gamma$ with $\cork(F(\gamma))\leq n.$ In \autoref{subsection: final proof}, we prove \autoref{theorem: Tutte's homotopy theorem} by contradiction following a minimal counterexample. 
In this section, we prove that a certain special type of Tutte path is null-homotopic. 

\begin{df}\label{theorem: lemma for homotopy theorem}
Let $M$ be a matroid and let $\Gamma$ be a modular cut in $M$. A \textit{special path} is a Tutte path $\delta=(W,X,Y,Z,W)$ off $\Gamma$ such that $W\cap X \cap Y$ and $Y\cap Z \cap W$ are indecomposable corank 3 flats, and such that $W\cap Y $ is a decomposable corank 2 flat. 
\end{df}

The following can be found in \cite[pp. 153-154]{Tutte58a} or~\cite[Section 6.3]{Tutte65}.

\begin{lemma}\cite[Lemma]{Tutte58a} \label{theorem: good paths}
Let $M$ be a matroid with modular cut $\Gamma,$ and assume that $n \geq 3$. If every closed Tutte path $\gamma$ in $M$ off $\Gamma$ with $\cork(F(\gamma)) \leq n$ is null-homotopic, then every special path $\delta$ with $\cork(F(\delta)) = n+1$ is null-homotopic.
\end{lemma}
\begin{rem} 
When we apply \autoref{theorem: good paths}, we do not need to manually prove the homotopy theorem for paths of corank 2 or 3 beforehand. Rather, when we apply it, the special paths under consideration will have corank at least 4. Therefore, once we encounter it in the proof, we will already be under the assumption that $\cork (F(\epsilon))\geq 4$ for the minimal non-null-homotopic path $\epsilon$, so the assumptions of \autoref{theorem: good paths} will be satisfied.
\end{rem}

We split the proof of \autoref{theorem: good paths} into a number of cases, resulting in easier lemmas through which the structure of the lattice above $F(\delta)$ will be determined. The strategy is always the same: we decompose $\delta \sim \delta_1\cdots \delta_k$ into closed Tutte paths $\delta_i$, which are known to be null-homotopic, and therefore $\delta$ is null-homotopic as well. 

\medskip

The setting of all lemmas in this section is that we have a special path $\delta = (W,X,Y,Z,W)$ of corank $n+1$ with $F_1 = W\cap X \cap Y$ and $F_2 = Y\cap Z \cap W$, and we assume that $n \geq 3.$ We denote any trivial path, by which we mean a path having only one term, by 0. We assume for the sake of contradiction that $\delta$ is {\bf not} null-homotopic.

\begin{lemma}\label{theorem: QQ'}
If $\delta'= (W,X',Y, Z', W)$ is a Tutte path off $\Gamma$ with $X'\supsetneq F_1$ and $Z' \supsetneq F_2,$ then $\delta' \sim \delta.$ 
\end{lemma}

\begin{proof}
Note that
    $$\delta'\sim (W,X',Y)(Y, X, W)(W, X, Y)(Y, Z, W)(W, Z, Y)(Y, Z', W),$$
   where $(W,X',Y)(Y, X, W)$ and $(W, Z, Y)(Y, Z', W)$ are closed Tutte paths on $F_1$ and $F_2$, respectively. Because $\cork(F_1) = \cork(F_2) = 3 \leq n$, the paths $(W,X',Y)(Y, X, W)$ and $(W, Z, Y)(Y, Z', W)$ are null-homotopic by assumption. 
\end{proof}

We now define two special types of flats above $F(\delta)$ of coranks $n$ and $n-1$, respectively, which will later serve as flats of smaller corank on which the paths $\delta_i$ will lie.

\begin{df}\label{definition: Tn}
    A \textit{type (a) transversal} is an indecomposable flat $A$ of corank $n$ above $F(\delta)$ that is not contained in both $W$ and $Y$.
\end{df}

\begin{lemma}\label{theorem: TnlorF1}
    If $A$ is a type (a) transversal, then $A\lor F_1$ and $A \lor F_2$ are indecomposable corank 2 flats.
\end{lemma}

\begin{proof}
        By symmetry, it suffices to prove the lemma for $A \lor F_1$. Observe that $\rk (   A \lor F_1)>\rk (   F_1) = \rk (   E)-3$; otherwise $A \subseteq F_1$, which contradicts the definition of $A$, 
        since it would then be contained in both $W$ and $Y.$ 
        By the same reasoning, we find that $A \cap F_1$ is a proper subset of $A,$ hence $A \cap F_1 = F(\delta).$ 
        The submodular inequality now implies 
   \begin{equation*}
   \begin{split}
    \rk (   A\lor F_1) &\leq \rk (   A) + \rk (   F_1) - \rk (   A \cap F_1) \\
     &= \rk (   E)-n + \rk (   E)-3 - ( \rk (   E)-(n+1)) \\    
     &=\rk (   E) -2,
   \end{split} 
   \end{equation*}
   and therefore $A \lor F_1$ is a corank 2 flat. It follows from \autoref{theorem: connected corank 3 flats} that $A \lor F_1$ is indecomposable.
\end{proof}
    
\begin{df}\label{definition: Tn-1}
    A \textit{type (b) transversal} is an indecomposable flat $B$ of corank $n-1$ above $F(\delta)$ for which both $B \not\subseteq W$ and $B\not\subseteq Y$ hold. 
\end{df}

\begin{lemma}\label{theorem: Tn-1lorF1}
    If $B$ is a type (b) transversal, then $B\lor F_1$ and $B \lor F_2$ are hyperplanes. 
    \end{lemma}

\begin{proof}
     It suffices to prove the result for $B \lor F_1$. By an application of submodular inequality analogous to the proof of \autoref{theorem: TnlorF1}, we find that $\rk (B\lor F_1) \leq \rk (E)-1,$ and since $B \not \subseteq F_1$ we have $\rk (E)-3 < \rk (B\lor F_1).$ Assume, for the sake of contradiction, that $\rk (B\lor F_1) = \rk (   E)-2.$ It follows from \autoref{theorem: connected corank 3 flats} that $B\lor F_1$ is an indecomposable corank 2 flat, and that if $H'$ is any hyperplane above $B\lor F_1$, $H' \cap W$ and $H' \cap Y$ are the only corank 2 flats above $F_1$ and below $H'$. This is a contradiction, since $B$ is not below $W$ or $Y.$ We conclude that $\rk (B\lor F_1) = \rk (E)-1$. 
\end{proof}

\begin{df}
    If $B$ is a type (b) transversal, we call the flats $B\lor F_1$ and $B\lor F_2$ the \textit{poles} of $B.$
\end{df}

\begin{lemma}\label{theorem: tn-1 poles}
    Let $B$ be a type (b) transversal. Then at least one of its poles is in $\Gamma.$
\end{lemma}
  \begin{proof}

    Let the poles of $B$ containing $F_1$ and $F_2$ be $X'$ and $ Z'$, respectively; we know they exist by \autoref{theorem: Tn-1lorF1}. Using \autoref{theorem: carrier of path connected}, we know that $F(\delta)$ is indecomposable. By \autoref{prop:conn-diamond}, there exist indecomposable flats $G_1$ and  $G_2$ of corank $n$ such that $F(\delta) = G_1 \cap G_2$ and $B = G_1 \lor G_2.$
    
    Flats $G_1$ and $G_2$ cover $F(\delta)$ so there exist elements $a_1,a_2\in E$ such that $G_i = \gen{F(\delta)\cup a_i}$. Because $W$ and $Y$ are hyperplanes such that $W\cap Y$ is an decomposable corank 2 flat, we get $W\cup Y = E.$ Therefore $a_1$ is either in $W$ or in $Y$ implying that $G_1$ is either below $W$ or below $Y,$ analogously for $G_2.$ If both $G_1$ and $G_2$ are below the same hyperplane in $\{W, Y\}$, we get that $B = G_1 \lor G_2$ is also below that hyperplane, which is a contradiction. 
    
    Hence, we may assume without loss of generality that $G_1$ is contained in $W$ and that $G_2$ is contained in $Y.$ Thus, the flats $G_1$ and $G_2$ are both type (a) transversals and are equal to $B\cap W$ and $B\cap Y$, respectively. 
    
    Hence we know that $G_i\lor F_j$ for $1\leq i,j \leq 2$ are $W \cap X',$ $ W \cap Z',$ $ Y \cap X',$ and $ Y \cap Z'$, and they are all indecomposable corank 2 flats, because the unique decomposable corank 2 flat above both $F_1$ and $F_2$ is $W\cap Y$ by \autoref{theorem: connected corank 3 flats}. The path $(W, X', Y, Z', W)$ is thus a Tutte path.

  Assume, for the sake of contradiction, that $X'\notin \Gamma$ and $Z'\notin \Gamma$. Because the poles $X'$ and $Z'$ lie above an indecomposable flat $B$, there exists a Tutte path $\epsilon$ on $B$ and off $\Gamma$ from $X'$ to $Z'$ by \autoref{thm: Tutte's path theorem}. Hence there exists a Tutte path on $B\cap W$ ($B\cap W$ is a type (a) transversal so indecomposable) given by $(W, X')\epsilon (Z', W)$ and a Tutte path $(X', Y, Z')\epsilon^{-1}$ on $B\cap Y$ ($B\cap Y$ is a type (a) transversal thus indecomposable). Because the coranks of these paths are at most $n$, we know that they are null-homotopic, and hence
    \[
    \delta\sim (W, X', Y, Z', W) \sim (W, X')(X', Y, Z')(Z', W) \sim 
(W, X')\epsilon (Z', W)   \sim 0 \]
    is also null-homotopic, a contradiction. Therefore, for any type (b) transversal, at least one of its poles has to be in $\Gamma.$  
 \end{proof}
 
\begin{lemma}
    There is a type (a) transversal $A$ not contained in $Y$ which is the intersection of two type (b) transversals $B$ and $B'$ such that $B\lor F_1 =X'\notin \Gamma,$ $B \lor F_2 = U_2 \in \Gamma,$ $B'\lor F_1 =U_1\in \Gamma,$ and $B' \lor F_2 = Z' \notin \Gamma.$
\end{lemma}
\begin{proof}
    
 Using \autoref{prop:conn-complement} applied to $[W, Y, F(\delta)]$, we obtain an indecomposable flat $A$ of corank $n$ such that $W \supseteq A \supsetneq F(\delta)$ but $Y \not \supseteq A.$ In particular, $A$ is a type (a) transversal, and hence $A \lor F_i = L_i$ are indecomposable corank 2 flats for $i = 1,2$ by \autoref{theorem: TnlorF1}. There exists a hyperplane $X'\supsetneq L_1$ such that $X'\neq W$ and $X'\notin \Gamma$. (The latter assertion follows because, as a join of $A$ and $F_1$, both of which are below $W$, $L_1$ must also be below $W$.) Hence, there can be at most one member of $\Gamma$ above $L_1$ (since otherwise two such hyperplanes would form a modular pair with the intersection $L_1$, resulting in $L_1 \in \Gamma$ and hence $W \in \Gamma$, a contradiction).
 Finally, since $L_1$ is indecomposable, we know that there are at least three hyperplanes above it, thus we can pick the desired $X'.$

 Applying \autoref{prop:conn-complement} to $[X', W, A]$, there is a type (b) transversal $B$ not contained in $W$ such that $X'\supseteq B \supsetneq A.$ Because its pole $B \lor F_1 = X'$ is not in $\Gamma$, we know that the other pole $U_2  = B \lor F_2$ has to be in $\Gamma$ by \autoref{theorem: tn-1 poles}. Similarly, there exists a hyperplane $Z' \ne W$ above $L_2$ which is not in $\Gamma,$ because $L_2$ is indecomposable and it is below $W\notin \Gamma.$ Applying \autoref{prop:conn-complement} to $[Z', W, A]$, we get a type (b) transversal $B' \supseteq A$ such that $Z' = B'\lor F_2 \notin \Gamma$ and $U_1 = B'\lor F_1 \in \Gamma.$ We know $B \ne B'$ from the construction. Therefore, we have $B \cap B' = A$. 
 \end{proof}
\begin{lemma}\label{theorem: corank of T is n-2}
    Let $B$ and $B'$ be as above, and let $T = B\lor B'$. Then $\cork(T) = n-2.$
\end{lemma}
\begin{proof} By the submodular inequality, we have $\rk (   T)\leq \rk (   B)+1$. Because $B \lor B'$ contains both $B$ and $B'$ as proper subsets, it follows that $B \lor B'$ is of corank $n-2.$
\end{proof}

 \begin{lemma}
 Let $T$ be as above, and let $\cS$ be the set of all hyperplanes above $T$ that are not in $\Gamma.$ Then $\cS$ is nonempty.
 \end{lemma}
 \begin{proof}
 Assume, for the sake of contradiction, that $\cS = \emptyset$. Then $T \in \Gamma.$ Since $U_1 \not\supseteq B$, we have $U_1 \not\supseteq T$. By submodularity inequality, we must have that $U_1 \cap T = B'$ and that $(U_1, T)$ forms a modular pair, which implies that $B' \in \Gamma.$ Thus $Z'\supseteq B'$ is in $\Gamma$, contradicting the definition of $Z'$. 
\end{proof}

\begin{lemma}\label{theorem: Y cap T is disconnected}
Let $\cS$ be as above, and let $I \in \cS.$ Then $Y \cap I$ is a decomposable corank 2 flat.
\end{lemma}

\begin{proof}
By \autoref{theorem: cap of points is disconnected.}, it suffices to show that $Y \cap I$ is decomposable. Assume for the sake of contradiction that $Y \cap I$ is indecomposable. 
Then there exists a Tutte path $\epsilon_0$ from $Y$ to $I$ on $Y \cap I$ off $\Gamma$ by the path theorem. The path theorem also shows the existence of Tutte paths $\epsilon_1$ from $X'$ to $I$ on $B$ and $\epsilon_2$ from $Z'$ to $I$ on $B'$, both off $\Gamma$. Notice that the Tutte paths $(X', Y)\epsilon_0\epsilon_1^{-1}$ and $(Y, Z')\epsilon_2\epsilon_0^{-1}$ are on $B \cap Y$ and $B'\cap Y$, respectively, which are both type (a) transversals (this is a consequence of the fact that $B$ is a type (b) transversal, as described in the proof of \autoref{theorem: tn-1 poles}). Thus, the closed Tutte paths $(X', Y)\epsilon_0\epsilon_1^{-1}$ and $(Y, Z')\epsilon_2\epsilon_0^{-1}$ are null-homotopic, and we have
 \begin{equation*}
 \begin{split}
 \delta \sim \delta' &= (W, X', Y, Z', W)\\
&\sim (W,X')(X', Y)(Y, Z')(Z', W) \\
&\sim (W,X')\epsilon_1\epsilon_0^{-1}\epsilon_0\epsilon_2^{-1}(Z', W) \\
&\sim (W,X')\epsilon_1\epsilon_2^{-1}(Z', W), 
  \end{split}
  \end{equation*}
which is a closed Tutte path on $A = B \cap B'$, an indecomposable flat of corank $n$, hence null-homotopic. This is a contradiction; therefore $Y \cap I$ is decomposable. 
\end{proof}

\begin{lemma}\label{theorem: W cap T is disconnected} Let $I \in \cS$. Then $W \cap I$ is a decomposable corank 2 flat. 
\end{lemma}

\begin{proof}
We wish to repeat the proof of \autoref{theorem: Y cap T is disconnected} in the new setting by replacing $Y\cap I $ with $W\cap I$. For this, we need to find a type (a) transversal that has the same properties for $W$ as the transversal $A$ has for $Y$. In particular, we need it to not be contained in $W.$

The natural candidate for such a type (a) transversal is $B\cap Y.$ Following the proof of \autoref{theorem: tn-1 poles}, the flat $B \cap Y$ is a type (a) transversal, since $B$ is a type (b) transversal. We repeat the argument in the proof of \autoref{theorem: Y cap T is disconnected} starting with the transversal $A$ by replacing it with transversal $B\cap Y.$ We need type (b) transversals $B_1$ and $B_2$ above $B\cap Y$ having the analogous properties as $B$ and $B'$ in comparison to $A.$

First, let $(B\cap Y)\lor F_i = L_i'$ for $i = 1,2$. We have $L_1'\subsetneq Y$ and $L_1'\subsetneq B\lor F_1 = X'.$ Notice that $X'\supseteq B \supseteq B\cap Y$ and $X' \neq Y$; this holds because $L_1$ is indecomposable, so it cannot have both $W$ and $Y$ above it.
Hence, we may choose $B_1 = B$.    

We observe that $B_2 = B'$ does not work because $B'\not \supseteq B \cap Y$.
Otherwise, we have $B \cap B'= A = B\cap Y$, implying that $A$ is below $W$ and $Y$, a contradiction. Therefore, $(B\cap Y)\lor F_2 = L_2' = U_2 \cap Y$ is an indecomposable corank 2 flat below the two hyperplanes $U_2\in \Gamma$ and $Y \notin \Gamma$. Hence, there exist a hyperplane $Z'' \notin \Gamma$ and a type (b) transversal $B''$ above $B \cap Y$ such that $B'' \lor F_2 = Z''$ and $B'' \lor F_1 \in \Gamma.$ We set $B_2 = B''$ and let $T' = B \lor B''$.

By the submodular inequality and the fact that $B''\not\subseteq Y$, $B'' \lor (I\cap Y)$ is a hyperplane. Since it is a hyperplane above a decomposable corank 2 flat, it is equal to either $I$ or $Y$; however, the latter is impossible because it is not above $B''.$ Hence $B'' \subseteq I$. Combining with the fact that $B\subseteq I$, we get $T' = B \lor B''  \subseteq I.$ Therefore, the flat $I$ is a hyperplane above $T'$, which is not in $\Gamma.$ By similar reasoning as in \autoref{theorem: Y cap T is disconnected}, with the flats $T$ and $Y\cap I$ replaced by $T'$ and $W\cap I$, respectively, we conclude that $W \cap I$ is a decomposable corank 2 flat.
\end{proof}

    Consider an indecomposable flat $G$ above $F(\delta)$ that is contained in a hyperplane of $\cS$ and is either below or above $F_1$, and having minimal corank among flats satisfying such properties. (We can find such a flat, because $F(\delta)$ itself satisfies all of the required properties.)

\begin{lemma}
   The corank of $G$ is $4$. 
\end{lemma}

\begin{proof}
    
 Fix $I \in \cS$ with $G \subseteq I$. First notice that $I$ is not above $F_1$; otherwise, $I \cap Y $ is a decomposable flat above $F_1,$ but we know that the unique such flat is $W\cap Y,$ hence $I = W$ implying that $I \cap W$ is not decomposable. Therefore we must have $F_1 \supseteq G \supseteq F(\delta)$. 
 
 We can bound the corank of $G$ as follows. First, we have $\cork(G)>\cork(F_1) = 3$ and 
\begin{equation*}
    \begin{split}
        \rk (   G \lor T) &\leq \rk (   G)+\rk (   T)-\rk (   G\cap T) \leq \rk (   G)+3,
    \end{split}
\end{equation*}
 which implies that $\cork(G \lor T) \geq \cork(G)-3 >0.$ Therefore, we can pick a hyperplane $N \supseteq G \lor T$; additionally, let $N$ be in $\Gamma$ if this is possible. We claim that $F_1 \lor N = E$; in fact, we have the stronger $F_1 \lor T = E$, because $F_1 \lor T$ contains two distinct hyperplanes $X'$ and $U_1$. 
 Our goal is to show that $G \lor T$ is a hyperplane.
     By \autoref{prop:conn-complement} applied to $[F_1,N,G]$, we get an indecomposable flat $G'$ not in $N$ with $F_1 \supseteq G' \supsetneq G$ and $\rk (   G') = \rk (   G)+1.$ By the submodular inequality and the fact that $N \supseteq G \lor T$ but $N \not\supseteq G'\lor T$, we get $\rk (   G'\lor T) = \rk (   G\lor T)+1.$ We have either $\rk (   G'\lor T) = \rk (   E)$, or, by the definition of $G,$ that all hyperplanes above $G'\lor T$ are in $\Gamma.$ The latter case implies that $G'\lor T \in \Gamma.$ If $N\in \Gamma$, this leads to a contradiction because then $(N,G'\lor T)$ is a modular pair, implying that $G \lor T$ and $I$ are in $\Gamma$. If we could not pick $N \in \Gamma$, then all hyperplanes above $G \lor T$, which includes all hyperplanes above $G' \lor T$, are not in $\Gamma$. This contradicts the assumption that all hyperplanes above $G' \lor T$ are in $\Gamma.$ Thus, we must have $G'\lor T = E$, showing that $\rk (   G\lor T) = \rk (   E)-1$ and that $G\lor T$ is a hyperplane. 
     Notice that \[
\rk (   E)-1 = \rk (   G \lor T) \leq \rk (   G)+\rk (   T)-\rk (   G\cap T) \leq \rk (   G)+3, 
\]
implying that $\cork(G) \leqslant 4.$ Combining with $\cork(G) > 3$, we see that $\cork(G) = 4$. 
\end{proof}

From now on, $I$ does not refer to an arbitrary element of $\mathcal{S}$, but rather we set $I = N = G\lor T$. 

\begin{lemma}\label{lemma:large-good-path}
    If $n +1 \geq 5$ then $\delta$ is null-homotopic.
\end{lemma}
\begin{proof}
    
If $\cork(F(\delta)) = n + 1 \geq 5,$ then $F_2 \not\supseteq G$, since otherwise, $F_1, F_2 \supsetneq G$ and thus $F(\delta) = F_1 \cap F_2 \supseteq G$, but then the coranks do not match. Applying \autoref{prop:conn-complement} to $[F_2, I, F(\delta)]$ (note that $F_2 \lor I \supseteq F_2 \lor T \supseteq U_2 \lor Z' = E$), we get an indecomposable flat $G''$ of rank $\rk (   F(\delta))+1$ such that $F_2\supsetneq G'' \supsetneq F(\delta)$, and such that $G''$ is not below $I$. Let $F_3 = G \lor G''.$ By the submodular inequality and the fact that $G \not\supseteq G'',$ we obtain 
\begin{equation*}
    \begin{split}
     \rk (   G\lor G'') &\leq \rk (   G)+\rk (   G'' )-\rk (   F(\delta)) \\ &= \rk (   E)-4+\rk (   E)-n -(\rk (   E)-(n+1))\\ &= \rk (   E) -3,
    \end{split}
\end{equation*}
and therefore $\rk (   G\lor G'' ) = \rk (   E)-3$, implying that $F_3$ is a corank 3 flat. Notice that $F_3 \subsetneq W \cap Y$ since $ W,$ $ Y \supsetneq F_1 \supsetneq G$ and $ W,$ $ Y \supsetneq F_2 \supsetneq G''.$ Applying \autoref{prop:conn-complement} to $[I, W \cap Y, G]$, there exists an indecomposable corank 2 flat $L$ above $G$ and below $I$ such that $L \lor (W\cap Y) = E.$

Let $i \in \{1,3\}$. We know that $L \cap F_i \supseteq G$. Hence \[
\rk (   L \lor F_i) \leq -\rk (   L \cap F_i)+\rk (   L)+\rk (   F_i)  \leq 2\rk (   E)-5 - \rk (   E)+4 = \rk (   E)-1. 
\]
Since $F_1 \not\subseteq I$ and $G'' \not\subseteq I$, we know that neither $F_1$ nor $F_3$ is contained in $L$. Hence, we may let $L \lor F_i = X_i$ for $i \in \{1,3\}$ which are hyperplanes because of the submodular inequality. 

Notice that neither $X_1$ nor $X_3$ can be equal to $I$. Indeed, if $X_1 = I$ we obtain a contradiction with the definition of $G$, because $F_1 \subsetneq I$ has a smaller corank. And if $X_3 = I$, we get a contradiction because $I \supsetneq F_3 = G \lor G''$ but $G''$ is not contained in $I.$ 
Because $X_3$ is above $L$, and $L$ is contained in neither $W$ nor $Y$, we also get $F_3 = W \cap Y \cap X_3.$ Because $X_1$ is above $F_1$, which is indecomposable, we know by \autoref{theorem: connected corank 3 flats} that $X_1 \cap W$ and $X_1 \cap Y$ are indecomposable corank 2 flats. Therefore $X_1 \cup W \neq E$ and $X_1 \cup Y \neq E.$ Because $W\cap I$ and $Y \cap I$ are decomposable, we get $W\cup I = E$ and $Y \cup I = E.$ Finally, notice that $L = X_3 \cap I = X_1 \cap I \neq \emptyset.$

Let $a \notin X_1 \cup W$. From $W\cup I = E$, we see that $a \in I.$ From $L =X_1 \cap I =  I \cap X_3$, we get $ a\notin L$ from the first equality, and thus $a \notin X_3$ from the second. Finally, $a \notin W \cup X_3$, or in other words, $X_3 \cup W \neq E.$ A similar argument shows that  $X_3 \cup Y \ne E$. Applying~\autoref{prop:connected-intersection} twice, we obtain that $F_3 = W \cap X_3 \cap Y$ is indecomposable. 

To finish off, notice that $W$ and $Y$ are above an indecomposable flat $F_3$. By 
\autoref{thm: Tutte's path theorem}, there exists a Tutte path $\epsilon$ from $Y$ to $W$ on $F_3$ off $\Gamma$. Notice that $G \subseteq F_1$ and $G'' \subseteq F_2$, where both $G$ and $G''$ are indecomposable and have corank at most $n.$ Therefore, we can decompose $\delta$ as 
\[
\delta \sim (W,X,Y)\epsilon\epsilon^{-1}(Y,Z,W) \sim 0,
\]
where the first path $(W,X,Y)\epsilon$ is null-homotopic because it is on $G$, and the second path $\epsilon^{-1}(Y,Z,W)$ is null-homotopic because it is on $G''.$ Hence $\delta$ is null-homotopic, which is a contradiction. 
\end{proof}

From now on, we assume that $n+1 = \cork(F(\delta)) = 4$.

 Our goal is to determine the structure of the sublattice above $F(\delta)$. Eventually we cover the case where the sublattice is equal to the lattice appearing in the definition of an elementary Tutte path of the fourth kind.

 Recall that we are given a type (a) transversal $A$ of corank 3, type (b) transversals $B$ and $B'$ of corank 2, and the flat $T = B\lor B'$. Since $\cork(T) = 4 - 3 = 1,$ $T$ is a hyperplane and $T = I \notin \Gamma$ is the unique hyperplane above $T.$ Note that the corank 2 flats $W \cap Y,$ $ W\cap T$, and $Y \cap T$ are all decomposable (and $T$ is neither $W$ nor $Y$ because it is the join of two type (b) transversals). Hence, by \autoref{theorem: connected corank 3 flats}, we conclude that $W\cap Y \cap T$ is not an indecomposable corank 3 flat. 

\begin{lemma}\label{theorem: corank 3 flet is in one of three}
    Any corank 3 flat $P$ above $F(\delta)$ is below one of the corank 2 flats $W\cap Y,$ $W\cap T$, or $Y\cap T.$
\end{lemma}
\begin{proof}

Flat $P$ covers $F(\delta)$ thus $P = \gen{F(\delta)\cup a}$ for some $a \in E.$ Because $W\cap Y,$ $W\cap T$ and  $Y\cap T$ are all decomposable corank 2 flats, we get that $W \cup Y = W \cup T = Y \cup T = E$. The element $a$ thus lies in at least two out of $W, Y$ and $T$ which is what we wanted to prove.
\end{proof}

\begin{lemma}\label{theorem: corank 2 flat in one of hyperplanes}
    Any corank 2 flat $L$ above $F(\delta)$ is below one of the hyperplanes $W,$ $Y$, or $T.$
\end{lemma}
\begin{proof}
Any corank 2 flat $L$ is above a corank 3 flat $P$ such that $L \supset P \supset F(\delta)$. By \autoref{theorem: corank 3 flet is in one of three}, we see that $P$ is below one of the flats $W\cap Y,$ $W\cap T$, or $Y \cap T.$ If $P$ is contained in $W \cap Y,$ for instance, we find that $L \lor (W\cap Y)$ is either equal to $W\cap Y,$ in which case we are done, or to one of the hyperplanes $W$ or $Y,$ because of the submodular inequality. The same goes for the other two cases.
\end{proof}

\begin{lemma}\label{theorem: two members of gama above 3 -transversal}
    Every type (a) transversal $F$ is below two hyperplanes of $\Gamma.$ Each of the two corank 2 flats $F\lor F_1$ and $F\lor F_2$ is contained in three hyperplanes, one of which is in $\Gamma.$
\end{lemma}
\begin{proof}

Consider an arbitrary type (a) transversal $F$. By \autoref{theorem: corank 3 flet is in one of three} $F$ is below one of the two hyperplanes $W$ or $Y$ and because it is a type (a) transversal, it is under exactly one. Let $Q\in \{W,Y\}$ denote the hyperplane above $F.$

Let $L_i = F\lor F_i$ be the indecomposable corank 2 flats and let all hyperplanes above $L_1$ other than $Q$ be $X_1,\ldots, X_k.$
Observe that none of the flats $X_i$ can be equal to $T$; indeed, if $X_i = T$, then $L_1 \subseteq Q \cap T$. Both $L_1$ and $Q \cap T$ have corank 2 so $L_1 = Q \cap T$, contradicting the indecomposability of $L_1$. 

For each $i = 1, \dots, k$, by \autoref{prop:conn-diamond} applied to $X_i$ and $ F$, there exist indecomposable corank 2 flats $C_i,$ $D_i$ such that $C_i \lor D_i = X_i$ and $X_i \supset C_i,D_i \supset F.$ Notice that, if for any $i$, $C_i$ or $D_i$ is contained in any hyperplane from $\{W,Y\}$ it has to be in $Q$, otherwise $F$ which is below $C_i$ and $D_i$ would be contained in both $W$ and $Y$, contradicting the fact that $F$ is a type (a) transversal. For each $i$, one of $\{C_i,D_i\}$ has to be a type (b) transversal; otherwise, both are contained in $Q$, which would imply $C_i\lor D_i =Q $, which is not $X_i.$ Pick the one that is a type (b) transversal and call it $B_i.$ 

For each $i$ the flat $B_i$ is a corank 2 flat contained in neither $W$ nor $Y$ hence, because we know by \autoref{theorem: corank 2 flat in one of hyperplanes} that $B_i$ is contained in one of $W$, $Y$, or $T,$ we must have $B_i \subsetneq T$ and therefore $B_i = X_i \cap T$.  We define for each $i$ the flat $X_i' = B_i \lor L_2$, remember that one of the poles $X_i,$ $ X_i' $ has to be in $\Gamma.$ The flats $X_i'$ for $i=1,\ldots, k$ are pairwise distinct since $X_i'\cap T = B_i$ and the latter are pairwise distinct.

 If $k \geq 3$, we find that above one of $L_1$ or $L_2$ there must be at least two hyperplanes of $\Gamma.$ We get a contradiction because it would then follow that either $L_1 \in \Gamma$ or $L_2 \in \Gamma$ and therefore that $Q \in \Gamma$. So we must have $k \leq 2$. 
 
 Notice that $k \leq 1$ is impossible, because the flat $L_1$ is indecomposable and the only hyperplanes containing $L_1$ are $X_1, \ldots, X_k$, and $Q$. 
 
 Therefore $k = 2$. Without loss of generality, we may let $X_1,$ $ X_2' \in \Gamma$ and $X_2,$ $  X_1' \notin \Gamma.$ We claim that for any two indecomposable corank 2 flats $L'\neq L''$ between $T$ and $F$, the joins $L' \lor L_1$ and $L'' \lor L_1$ are distinct. This is because if $L_1 \lor L' = L_1 \lor L''$, then $ L' = (L_1 \lor L') \cap T = (L_1 \lor L'')\cap T = L''.$ Therefore, because all hyperplanes above $L_1$ are $X_1,$ $X_2$, and $Q$, we see that the only indecomposable corank 2 flats between $T$ and $F$ are $B_1 =X_1 \cap T$ and $B_2 = X_2 \cap T$ ($Q\cap T$ is decomposable). Because $F$ is indecomposable and contained in a decomposable corank 2 flat $Q\cap T$, we know by \autoref{theorem: connected corank 3 flats} that
each hyperplane $H$ above $F$ other than $Q$ or $T$ is above two indecomposable corank 2 flats: $H \cap T$ and $H\cap Q$. In particular, $H \cap T$ is an indecomposable flat between $F$ and $T$, and thus it is equal to $B_1$ or $B_2$.

If $H$ is a third hyperplane of $\Gamma$ above $F$ other than $X_1$ and $X_2'$, we find that in the case of $H\cap T = B_1$ that $H$ and $X_1$ are above $B_1$ which is a corank 2 flat and hence $T \in \Gamma$, a contradiction. The case $H\cap T = B_2$ also leads to contradiction, because both $H$ and $X_2'$ are above $B_2$ which again implies $T \in \Gamma.$ Thus, each type (a) transversal is below two hyperplanes of $\Gamma.$
\end{proof}

\begin{lemma}\label{theorem: two member of gamma above fi}
    For $i= 1,2$, there are precisely two indecomposable corank 2 flats between $F_i$ and $Y$ and exactly two indecomposable corank 2 flats between $F_i$ and $W$. Each indecomposable corank 2 flat between $F_i$ and $W$ or between $F_i$ and $Y$ lies above a type (a) transversal. The flat $F_i$ is below two hyperplanes of $\Gamma.$
\end{lemma}
\begin{proof}

Let $L$ be an indecomposable corank 2 flat above $F_1.$ Since $L$ and $F(\delta)$ are indecomposable, by \autoref{prop:conn-diamond}, there exist indecomposable flats $K_1$ and $K_2$ between them. One of them is a type (a) transversal, because $L$ is not below both $W$ and $Y.$ Therefore, any indecomposable corank 2 flat $L$ above $F_1$ can be written as $L = F_1 \lor K_1$, where $K_1 $ is a type (a) transversal. By \autoref{theorem: two members of gama above 3 -transversal}, this means that $L$ is contained in exactly three hyperplanes, and one of them is in $\Gamma.$ Also, if $L'$ is a fixed indecomposable corank 2 flat between $W$ and $F_1$ and we have two distinct indecomposable corank 2 flats $L'', L'''$ between $Y$ and $F_1$, we know that $L'' \lor L'$ and $L''' \lor L'$ are distinct hyperplanes above $L'$; if this is not the case, we have $L'' = (L'' \lor L') \cap Y=(L''' \lor L') \cap Y=L'''.$ 

We observe that there are at most two indecomposable corank $2$ flats between $F_1$ and $W$; in fact, there are exactly two by \autoref{prop:conn-diamond} applied to $W$ and $F_1$. To see this, take a fixed indecomposable corank 2 flat $L''$ between $F_1$ and $Y$ and let $G_1,\ldots, G_k$ denote all indecomposable corank 2 flats between $F_1$ and $W.$ Then $G_1\lor L'',\ldots, G_k \lor L''$ are pairwise distinct hyperplanes above $L''$ distinct from $Y$, and we know by \autoref{theorem: two members of gama above 3 -transversal} that there are precisely two of them, so $k\leq 2$. Analogously, there are at most two indecomposable corank 2 flats between $F_1$ and $Y$ (and again, we find that there are exactly 2). 

Suppose, for the sake of contradiction, that there are at least three hyperplanes $H_1,H_2$, and $H_3$ of $\Gamma$ above $F_1.$ Then at least two of them intersect with $W$ in the same indecomposable corank 2 flat below $W.$ Hence $W \in \Gamma$, which is a contradiction. Next, we have to show there are exactly two. We know that there are two indecomposable corank 2 flats $L'$ and $L''$ between $F_1$ and $W.$ Each of the flats $L'$ and $L''$ lies on a type (a) transversal. By \autoref{theorem: two members of gama above 3 -transversal} each of $L'$, $L''$ flats is below one hyperplane of $\Gamma$. If both of them are under the same one, then $W =L'\lor L''$ would be in $\Gamma,$ which is a contradiction.

Therefore, there are exactly two hyperplanes of $\Gamma$ above $F_1.$
\end{proof}

By \autoref{theorem: two member of gamma above fi}, we have in total at least four type (a) transversals. This is because each of the four indecomposable corank 2 flats between $F_1$ and $W$ or between $F_1$ and $Y$ lies above a type (a) transversal. These type (a) transversals are pairwise distinct, because their joins with $F_1$ are distinct indecomposable corank 2 flats; indeed, we have $(A\lor F_1)\cap T = A$ for every type (a) transversal $A$.

If $F_1,$ $F_2$ and these four type (a) transversals constitute all of the indecomposable corank $3$ flats above $F(\delta)$, then we are done:

\begin{lemma}\label{lemma:dual-fano-good-path}
    Assume $F_1,$ $F_2$ and four type (a) transversals corresponding to indecomposable corank 2 flats between $F_1$ and $W$ or between $F_1$ and $Y$ are all of the indecomposable corank $3$ flats above $F(\delta).$ Then $\delta$ is an elementary Tutte path of the fourth kind and hence null-homotopic.
\end{lemma}

\begin{proof}
The notation we use here is from \autoref{remark: tutte version of fourth path}. All of the conditions of an elementary path of the fourth kind are satisfied: $D = F(\delta)$ is a corank 4 flat; the three pairwise intersections of the three hyperplanes $A = W,$ $ B = Y$, and $C = T$ are all decomposable corank 2 flats; there are six indecomposable corank 3 flats above $F(\delta),$ namely $F_1,$ $F_2$, and the four type (a) transversals; $W,$ $Y,$ $T \notin \Gamma,$ but above each indecomposable corank 3 flat there are exactly two hyperplanes of $\Gamma$ by \autoref{theorem: two members of gama above 3 -transversal} and \autoref{theorem: two member of gamma above fi}. The recognition result described in \autoref{remark: tutte version of fourth path} therefore identifies $\delta$ with an elementary path of the fourth kind, so $\delta$ is null-homotopic.
\end{proof}

\begin{proof}[Proof of \autoref{theorem: good paths}]

In light of \autoref{lemma:large-good-path} and \autoref{lemma:dual-fano-good-path}, there is one last case to consider. More precisely, suppose that $\cork(F(\delta)) = 4$ and that there are more than 6 indecomposable corank 3 flats above $F(\delta)$. 
Let $F_3$ denote the indecomposable flat above $F(\delta)$ that is not equal to $F_1, F_2$, or the four type (a) transversals. Assume for the sake of contradiction that $\delta$ is not null-homotopic. 

Because every type (a) transversal is below $T$ (it is below one of the corank 2 flats $ \{W \cap Y,W\cap T,Y \cap T \}$ and not below $W \cap Y$), we know that there are at most two type (a) transversals below $W$ and at most two type (a) transversals below $Y.$ This is because, for any type (a) transversal $T'$,  we have $T' = (T' \lor F_1)\cap T$, and there are at most four indecomposable corank 2 flats $T' \lor F_1$ between $F_1$ and $W$ or between $F_1$ and $Y.$

Therefore, the flat $F_3$ is not a type (a) transversal, hence it has to be below both $W$ and $Y.$ 
In particular, $F_3$ is below $W\cap Y.$ Let $F$ be a type (a) transversal and let $B_1, B_2, X_1, X_2, X_1',X_2'$ be as in the proof of \autoref{theorem: two members of gama above 3 -transversal}. This means that for $i = 1,2$ the flat $B_i $ is a type (b) transversal with $F \subset B_i \subset X_i$. We also have $B_i \subset T$. It holds that $F_3$ is not below $B_i$, otherwise $F_3$ is below $W,Y$ and $T$, which would imply that it is an indecomposable corank 3 flat below 3 decomposable corank 2 flats, which is a contradiction. By the submodular inequality, we thus find that $B_i \lor F_3 = X_i''$ are hyperplanes and neither of them is in $\Gamma.$ 
Indeed, assume for the sake of contradiction that the latter statement is false. Then $B_1 \lor F_3 = X_1$ (the only hyperplane above $B_1$ in $\Gamma$) and we see that $B_1\lor F_1 = B_1 \lor F_3 \supseteq F_1 \lor F_3 = Y \cap W.$ Because $Y\cap W$ is decomposable, we must have $X_1 \in \{W, Y\}$, which is a contradiction. Similar reasoning works in the case of $X_2'',$ namely if $B_2\lor F_3 = X_2'$ then $B_2 \lor F_2 = B_2\lor F_3 \supset F_2 \lor F_3 = Y \cap W$. This implies that $X_2' \in \{W,Y\}$ which is a contradiction. 

Lastly, notice that $\delta' =(W,X,Y,X_1'', W) = (W,X,Y)(Y,X_1'',W)$ is null-homotopic. Assume for contradiction that $\delta'$ is not null-homotopic. To prove this, we first show that $\delta'$ is a special path. In particular, if it is a special path, we can apply \autoref{theorem: tn-1 poles} with $F_3$ replacing $F_2$; the conditions for $\delta'$ being a special path (\autoref{theorem: lemma for homotopy theorem}) are met ($F_1 = W\cap X \cap Y$ and $ F_3 = Y \cap X_1''\cap W$ are indecomposable corank 3 flats and $W\cap Y$ is decomposable of corank 2), $\cork(F(\delta')) = n+1$ and $\delta'$ is assumed to not be null-homotopic, so we are in the setting of \autoref{theorem: QQ'}. But there is a type (b) transversal $B_2$ with the property that neither of its poles $B_2\lor F_1 = X_2$ and $B_2 \lor F_3 = X_2''$ belong to $\Gamma$. This contradicts \autoref{theorem: tn-1 poles}, and hence the path $\delta'$ is null-homotopic.

By analogous reasoning, we find that $\delta''=(Y, Z, W, X_1'', Y) = (Y, Z, W)(W, X_1'', Y)$ is null-homotopic. In particular, using a type (b) transversal $B_1$ with $F_2 = Y\cap Z\cap W$ and $F_3 =W\cap X_1''\cap Y$ we get the poles $X_1'$ and $X_1''$  which are both not in $\Gamma$, contradicting \autoref{theorem: tn-1 poles}. Therefore, \[
\delta = (W, X, Y)(Y, Z, W) \sim (W,X_1'', Y)(Y, X_1'', W) \sim 0
\]
is null-homotopic, which is the final contradiction. 
\end{proof}

\subsection{The final proof}\label{subsection: final proof}

\begin{proof}[Proof of \autoref{theorem: Tutte's homotopy theorem}]
The homotopy theorem is true for any closed Tutte path of corank 1. Assume, for the sake of contradiction, that \autoref{theorem: Tutte's homotopy theorem} is false for a closed Tutte path $\gamma$ in a matroid $M$ with modular cut $\Gamma$; choose $\gamma$ of least corank among those closed Tutte paths that are not null-homotopic. Let the origin of the path $\gamma$ be $X_0$ and $\cork(F(\gamma))=n+1 \geq 2$. By the special lemma, $\gamma$ cannot be a special path. By \autoref{theorem: carrier of path connected}, we know that $F(\gamma)$ is indecomposable. By \autoref{cor:conn-chain}, there is an indecomposable flat $G$ with $X_0 \supseteq G \supsetneq F(\gamma)$ and $\cork(G) = n.$ 

For every closed Tutte path $\gamma'=(X_0,\ldots,X_m, X_0)$ on $F(\gamma)$, 
we define $u(\gamma')$ as the number of indices $j$ such that $X_j \not\supseteq G.$ If $u(\gamma')>0$ and $i$ denotes the smallest index such that $X_i \not\supseteq G$, we define $v(\gamma')= \cork(X_{i-1}\cap X_i \cap X_{i+1})$, where the subscripts are read modulo $m + 1$. Else, we define $v(\gamma') =0$.

We pick a closed Tutte path $\epsilon$ on $F(\gamma)$ with origin $X_0$ such that:
\begin{enumerate}
    \item[(a)]\label{a} $\epsilon\sim \gamma.$
    \item[(b)]\label{b} $u(\epsilon)$ is minimal among all paths satisfying (a).
    \item[(c)]\label{c} $v(\epsilon)$ is minimal among all paths satisfying (b).
\end{enumerate}

We split the proof into cases. In each case, we will derive a contradiction.

\bigskip\noindent\textbf{Case 1}
     Assume $u(\epsilon) = 0.$ Then $\epsilon$ lies on the indecomposable corank $n$ flat $G$. Hence $\epsilon$ is null-homotopic by assumption, implying that $\gamma$ is null-homotopic, which is a contradiction.
     
 \bigskip\noindent\textbf{Case 2}
   Assume $u(\epsilon)>0,$ which implies that $v(\epsilon)>0$. We define $F = X_{i-1}\cap X_i \cap X_{i+1}.$ Since $\epsilon$ is a Tutte path, we know that $X_{i-1}$ and $X_{i}$ are distinct hyperplanes and therefore $\cork(X_{i-1}\cap X_i)\geq 2$, implying that $v(\epsilon) \geq 2.$
   
\bigskip\noindent\textbf{2.1}
Assume $v(\epsilon) = 2.$ 
    
     \bigskip\noindent\textbf{2.1.1}
 If $X_{i-1} = X_{i+1}$, then $(X_{i-1}, X_{i}, X_{i+1})$ is an elementary Tutte path of the first kind, implying that $\epsilon = \epsilon_1(X_{i-1}, X_{i}, X_{i+1})\epsilon_2 \sim \epsilon_1\epsilon_2.$ But $\epsilon_1\epsilon_2$ satisfies condition (a) with $u(\epsilon_1\epsilon_2)\leq u(\epsilon)-1<u(\epsilon)$, which is a contradiction.
 
\bigskip\noindent\textbf{2.1.2}
  If $X_{i-1}\neq X_{i+1}$, then $\delta = (X_{i-1}, X_{i}, X_{i+1}, X_{i-1})$ is an elementary Tutte path of the second kind and \begin{equation*}
        \begin{split} 
        \epsilon &= \epsilon_1 (X_{i-1}, X_{i}, X_{i+1})\epsilon_2 \\ &\sim \epsilon_1(X_{i-1}, X_{i+1})\epsilon_2.
        \end{split}
        \end{equation*}
        But $\epsilon_1(X_{i-1}, X_{i+1})\epsilon_2$ satisfies (a) and $u(\epsilon_1(X_{i-1}, X_{i+1})\epsilon_2) = u(\epsilon)-1<u(\epsilon),$ which is a contradiction.
       
    \bigskip\noindent\textbf{2.2}
 Assume $v(\epsilon) = 3.$ Then the flat $F$ is an indecomposable corank 3 flat by \autoref{theorem: carrier of path connected}, and $G\lor F =L$ is a corank 2 flat because $F$ does not contain $G$. Furthermore, we have
        \begin{equation*}
        \begin{split}
        \rk (   G\lor F) & \leq \rk (   G)+\rk (   F)-\rk (   G\cap F) \\ &=  \rk (   E)-n + \rk (   E)-3-(\rk (   E)-(n+1)) \\  &= \rk (   E)-2.
        \end{split}
        \end{equation*}
         Let $Z = L\lor (X_i \cap X_{i+1})$. Then $Z$ is a hyperplane, because $X_i \cap X_{i+1}$ is not above $G.$ 
         
     \bigskip\noindent\textbf{2.2.1}
            Assume $Z\notin \Gamma.$ If $Z = X_{i+1}$, we define $\delta  = (Z)$; if not, let $\delta  = (Z, X_{i+1})$. Either way, $\delta$ is a Tutte path. Notice that if $L$ is indecomposable, then the path $(X_{i-1},X_i,Z, X_{i-1})$ is elementary. Thus, we have
         \begin{equation*}
        \begin{split} 
        \epsilon &= \epsilon_1 (X_{i-1}, X_{i}, X_{i+1})\epsilon_2 \\ 
        &\sim \epsilon_1 (X_{i-1}, X_{i}, Z, X_{i-1})(X_{i-1}, Z)\delta \epsilon_2 \\
        &\sim \epsilon_1(X_{i-1}, Z)\delta \epsilon_2.
        \end{split}
        \end{equation*}
        But the path $\epsilon_1(X_{i-1}, Z)\delta \epsilon_2$ satisfies (a) and $u(\epsilon_1(X_{i-1}, Z)\delta \epsilon_2) \leq u(\epsilon)-1,$ which is a contradiction.
        
        If $L\supset G$ is not indecomposable, it follows from \autoref{theorem: disconnected corank2 on connected corank 3} that there is an indecomposable corank 3 flat $F'$ such that $L \supsetneq F'\supseteq G.$ We know by \autoref{theorem: connected corank 3 flats} that there is an indecomposable corank 2 flat $L'$ between $F'$ and $X_{i-1}$; note that the flat $X_{i-1}$ is above at least two corank 2 flats above $F'$, and $L$ is the unique decomposable corank 2 flat above $F'$ by \autoref{theorem: connected corank 3 flats}. There is a hyperplane $T \notin \Gamma$ above $L'$ that is not equal to $X_{i-1}$ because $ X_{i-1}\notin \Gamma.$ We also know from \autoref{theorem: connected corank 3 flats} that $X_{i-1}\cap T$ and $Z \cap T$ are indecomposable corank 2 flats. But then $(X_{i-1}, X_{i}, Z, T, X_{i-1})$ is a special path, and $X_{i-1}\cap X_i\cap Z = F$, which is distinct from $X_{i-1}\cap Z\cap T=F'$. If the corank of this special path is at most $n$, then it is null-homotopic by the inductive hypothesis, and if the corank is $n+1$, then it is null-homotopic by \autoref{theorem: good paths}. Together with the elementary Tutte path $(X_i, Z, X_{i+1}, X_i)$ when $Z \ne X_{i+1}$,
        we thus have
        \begin{equation*}
    \begin{split} 
        \epsilon &= \epsilon_1 (X_{i-1}, X_{i}, X_{i+1})\epsilon_2 \\ 
         & \sim \epsilon_1 (X_{i-1}, X_{i}, Z)\delta\epsilon_2 \\
        &\sim \epsilon_1(X_{i-1},T, Z)\delta\epsilon_2 \\
        & \sim \epsilon_3\delta\epsilon_2,
        \end{split}
        \end{equation*}
        where $\epsilon_3$ is a Tutte path on $G$ since $T , Z \supseteq G.$ Therefore the path $ \epsilon_3 \delta \epsilon_2$ satisfies (a) and $u(\epsilon_3 \delta \epsilon_2) \leq u(\epsilon)-1,$ which is a contradiction.
        
\bigskip\noindent\textbf{2.2.2}
         Assume $Z \in \Gamma.$
        By \autoref{prop:conn-diamond}, there exists an indecomposable corank 2 flat $L'$ between $X_{i+1}$ and $F$ other than $X_i\cap X_{i+1}$. If $L'$ is below $X_{i-1},$ then $(X_{i-1}, X_{i}, X_{i+1}, X_{i-1})$ is an elementary path of the second kind and \begin{equation*}
        \begin{split}
        \epsilon  &= \epsilon_1 (X_{i-1}, X_{i}, X_{i+1})\epsilon_2 \\ 
        &\sim \epsilon_1(X_{i-1}, X_{i+1})\epsilon_2,
        \end{split}
        \end{equation*}
        meaning that $\epsilon_1(X_{i-1}, X_{i+1})\epsilon_2$ satisfies condition (a) but $u(\epsilon_1(X_{i-1}, X_{i+1})\epsilon_2) = u(\epsilon)-1<u(\epsilon)$, which is a contradiction. Therefore, $(X_i\cap X_{i-1}) \lor L'  = U$ and $L\lor L' = V$ are distinct hyperplanes  not equal to any of $X_{i-1},X_i$, or $X_{i+1}.$ Notice that $V \notin \Gamma$ because $Z  \supsetneq L$ is in $\Gamma$ but $X_{i-1}\supsetneq L$ is not in $\Gamma.$ 

First assume that $U \notin \Gamma.$ Because $(X_{i-1}, U, X_i, X_{i-1})$, $(X_{i}, U, X_{i+1}, X_i)$, $(U, V, X_{i-1}, U)$, and $(U, V, X_{i+1}, U)$ are elementary Tutte paths, we have
        \begin{equation*}
            \begin{split}
            \epsilon  &= \epsilon_1 (X_{i-1}, X_{i}, X_{i+1})\epsilon_2 \\
               &\sim \epsilon_1(X_{i-1},U, X_i , U, X_{i+1})\epsilon_2 \\
               & \sim \epsilon_1(X_{i-1}, U, X_{i+1})\epsilon_2 \\
               & \sim \epsilon_1(X_{i-1},V, U, X_{i+1})\epsilon_2 \\
                & \sim \epsilon_1(X_{i-1},V,X_{i+1})\epsilon_2.
            \end{split}
        \end{equation*}
            But $V \supsetneq G$. Hence $\epsilon_1(X_{i-1},V,X_{i+1})\epsilon_2$ satisfies (a) and \[u(\epsilon_1(X_{i-1},V,X_{i+1})\epsilon_2) = u(\epsilon)-1<u(\epsilon),
            \] which is a contradiction.
    
            Assume that $U \in \Gamma.$ Notice that if all indecomposable corank 2 flats above $F$ are below either $U$ or $Z$, then the path $(X_{i-1}, X_i, X_{i+1}, V, X_{i-1})$ is an elementary path of the third kind.
            Therefore \begin{equation*}
                \begin{split}
                      \epsilon  &= \epsilon_1 (X_{i-1}, X_{i}, X_{i+1})\epsilon_2 \\
                      & \sim \epsilon_1 (X_{i-1}, V, X_{i+1})\epsilon_2, 
                \end{split}
            \end{equation*}
            which again means that $\epsilon_1 (X_{i-1}, V, X_{i+1})\epsilon_2$ satisfies (a) with \[u(\epsilon_1 (X_{i-1}, V, X_{i+1})\epsilon_2)<u(\epsilon),\] a contradiction. Thus, we may assume without loss of generality that there exists another indecomposable corank 2 flat $L''$ above $F$ which is below neither $U$ nor $Z.$ 
            
           \bigskip\noindent\textbf{2.2.2.1} First, assume that $X_{i+1}\supsetneq L''.$ We can then repeat the argument following the definition of  $L'$ 
           in 2.2.2 with $L''$ replacing $L'$ 
           and we find that $L'' \lor (X_{i-1}\cap X_i) \neq U$. Hence $L'' \lor (X_{i-1}\cap X_i)$ 
           cannot be in $\Gamma$, since $U$ is above $X_i \cap X_{i-1}$ and in $\Gamma.$ Repeating the argument for the case $U \notin \Gamma$ with $L''$ instead of $L'$ leads us to a contradiction.

\bigskip\noindent\textbf{2.2.2.2} Second, assume that $X_i \supsetneq L''$. Then $L''\lor L = W_1$ is a hyperplane above $G$ by the submodular inequality, and $W_1$ is not equal to  $X_{i-1}$ or $Z$ because $X_{i-1}\cap X_i \neq L''$ and $ Z\cap X_i \neq L''.$ Notice that $W_1 \notin \Gamma$, because the hyperplanes $W_1$ and $Z$ intersect at $L$, which is not in $\Gamma$. We then deduce, because $(X_{i-1}, W_1, X_i, X_{i-1})$ is an elementary Tutte path, that
\begin{equation*}
    \begin{split}
          \epsilon  &= \epsilon_1 (X_{i-1}, X_{i}, X_{i+1})\epsilon_2 \\
          & \sim \epsilon_1 (X_{i-1}, W_1, X_i, X_{i+1})\epsilon_2 \\
          & = \epsilon'.
    \end{split}
\end{equation*}
            Now observe that $\epsilon' = \epsilon_1 (X_{i-1}, W_1, X_i, X_{i+1})\epsilon_2$ satisfies (a) and (b), because $u(\epsilon) = u(\epsilon')$ and (c) because $$v(\epsilon_1 (X_{i-1}, W_1, X_i, X_{i+1})\epsilon_2) = \cork( W_1 \cap X_i \cap X_{i+1}) = \cork(L''\cap X_i\cap X_{i+1})= \cork(F) = 3.$$ 
We claim that we can deduce a contradiction for $\epsilon'$ as in either 2.2.1 or the ``$U \notin \Gamma$'' case of 2.2.2. In fact, let $F^\circ = W_1 \cap X_i \cap X_{i+1}$, $L^\circ = G \vee F^\circ$, and $Z^\circ = L^\circ \lor (X_i \cap X_{i+1})$. The case in which $Z^\circ \notin \Gamma$ is just 2.2.1. If $Z^\circ \in \Gamma$, then we must have $U^\circ = (X_i \cap W_1) \vee L' = L'' \lor L' \notin \Gamma$. Suppose for the sake of contradiction that $U^\circ \in \Gamma$, then $U^\circ \cap U = L' \in \Gamma$, implying that $X_{i+1} \in \Gamma$, a contradiction. Therefore, 2.2.2.2 can be reduced to the previous cases. 

\bigskip\noindent\textbf{2.2.2.3} Third, assume that $L'' \subsetneq X_{i-1}$. Consider the hyperplane $L'' \lor (X_i\cap X_{i+1}) = W_2$. $W_2$ does not contain $G$; otherwise, $W_2$ contains both $L = G \vee F$ and $X_i \cap X_{i+1}$, implying that $W_2 = Z$. We also know that $W_2 \notin \Gamma$, since $W_2$, $Z$, and $X_i$ intersect at a common flat of corank 2. Since $(X_{i-1}, X_i, W_2, X_{i-1})$ and $(X_i, W_2, X_{i+1}, X_i)$ are elementary Tutte paths, we find that
            \begin{equation*}
                \begin{split}
                      \epsilon &= \epsilon_1 (X_{i-1}, X_{i}, X_{i+1})\epsilon_2 \\
                      &\sim \epsilon_1 (X_{i-1}, W_2, X_i, X_{i+1})\epsilon_2 \\
                      & \sim \epsilon_1 (X_{i-1}, W_2, X_{i+1})\epsilon_2 \\
                      & = \epsilon'.
                \end{split}
            \end{equation*}
            Since $W_2$ does not contain $G$, $\epsilon'$ satisfies (a), (b), and (c). Similar to 2.2.2.2,  we have $(W_2 \cap X_{i-1}) \lor L' = L'' \lor L' \notin \Gamma$. Therefore, we can reduce 2.2.2.3 to the previous cases.

            \bigskip\noindent\textbf{2.2.2.4} We assume that $L''$ is contained in none of $X_{i-1}$, $X_i$, and $X_{i+1}$. 
            We claim that the hyperplanes $L''\lor L = W_1$, $L'' \lor (X_i \cap X_{i+1}) =W_2$, and $L'' \lor (    X_{i-1}\cap X_i) = W_3$ are pairwise distinct. For instance, assume for the sake of contradiction that $W_1 = W_2$; then $W_1$ is above $X_{i+1} \cap X_i$ and $L,$ and thus above $(X_{i+1} \cap X_i)\lor L$, which is $Z,$ but $L''$ is not below $Z$. The cases $W_1 = W_3$ and $W_2 = W_3$ can be similarly disproved. Moreover, all three hyperplanes $W_1$, $W_2$, and $W_3$ are not in $\Gamma$, because each of them lies above some corank 2 flat which is the intersection of some hyperplane not in $\Gamma$ and one of $Z$ and $U$. For instance, $W_1$, $Z$, and $X_{i-1}$ intersect at $L$. 
            
            Note that $(X_{i-1}, W_3, X_i, X_{i-1})$, $(X_{i+1}, W_2, X_i, X_{i+1})$, $(X_i, W_2, W_3, X_{i})$, $(W_1, W_2, W_3, W_1)$, and $(X_{i-1}, W_1, W_3, X_{i-1})$ are all elementary Tutte paths. Consider
            \begin{equation*}
                \begin{split}
                    \epsilon &= 
                    \epsilon_1(X_{i-1},X_i, X_{i+1})\epsilon_2 \\
                    & \sim \epsilon_1(X_{i-1} ,W_3, X_i, W_2, X_{i+1})\epsilon_2 \\
                    &\sim \epsilon_1(X_{i-1}, W_3, W_2, X_{i+1}) \epsilon_2 \\
                    &\sim \epsilon_1(X_{i-1}, W_3, W_1, W_2, X_{i+1}) \epsilon_2 \\
                   & \sim \epsilon_1(X_{i-1}, W_1, W_2, X_{i+1}) \epsilon_2 \\
                    & =\epsilon'.
                \end{split}
            \end{equation*}
           The Tutte path $\epsilon'$ satisfies (a), (b) (this is because $W_1 \supseteq L \supseteq G$, but $W_2$ does not contain $G$), and (c) with $v(\epsilon') = 3$. Similar to 2.2.2.2,  we have $(W_1 \cap W_2) \lor L' = L'' \lor L' \notin \Gamma$. Therefore, we can reduce 2.2.2.4 to the previous cases. 

    \bigskip\noindent\textbf{2.3}
 Assume $v(\epsilon)>3.$ Because $X_{i-1}\cap X_i$ and $F = X_{i-1}\cap X_i\cap X_{i+1}$ are indecomposable, there exists an indecomposable corank 3 flat $K$ with $X_{i-1}\cap X_i\supsetneq K \supsetneq F$ by \autoref{prop:conn-path}. The flat $K \lor G = L$ is of corank $2$ because of the submodular inequality. Notice that $X_{i-1}\supseteq L$, since $X_{i-1}\supseteq G$ and $X_{i-1}\supseteq K$. Pick a hyperplane $T$ above $L$ that is not equal to $X_{i-1}$ and, if possible, pick $T$ that is in $\Gamma.$

    Note that $(X_i \cap X_{i+1}) \lor T = E$. By \autoref{prop:conn-complement} applied to $[X_i\cap X_{i+1}, T, F]$, we get an indecomposable corank $v(\epsilon)-1$ flat $F'$ such that $X_i\cap X_{i+1}\supseteq F'\supsetneq F$ and $F' \lor T = E.$ Observe that $X_{i-1}\not\supseteq F'$, because otherwise $F = X_{i-1} \cap X_i\cap X_{i+1} \supseteq F'$, which is a contradiction. Therefore, $F'\lor L = T'$ is a hyperplane not equal to $T$ nor $X_{i-1}.$ Note that $L$ is indecomposable, since it is contained in three different hyperplanes $T'$, $T$, and $X_{i-1}$. Additionally, $T' \notin \Gamma$ holds, because if we could pick $T\in \Gamma$ then it is the only hyperplane above $L$ in $\Gamma$ as $X_{i-1}\notin\Gamma, $ and if we could not, then there are no members of $\Gamma$ which are above $L.$

    By the submodular inequality, we get that $K\lor F'= L'$ is a corank 2 flat. Notice that $T' \supsetneq L'$, because $T' \supsetneq F'$ and $T'\supsetneq L \supsetneq K$. Furthermore, $X_i\supsetneq L'$, because $X_i\cap X_{i+1}  \supsetneq F'$ and $X_i\cap X_{i-1}  \supsetneq K.$

    Assume that $L'$ is indecomposable. Then because $(X_{i-1}, T', X_i, X_{i-1})$ is an elementary Tutte path, we have 
\begin{equation*}
    \begin{split}
        \epsilon & = \epsilon_1(X_{i-1}, X_i, X_{i+1})\epsilon_2 \\
        & \sim \epsilon_1(X_{i-1}, T', X_i, X_{i+1})\epsilon_2 \\
        & = \epsilon'.
    \end{split}
\end{equation*}
We have $T'\supsetneq G$. Hence $\epsilon'$ satisfies (a) and (b) with $u(\epsilon') = u(\epsilon)$. However, in (c) we obtain 
\[
v(\epsilon') = \cork (T' \cap X_i \cap X_{i+1}) \leq \cork (F') = v(\epsilon) - 1, 
\]
which is a contradiction.

Finally, we assume that $L'$ is decomposable. Using \autoref{theorem: disconnected corank2 on connected corank 3}, we get an indecomposable corank 3 flat $K'$ with $L' \supsetneq K' \supseteq F'$, since $F'$ is indecomposable. By the submodular inequality, and because $K' \subseteq X_i$, we see that $K' \lor G = L'' $ is a corank 2 flat, which is indecomposable by \autoref{theorem: connected corank 3 flats} (since $L'$ is the unique decomposable corank 2 flat above $K'$). Therefore, we can pick a hyperplane $U$ on $L''$ which is not equal to $T'$ and $U\notin \Gamma$ because $L''$ is indecomposable and $T'\notin \Gamma$. 

Observe that $(T', U, X_i, X_{i-1}, T')$ is a special path, because $T'\cap X_i = L'$ is decomposable and $T'\cap U\cap X_i =K'$, $X_i\cap X_{i-1}\cap T' = K$ are distinct indecomposable corank 3 flats. If the corank of this special path is at most $n$, then it is null-homotopic by the inductive hypothesis, and if the corank is $n+1$, then it is null-homotopic by \autoref{theorem: good paths}. 
Now we have
\begin{equation*}
    \begin{split}
        \epsilon & = \epsilon_1(X_{i-1}, X_i, X_{i+1})\epsilon_2 \\
        & \sim \epsilon_1(X_{i-1}, T', U, X_i, X_{i+1})\epsilon_2 \\
        & = \epsilon'.
    \end{split}
\end{equation*}
Since both $T'$ and $U$ contain $G$, we know that $\epsilon'$ satisfies (a) and (b) with 
\[
v(\epsilon')= \cork(U \cap X_i \cap X_{i+1}) \leq \cork(F') = v(\epsilon)-1<v(\epsilon),
\]
which is a contradiction. This concludes the proof of the homotopy theorem.
\end{proof}

\section{Tutte's thought process, in his own words}

At this point, the reader may be wondering how on earth Tutte came up with the homotopy theorem and its remarkably intricate proof in the first place. In this final section, we quote Tutte's own writing on the subject from \cite{Tutte71}\footnote{There is a PDF of this reference available online, but it seems that Section 6.4 appears only in the print version of the book, which is why we've chosen to excerpt it more or less in full here.}.

In the preface to \cite{Tutte71}, Tutte writes,

\begin{quote}
Chapter 6 is supplementary. It is meant to give very short descriptions of some parts of matroid theory that are not dealt with in the other five chapters. In particular it is concerned with the `homotopy theorem' and the characterization of regular and graphic matroids. The author has been informed that his treatment of these matters in his papers on matroids is exceptionally obscure. He hopes that a perusal of Chapter 6 may make it easier to read the detailed proofs.
\end{quote}

\medskip

Later, in \cite[Section 6.4]{Tutte71}, he continues:

\begin{quote}
We suppose given a matroid $M$
and a linear subclass
$C$ of $M$. We study re-entrant paths off $C$. Suppose we have two such paths,
$P$ and $Q$, of the following forms:
\[
\begin{aligned}
P &= (X_1, X_2, \ldots, X_n, X_1), \\
Q &= (X_i, Y_1, Y_2, \ldots, Y_k, X_i),
\end{aligned}
\]
where $1 \le i \le n$. Then another re-entrant path off $C$ is
\[
R = (X_1, X_2, \ldots, X_{i-1}, X_i, Y_1, Y_2, \ldots, Y_k, X_i, X_{i+1}, \ldots, X_n, X_1).
\]

We say that $P$ is \emph{deformed} into $R$ by the adjunction $Q$, or that
$R$ is \emph{deformed into $P$} by the deletion of $Q$. 
In homotopy theory we specify a class $U$
of re-entrant paths off $C$ called ``elementary''. Two paths $P$ and $R$ are
then said to be \emph{homotopic} if $R$ can be transformed into $P$ by a finite
sequence of operations, each of which adjoins or deletes an elementary path.

The problem of homotopy first arose in the following form: can we choose $U$
in some simple way so as to make all re-entrant paths off $C$ null-homotopic?

In choosing $U$ it seemed natural to include all paths off $C$ of the following
forms: $(X,Y,X)$ on a line and $(X,Y,Z,X)$ on a plane\footnote{In Tutte's original account, paths consist of circuits of a matroid $M$, which correspond to hyperplanes of the dual matroid $M^\ast$. Therefore, the lattice that Tutte considered is actually the lattice of unions of circuits of $M$, which by \cite[Propositions 1.7.8 and 2.1.6]{Oxley11} 
is the opposite lattice of the lattice of flats of $M^\ast$. If $F$ is a flat of corank $r$ in $M^\ast$, then the corresponding union of circuits $C = E - F$ in $M$ is by definition of \emph{dimension} $d = r - 1$. \emph{Lines} are unions of circuits that are of dimension $1$. \emph{Planes} are unions of circuits that are of dimension $2$.} . These are the
\emph{elementary paths} of the first and second kinds, respectively.
Attention was then drawn to paths, off $C$, on a plane $P$ of the form
\[
(X, Y, Z, T, X),
\]
where $X, Y, Z$ and $T$ are distinct, there are two distinct points $A$ and $B$
on $P$ such that each connected line on $P$ is on either $A$ or $B$,
$X \cup Y$ and $Z \cup T$ are lines on $A$, and $Y \cup Z$ and $T \cup X$ are lines on $B$.
It was found to be impossible to transform such a path into the null path by adjoining and deleting elementary paths of the
first and second kinds. Such paths were therefore included in $U$ as \emph{elementary paths of the third kind}.

An attempt was next made to show that a re-entrant path off $C$, confined to a
flat\footnote{In the spirit of the previous footnote, when Tutte talks about a \emph{flat} of $M$, he is referring to a union of circuits of $M$.} of $M$ of dimension $d$ (that is, rank $d+1$), could always be deformed into
a path in a flat of lower dimension by adjoining and deleting elementary
paths of the first, second and third kinds. This attempt succeeded only
partially. It was found that the operation is possible for all $d \geq 0$ if it is possible for $d=3$. But a close
investigation of the three-dimensional case disclosed a class of paths not
deformable into the null path by adjoining and deleting elementary paths
already recognized. These paths were therefore included in $U$ as elementary
paths of the fourth kind. It could then be shown that, with respect to $U$, all re-entrant
paths off $C$ were null-homotopic. This is the result that we have referred to
as the Homotopy Theorem.
\end{quote}


\begin{small}
 \bibliographystyle{plain}
 \bibliography{matroid}
\end{small}


\end{document}